\documentclass[12pt]{report}

\usepackage{tamuconfig}

\usepackage{times}
\usepackage[T1]{fontenc}

\usepackage{graphicx}

\DeclareGraphicsExtensions{.png}

\graphicspath{ {./graphic/} }

\usepackage[hidelinks]{hyperref}

\makeatletter
\g@addto@macro\@floatboxreset\centering
\makeatother

\usepackage{etoolbox}
\makeatletter
\patchcmd{\ttlh@hang}{\parindent\z@}{\parindent\z@\leavevmode}{}{}
\patchcmd{\ttlh@hang}{\noindent}{}{}{}
\makeatother
\usepackage{amsmath}
\usepackage{amssymb}
\usepackage{amsthm}
\usepackage{latexsym}
\usepackage{mathtools}
\usepackage{amsfonts}
\usepackage{wrapfig}
\usepackage{multicol}
\usepackage{comment}
\usepackage{mathrsfs}
\usepackage{textcomp}
\usepackage{diagbox}
\usepackage[T1]{fontenc}
\usepackage{graphicx}
\usepackage{setspace}
\usepackage{nicefrac}
\usepackage{indentfirst}
\usepackage{wasysym}
\usepackage{upgreek}
\usepackage{booktabs}
\usepackage{etoolbox}
\usepackage{wrapfig}
\usepackage{floatflt}
\usepackage{lscape}
\usepackage{mathrsfs}
\usepackage{textcomp}
\usepackage[T1]{fontenc}
\usepackage{graphicx}
\usepackage{setspace}
\usepackage{nicefrac}
\usepackage{indentfirst}
\usepackage{enumerate}
\usepackage{wasysym}
\usepackage{upgreek}
\usepackage{hyperref}
\usepackage{booktabs}
\usepackage{etoolbox}
\usepackage[labelfont=bf, font={small}]{caption}[2005/07/16]
\usepackage[all,cmtip]{xy}
\usepackage{hyphenat}
\usepackage{color}

\usepackage{tikz-cd}
\usepackage{tikz}
\usetikzlibrary{tikzmark,decorations.pathreplacing}
\usetikzlibrary{automata}
\usetikzlibrary{arrows,matrix}
\usetikzlibrary{positioning}
\usetikzlibrary{decorations}
\usepackage{enumitem}
\begin{document}

\renewcommand{\tamumanuscripttitle}{Newton polytopes and numerical algebraic geometry}

\renewcommand{\tamupapertype}{Dissertation}

\renewcommand{\tamufullname}{Taylor Christian Brysiewicz}

\renewcommand{\tamudegree}{Doctor of philosophy}
\renewcommand{\tamuchairone}{Frank Sottile}

\renewcommand{\tamumemberone}{Laura Matusevich}
\newcommand{\tamumembertwo}{Andrea Bonito}
\newcommand{\tamumemberthree}{Christopher Menzel}
\renewcommand{\tamudepthead}{Sarah Witherspoon}

\renewcommand{\tamugradmonth}{May}
\renewcommand{\tamugradyear}{2020}
\renewcommand{\tamudepartment}{Mathematics}

%
%
%
%


\providecommand{\tabularnewline}{\\}

\begin{titlepage}
\begin{center}
\MakeUppercase{\tamumanuscripttitle}
\vspace{4em}

A \tamupapertype

by

\MakeUppercase{\tamufullname}

\vspace{4em}

\begin{singlespace}

Submitted to the Office of Graduate and Professional Studies of \\
Texas A\&M University \\

in partial fulfillment of the requirements for the degree of \\
\end{singlespace}

\MakeUppercase{\tamudegree}
\par\end{center}
\vspace{2em}
\begin{singlespace}
\begin{tabular}{ll}
 & \tabularnewline
& \cr
Chair of Committee, & \tamuchairone\tabularnewline
Committee Members, & \tamumemberone\tabularnewline
 & \tamumembertwo\tabularnewline
 & \tamumemberthree\tabularnewline
Head of Department, & \tamudepthead\tabularnewline

\end{tabular}
\end{singlespace}
\vspace{3em}

\begin{center}
\tamugradmonth \hspace{2pt} \tamugradyear

\vspace{3em}

Major Subject: \tamudepartment \par
\vspace{3em}
Copyright \tamugradyear \hspace{.5em}\tamufullname 
\par\end{center}
\end{titlepage}
\pagebreak{}

%
%
%
%

\chapter*{ABSTRACT}

\addcontentsline{toc}{chapter}{ABSTRACT} 

\pagestyle{plain} 
\pagenumbering{roman} 
\setcounter{page}{2}

\indent 

We develop a collection of numerical algorithms which connect ideas from polyhedral geometry and algebraic geometry. The first algorithm we develop functions as a numerical oracle for the Newton polytope of a hypersurface and is based on ideas of Hauenstein and Sottile. Additionally, we construct a numerical tropical membership algorithm which uses the former algorithm as a subroutine. Based on recent results of Esterov, we give an algorithm which recursively solves a sparse polynomial system when the support of that system is either lacunary or triangular.  
Prior to explaining these results, we give necessary background on polytopes, algebraic geometry, monodromy groups of branched covers, and numerical algebraic geometry.

\pagebreak{}

%
%
%
%

\chapter*{DEDICATION}
\addcontentsline{toc}{chapter}{DEDICATION}  

\begin{center}
\vspace*{\fill}
To my father
\vspace*{\fill}
\end{center}

\pagebreak{}

%
%
%
%

\chapter*{ACKNOWLEDGMENTS}
\addcontentsline{toc}{chapter}{ACKNOWLEDGMENTS}  

\indent 

I offer my deepest gratitude to my family, friends, teachers, and mentors whom have helped me along my journey. Without their love and support, none of this work would have begun.

To my parents, Kathleen and Robert Brysiewicz, for their unconditional love.
To Nicholas, Bobbi, Shelby, Alexandra, and Tobias, for constantly supporting me, always picking up the phone, and generally mapping out corners of the world before I need to.

To Alex, Hank, Fulvio, and all of the other friends and colleagues I have met throughout graduate school, for the countless conversations about mathematics. 
To my best friend, Jamie, for her unwavering love and reminders to breathe. 

To my teachers, for believing in me, in particular, Alexandra Brysiewicz, Ginger Benning, Karen Yerly, Sue Samonds, Joan Kustak, Andrew Wang, Deepak Naidu, Michael Geline, and Seth Dutter.
To Peter Howard and Monique Stewart for helping me navigate graduate school.
To Laura Matusevich, Andrea Bonito, and Christopher Menzel for serving on my committee.
To Jonathan Hauenstein, Michael Burr, Christopher O'Neill, Timo de Wolff, Laura Matusevich, Anton Leykin, Bernd Sturmfels, and Cynthia Vinzant, for their general mentorship. 

Finally, to Frank (Sottile), whose commitment, thoughtfulness, and mathematical insight are more than I could ask for in an advisor. He has helped me grow in every facet of what it means to be a professional mathematician, and for that, I am forever grateful.

\pagebreak{}
%
%
%
%

\chapter*{CONTRIBUTORS AND FUNDING SOURCES}
\addcontentsline{toc}{chapter}{CONTRIBUTORS AND FUNDING SOURCES}  

\subsection*{Contributors}
This work was supported by a dissertation committee consisting of Professor Frank Sottile [advisor], Professor Laura Matusevich, and Professor Andrea Bonito of the Department of Mathematics and Professor Christopher Menzel of the Department of Philosophy.
The material in Section $8$ is joint work with Jose Rodriguez, Frank Sottile, and Thomas Yahl.

All other work conducted for the dissertation was completed by the student independently.
\subsection*{Funding Sources}
Graduate study was supported by a graduate fellowship from Texas A\&M University. The material in Section $7$ was supported by NSF grant DMS-1501370 and completed during the ICERM-2018 semester on nonlinear algebra. 
\pagebreak{}

%
%
%
%

\phantomsection
\addcontentsline{toc}{chapter}{TABLE OF CONTENTS}  

\begin{singlespace}
\renewcommand\contentsname{\normalfont} {\centerline{TABLE OF CONTENTS}}

\setcounter{tocdepth}{4} 

\setlength{\cftaftertoctitleskip}{1em}
\renewcommand{\cftaftertoctitle}{%
\hfill{\normalfont {Page}\par}}

\tableofcontents

\end{singlespace}

\pagebreak{}


\phantomsection
\addcontentsline{toc}{chapter}{LIST OF FIGURES}  

\renewcommand{\cftloftitlefont}{\center\normalfont\MakeUppercase}

\setlength{\cftbeforeloftitleskip}{-12pt} 
\renewcommand{\cftafterloftitleskip}{12pt}

\renewcommand{\cftafterloftitle}{%
\\[4em]\mbox{}\hspace{2pt}FIGURE\hfill{\normalfont Page}\vskip\baselineskip}

\begingroup

\begin{center}
\begin{singlespace}
\setlength{\cftbeforechapskip}{0.4cm}
\setlength{\cftbeforesecskip}{0.30cm}
\setlength{\cftbeforesubsecskip}{0.30cm}
\setlength{\cftbeforefigskip}{0.4cm}
\setlength{\cftbeforetabskip}{0.4cm}



\listoffigures

\end{singlespace}
\end{center}

\pagebreak{}

%
\phantomsection
\addcontentsline{toc}{chapter}{LIST OF TABLES}  

\renewcommand{\cftlottitlefont}{\center\normalfont\MakeUppercase}

\setlength{\cftbeforelottitleskip}{-12pt} 

\renewcommand{\cftafterlottitleskip}{1pt}

\renewcommand{\cftafterlottitle}{%
\\[4em]\mbox{}\hspace{2pt}TABLE\hfill{\normalfont Page}\vskip\baselineskip}

\begin{center}
\begin{singlespace}

\setlength{\cftbeforechapskip}{0.4cm}
\setlength{\cftbeforesecskip}{0.30cm}
\setlength{\cftbeforesubsecskip}{0.30cm}
\setlength{\cftbeforefigskip}{0.4cm}
\setlength{\cftbeforetabskip}{0.4cm}

\listoftables 

\end{singlespace}
\end{center}
\endgroup
\pagebreak{}  

\pagestyle{plain} 
\pagenumbering{arabic} 
\setcounter{page}{1}

\chapter{INTRODUCTION \label{cha:introduction}}
Understanding the solution sets of polynomial systems,
\begin{equation}
\label{eq:polynomialsystem}
f_1(x_1,\ldots,x_n)=f_2(x_1,\ldots,x_n)=\cdots=f_k(x_1,\ldots,x_n)=0,
\end{equation} is a ubiquitous need throughout mathematics, as well as the primary goal of algebraic geometry. 
Such solution sets, $$\V(f_1,\ldots,f_k)=\{(a_1,\ldots,a_n) \in \C^n \mid f_i(a_1,\ldots,a_n)=0 \text{ for } i=1,\ldots,k\},$$ are called varieties. One way to study varieties is to partition them into families with respect to some structure so that most varieties in the same family have the same properties. Those that do not exhibit these generic properties may still be understood through the role they play in their family.  In this dissertation, we study families of varieties delineated via the monomials appearing in their defining polynomials. 

The support of a polynomial,
$$f(x_1,\ldots,x_n) = \sum_{\alpha=(\alpha_1,\ldots,\alpha_n)\in \Z^n} c_\alpha x_1^{\alpha_1}\cdots x_n^{\alpha_n}, \quad c_\alpha \in \C,$$ is the set $\supp(f)=\{\alpha \in \Z^n \mid c_\alpha \neq 0\}$. Studying a polynomial system $F=(f_1,\ldots,f_k)$ through its support $\Adot = (\supp(f_1),\ldots,\supp(f_k))$  endows it with the structure of a sparse polynomial system and identifies $F$ as a point in the coefficient space $\C^{\Adot}$. Sparse polynomial systems belonging to the same family share a striking number of properties, many depending only on the collection $P_\bullet$ of convex hulls of the supports in $\Adot$, called Newton polytopes. 

The polyhedral geometry of the Newton polytopes $P_\bullet$ encodes much information about $\V(F)$. For example, the famous Bernstein-Kushnirenko Theorem (Proposition \ref{prop:BKK}) states that when $F$ is a square system ($k=n$) the number of isolated points of $\V(F)$ in $(\C^\times)^n$ is bounded by a numerical value called the mixed volume of $P_\bullet$. It also states that this bound is almost always attained, inducing a branched cover
\begin{align}
\label{eq:introbranchedcover}
{\pi_{\Adot}}\colon X_{\Adot}&\to \C^{\Adot}\\
(x,F) &\mapsto F \nonumber
\end{align}
from the incidence variety ${X_\Adot}=\{(x,F) \mid x \in (\C^\times)^n, F(x)=0\}$ whose fiber $\pi_{\Adot}^{-1}(F)$ is identified with the solutions of $F=0$ in $(\C^{\times})^n$. This viewpoint gives geometric structure to families of sparse polynomial systems whereby we may understand their constituents.

More difficult than counting solutions of polynomial systems is computing them.
Over the last sixty years, mathematicians laid the groundwork for computational algebraic geometry, developing symbolic algorithms for studying and computing solutions of polynomials. More recently, techniques from numerical analysis joined  algebraic geometry to form a novel computational paradigm known as numerical algebraic geometry. While symbolic algorithms use the algebraic properties of a polynomial system to study its solutions, numerical algebraic geometry studies varieties by computing numerical approximations of points on them, thus providing a predominantly geometric viewpoint toward computations in algebraic geometry.

Due to the geometric nature of numerical algebraic geometry, many definitions and proofs from geometry translate directly to numerical algorithms. For example, the definition of the monodromy group of a branched cover immediately suggests a numerical method to compute it (Algorithm \ref{alg:extractmonodromyelements}). Another example is Huber and Sturmfels' proof of the Bernstein-Kushnirenko  Theorem in \cite{HuberSturmfels}  which produces the polyhedral homotopy algorithm~(Algorithm \ref{alg:polyhedralhomotopy}) for  computing all solutions of $F=0$.

In Section~\ref{section:solvingsparsedecomposablesystems} we give an algorithm which improves upon the polyhedral homotopy whenever the branched cover $\pi_{\Adot}$ decomposes into a composition of branched covers. This decomposition happens if and only if the monodromy group of $\pi_{\Adot}$ is imprimitive, a condition that Esterov \cite{Esterov} classified via computable conditions on $\Adot$. Our algorithm (Algorithm \ref{alg:SDS}) assesses whether or not $\pi_\Adot$ decomposes and recursively computes fibers of the decomposition to compute a fiber of $\pi_\Adot$, thus solving a sparse polynomial system with support $\Adot$.

Conversely, algorithms in numerical algebraic geometry can extract information about Newton polytopes. In 2012, Hauenstein and Sottile suggested a numerical algorithm (Algorithm \ref{alg:hsalgorithm}) which functions as a vertex oracle for the Newton polytope of the defining equation of a hypersurface. In Section~\ref{section:numericalNP}, we explain how this algorithm is stronger than a vertex oracle and as a consequence, introduce the notion of a numerical oracle. Based on ideas from Hept and Theobald \cite{HeptTheobald}, we develop a tropical membership test (Algorithm \ref{alg:TropicalMembership}) which relies on the algorithm of Hauenstein and Sottile as a subroutine. We analyze the convergence rates of each algorithm (Theorem \ref{thm:convergencerates}) and explain our implementation of them in Section~\ref{sec:numericalnp}. Finally, we use our implementation to investigate the colossal L\"uroth polytope (Section~\ref{sec:Luroth}) and determine the implicit equation  a hypersurface from algebraic vision (Section~\ref{sec:AlgebraicVision}). 

We provide all necessary background in Sections \ref{section:polytopes}-\ref{section:numericalalgebraicgeometry}.
Section~\ref{section:polytopes} includes elementary results regarding polytopes, numerical oracles, mixed volumes, and subdivisions. In Section~\ref{section:algebraicgeometry} we give a basic introduction to algebraic geometry necessary for the subsequent sections. In Section~\ref{section:branchedcoversandgroups} we discuss branched covers, decomposable branched covers, and monodromy/Galois groups; we also give a proof that the monodromy/Galois group of a branched cover is imprimitive if and only if the branched cover is decomposable.  In Section~\ref{section:newtonpolytopes} we connect the previous sections by introducing Newton polytopes, sparse polynomial systems, and tropical algebraic geometry. Section~\ref{section:numericalalgebraicgeometry} builds the theory of numerical algebraic geometry and contains an assembly of numerical algorithms, including Huber and Sturmfels' treatment of the polyhedral homotopy, as well as algorithms which use monodromy to solve polynomial systems.

\chapter{POLYTOPES \label{section:polytopes}}
\renewcommand*{\thefootnote}{\fnsymbol{footnote}}
We remark that a portion of the discussion of numerical oracles in this section also appears in the article \cite{Bry:NPtrop} by the author\footnote{Reprinted with permission from T. Brysiewicz, ``Numerical Software to Compute Newton polytopes and Tropical Membership,''  {\it{Mathematics in Computer Science,}} 2020. Copyright 2020 by Springer Nature.}.
\renewcommand*{\thefootnote}{\arabic{footnote}}
\section{Describing polytopes}
\label{subsection:representingpolytopes}
A subset $S \subset \R^n$ is \mydef{convex} if for any $p,q \in S$ the line segment between them $\mydefMATH{[p,q]}=\{\lambda p + (1-\lambda)q \mid 0 \leq \lambda \leq 1\}$ is also contained in $S$.  The \mydef{convex hull} of $S$ is 
$$\mydefMATH{\conv(S)} = \bigcap \{S' \subset \R^n \mid S \subset S', S' \text{ convex}\}.$$
\begin{lemma}
If $\mathcal A=\{\alpha_1,\ldots,\alpha_k\} \subset \R^n$ is finite then 
$$\conv(\mathcal A) = \left\{\sum_{i=1}^k \lambda_i\alpha_i \mymid \sum_{i=1}^k \lambda_i =1, \lambda_i \in \R_{\geq 0}\right\}.$$
\end{lemma}
\begin{proof}
The forward containment is true since the right-hand-side is a convex set containing $\mathcal A$. Indeed, if $p=\sum_{i=1}^k \lambda_i\alpha_i$ and $q = \sum_{i=1}^k \nu_i \alpha_i$ are elements of the right-hand-side and $\gamma \in [0,1]$, then $$\gamma p + (1-\gamma)q = \sum_{i=1}^k (\gamma \lambda_i + (1-\gamma) \nu_i) \alpha_i$$
is as well.

The reverse containment for $k=1$ or $k=2$ is true by definition. Assume it is true for $k-1$ and let $\alpha=\sum_{i=1}^k \lambda_i\alpha_i$ be an element of the right-hand-side. Without loss of generality, assume $\lambda_1\neq0$ so that $$\alpha = \lambda_1\alpha_1 + (1-\lambda_1)\left(\frac{\lambda_2}{1-\lambda_1} \alpha_2 + \cdots + \frac{\lambda_k}{1-\lambda_1} \alpha_k\right).$$
Since $p:=\alpha_1$ and $q:=\sum_{i=2}^k \frac{\lambda_i}{1-\lambda_1}\alpha_i$ are points in $\conv(\mathcal A)$ by induction, the segment $[p,q]$ containing $\alpha$ must be in $\conv(\mathcal A)$ as well. 
\end{proof}
\begin{definition}
A \mydef{polytope} is any subset  $P \subset \R^n$ that can be written as the convex hull of finitely many points. If these points can be taken to be in $\Z^n$, then $P$ is called an \mydef{integral polytope}. 
\end{definition}

\begin{example}
\label{ex:polytopeExample} 
For ease of reading, we will often encode points in $\R^n$ as the columns of a matrix.  
Let $\mathcal A={ \left(\begin{smallmatrix} 0 & 0 & 3/2 & 2 & 2 & 2 & 3 & 4 \\ 2 & 3 & 3/2 &0 & 3 & 4 & 2 & 0 \end{smallmatrix}\right)}\subset \R^2$. The polytope $Q=\conv(\mathcal A)$ shown in Figure \ref{fig:ConvexHull1} is an integral polytope since we may write $Q=\conv\bigl( \begin{smallmatrix} 0 & 0  & 2  & 2 & 4 \\ 2 & 3  &0  & 4  & 0 \end{smallmatrix}\bigr)$. 
\begin{figure}[htpb!]
\includegraphics[scale=0.6]{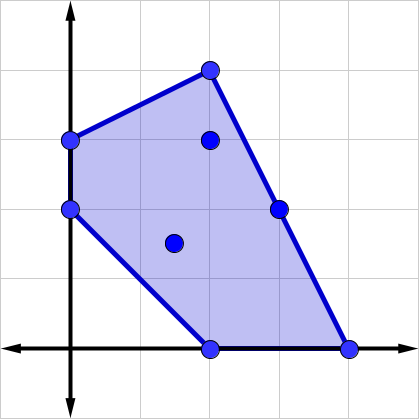}
\caption{An integral polytope $Q \subset \R^2$. }
\label{fig:ConvexHull1}
\end{figure}
\end{example}

The \mydef{dimension} of a subset $S \subset \R^n$, denoted \mydefMATH{$\dim(S)$}, is the dimension of its \mydef{affine span}, $$\mydefMATH{\R S}=\left\{\sum_{i=1}^k \lambda_is_i \mymid s_i \in S, \; \lambda_i \in \R, \; \sum_{i=1}^k \lambda_i = 1\right\},$$
and the \mydef{codimension} of $S$ is $\mydefMATH{\codim(S)}=n-\dim(S)$. \mydef{Polygons} are polytopes of dimension two. If $S$ is compact, we define the \mydef{support function} of $S$ as
\begin{align*}
\mydefMATH{h_S}\colon\R^n &\to \R \\
\omega &\mapsto \underset{x \in S}{\max} \langle x, \omega \rangle.
\end{align*}
Given $\omega \in \mathbb{R}^n$, the \mydef{subset of $S$ exposed by $\omega$} is 
$$\mydefMATH{S_\omega}=\{x \in S \mid \langle x, \omega \rangle = h_S(\omega)\}.$$
A \mydef{face} $\mathcal F$ of a polytope $P\subset \R^n$ is any subset of $P$ of the form $\mathcal F=\emptyset$ or $\mathcal F=P_\omega$ for some $\omega \in \R^n$.  Faces of dimensions $0,1, k, \dim(P)-1$ are called \mydef{vertices}, \mydef{edges}, \mydef{$k$-faces}, and \mydef{facets} respectively.  The set of vertices is denoted $\mydefMATH{\vertices(P)}$ and the set of facets is denoted $\mydefMATH{\facets(P)}$.

\begin{example}
Let $Q$ be as in Example \ref{ex:polytopeExample}. \begin{figure}[t]
\includegraphics[scale=0.8]{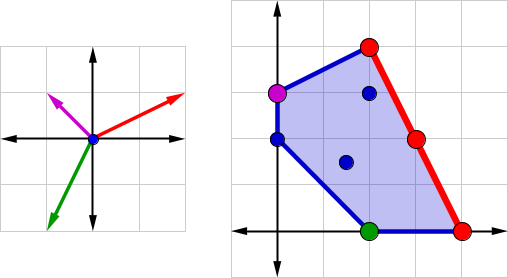}
\caption{
\label{fig:supportfunction}Left: Directions $\omega_1,\omega_2,\omega_3$, and $\omega_4=(0,0)$. Right: The polytope $Q \subset \R^2$ and three of its proper faces exposed by $\omega_1,\omega_2$, and $\omega_3$.}
\end{figure}The dimension of $Q$ is $2$ and its codimension is $0$. Let $\omega_1=(-1,1), \omega_2=(2,1), \omega_3=(-1,-2)$, and $\omega_4=(0,0)$. Then 
$$h_Q(\omega_1) = 3, \quad h_Q(\omega_2) = 8, \quad h_Q(\omega_3) = -2, \quad h_Q(\omega_4)=0,$$
and the faces exposed by $\omega_1,\omega_2,\omega_3$, and $\omega_4$ are 
$$Q_{\omega_1}=\{(0,3)\}, \quad  Q_{\omega_2}=\conv(\{(2,4),(4,0)\}), \quad Q_{\omega_3}=\{(2,0)\}, \quad Q_{\omega_4}=Q.$$  Figure \ref{fig:supportfunction} depicts these directions and faces.
In total, $Q$ has one empty face, five vertices, five facets (edges), and one $2$-face.
\hfill $\diamond$
\end{example}
Given a polytope $P \subset \R^n$, it is useful to collect directions $\omega \in \R^n$ which expose the same face into cones. A subset $C \subset \R^n$ is a \mydef{cone} if for any $p \in C$, we have that $\lambda p \in C$ for $\lambda \in \R_{\geq 0}$. A cone is a convex cone if it is closed under addition. Indeed if $p$ and $q$ are elements of a cone $C$ which is closed under addition and $\lambda \in [0,1]$ then $\lambda p + (1-\lambda)q \in C$ since each summand is in $C$.
The \mydef{(outer) normal fan} of a polytope $P$ is the collection
$$\mydefMATH{\mathcal N(P)}=\{C[\omega]\}_{\omega \in \R^n}$$ of convex cones $$\mydefMATH{C[\omega]}=\{ \omega' \in \mathbb{R}^n \mid P_{\omega} \subseteq P_{\omega'}\}.$$
We denote the set of all $C[\omega]$ of codimension at least $i$ by $\mathcal N^{(i)}(P)$.
\begin{example}
Figure \ref{fig:polytope} displays $Q$ along with its normal fan $\mathcal N(Q)$ which has one zero-dimensional cone (the origin), five one-dimensional cones, and five two-dimensional cones.\hfill $\diamond$
\begin{figure}[!htpb]
\includegraphics[scale=0.65]{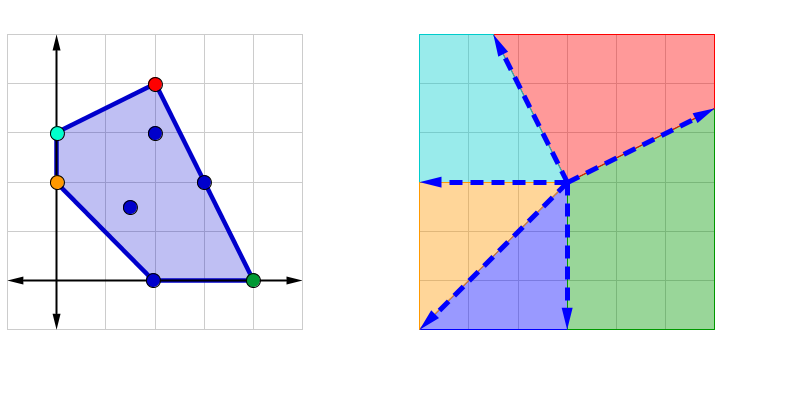}
\caption{A polytope and its corresponding normal fan}
\label{fig:polytope}
\end{figure}
\end{example}

\begin{lemma}\cite[Proposition 2.2]{Ziegler}
Every polytope may be written as 
\begin{equation}
\label{eq:vertexrep}
P = \conv(\vertices(P)).
\end{equation} If $\mathcal A \subset \R^n$ is finite, then $\vertices(\conv(\mathcal A)) \subseteq \mathcal A$.
\label{lem:Ziegler1}
\end{lemma}
\begin{lemma}\cite[Proposition 2.3]{Ziegler}
Let $F$ be a face of a polytope $P \subset \R^n$.
\begin{enumerate}
\item $F$ is a polytope with $\vertices(F) = F \cap \vertices(P)$.
\item Every intersection of  faces of $P$ is a face of $P$.
\item The faces of $F$ are exactly the faces of $P$ that are contained in $F$.
\item $F=P \cap \R F$.
\end{enumerate}
\label{lem:Ziegler2}
\end{lemma}
Lemma \ref{lem:Ziegler1} gives one way to canonically represent a polytope: as the convex hull of its vertices. This representation is called the \mydef{vertex representation} of a polytope.
Halfspaces provide another way to represent polytopes. A \mydef{halfspace} of $\R^n$ is any subset of the form
$$\mydefMATH{\R^n_{\omega,c}}=\{x \in \R^n \mid \langle x, \omega \rangle\leq c\}\subset \R^n,$$ 
for some $\omega \in \R^n$ and $c \in \R$. Given a polytope $P \subset \R^n$ and any direction $\omega \in \R^n$, the halfspace $\mydefMATH{H_P(\omega)}=\R^n_{\omega,h_P(\omega)}$ contains $P$. Note that $H_P(\omega)=H_P(\lambda \cdot \omega)$ for any $\lambda>0$.
\begin{lemma}\cite[Theorem 2.15]{Ziegler}
\label{lem:reptheoremforpolytopes}
Every polytope $P \subset \R^n$ may be written as
\begin{equation}
\label{eq:halfspacerep}
P=\R P \cap \left(\bigcap_{i=1}^m H_P(\omega_i)\right)
\end{equation} for any set $\{\omega_i\}_{i=1}^m \subset \R^n$ such that $\{P_{\omega_i}\}_{i=1}^m=\facets(P)$.
Conversely, any bounded intersection of halfspaces is a polytope.
\end{lemma}If a polytope is $n$-dimensional, then it has a unique representation of the form \eqref{eq:halfspacerep} since each facet is $(n-1)$-dimensional and is exposed by its unique outer-normal ray. Note that these are the one-dimensional cones in the normal fan of a polytope. If a polytope has positive codimension, then it has a unique representation of the form \eqref{eq:halfspacerep} within its affine hull (the $\omega_i$ in \eqref{eq:halfspacerep} are taken to be parallel with the affine hull of $P$). We call such a unique representation the \mydef{halfspace representation} of a polytope.
\begin{example}
The polytope $Q$ in Example \ref{ex:polytopeExample} has the halfspace representation,
$$Q= H_Q(2,1) \cap H_Q(0,-1) \cap H_Q(-1,-1)\cap H_Q(-1,0) \cap H_Q(-1,2).$$  Each of these halfspaces are displayed in Figure \ref{fig:HalfSpace}. 
\begin{figure}[htpb!]
\includegraphics[scale=0.4]{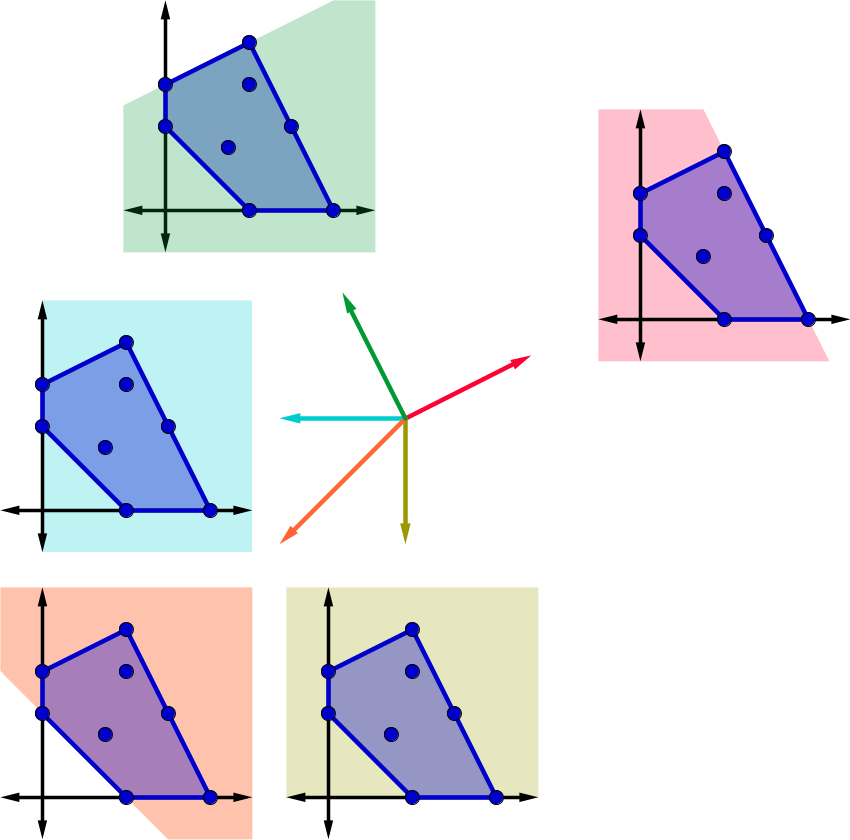}
\caption{Five halfspaces in $\R^2$ whose intersection is $Q$.}
\label{fig:HalfSpace}
\end{figure}\hfill $\diamond$
\end{example} 
\section{Oracles}
While the vertex and halfspace representations are the most common ways of expressing a polytope, other representations come from functions called oracles. Colloquially, an oracle is an entity which provides prophetic insight whenever queried. Likewise, the \mydef{vertex oracle} for a polytope $P \subset \R^n$ is the function 
$$ {\mathbb{V}}_P\colon \mathbb{R}^n \to  \mathbb{R}^n \cup \{\texttt{PFE}\} $$
$$\omega \mapsto \begin{array}{cc}
  \left\{  
    \begin{array}{lccl}
      P_\omega &&& \dim(P_\omega)=0\\
      \texttt{PFE} &\quad&& \text{otherwise}\\
    \end{array}
    \right.
\end{array}$$
where $\texttt{PFE}$ abbreviates the expression ``Positive dimensional Face Exposed''.  We remark that $\mathbb{V}_P(\omega)=\texttt{PFE}$ if and only if $\omega \in \mathcal N^{(1)}(P)$. The process of evaluating a vertex oracle is called \mydef{querying} the oracle.
\begin{remark}
\label{remark:oraclehalfspace}
When a vertex oracle query returns a vertex $\mathbb{V}_P(\omega)=v$, it implicitly returns the information that $h_P(\omega) = \langle v,\omega \rangle$ and therefore that $P \subset \R^n_{\omega,\langle v,\omega \rangle}=H_P(\omega)$.
\end{remark}
Let $\mydefMATH{\bzero}$ denote the all $0$'s vector in $\R^n$, $\mydefMATH{\bone}$ denote the all $1$'s vector in $\R^n$, and $\mydefMATH{e_i}$ denote the $i$-th coordinate vector in $\R^n$. For any $v \in \R^n$, let $\mydefMATH{|v|}$ denote the sum of its coordinates. Given a polytope $P \subset \R^n$, let $\mydefMATH{\mathcal L(P)}=P \cap \Z^n$ denote its set of \mydef{lattice points}. 
\begin{proposition}
\label{prop:oracletovertex}
If $P \subset \R^n$ is an integral polytope, then the vertex representation of $P$ can be recovered from the vertex oracle for $P$.
\end{proposition}
\begin{proof}
Let $P\subset \R^n$ be an integral polytope and $\mathbb{V}_P$ its vertex oracle. To prove the proposition, we first bound $P$ between two polytopes by querying the vertex oracle as follows.

Let $\omega^* = (\omega^*_1,\ldots,\omega^*_n) \in \R^n_{> 0}$ be a vector such that $\omega^*_1,\ldots,\omega^*_n$ are rationally independent (i.e. $\langle x, \omega^* \rangle \neq 0$ for any ${\bf 0} \neq x \in \Z^n$). Observe that $\mathbb{V}_P(\omega^*)$ must return a vertex: otherwise, there exist two vertices $p_1,p_2$ such that $\langle p_1,\omega^* \rangle = \langle p_2,\omega^* \rangle$ implying that $x=p_1-p_2$ is an integer point whose dot product with $\omega^*$ is nonzero. A consequence of Remark \ref{remark:oraclehalfspace} is that the halfspace $H_P(\omega^*)$ containing $P$ is computed as well. Since $\omega$ is in the positive orthant, $H_P(\omega^*)$ bounds $P\cap \R^n_{\geq 0}$.

Similarly, for every vertex $v$ of the hypercube $\mydefMATH{\text{cube}(n)}=[-1,1]^n$, we let $v   \circ   \omega^*$ denote the Hadamard (coordinate-wise) product so that the output $\mathbb{V}_P(v\circ\omega^*)$ is a vertex of $P$. Again, each oracle query bounds $P$ in the corresponding orthant of $\R^n$ so that the intersection 
$$P^*=\bigcap_{v \in \text{cube}(n)} H_P(v\circ \omega^*),$$ is bounded, and thus by Lemma \ref{lem:reptheoremforpolytopes}, is a polytope. Setting $P_*=\conv(\{\mathbb{V}_P(v\circ \omega) \mid v \in \text{cube}(n)\})$ gives containments
\begin{equation}
\label{eq:boundpolytope}
P_* \subseteq P \subseteq P^*.
\end{equation}

The proof proceeds algorithmically. Set $P^*=\conv\left(\mathcal L(P^*)\right)$ so that $P^*$ is integral. Since $P$ is integral, the containments \eqref{eq:boundpolytope} are still true. For every $p \in \vertices(P^*) \smallsetminus P_*$, pick $\omega$ such that $\mathbb{V}_{P^*}(\omega) = p$. Since $p$ is the unique point in $P^*$ obtaining a maximum dot product with $\omega$ and $P \subseteq P^*$ then $p \in P$ if and only if $\mathbb{V}_P(\omega)=p$. We have three cases: either $p \in P$ and so $\mathbb{V}_P(\omega)=p$ (case (i)), or $\mathbb{V}_P(\omega)$ returns $\texttt{PFE}$ (case (ii)) or $\mathbb{V}_P(\omega)$ returns another vertex $q\neq p$ (case (iii)).

\noindent
{\bf Case (i)}: If $\mathbb{V}_P(\omega) = p$ then set $P_*=\conv(P_* \cup p)$. Note that the containments \eqref{eq:boundpolytope} still hold and that the number of lattice points of $P_*$ has increased.

\noindent
{\bf Case (ii)}: If $\mathbb{V}_P(\omega) = \texttt{PFE}$, then $p \not\in P$ and so we may set $P^* =\conv(\mathcal L(P^*)\smallsetminus p)$ while preserving \eqref{eq:boundpolytope}. In this case, the number of lattice points of $P^*$ has decreased.

\noindent
{\bf Case (iii)}: If $\mathbb{V}_P(\omega) = q \neq p$, then we may set $P_* = \conv(P_* \cup q)$ and $P^* = \conv(\mathcal L(P^* \cap H_P(\omega)))$ while preserving \eqref{eq:boundpolytope}. In this case, the number of lattice points of $P_*$ may have increased depending on whether or not $q$ was already in $P_*$, but it will always be the case that the number of lattice points of $P^*$ has decreased. 

Each oracle query involves one of the above cases and each case preserves the containments \eqref{eq:boundpolytope} while either increasing the number of lattice points in $P_*$ or decreasing the number of lattice points in $P^*$. Thus, this process must terminate with  $\vertices(P^*)\smallsetminus P_* = \emptyset$, proving that these polytopes are equal to each other and so $P_*=P=P^*$.
\end{proof}

\boxit{
\begin{algorithm}[Vertex oracle $\to$ vertex representation] \nothing \\
\myline
{\bf Input:} \\
$\bullet$ The vertex oracle $\mathbb{V}_P$ for an integral polytope $P \subset \R^n_{\geq 0}$\\ \myline
{\bf Output:}\\
$\bullet$ The vertex representation for $P$\\
\myline
{\bf Steps:}
\begin{enumerate}[nosep]
\item[0] Pick $\omega^*=(\omega_1^*,\ldots,\omega^*_n) \in \R^n_{>0}$ with rationally independent coordinates
\item[1] \set $P_* = \emptyset$, \set $P^* = \R^n$
\item[2] \myfor each vertex $v \in \text{cube}(n)$ \mydo
\begin{enumerate}[nosep]
\item[2.1] \set $P_* = \conv(P_* \cup \mathbb{V}_P(v \circ \omega^*))$
\item[2.2] \set $P^* = P^* \cap H_P(v \circ \omega^*)$
\end{enumerate}
\item[3] \while $\mathcal L(P_*) \neq \mathcal L(P^*)$ \mydo 
\begin{enumerate}[nosep]
\item[3.1] \set $P^*=\conv(\mathcal L(P^*))$
\item[3.2] Pick $p \in \vertices(P^*) \smallsetminus P_*$
\item[3.3] Find $\omega \in \R^n$ such that $\mathbb{V}_{P^*}(\omega)=p$
\item[3.4] \myif $\mathbb{V}_P(\omega)=p$ \then \set $P_*=\conv(P_* \cup p)$

\item[3.5] \myif $\mathbb{V}_P(\omega)=\texttt{PFE}$ \then \set $P^* = \conv(\mathcal L(P^*) \smallsetminus p)$

\item[3.6] \myif $\mathbb{V}_P(\omega)=q \neq p$ \then
\begin{enumerate}[nosep]
\item[3.6.1] \set $P^*=P^* \cap H_P(\omega)$
\item[3.6.2] \myif $q \not\in P_*$ \then \set $P_*=\conv(P_* \cup p)$
\end{enumerate}
\end{enumerate}
\item[4] \return $\vertices(P_*)$
\end{enumerate}
\label{alg:notbeneathbeyond}
\end{algorithm}
}

\begin{example}
Figure \ref{fig:notbeneathbeyond} displays the steps required to complete Algorithm \ref{alg:notbeneathbeyond} on $Q$ from Example \ref{ex:polytopeExample}.  We use $\omega^*=(1,\sqrt{2})$ in step $(0)$ of the algorithm. Step $(2)$ in Algorithm \ref{alg:notbeneathbeyond} is represented by the top-left graphic showing the four vertex oracle queries on the vectors $\omega^*,-\omega^*,(-1,\sqrt{2}),$ and $(1,-\sqrt{2})$. Each query reveals a vertex of $Q$ and a halfspace containing $Q$. The intersection of all such halfspaces $Q^*$ is depicted in grey in the first image along with  $Q_*$ in green and $\conv(\mathcal L(Q^*))$ in red. 

\begin{figure}[!htpb]
\includegraphics[scale=0.62]{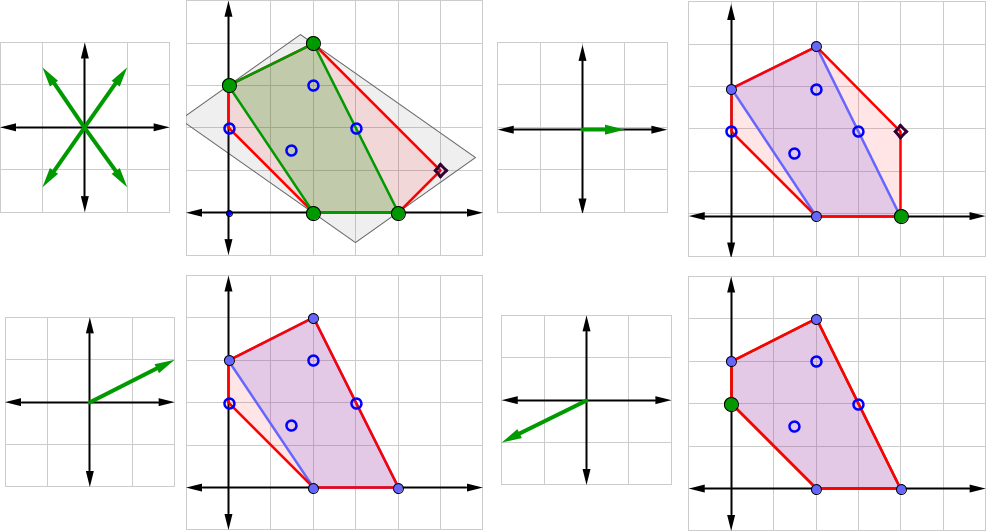}
\caption{A graphical interpretation of Algorithm \ref{alg:notbeneathbeyond} running on the polytope $Q$ in Example \ref{ex:polytopeExample}.}
\label{fig:notbeneathbeyond}
\end{figure}

The next image (to the right) displays the oracle query $\mathbb{V}_Q(1,0) = (4,0)$, revealing a vertex which was already found. Thus, this oracle query does not increase the size of $Q_*$ but it does establish that $(5,1)$ (a previous vertex of $Q^*$) is not contained in $Q$ and so the size of $Q^*$ is reduced. The third image (bottom left) attempts to establish whether or not $(4,2) \in Q$ by choosing $\omega=(2,1)$ so that $\mathbb{V}_{Q^*}(\omega)=(4,2)$ and querying $\mathbb{V}_Q(\omega)=\texttt{PFE}$. This does not find a new vertex of $Q$, nor does it find a new halfspace containing $Q$. It does, however, reveal that $(4,2) \not\in Q$ and so $Q^*$ is again reduced to $\conv(\mathcal L(Q^*)\smallsetminus (4,2))$. At this stage, $(0,2)$ is the unique vertex of $Q^*$ which is not in $Q_*$ and $\mathbb{V}_Q(-2,1)=(0,2)$ reveals that it is a vertex of $Q$. The outer polytope $Q^*$ is reduced again, the inner polytope $Q_*$ grows, and $Q^*$ becomes equal to $Q_*$, ending the algorithm.\hfill $\diamond$
\end{example}
\begin{remark}
\label{rem:oracletovertex}
Implementing Algorithm \ref{alg:notbeneathbeyond}, as is, requires the representation of a rationally independent vector $\omega^*$ on a computer for step $(2)$. Theoretically, a random $\omega \in \R^n$ will expose a vertex of $P$ with probability one and so in practice, we replace steps $(0)$ and $(2)$ by randomly querying the oracle in each orthant until a vertex is returned. This process bounds $P$ in a polytope $P^*$. We remark that probability one statements about the theory may not translate to probability one computations and we give a more detailed discussion in Remark \ref{rem:positivemeasure} in Section~\ref{section:numericalNP}. \hfill $\diamond$
\end{remark}
We denote the standard full-dimensional simplex in $\R^n$ by $\mydefMATH{\Delta_n}=\conv(\bzero,e_1,\ldots,e_n)$ and the dilation of $\Delta_n$ by a factor of $d$ by $\mydefMATH{d\Delta_n}=\conv(\bzero,d\cdot e_1,\ldots, d\cdot e_n)$.
The \mydef{degree} of a polytope $P \subset \R_{\geq 0}^n$, is \mydefMATH{$\deg(P)$}$=h_P({\bf 1})$. A polytope is \mydef{homogeneous} if $|p| =\deg(P)$ for all $p \in P$ and the \mydef{homogenization} of $P$ is $\mydefMATH{\widetilde{P}}=\{(p,\deg(P)-|p|) \mid p \in P\} \subset \R^{n+1}$.

\begin{definition}
\label{def:numericaloracle}
The \mydef{numerical oracle} for a polytope $P \in \R^n$ is the function
$$ {\mathcal O}_P\colon \mathbb{R}^n \to  \mathbb{R}^n \cup \{\texttt{EEP}\} $$
$$\omega \mapsto \begin{array}{ll}
  \left\{  
    \begin{array}{llll}
      P_\omega &&& \dim(P_\omega)=0\\
      \min(P_\omega) &&& 0 < \dim(P_\omega) < \dim(P) \\
      \texttt{EEP} &\quad&& P_\omega =P\\
    \end{array}
    \right.
\end{array}$$
where \mydefMATH{$\min(P_\omega)$ }is the coordinate-wise minimum of all points in $P_\omega$.
\end{definition} The expression \texttt{EEP} abbreviates \texttt{Exposes Entire Polytope}. This oracle is dubbed ``numerical'' because it arises naturally from the numerical HS-algorithm (Algorithm \ref{alg:hsalgorithm} of Section~\ref{section:numericalNP}).

Generally, one cannot distinguish whether the output of a numerical oracle for a polytope $P$ is a vertex $v=P_\omega$ or the coordinate-wise minimum $w=\min(P_\omega)$ of a positive-dimensional face. For example, the numerical oracle query $\mathcal O_{\Delta_2}(1,1)$ returns ${\bf 0}$ not because ${\bf 0}$ is a vertex, but because ${\bf 0}=\min(\conv(e_1,e_2))$. Thus, at first glance, a numerical oracle may seem weaker than a vertex oracle. However, when the polytope $P$ is homogeneous of degree $d$ these cases may be distinguished easily since the sum of the coordinates of a vector output of $\mathcal O_P(\omega)$ will be $d$ if and only if the vector is a vertex and it will be less than $d$ otherwise. Restricted to homogeneous polytopes, a numerical oracle gives strictly more information than a vertex oracle, implying the following corollary to Proposition \ref{prop:oracletovertex}.
\begin{corollary}
If $P$ is a homogeneous integral polytope then the vertex representation of $P$ may be recovered from its numerical oracle.
\end{corollary}
Other oracles for polytopes exist and are well-studied. For example, Emiris et. al. \cite{OToV} developed an algorithm similar to Algorithm \ref{alg:notbeneathbeyond} for oracles which are stronger than vertex oracles: instead of returning $\texttt{PFE}$, they return a vertex on the corresponding positive-dimensional face.

\section{Mixed volume}
\label{subsection:mixedvolume}
We develop some of the theory of mixed volumes of polytopes and include multiple formulas and characterizations of mixed volume. We list them here for convenience.
\begin{enumerate}
\item Coefficient of a volume function (Definition \ref{def:mixedvolume}).
\item Volume alternating sum formula (Lemma \ref{lem:mixedvolume2}).
\item Axiomatic characterization (Lemma \ref{lem:mixedvolume3}).
\item Lattice point alternating sum formula for integral polytopes (Lemma \ref{lem:mixedvolume4}).
\item Sum of volumes of mixed cells formula (Lemma \ref{lem:mixedvolume5}).
\end{enumerate}
We give a sixth way of computing mixed volume in Section~\ref{section:newtonpolytopes} via the Bernstein-Kushnirenko  Theorem (Proposition \ref{prop:BKK}). 

We begin our discussion by introducing two natural operations on subsets of $\R^n$. 
Let $S_1,S_2 \subset \R^n$ and $\lambda \in \R_{\geq 0}$. The set
$$\mydefMATH{\lambda S_1} = \{\lambda s \mid s \in S_1\},$$ is the \mydef{scaling} of $S_1$ by $\lambda$.
The set 
$$\mydefMATH{S_1+S_2}=\{s_1+s_2\mid s_1 \in S_1, s_2 \in S_2\},$$ is the \mydef{Minkowski sum} of $S_1$ and $S_2$.
The scaling of a polytope $P=\conv(\mathcal A)$ by $\lambda \in \R_{\geq 0}$ is clearly a polytope given as $\lambda P=\conv(\lambda \mathcal A)$. The following lemma proves an analogous result for Minkowski sums of polytopes.
\begin{lemma}
\label{lem:minkowskiproperties}
Let $P,Q \subset \R^n$ be polytopes. 
\begin{enumerate}
\item The support functions of $P$ and $Q$ are additive: $h_{P+Q}=h_{P}+h_{Q}$.
\item The Minkowski sum $P+Q$ is a polytope which may be written as $\conv(\vertices(P)+\vertices(Q))$.
\item If $F\subset P+Q$ is a face, then there exist unique faces $F_P\subseteq P$ and $F_Q\subseteq Q$ such that $F=F_P+F_Q$. 
\item If $P$ and $Q$ are integral, so is $P+Q$.
\end{enumerate}
\end{lemma}
\begin{proof}
Additivity of support functions is immediate since $$h_{P+Q}(\omega) = \max_{s \in P+Q} \langle s,\omega \rangle= \max_{p \in P, q \in Q} \langle p+q,\omega \rangle = \max_{p \in P} \langle p,\omega \rangle + \max_{q \in Q} \langle q, \omega \rangle.$$

To show that $P+Q$ is a polytope, we first show $P+Q$ is convex.  Let $a=p_1+q_1$ and $b=p_2+q_2$ for $p_1,p_2\in P$ and $q_1,q_2 \in Q$. Then $v \in [a,b]$ implies
\begin{align*}
v &= \lambda a + (1-\lambda)b \\
&= \lambda(p_1+q_1) + (1-\lambda)(p_2+q_2) \\
&= (\lambda p_1 +(1-\lambda p_2)) + (\lambda q_1+(1-\lambda) q_2) \in P+Q,
\end{align*}
proving that $P+Q$ is convex. To see that $P+Q\subset \conv(\vertices(P)+\vertices(Q))$, suppose towards contradiction that there exists $v \in P+Q \smallsetminus \conv(\vertices(P)+\vertices(Q))$. Then there exists a halfspace of $\conv(\vertices(P)+\vertices(Q))$ not containing $v$. In other words, there exists $\omega$ such that $\langle v, \omega \rangle = h_{P+Q}(\omega)> h_P(\omega)+h_Q(\omega)$, a contradiction by part $(1)$. Thus, 
$$\vertices(P)+\vertices(Q) \subset P+Q \subset \conv(\vertices(P)+\vertices(Q)),$$
and taking the convex hull of this containment proves parts $(2)$ and $(4)$.

To prove part $(3)$, observe that for any $\omega \in \R^n$ we have $(P+Q)_\omega=P_\omega+Q_\omega$ by part $(1)$. Suppose $$(P+Q)_\omega = P_{\omega'}+Q_{\omega''},$$ for some other $\omega',\omega''\in \R^n$. The evaluation of $x \mapsto \langle x,\omega \rangle$ at any point on the right-hand-side must equal $h_{P}(\omega)+h_{Q}(\omega)$, implying that $P_{\omega'}=P_\omega$ and $Q_{\omega''}=Q_\omega$.
\end{proof}

To fix notation, let $\mydefMATH{\Pdot}=\{P_1,\ldots,P_n\}$ be a collection of $n$ polytopes in $\R^n$. We denote the set $\{1,\ldots,n\}$ by $\mydefMATH{[n]}$. The following result is due to Minkowski when $d=3$ \cite{Minkowski}.

\begin{lemma}[H. Minkowski \cite{Minkowski}]
The function $$\mydefMATH{V(\Pdot)}\colon\R^n_{\geq 0} \to \R$$
$$V(\Pdot)(\lambda_1,\ldots,\lambda_n) = \vol(\lambda_1P_1+\cdots+\lambda_nP_n)$$ 
is a homogeneous polynomial of degree $n$ in $\R[\lambda_1,\ldots,\lambda_n]$ where $\mydefMATH{\vol}$ denotes the $n$-dimensional Euclidean volume.
\end{lemma}
\begin{definition}
\label{def:mixedvolume}
The \mydef{mixed volume} of  $\Pdot$, denoted $\mydefMATH{\MV(\Pdot)}$, is the coefficient of $\lambda_1 \lambda_2\cdots \lambda_n$ in $V(\Pdot)$.
\end{definition}
\begin{example}
\label{ex:MixedVolume}
Consider $A=\conv(\bzero,e_1,e_2,e_1+e_2)$ and $B=\conv(\bzero,e_1,e_2)$ as displayed in Figure \ref{fig:mixedvolume}. Then $V(A,B) = \lambda_1^2+2\lambda_1\lambda_2+\frac{1}{2}\lambda_2^2$ and so $\MV(A,B)=2$. 
\begin{figure}[htpb!!]
\includegraphics[scale=0.6]{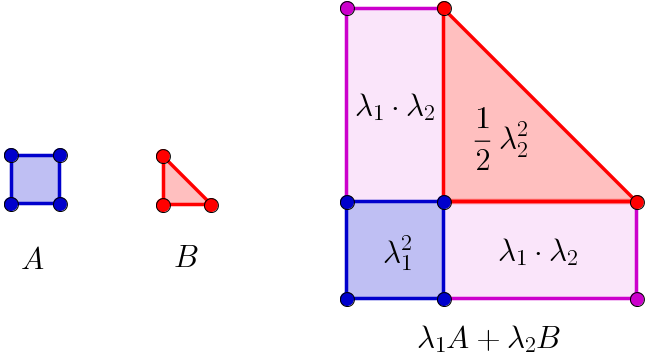}
\caption{A graphic expressing $\text{vol}(\lambda_1A+\lambda_2B)$ for two polygons $A,B \subset \R^2$.}
\label{fig:mixedvolume}
\end{figure}\hfill $\diamond$
\end{example}
\begin{lemma}
\label{lem:properties of mixed volume}
Let $P,P_1,\ldots,P_n,Q \subset \R^n$ be polytopes and let $a \in \R_{\geq 0}$. Then,
\begin{enumerate}
\item $\MV(P,\ldots,P) = n! \vol(P)$.
\item $\MV$ is symmetric in its arguments.
\item $\MV$ is multilinear:
$$\MV(aP_1+Q,P_2,\ldots,P_n) = a\MV(P_1,\ldots,P_n)+\MV(Q,P_2,\ldots,P_n).$$
\end{enumerate}
\end{lemma}
\begin{proof}
Note that $\vol(\lambda_1P+\cdots+\lambda_nP) = \vol((\lambda_1+\cdots+\lambda_n)P) = (\lambda_1 + \cdots + \lambda_n)^n \vol(P)$ and so the coefficient of $\lambda_1 \cdots \lambda_n$ is $n!\vol(P)$.
Part $(2)$ is immediate from the definition of mixed volume. For a proof of part $(3)$, see  \cite[Lemma 3.6]{Ewald}.
\end{proof}

\begin{lemma}\cite[Theorem 3.7]{Ewald}
\label{lem:mixedvolume2}
Given a collection of polytopes $P_{1},\ldots,P_n$,
$$\MV(P_1,\ldots,P_n) = \sum_{I \subset [n]} (-1)^{n-|I|}\vol\left(\sum_{i \in I} P_i \right).$$
\end{lemma}
\begin{proof} We restate the proof given in \cite{Ewald}.
Due to precisely the properties of mixed volume in  Lemma \ref{lem:properties of mixed volume}, we may treat the statement in the theorem as the polynomial equation
\begin{equation}
\label{eq:alternatingaspolynomial}
n!x_1\cdots x_n = (x_1+\cdots+x_n)^n-\sum_{i=1}^n(x_1+\cdots+x_{i-1}+x_{i+1}+\cdots x_n)^n +-\cdots
\end{equation}
$$\cdots +(-1)^{n-2} \sum_{i < j} (x_i+x_j)^n + (-1)^{n-1}\sum_{i=1}^n x_i^n,$$
where $x_{i_1}\cdots x_{i_N}\leftrightarrow \vol\left( \lambda_{i_1} \cdot P_{i_1}+ \cdots +\lambda_{i_N}\cdot P_{i_N}\right)$.
To verify \eqref{eq:alternatingaspolynomial}, we may simply check how many times each monomial appears in the right-hand-side. The monomial $x_i^n$ appears once in the first term, $n-1$ times in the second, and so on to give a total of 
$$1-(n-1)+{{n-1}\choose{2}}-\cdots+(-1)^{n-2}(n-1)+(-1)^{n-1} = (1-1)^{n-1} = 0.$$
Similarly, every term on the right-hand-side cancels except for the mixed term $x_1\cdots x_n$ which appears $n!$ times. 
\end{proof}
Since the formula in Lemma \ref{lem:mixedvolume2} is short when $n=2$, we state it as a corollary.
\begin{corollary}
The mixed volume of two convex polygons $P_1,P_2 \subset \R^2$ is $$\MV(P_1,P_2)=\vol(P_1+P_2)-\vol(P_1)-\vol(P_2).$$
\end{corollary}
\begin{lemma}
\label{lem:mixedvolume3}
The only function from $n$-tuples of polytopes to $\R$ satisfying the properties in Lemma \ref{lem:properties of mixed volume} is $\MV$.
\end{lemma}
\begin{proof}
The proof of the formula of Lemma \ref{lem:mixedvolume2} relied precisely on the properties in Lemma \ref{lem:properties of mixed volume}. Thus, any other function satisfying those properties will have the same formula.
\end{proof}
When each polytope in a collection $\Pdot$ is integral, there is a discrete analog of Lemma \ref{lem:mixedvolume2} involving lattice point enumeration.
\begin{lemma}\cite[Corollary 3.10]{Steffens}
\label{lem:mixedvolume4}
Given a collection of integral polytopes $P_{1},\ldots,P_n$,
$$\MV(P_1,\ldots,P_n) =(-1)^n+ \sum_{\emptyset \neq I \subset [n]} (-1)^{n-|I|}\Bigl{|}\mathcal L\Bigl(\sum_{i \in I} P_i \Bigr)\Bigr|.$$ 
\end{lemma}
\section{Subdivisions} Following \cite{HuberSturmfels} we give the notion of subdivisions of collections of finite subsets of $\R^n$.  The combinatorial constructions in this section provide a fifth description of the mixed volume of a collection of polytopes and are fundamentally important for Algorithm \ref{alg:polyhedralhomotopy} of Section~\ref{subsection:polyhedralhomotopy}. 

Let $\mydefMATH{\Adot}= (\calA_1,\ldots,\calA_k)$ be a collection of finite subsets of $\R^n$ whose union affinely spans $\R^n$. 
A \mydef{cell} of $\Adot$ is a tuple $\mathcal C_\bullet = (\mathcal C_1,\ldots,\mathcal C_k)$ of nonempty subset $\mathcal C_i \subset \mathcal A_i$. We define
\begin{align*}
\mydefMATH{\type(\mathcal C_\bullet)} &= (\dim(\conv(\mathcal C_1)),\ldots,\dim(\conv(\mathcal C_k))), \\
\mydefMATH{\conv(\mathcal C_\bullet)} &= \conv(\mathcal C_1+ \cdots + \mathcal C_k), \\
\mydefMATH{|\mathcal C_\bullet|} &= |\mathcal C_1| + |\mathcal C_2| + \cdots + |\mathcal C_k|,\\
\mydefMATH{\vol(\mathcal C_\bullet)} &= \vol(\conv(\mathcal C_\bullet)).
\end{align*}
\begin{definition}
\label{def:subdivision}
A \mydef{subdivision} of $\Adot$ is a collection $S^\bullet=\left\{\Cdot^{(1)},\ldots,\Cdot^{(m)}\right\}$ of cells satisfying 
\begin{enumerate}
\item $\dim\left(\conv\left(\Cdot^{(i)}\right)\right) = n$ for all $i=1,\ldots,m$.
\item $\conv\left(\Cdot^{(i)}\right) \cap \conv\left(\Cdot^{(j)}\right)$ is a proper face of $\conv\left(\Cdot^{(i)}\right)$ and $\conv\left(\Cdot^{(j)}\right)$ for all $i\neq j \in [m]$.
\item $\bigcup_{i=1}^m \conv\left(\Cdot^{(i)}\right) = \conv(\Adot).$
\end{enumerate}
If $S^\bullet$ additionally satisfies 
\begin{itemize}
\item[(4)] $\left|\type\left(\Cdot^{(i)}\right)\right| = n$ for all $i=1,\ldots,m$,
\end{itemize}
then we say it is a \mydef{mixed subdivision}. Even stronger, if $S^\bullet$ additionally satisfies
\begin{itemize}
\item[(5)] $\sum_{i=1}^k\left(\left|\mathcal C^{(j)}_i\right|-1\right)=n$ for all $j=1,\ldots,m$,
\end{itemize}
then we say it is a \mydef{fine mixed subdivision}.
\end{definition}
A cell $\mathcal C_\bullet$ of a subdivision $S^{\bullet}$ is called a \mydef{mixed cell} when $\min(\type(\mathcal C_\bullet))>0$ and a \mydef{fine mixed cell} if it additionally satisfies $\sum_{i=1}^k (|\mathcal C_i|-1)=n$. When $k=n$, a cell $\mathcal C_\bullet$ is mixed if $\type(\mathcal C_\bullet)=\bone$ and it is fine mixed if $|\mathcal C_i|=2$ for $i=1,\ldots,k$.   
\begin{example}
When $k=1$, every subdivision of $\Adot$ is a mixed subdivision because parts $(1)$ and $(4)$ of Definition \ref{def:subdivision} become the same statement. The fine mixed subdivisions of $\Adot$ are those with the property that the convex hull of each cell is an $n$-simplex. Such subdivisions comprise a rich family of combinatorial objects called \mydef{triangulations} \cite{triangulations}. \hfill $\diamond$
\end{example} The definitions above provide a new description of mixed volume.
\begin{lemma}\cite[Theorem 2.4]{HuberSturmfels}
\label{lem:mixedvolume5}
Suppose $\Adot=(\mathcal A_1,\ldots,\mathcal A_k)$ is a collection of finite subsets of $\R^n$ whose union affinely spans $\R^n$ and let $P_i=\conv(\mathcal A_i)$. If $S^\bullet$ is a mixed subdivision of $\Adot$ and  $r=(r_1,\ldots,r_k)\subset \mathbb{N}^k$ such that $|r|=n$, then the mixed volume of $$P=(\underbrace{P_1,\ldots,P_1}_{r_1},\underbrace{P_2,\ldots,P_2}_{r_2},\ldots,\underbrace{P_k,\ldots,P_k}_{r_k}),$$ is the sum of the volumes of the mixed cells in $S^\bullet$ of type $(r_1,r_2,\ldots,r_k)$:
$$\MV(P) = \sum_{\substack{\Cdot \in S^\bullet \\ \type(\Cdot)=(r_1,\ldots,r_k)}} \vol(\mathcal C_\bullet).$$
\end{lemma}

We  describe a process which produces subdivisions from functions. 
Let $\mathcal A \subset \R^n$ be a finite set and let $\ell\colon\mathcal A \to \R$ be any function. Let $\mydefMATH{\Gamma_\ell}\colon \mathcal A \to \R^{n+1}$ be the function ${\Gamma_\ell}(\alpha)=(\alpha,\ell(\alpha))$. We call $\ell$ a \mydef{lifting function} and we call the polytope 
$$\mydefMATH{\conv_\ell(\mathcal A)}=\conv(\Gamma_\ell(\mathcal A))\subset \R^{n+1},$$ the \mydef{lift of $\mathcal A$ by $\ell$}.
Similarly, given a set of functions $\mydefMATH{\ell_\bullet}=(\ell_1,\ldots,\ell_k)$ with $\ell_i\colon\mathcal A_i \to \R$, let $\mydefMATH{\Gamma_{\ell_\bullet}}\colon\Adot \to \R^{n+1}$ be the function $\Gamma_{\ell_\bullet}(\alpha_1,\ldots,\alpha_k)=\sum_{i=1}^k \Gamma_{\ell_i}(\alpha_i)$. Analogously, define
$$\mydefMATH{\conv_{\ell_\bullet}(\mathcal A_\bullet)}=\conv(\Gamma_{\ell_\bullet}(\Adot)) = \sum_{i=1}^k \conv_{\ell_i} (\mathcal A_i).$$
For any polytope $P \subset \R^{n+1}$, the \mydef{lower hull} of $P$ is the set $$\mydefMATH{\underline{\text{hull}}(P)}=\{P_\omega \mid \omega \in \R^{n+1} \text{ and } \langle \omega, e_{n+1} \rangle <0\}.$$
The $n+1$ above is suggestive in that we will often take lower hulls of lifts of polytopes.
\begin{lemma}
\label{lem:lowerhullsurjects}
Let $\mathcal A \subset \R^n$ be a finite collection of points and $\ell\colon \mathcal A \to \R$ a function. The projection of the lower hull of $\conv_\ell(\mathcal A)$ onto the first $n$ coordinates is $\conv(\mathcal A)$. 
\end{lemma}
\begin{proof}
Since $\conv(\mathcal A)$ is full-dimensional in its affine span, we may assume $\dim(\conv(\mathcal A))=n$ and show that $\alpha \in \vertices(\conv(\mathcal A)) \implies \Gamma_\ell(\alpha) \in \underline{\text{hull}}(\conv_{\ell}(\mathcal A))$.

Let $\alpha \in \vertices(\conv(\mathcal A))$ and $\omega \in \R^n$ so that $\conv(\mathcal A)_\omega= \alpha$. Then $(\omega,0)$ exposes $\Gamma_\ell(\alpha)$ and is in the interior of the $(n+1)$-dimensional cone $C[(\omega,0)]$. Thus, there exists a direction with negative last coordinate which exposes $\Gamma_\ell(\alpha)$ implying that $\Gamma_\ell(\alpha) \in \underline{\text{hull}}(\conv_{\ell}(\mathcal A))$. 
\end{proof}

\begin{definition}
\label{def:inducedsubdivision}
Given a set $\ell_\bullet$ of lifting functions $\ell_i\colon\mathcal A_i \to \R$, let $\mydefMATH{S^{\ell_\bullet}}$ be the set of maximal (with respect to inclusion) cells $\mathcal C_\bullet$ of $\mathcal A_\bullet$ satisfying
\begin{enumerate}
\item $\dim(\conv_{\ell_\bullet}(\mathcal C_\bullet))=n$,
\item $\conv_{\ell_\bullet}(\mathcal C_\bullet) \in \underline{\text{hull}}(\conv_{\ell_\bullet}(\mathcal A_\bullet))$.
\end{enumerate}
\end{definition}
We remark that the maximality condition in Definition \ref{def:inducedsubdivision} ensures that $\{\conv_{\ell_\bullet}(\mathcal C_\bullet)\}_{\mathcal C_{\bullet} \in S^{\ell_\bullet}}$ are distinct. Indeed if $\conv_{\ell_\bullet}(\mathcal C_\bullet) = \conv_{\ell_\bullet}(\mathcal C'_\bullet)$ but $\mathcal C_\bullet \neq \mathcal C'_\bullet$ then the union $\mathcal C_\bullet \cup \mathcal C'_\bullet$ satisfies conditions $(1)$ and $(2)$ of Definition \ref{def:inducedsubdivision} and contains each cell, contradicting maximality.
\begin{lemma}The set $S^{\ell_\bullet}$ is a subdivision of $\mathcal A_\bullet$.
\end{lemma}
\begin{proof}
If $\conv_{\ell_\bullet}(\Adot)$ is only $n$-dimensional, it must lie in a hyperplane implying that $S^{\ell_\bullet}=\Adot$ is the trivial subdivision. 

Let $\pi\colon\R^{n+1} \to \R^n$ be the projection onto the first $n$ coordinates.
Suppose $\mathcal C_\bullet \in S^{\ell_\bullet}$ and $\conv_{\ell_\bullet}(\mathcal C_\bullet)$ is exposed by $\omega \in \R^{n+1}$ where $\omega$ has negative last coordinate. Since $\dim(\conv_{\ell_\bullet}(\mathcal C_\bullet))=n$, its projection under $\pi$ has dimension at most $n$. Moreover, its projection has dimension less than $n$ only if the affine span of $\conv_{\ell_\bullet}(\mathcal C_\bullet)$ contains a line which projects to a point under $\pi$. But no such line exists because $\omega$ has negative last coordinate and exposes $\conv_{\ell_\bullet}(\mathcal C_\bullet)$. Thus, $S^{\ell_\bullet}$ satisfies $(1)$ of Definition \ref{def:subdivision}.

Given distinct $\Cdot$ and $\Cdot'$ in $S^{\ell_\bullet}$, both $\conv_{\ell_\bullet}(\Cdot)$ and $\conv_{\ell_\bullet}(\Cdot')$ are facets of $\conv_{\ell_\bullet}(\Adot)$, and so by part $(2)$ of Lemma \ref{lem:Ziegler2}, their intersection is a face of of $\conv_{\ell_\bullet}(\Adot)$ as well.  By part $(3)$ of Lemma \ref{lem:Ziegler2}, that intersection is a face of both $\conv_{\ell_\bullet}(\Cdot)$ and $\conv_{\ell_\bullet}(\Cdot')$ . It is proper since $\conv_{\ell_\bullet}(\mathcal C_\bullet)$ and $\conv_{\ell_\bullet}(\mathcal C_\bullet')$ are distinct. Thus, $S^{\ell_\bullet}$ satisfies $(2)$ of Definition \ref{def:subdivision}. 
Part $(3)$ of Definition \ref{def:subdivision} follows from Lemma \ref{lem:lowerhullsurjects}.
\end{proof} Any subdivision of the form $S^{\ell_\bullet}$ is called the \mydef{coherent subdivision} of $\mathcal A_\bullet$ \mydef{induced} by $\ell_\bullet$.

\begin{example}
\label{ex:liftexample}
Consider the set $\mathcal A_\bullet=\{\mathcal A_1\}$ where $\mathcal A_1$ consists of all lattice points in the $3$-dilate of the unit square in $\mathbb{R}^2$. Let $\ell_\bullet = \{\ell_1\}$ where $\ell_1\colon\mathcal A_1 \to \R$ is defined by $$\ell_1(\alpha)= \begin{array}{cc}
  \left\{  
    \begin{array}{cccl}
      \pi &&&  \alpha \text{ is in the boundary of } \conv(\mathcal A_1)\\
      1 &\quad&& \text{otherwise}\\
    \end{array}
    \right.
\end{array}.$$
Then $$\conv_{\ell_\bullet}(\mathcal A_\bullet) =\conv_{\ell_1}(\mathcal A_1) = \conv{ \begin{pmatrix} 0&0&3&3&1&1&2&2 \\ 0&3&0&3&1&2&1&2 \\ \pi&\pi&\pi&\pi&1&1&1&1 \end{pmatrix}}.$$ The lower hull of $\conv_{\ell_\bullet}(\mathcal A_\bullet)$ consists of five facets exposed by the directions $$(0,0,-1),(0,1-\pi,-1),(1-\pi,0,-1),(0,\pi-1,-1),(\pi-1,0,-1),$$ which project down to $\conv(\mathcal A_1)$, producing a description of the subdivision $S^{\ell_\bullet}=\left\{\Cdot^{(1)},\ldots,\Cdot^{(5)}\right\}$. The collection $\left\{\conv\left(\Cdot^{(i)}\right)\right\}_{i=1}^5$ consists of five quadrangles displayed in blue in Figure \ref{fig:subdivision1}.
\begin{figure}[htpb!!]
\includegraphics[scale=0.6]{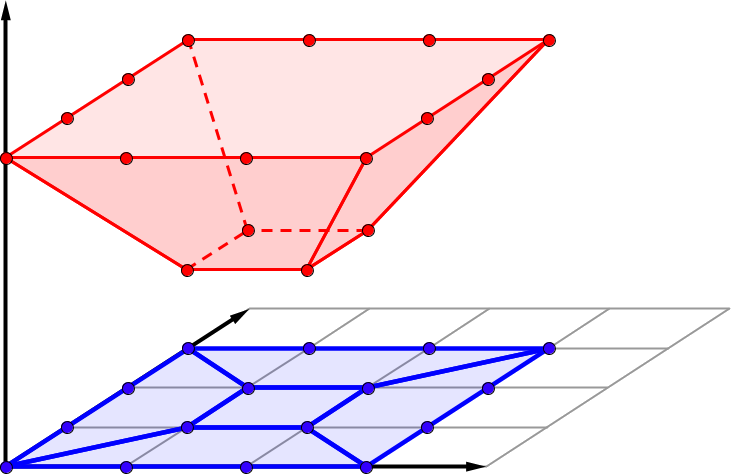}
\caption{A lifting of a dilated square and the corresponding polyhedral subdivision.}
\label{fig:subdivision1}
\end{figure}\hfill $\diamond$
\end{example}

\begin{example}
\label{ex:subdivision}
Let $\Adot = \{\mathcal A_1,\mathcal A_2\}$ where 
\begin{align*}
\mathcal A_1 &= \{(0,0),(0,1),(1,0),(1,1)\},\\
\mathcal A_2 &= \{(0,0),(1,2),(2,1)\}.
\end{align*}
Let $\ell_\bullet=(\ell_1,\ell_2)$ be the functions defined by 
\begin{align*}
\ell_1(0,0)&=2,\quad \ell_1(0,1)=3, \quad \ell_1(1,0)=3, \quad \ell_1(1,1) = 3,\\
\ell_2(0,0)&=1, \quad \ell_2(1,2)=1, \quad \ell_2(2,1) = 1.
\end{align*}
Figure \ref{fig:subdivision2} displays $\conv(\mathcal A_1)$ and $\conv(\mathcal A_2)$ along with the lower hulls of the convex hulls of their lifts in the first two images.   The third image displays the lower hull of $\conv_{\ell_\bullet}(\mathcal A_\bullet)$ along with the two points of $\Gamma_{\ell_1}(\mathcal A_1)+\Gamma_{\ell_2}(\mathcal A_2)$ which do not belong to any facet in the lower hull. The third image also contains a depiction of the induced subdivision on $\Adot$. The green parallelograms and the pink diamond are the mixed cells of the subdivision. The sum of their areas is equal to $4=\MV(\Adot)$ verifying Lemma \ref{lem:mixedvolume5}. Figure \ref{fig:subdivision3} shows the projections of these lower facets.
\begin{figure}[htpb!!]
\includegraphics[scale=0.45]{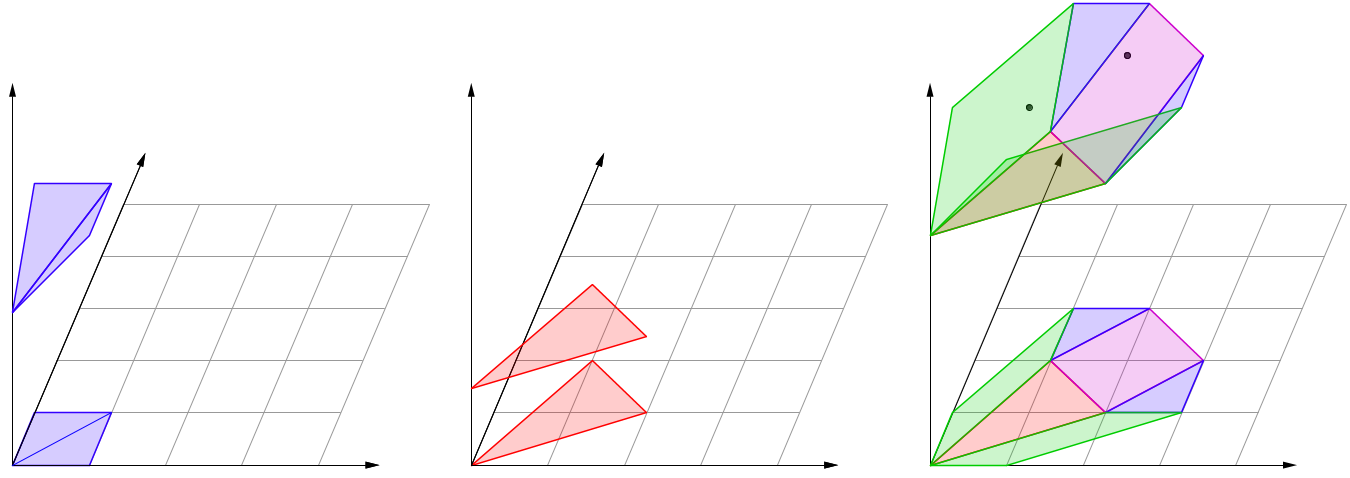}
\caption{A coherent fine mixed subdivision.}
\label{fig:subdivision2}
\end{figure}\begin{figure}[htpb!!]
\includegraphics[scale=0.5]{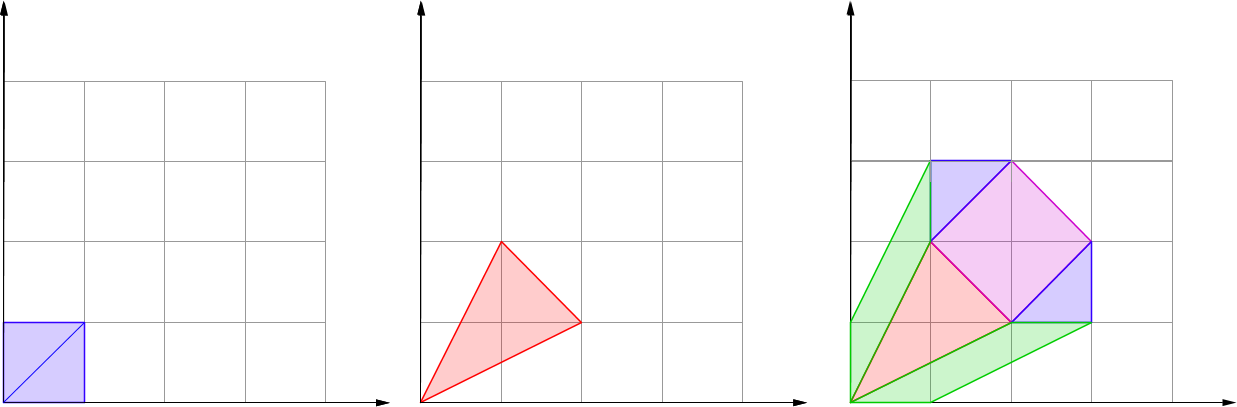}
\caption{The projection of a coherent fine mixed subdivision.}
\label{fig:subdivision3}
\end{figure}\hfill $\diamond$
\end{example}
\section{Monotonicity and positivity of mixed volume}
\label{subsection:monotonicity}
The \mydef{defect} of a collection of polytopes $P=\{P_1,\ldots,P_k\}$ is
$$\mydefMATH{d(P)}=\dim\left(\sum_{i=1}^k P_i\right)-k.$$
We say $P$ is \mydef{essential} if the defect of any nonempty subset of $P$ is nonnegative. 
\begin{lemma}
\label{lem:whenmixedvolumeiszero}
A collection of polytopes $\Pdot=\{P_1,\ldots, P_n\}$ in $\R^n$ has positive mixed volume if and only if $\Pdot$ is essential.
\end{lemma}

Mixed volume is monotonic with respect to inclusion: if $P_1$ and $Q_1,\ldots,Q_n$ are polytopes in $\R^n$ where $P_1 \subset Q_1$, then 
\begin{equation}
\label{eq:monotonicity} 
\MV(P_1,Q_2,\ldots,Q_n) \leq \MV(Q_1,\ldots,Q_n).
\end{equation} On the other hand, $P_1 \subsetneq Q_1$ does not imply that the inequality \eqref{eq:monotonicity} is strict. 

Conditions for strict monotonicity were originally determined by Maurice Rojas \cite{Rojas} in 1994 but have since been rediscovered  for the unmixed case \cite{EsterovMonotonicity}  ten years later and again rediscovered and explained in the mixed case \cite{Chen,Bihan} another ten years after that. The following version comes from \cite{Bihan}.  
 
A subset $S \subset P$ of a convex polytope \mydef{touches} a face $F$ of $P$ whenever $S \cap F$ is nonempty.
\begin{lemma}\cite[Proposition 3.2]{Bihan}
\label{lem:littlemonotonicity}
Let $P_1$ and $Q_\bullet=(Q_1,\ldots,Q_n)$ where $P_1,Q_1,\ldots,Q_n$ are polytopes in $\R^n$ such that $P_1 \subset Q_1$. Then $\MV(P_1,Q_2,\ldots Q_n) = \MV(Q_\bullet)$ if and only if $P_1$ touches every face $(Q_1)_\omega$ for $\omega$ in the set 
$$U=\{\omega \in \R^n \mid \{(Q_2)_\omega,\ldots,(Q_n)_\omega\} \text{ is essential}\}.$$
\end{lemma}

\begin{figure}[htpb!]
\includegraphics[scale=0.4]{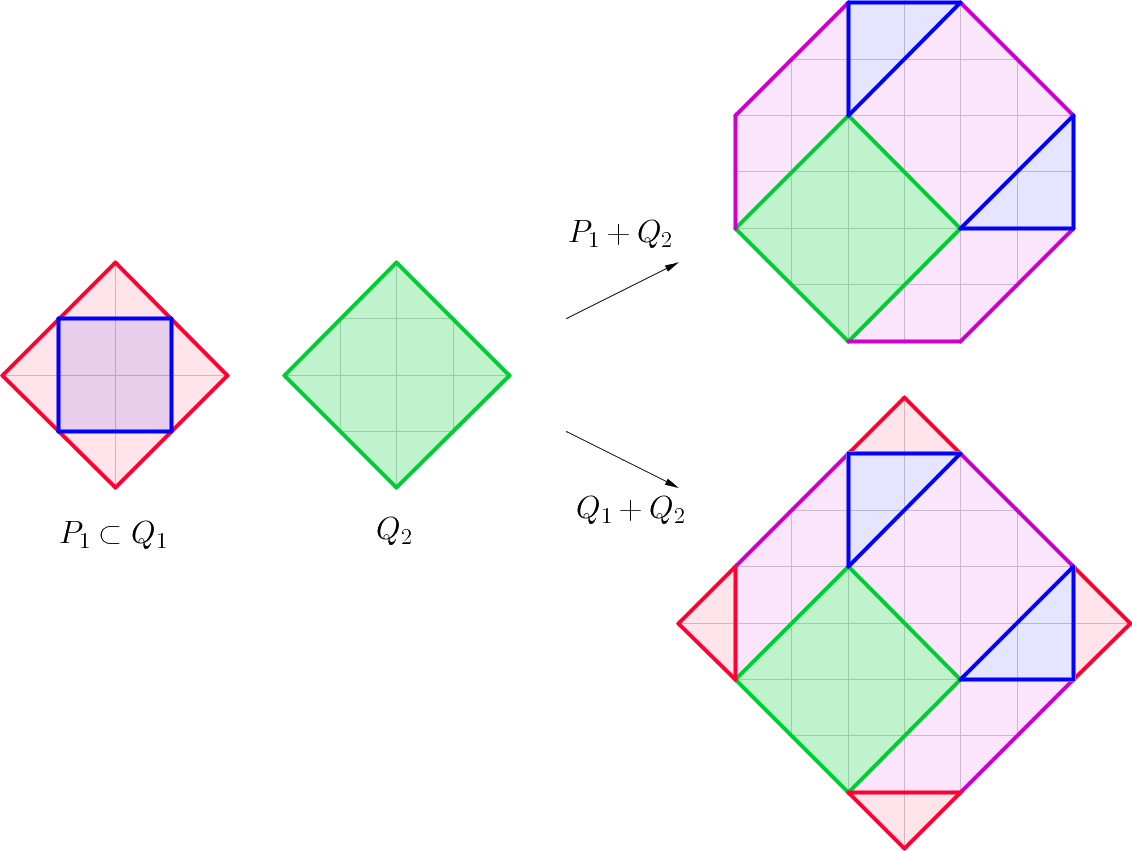}
\caption{Polytopes $P_1,Q_1,$ and $Q_2$ as in Example \ref{ex:monotonicity} along with a fine mixed subdivision displaying that $\MV(P_1,Q_2)=\MV(Q_1,Q_2)$.}
\label{fig:monotonicity}
\end{figure} 
\begin{example}
\label{ex:monotonicity}
Let $Q_1=Q_2=\conv(e_1,-e_1,e_2,-e_2)$ and let $P_1 = [0,1]^2$. 
The collection $U$ in Lemma \ref{lem:littlemonotonicity} is the set $\{e_1+e_2,e_1-e_2,-e_1+e_2,-e_1-e_2\}$ of directions exposing the facets of $Q_1$. Indeed, for every $\omega \in U$, we have $P_1 \cap (Q_1)_\omega \neq \emptyset$, so $P_1$ touches each facet. Figure \ref{fig:monotonicity} displays $P_1 \subset Q_1$ and $Q_2$ along with a depiction of a mixed subdivision of each of the sums $P_1+Q_2$ and $Q_1+Q_2$. The mixed cells of each subdivision are the same and so the pairs of polytopes have the same mixed volume, $MV(P_1,Q_2)=\MV(Q_1,Q_2) = 16$.\hfill $\diamond$
\end{example}

Given two collections of polytopes $P_\bullet = (P_1,\ldots,P_n)$ and $Q_\bullet=(Q_1,\ldots,Q_n)$ in $\R^n$ with $P_i \subset Q_i$, one may either iterate Lemma \ref{lem:littlemonotonicity} to determine strict monotonicity or use the following generalized version.
\begin{lemma}\cite[Theorem 3.3]{Bihan}
\label{lem:monotonicity}
Let $\Pdot=(P_1,\ldots,P_n)$ and $Q_\bullet=(Q_1,\ldots,Q_n)$ be collections of polytopes in $\R^n$ such that $P_i \subset Q_i$ for $i=1,\ldots,n$. 
For $\omega \in \R^n$ let $$T_\omega=\{i \in [n] \mid P_i \text{ touches } (Q_i)_\omega\}.$$
Then $$\MV(\Pdot) < \MV(Q_\bullet)$$ if and only if there exists $\omega$ such that the collection $\{(Q_i)_\omega \mid i \in T_\omega\} \cup \{Q_i \mid i \in [n] \smallsetminus T_\omega\}$ is essential. 
\end{lemma}

\chapter{ALGEBRAIC GEOMETRY \label{section:algebraicgeometry}}

Algebraic geometry is the study of solution sets of polynomial equations. Such sets are called varieties and there is an intimate dictionary between the algebraic properties of collections of polynomials and the geometric properties of the varieties they define. 

We explain a small subset of algebraic geometry relevant to this dissertation. For a more thorough treatment of algebraic geometry we invite the reader to consult \cite{CLO,HarrisBook,Hartshorne,Shafarevich}. In particular,  \emph{Ideals, Varieties, and Algorithms} by Cox, Little, and O'shea \cite{CLO} takes a concrete and computational approach to solutions of polynomial equations that is suitable for undergraduates.

Throughout this section, we write $\mathbb{C}[x]$ for the polynomial ring $\mathbb{C}[x_1,\ldots,x_n]$ in $n$ variables with coefficients in $\C$. Given a collection $F \subset \C[x]$, we write $\langle F \rangle$ for the ideal in $\C[x]$ generated by all elements of $F$.  When working in few variables, we use the more familiar variables of $x,y,z,$ and $w$ in that order.

\section{Affine varieties} 
We denote $n$-dimensional complex affine space by
$$\mydefMATH{\C^n}=\{(a_1,a_2,\ldots,a_n) \mid a_i \in \C, \quad i=1,\ldots,n\}.$$ For any subset $F \subset \mathbb{C}[x]$, the \mydef{affine variety} defined by $F$ is
$$\mydefMATH{\V(F)} = \{(a_1,\ldots,a_n) \in \mathbb{C}^n \mid f(a_1,\ldots,a_n)=0 \text{ for all } f \in F \}\subset \C^n.$$
We also refer to $\V(F)$ as the \mydef{vanishing locus} of $F$, the \mydef{zero set} of $F$, or the affine variety \mydef{cut out} by $F$. It is worth mentioning that many texts refer to $\V(F)$ as an ``affine algebraic set'' and reserve the term ``affine variety'' for a more specific object. If $f \in \C[x]$ and $f(a)=0$ for some $a \in \C^n$, then we say that $f$ \mydef{vanishes} at $a$.  We sometimes will decorate the notation $\C^n$ with subscripts to indicate the coordinates involved. For example, $\V(y-x^2) \subset \C_{x,y}^2$.

If $X\subset Y$ are both affine varieties, we say $X$ is a \mydef{subvariety} of $Y$. The set $\C^n$ is an affine variety cut out by $\{0\} \subset \C[x]$.
We list some affine subvarieties of $\C^2$ in Figures \ref{fig:Cn}-\ref{fig:point}.
\begin{figure}[!htpb]
\begin{minipage}{0.3\textwidth}
\begin{center}
\includegraphics[scale=0.35]{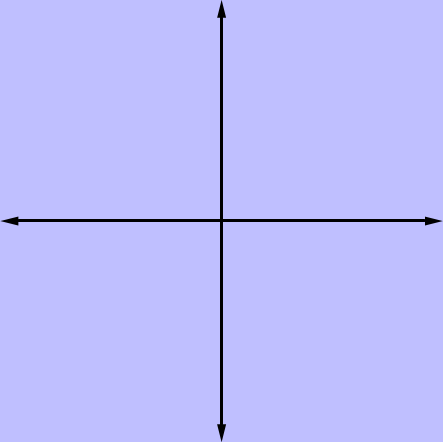}
\caption{The set $\V(0)$ defines $\C^2 \subset \C^2$.}
\label{fig:Cn}
\end{center}
\end{minipage}
\hspace{0.1 in}
\begin{minipage}{0.3\textwidth}
\begin{center}
\includegraphics[scale=0.35]{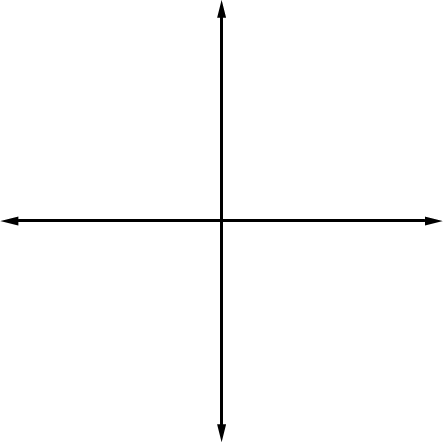}
\caption{The set $\V(1)$ defines $\emptyset \subset \C^2$.}
\label{fig:emptyset}
\end{center}
\end{minipage}
\hspace{0.1 in}
\begin{minipage}{0.3\textwidth}
\begin{center}
\includegraphics[scale=0.35]{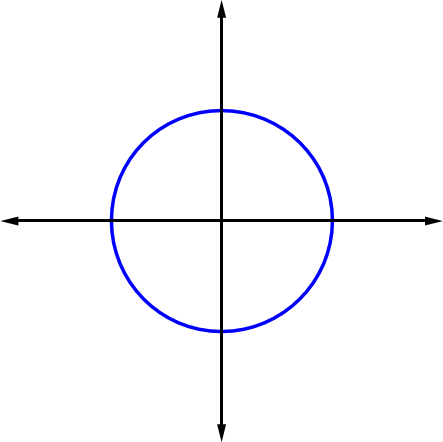}
\caption{The set $\V(x^2+y^2-1)$ defines the unit circle in $\C^2$.}
\label{fig:unitCircle}
\end{center}
\end{minipage}
\end{figure}
\begin{figure}[!htpb]
\begin{minipage}{0.45\textwidth}
\begin{center}
\includegraphics[scale=0.5]{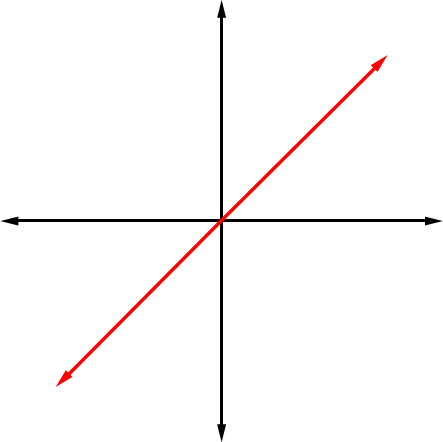}
\caption{The set $\V(x-y)$ defines a line in $\C^2$.}
\label{fig:line}
\end{center}
\end{minipage}
\hspace{0.1 in}
\begin{minipage}{0.45\textwidth}
\begin{center}
\includegraphics[scale=0.5]{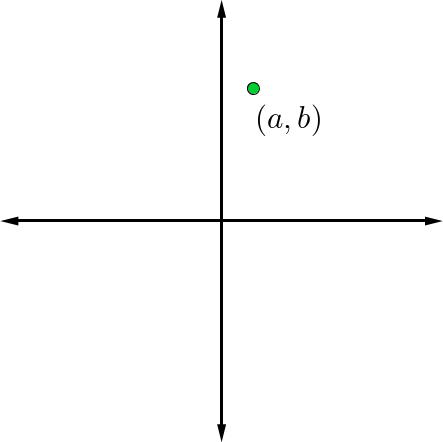}
\caption{The set $\V(x-a,y-b)$ defines a single point $(a,b) \subset \C^2$.}
\label{fig:point}
\end{center}
\end{minipage}
\end{figure}

For any subset $S\subset \C^n$ (not necessarily a variety), we denote the set of all polynomials which vanish on $S$ by 
$$\mydefMATH{\mathcal I(S)}=\{f \in \C[x] \mid f(s)=0 \text{ for all } s \in S\}.$$
This set is an ideal in the polynomial ring $\C[x]$ since if $f,g \in \mathcal I(S)$ and $h \in \C[x]$, then $f+hg \in \mathcal I(S)$ because $$f(s)+h(s)g(s) = 0+h(s)\cdot 0=0 \quad \text{ for all } s \in S.$$ Hence, we call $\mathcal I(S)$ the \mydef{ideal} of $S$. At this point, we may think of $\V$ and $\I$ as the functions,
\begin{align*}
\V &\colon\left\{\text{subsets of }\C[x]\right\} \to \left\{\text{subsets of }\C^n\right\}\\
\I &\colon\left\{\text{subsets of }\C^n\right\} \to \left\{\text{subsets of }\C[x]\right\}.
\end{align*}
\begin{lemma}
\label{lem:reversecorrespondence}
The functions $\V$ and $\I$ are inclusion reversing:
\begin{enumerate}
\item If $S_1 \subset S_2 \subset \C^n$ then $\I(S_2) \subset \I(S_1)$.
\item If $F_1 \subset F_2 \subset \C[x]$ then $\V(F_2) \subset \V(F_1)$. 
\end{enumerate}
\end{lemma}
\begin{proof}
Suppose $S_1 \subset S_2 \subset \C^n$. Then any polynomial vanishing on $S_2$ vanishes on the subset $S_1$ and so $\I(S_2) \subset \I(S_1)$. 
Suppose that $F_1 \subset F_2 \subset \C[x]$. Then if every element of $F_2$ vanishes at some $a \in \C^n$, then every element of the subset $F_1$ vanishes at $a$ as well.
\end{proof}

\begin{lemma}
\label{lem:vanishingisvanishingofideal}
For any subset $F \subset \C[x]$, we have $\V(F)=\V(\langle F \rangle)$.
\end{lemma}
\begin{proof}
If $g \in \langle F \rangle$, then 
\begin{equation}
\label{eq:inideal}
g=\sum_{f \in F} h\cdot f, \quad h \in \C[x],
\end{equation}
and so evaluating the sum at a point $a \in \V(F)$ shows that $g(a)=\sum\limits_{f \in F} h(a)\cdot 0=0$ and thus $\V(F) \subseteq \V(\langle F \rangle)$. Conversely, since $F \subset \langle F \rangle$, we have $\V(\langle F \rangle) \subseteq \V(F)$ proving equality.
\end{proof}

\begin{proposition}[Hilbert's Basis Theorem \cite{Hilbert}]
\label{prop:hilbertbasis}
Every ideal $I \in \C[x]$ may be written as $I=\langle f_1,\ldots,f_k \rangle$ for some $k \in \mathbb{N}$ and $f_i \in \C[x]$.
\end{proposition}
A more general version of Hilbert's Basis Theorem states that the polynomial ring $R[x]$ over any Noetherian ring $R$ is Noetherian. Hilbert proved the case when $R$ is either a field or the ring of integers \cite{Hilbert}. 
Consequently, when studying affine varieties $X=\V(F)\subset \C^n$, we may assume that $F$ is finite.

Given an affine variety $X=\V(f_1,\ldots,f_k)\subset \C^n$, declare that the subvarieties of $X$ of the form $X \cap \V(g_1,\ldots,g_m)$ for some $g_1,\ldots,g_m \in \C[x]$ are \mydef{closed}.  Lemma \ref{lem:unionsandintersections} along with the facts that $\emptyset$ and $\C^n$ are affine varieties prove that this gives a topology on $X=\C^n$, which we call the \mydef{Zariski topology}. 

\begin{lemma}
\label{lem:unionsandintersections}
Finite unions and arbitrary intersections of closed affine subvarieties of $\C^n$ are closed affine subvarieties of $\C^n$. 
\end{lemma}
\begin{proof}
Let $F,G\subset \C[x]$ be finite generating sets for the ideals $I$ and $J$ respectively. Then $$\V(I)\cap \V(J) = \V(I+J),$$
equivalently,
$$\V(F) \cap \V(J) = \V(F \cup G).$$
These intersections may be taken to be arbitrary by Hilbert's Basis Theorem.
Finite unions are also varieties since, 
$$\V(I)\cup \V(J) = \V(IJ),$$
or equivalently,
$$\V(F) \cup \V(G) = \V(\{f\cdot g \mid f \in F, g \in G\}),$$
completing the proof.
\end{proof}
Figures \ref{fig:intersection} and \ref{fig:union} display examples of unions and intersections of varieties.

\begin{figure}[!htpb]
\begin{minipage}{0.45\textwidth}
\begin{center}
\includegraphics[scale=0.5]{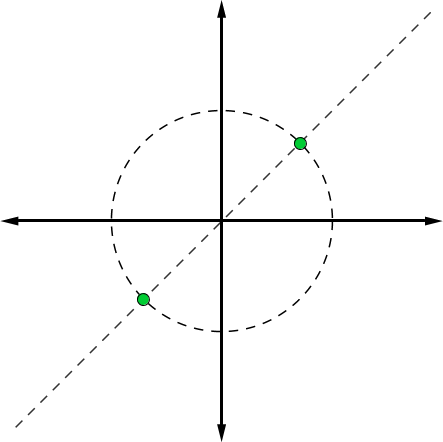}
\caption{The set $\V(x^2+y^2-1,x-y)$ defines two points.}
\label{fig:intersection}
\end{center}
\end{minipage}
\hspace{0.1 in}
\begin{minipage}{0.45\textwidth}
\begin{center}
\includegraphics[scale=0.5]{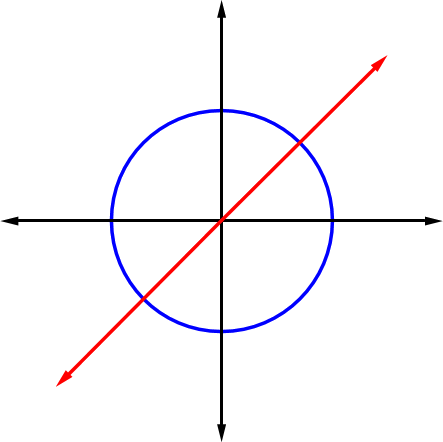}
\caption{The set $\V((x^2+y^2-1)\cdot (x-y))$ defines the union of the unit circle and a line.}
\label{fig:union}
\end{center}
\end{minipage}
\end{figure}
The Zariski topology on a closed subvariety of $\C^n$ is the subspace topology inherited from the Zariski topology on $\C^n$.
Affine varieties come equipped with a second topology: the subspace topology inherited from the Euclidean topology on $\C^n \cong \R^{2n}$. The Zariski topology is weaker than the Euclidean topology in the sense that closed/open sets in the Zariski topology are closed/open in the Euclidean topology but the converse is very much not true.

For any subset $S \subset \C^n$, denote its closure in the Zariski topology by $\overline S$. The following lemma is dual to Lemma \ref{lem:vanishingisvanishingofideal}.
\begin{lemma}
For any subset $S \subset \C^n$ we have $\I(S)=\I(\overline S)$. 
\end{lemma}
\begin{proof}
We have $\I(S) \supset \I(\overline S)$ immediately. 
Suppose $f \in \I(S)$ so that $f(s)=0$ for all $s \in S$.  If $f \not\in \I(\overline S)$ then there exists some point $s' \in \overline S$ such that $f(s') \neq 0$ implying that $\overline S \cap \V(f)$ is a variety which is strictly smaller than $\overline S$ and contains $S$, a contradiction.
\end{proof}
Even when restricted to ideals and closed affine varieties, the functions $\V$ and $\I$ are not inverses of each other. It is true that $\V(\I(X))=X$ for any closed affine variety $X \subset \C^n$, but it is not true that $\I(\V(I))=I$ for any ideal $I \subset \C[x]$. For example $\I(\V(\langle x^2 \rangle)) = \langle x \rangle$. For $\V$ and $\I$ to be inverses of each other, we must restrict the domain of $\V$ to the subset of ideals satisfying $f^m \in I \iff f \in I$, called \mydef{radical ideals}. For any ideal $I$, the set $\mydefMATH{\sqrt{I}}=\{f \in \C[x] \mid f^m \in I \text{ for some } m \in \mathbb{N}\}$ is a radical ideal called the \mydef{radical} of $I$.
\begin{proposition}[Hilbert's Nullstellensatz \cite{HilbertNull}]
Given an ideal $I \subset \C[x]$,
$$\I(\V(I)) = \sqrt{I}.$$
\end{proposition}
The Nullstellensatz implies that with a further restriction to radical ideals, the functions
\begin{align*}
\V &\colon\left\{\text{radical ideals in }\C[x]\right\} \to \left\{\text{closed affine subvarieties of }\C^n\right\}\\
\I &\colon\left\{\text{closed affine subvarieties of }\C^n\right\} \to \left\{\text{radical ideals of }\C[x]\right\} 
\end{align*}
are inverses.
As corollaries we have that $\V(I)=\V(\sqrt{I})$ and that $\V(I)\subset \C^n$ is empty if and only if $I=\C[x]$.

Every polynomial $f \in \C[x]$ defines a function \begin{align*} f\colon\C^n &\to \C\\ x &\mapsto f(x).
\end{align*}
A \mydef{regular function} on an affine variety $X \subset \C^n$ is the restriction of a polynomial function on $\C^n$ to $X$. Two regular functions $f$ and $g$ on $X$ are the same if and only if $f-g \in \I(X)$ and so the regular functions on $X$ are identified with equivalence classes in the quotient ring $\mydefMATH{\C[X]}=\C[x]/\I(X)$ called the \mydef{coordinate ring} of $X$. Just as regular functions on affine varieties are restrictions of polynomials, a \mydef{regular map}  of affine varieties $X \subset \C^n, Y \subset \C^m$ is any function
\begin{align*}
\varphi\colon X &\to Y \\
x &\mapsto (\varphi_1(x),\ldots,\varphi_m(x))
\end{align*} where each $\varphi_i\colon X \to \C$ is a regular function. We say $\varphi$ is an \mydef{isomorphism} if it is bijective and its inverse is also a regular map.

A regular map $\varphi\colon X \to Y$ of affine varieties naturally induces a $\C$-algebra homomorphism on the coordinate rings of $X$ and $Y$ in the opposite direction: 
\begin{align*}
\mydefMATH{\varphi*}\colon \C[Y] &\to \C[X]\\
 f & \mapsto f \circ \varphi.
\end{align*}  Conversely, given any $\C$-algebra homomorphism $\phi\colon\C[Y] \to \C[X]$, with $\C[Y]=\C[y]/\I(Y)$ and $\C[X]=\C[x]/\I(X)$ let $[g_i] \in \C[X]$ be the image of $[y_i]$ under $\phi$. The map, 
\begin{align*}
\mydefMATH{\phi^\#}\colon X &\to Y\\
x &\mapsto (g_1(x),\ldots,g_m(x)),
\end{align*}
is a regular map of varieties. 
Note then that $\varphi$ is an isomorphism of affine varieties if and only if $\varphi^*$ is a $\C$-algebra isomorphism. 

\begin{example}
Given an affine variety $\V(f_1,\ldots,f_k) = X \subset \C^n$, subvarieties of $X$ are not always closed. To see this, consider the open subset $U_f=X \smallsetminus \V(f)$ for some $0 \neq f \in \C[X]$.  While $U_f$ cannot be expressed as $X \cap \V(g_1,\ldots,g_r)$ for any collection $g_1,\ldots,g_r \in \C[x]$ ($U_f$ is not closed) it can still be given the structure of a variety in the following way. 

Introduce a new variable $z$ and consider $Y=\V(f_1,\ldots,f_k) \cap \V(fz-1) \subset \C^{n+1}$. Here, $Y$ is a closed subvariety of the variety cut out by the same equations as $X$ considered in a higher dimensional space. The coordinate ring $\C[Y]$ is isomorphic to $\C[X][\frac{1}{f}]$ via the map $z \mapsto \frac{1}{f}$. This gives $U_f$ the structure of an affine variety and hence we call it a \mydef{principal affine open} subvariety of $X$. 
\end{example}
\section{Projective varieties}
The fundamental theorem of algebra states that a univariate polynomial of degree $d$ has $d$ complex zeros, counted with multiplicity. This fact does not hold over the real numbers and so extending the notion of polynomial equations over $\R$ to those over $\C$ casts the real case into a larger picture which is better behaved. Similarly for varieties, we extend the notion of affine varieties to projective varieties. Doing so produces a more unified understanding of affine varieties.

We wish to keep the notation of $\C[x]$ for a polynomial ring in $n$ variables, and so many of our statements will involve $\P^{n-1}$ rather than $\P^n$. When we write this, we assume $n\geq 2$. 
\begin{definition}
Define the equivalence $\sim$ on the set
 $\C^{n} \smallsetminus \{\bzero\}$ by setting $x=(x_1,\ldots,x_n) \sim (y_1,\ldots,y_n)=y$ if and only if $y=\lambda x$ for some $\lambda \in \C \smallsetminus \{0\}$.  \mydef{Projective $(n-1)$-space} is the quotient
$$\mydefMATH{\P^{n-1}}=(\C^{n}\smallsetminus \{\bzero\})/\sim.$$
We write the equivalence class of $(a_1,\ldots,a_n)$ in $\P^{n-1}$ as $[a_1:\cdots:a_n]$.  
\end{definition}
The zeros of a polynomial $f \in \C[x]$ are well-defined on $\P^{n-1}$ whenever $f$ satisfies the condition
\begin{align*}
f(x)=0\text{ if and only if }f(\lambda x)=0\text{, for any }\lambda \in \C \smallsetminus \{0\}.
\end{align*} This property is equivalent to $f$ being homogeneous.
A polynomial $$f=\sum_{\alpha \in \mathcal A} c_\alpha x_1^{\alpha_1}\cdots x_n^{\alpha_n} \in \C[x], \quad c_\alpha \in \C \smallsetminus \{0\}$$ is \mydef{homogeneous} of degree $d$ if $|\alpha|=d$ for all $\alpha \in \mathcal A$. Denote the set of homogeneous polynomials of degree $d$ by $\mydefMATH{\C[x]_d}$. If a polynomial is not homogeneous, we say it is \mydef{inhomogeneous}. 

One may erroneously guess that since the zeros of $f \in \C[x]$ are well-defined on $\P^{n-1}$ if and only if $f$ is homogeneous, then the zeros of $\{f_1,\ldots,f_k\}\subset \C[x]$ are well-defined if and only if $f_1,\ldots,f_k$ are homogeneous, but this is \textbf{not} necessary. For example, the zero set of $\{x^3+x^2+y^2-z^2,x\}$ are the points $\{[0:1:-1],[0:-1:1]\} \in \P^2$. Due to the argument in Lemma \ref{lem:vanishingisvanishingofideal}, the zeros of $\{x^3+x^2+y^2-z^2,x\}$ are the same as the zeros of $I=\langle x^3+x^2+y^2-z^2,x\rangle=\langle x^2+y^2-z^2,x \rangle$. Ideals such as $I$ which can be generated by homogeneous elements are called \mydef{homogeneous ideals}. The common zeros of a collection $F \subset \C[x]$ are well-defined on projective space exactly when $\langle F \rangle$ is a homogeneous ideal.
\begin{definition}
Let $F \subset \C[x]$ be a collection of polynomials such that $\langle F \rangle$ is a homogeneous ideal. The \mydef{projective variety} defined by $F$ is
$$\mydefMATH{\V(F)} = \{[a_1:\cdots:a_n] \in \P^{n-1} \mid f([a_1:\cdots:a_n]) = 0 \text{ for all } f \in F\} \subset \P^{n-1}.$$
\end{definition}
Since $\P^{n-1}$ is a quotient of $\C^n\smallsetminus \{\bzero\}$ with projection $\pi\colon \C^{n}\smallsetminus \{\bzero\} \to \P^{n-1}$, any subset $S \subset \P^{n-1}$ may be pulled back to the subset $$\mydefMATH{\mathcal C X}=\overline{\pi^{-1}(S)}\cup \bzero \subset \C^n,$$ called the \mydef{affine cone} over $S$. 
\begin{figure}[!htpb]
\includegraphics[scale=0.35]{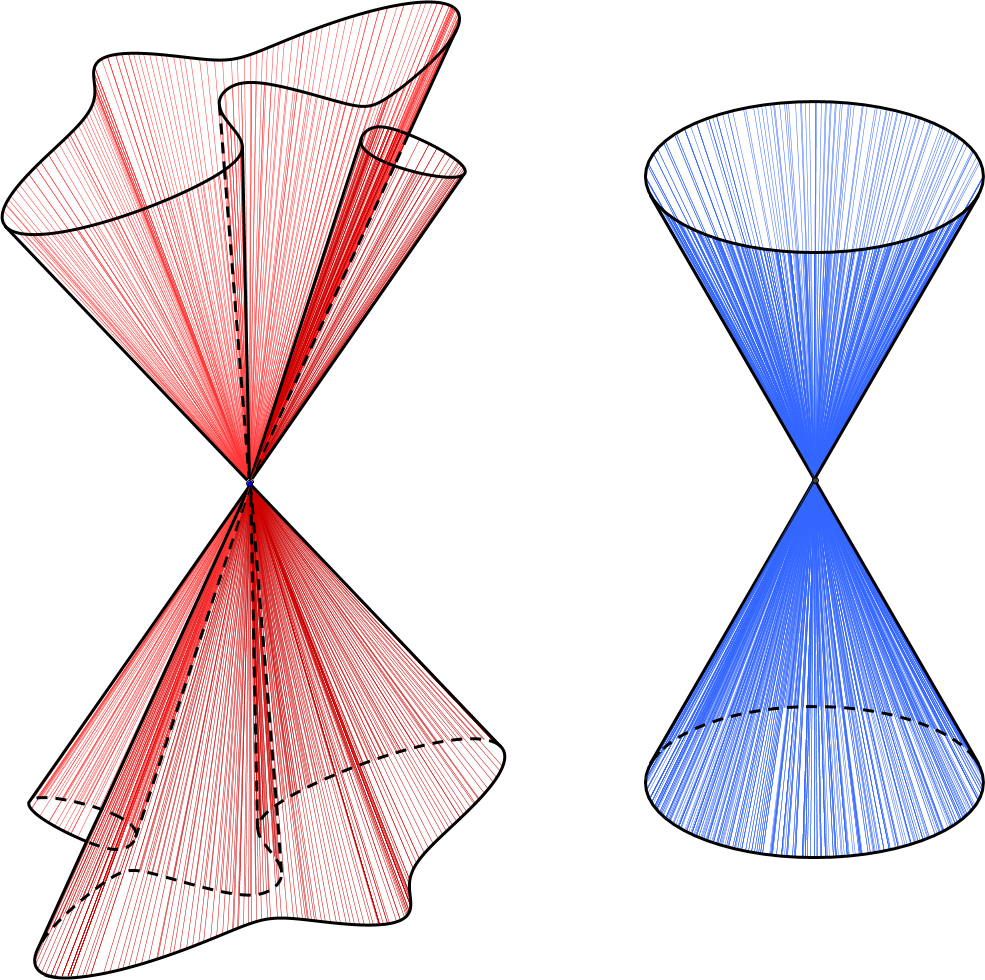}
\caption{Left: An affine cone over a quartic curve in $\P^2$. Right: An affine cone over a circle in $\P^2$.}
\label{fig:projective}
\end{figure}
For any subset $S \subset \P^{n-1}$, we define the set $\mydefMATH{\I(S)}$ to be the set of polynomials which vanish on the cone over $S$. This is an ideal, and the following proves something stronger.
\begin{lemma}
\label{lem:idealishomogeneous}
If $S \subset \mathbb{P}^{n}$, then the ideal $\I(S)$ is homogeneous.
\end{lemma}
\begin{proof}
Suppose $f$ is a non-homogeneous generator of $\I(S)$ given as  
$$f=\sum_{i=0}^k f_i(x),$$ where each $f_i(x)$ is homogeneous. Since $f$ vanishes on a subset of projective space, for any $s \in S$, we have $f(s)=0$ if and only if $f(\lambda s)=0$ for any $\lambda \in \C \smallsetminus \{0\}$. On the other hand 
$$f(\lambda s) = \sum_{i=0}^k \lambda^{\deg(f_i)}f_i(s),$$ is a polynomial in $\lambda$ which must vanish whenever $s \in S$. Thus, thinking of $\lambda$ as a variable, each coefficient $f_i(s)$ of $\lambda^i$ must be zero. This implies that $f_i(x) \in \I(S)$ and in particular, $f_0=0$. Thus, $f$ may be replaced as a generator of $\I(S)$ with the finite set $\{f_i\}_{i=1}^k$ since $f \in \langle f_1,\ldots,f_k \rangle\subset \I(S)$.
\end{proof}
The sets $\P^{n-1}=\V(0)$ and $\emptyset=\V( x_1,\ldots,x_n ) \subset \P^{n-1}$ are projective varieties. For any projective variety $X \subset \P^{n-1}$, declaring subvarieties of the form $X \cap \V(I)\subset X$  to be closed gives a topology by the same arguments as in the affine case. This topology is also called the \mydef{Zariski topology}.
The ideal $\mydefMATH{\mathfrak m_0}=\langle x_1,\ldots,x_n \rangle$ is called the \mydef{irrelevant ideal} since for any homogeneous ideal $I$,  $\V(I\cdot \mathfrak m_0) \subset \P^{n-1}$ is the same as $\V(I)$. Since $\bzero$ is always contained in the cone over $S$ the ideal $\I(S)$ is always contained in the irrelevant ideal.

The same arguments as in the affine case show that the following basic properties of the functions 
\begin{align*}
\V&\colon \{\text{homogeneous ideals in } \C[x]\} \to \{\text{closed projective subvarieties of }\P^{n-1}\} \\
\I&\colon \{\text{subsets of }\P^{n-1}\} \to \{\text{homogeneous ideals in } \C[x] \text{ containing } \mathfrak m_0 \}
\end{align*}
hold projectively.
\begin{enumerate}
\item $\V$ and $\I$ are inclusion reversing,
\item $\V(F) = \V(\langle F \rangle)$,
\item $\I(S) = \I(\overline S)$,
\item $\sqrt{I\cdot \mathfrak m_0} = \I(\V(I))$ (projective Nullstellensatz).
\end{enumerate}
Thus, the functions
\begin{align*}
\V&\colon\{\text{homog. radical ideals in }\C[x] \text{ contained in } \mathfrak m_0 \} \to \{\text{closed projective subvarieties of }\P^{n-1}\} \\
\I&\colon \{\text{closed projective subvarieties of }\P^{n-1}\} \to \{\text{homog. radical ideals in }\C[x] \text{ contained in } \mathfrak m_0 \} 
\end{align*}
are inclusion-reversing inverses.

\section{Charts on projective space}
Consider the open set $$\mydefMATH{U_i} = \P^{n-1}\smallsetminus \V(x_i).$$
Since any point in projective space has some nonzero coordinate, the $U_i$ cover $\P^{n-1}$ and every point in $U_i$ has a unique representative  of the form
$$\left(\frac{x_1}{x_i}, \ldots ,\frac{x_{i-1}}{x_i},1,\frac{x_{i+1}}{x_i}, \ldots,\frac{x_n}{x_i}\right).$$ The maps 
\begin{align*}\varphi_i\colon U_i &\to \C^{n-1} \\
[x]&\mapsto \left(\frac{x_1}{x_i},\ldots,\frac{x_{i-1}}{x_i},\frac{x_{i+1}}{x_i},\ldots,\frac{x_{n}}{x_i} \right)
\end{align*} are charts for $\P^{n-1}$ as a manifold. We call these the \mydef{standard affine open charts} for $\P^{n-1}$ as they identify each $U_i$ with an affine space. We sometimes will refer to the $U_i$ themselves as charts.

\section{Homogenizing and dehomogenizing}
Let $X \subset \P^{n-1}$ be a projective variety. For $X=\V(f_1,\ldots,f_k)\subset \P^{n-1}$ with $f_i \in \C[x]$, the affine cone of $X$ is simply $\mathcal C X=\V(f_1,\ldots,f_k) \subset \C^n$. For any variable $x_i$, the intersection of the affine cone of $X$ with the hyperplane $\mydefMATH{H_i}=\V(x_i-1) \cong \C^{n-1}$ is a closed affine subvariety $\mathcal CX \cap \V(x_i-1)  \subset \C^{n-1}$ called the \mydef{dehomogenization} of $X$ with respect to $x_i$. Identifying $H_i$ with $\C^{n-1}$ via the standard affine open chart $\varphi_i$, the dehomogenization of $X$ with respect to $x_i$ is the same as the image of $\varphi_i\colon X \cap U_i \to \C^{n-1}$.

Conversely, suppose $X$ is an affine subvariety of $\C^{n-1}$.  By introducing a new coordinate $x_{n}$ we define the \mydef{projective closure} of $X$ as
\begin{equation}
\label{eq:projectiveclosure}
\mydefMATH{\overline{X}} = \overline{\{[x:1] \mid x \in X\}} \subset \P^{n-1}.
\end{equation} This is the same as taking the closure $\overline{\varphi_n^{-1}(X)}$ where $\varphi_n$ is the standard affine open chart on $\P^{n-1}$. When considering the projective closure \eqref{eq:projectiveclosure} of an affine variety, we call the hyperplane $\mydefMATH{H^{\infty}_{n}} = \V(x_n)$ the \mydef{hyperplane at infinity}.  Of course, the processes of projectively closing an affine variety and dehomogenizing a projective variety may be done with respect to any hyperplane $H \subset \C^n$ not passing through the origin, via the exact same geometric procedure. In these cases, the corresponding hyperplane at infinity is the hyperplane $\mydefMATH{H^\infty}$ through the origin with the same normal direction as $H$.

The dehomogenization of $\overline{X}\subset \P^{n-1}$ with respect to $x_{n}$ is exactly $X$ and writing equations for a dehomogenization is straightforward: if $F=\{f_1,\ldots,f_k\} \subset \C[x]$ is a collection of homogeneous polynomials, then the dehomogenization of $\V(F)$ with respect to the variable $x_i$ is the affine variety  $$\V(g_1,\ldots,g_k) \subset \C^{n-1},$$ where $\mydefMATH{g_j}:=f_j(x_1,\ldots,x_{i-1},1,x_{i+1},\ldots,x_n)$ is the \mydef{dehomogenization} of $f_j$ with respect to $x_i$.
The inverse task of producing the algebraic equations for $\overline{X}$ from those for $X$ is much more difficult. Given a polynomial $$f=\sum_{\alpha \in \mathcal A} c_\alpha x_1^{\alpha_1}\cdots x_{n-1}^{\alpha_{n-1}} \in \C[x_1,\ldots,x_{n-1}], \quad c_\alpha \in \C \smallsetminus \{0\}$$ of degree $d$, the \mydef{homogenization} of $f$ with respect to a new variable $x_n$ is the polynomial $$\mydefMATH{\tilde{f}}=\sum_{\alpha \in \mathcal A} c_\alpha x_1^{\alpha_1}\cdots x_{n-1}^{\alpha_{n-1}} \cdot x_{n}^{d-|\alpha|}\in \C[x].$$ Similarly, for a subset $F$ of polynomials, let $\mydefMATH{\widetilde{F}}=\{\tilde{f}\}_{f \in F}$ be the \mydef{homogenization} of $F$. It is easy to see that dehomogenizing $\tilde{f}$ with respect to $x_{n}$ recovers $f$ and so the dehomogenization of $\V(\widetilde{F})\subset \P^{n-1}$ with respect to $x_{n}$ is $\V(F) \subset \C^{n-1}$. Unfortunately, $\V(\widetilde{F}) \neq \overline{\V(F)}$ and so merely homogenizing equations for an affine variety does not produce equations for its projective closure. In order for this to work, we must homogenize the ideal $I=\langle F \rangle$ generated by $F$; that is, $\V(\widetilde{I}) = \overline{\V(F)}$. Thus, the homogenization of the ideal generated by a collection of polynomials is \emph{not} the ideal generated by the homogenizations of those polynomials. We illustrate the failure of the na\"ive homogenization $\widetilde{F}$ to cut out the projective closure $\overline{\V(F)}$ in the following example.
\begin{example}
Let $F=\{xy-1,z-x^2\} \subset \C[x,y,z]$ define the set $C \subset \C^3$ called the twisted cubic displayed in Figure \ref{fig:twistedcubic}. 
\begin{figure}[!htpb]
\includegraphics[scale=0.6]{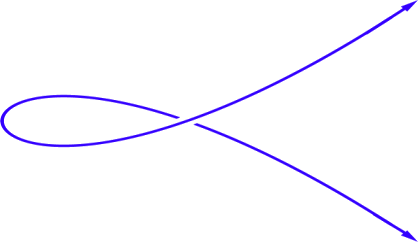}
\caption{A twisted cubic.}
\label{fig:twistedcubic}
\end{figure}
Let $I=\langle F \rangle$. The homogenization of $I$ with respect to $w$ is $\widetilde{I} = \langle xy-w^2,zw-x^2,yz-xw \rangle$, but the ideal generated by the homogenization of $F$ with respect to $w$ is $\langle \widetilde{F}\rangle =\langle xy-w^2,zw-x^2\rangle \subset \widetilde{I}$.  Notice that the line $\{[0:s:t:0] \mid [s:t] \in \P^1\}$ is contained in $\V(xy-w^2,zw-x^2)$ but not in $\V(xy-w^2,zw-x^2,yz-xw)$.\hfill $\diamond$
\end{example}

\section{Regular functions}

We define the \mydef{homogeneous coordinate ring} of a projective variety $X \subset \P^{n-1}$ to be the graded quotient ring $$\mydefMATH{\C[X]}=\C[x]/\I(X).$$ Note that this is the coordinate ring of $\mathcal C X$. 
Contrary to the affine case, most polynomials are not functions on any projective variety $X\subset \P^{n-1}$, rather, the \emph{only} polynomial functions on $X$ are constants: if $f \in \C[x]$ is not constant, then for $\lambda \in \C \smallsetminus \{0,1\}$ we have $f(\lambda x)=\lambda^{\deg(f)}f(x) \neq f(x)$.  

Despite there being almost no polynomial functions on a projective variety, there are still functions between projective varieties which are locally given by polynomials.
Given a collection $\{f_1,\ldots,f_k\} \subset \C[X]_d$ of polynomials of the same degree, the function 
\begin{align*}
f\colon X\smallsetminus \V(f_1,\ldots,f_k) &\to \P^{k-1}\\
x &\mapsto [f_1(x):\cdots:f_k(x)]
\end{align*}
is well-defined. If $\varphi\colon X \to \P^{k-1}$ is a function such that for every $x \in X$ there exist $f_1,\ldots,f_k \in \C[x]$ of the same degree such that $x \not\in \V(f_1,\ldots,f_k)$ and $$\varphi(y)=[f_1(y):\cdots:f_k(y)],\text{ for all } y \in X \smallsetminus  \V(f_1,\ldots,f_k), $$ then we say the map $\varphi$ is \mydef{regular}. Two projective varieties are isomorphic if there exist regular maps $\varphi\colon X \to Y$ and $\psi\colon  Y \to X$ which are inverses of each other. Regular maps of affine/projective varieties are continuous maps under the Zariski topology.

\section{Irreducibility and dimension}
The term ``variety'' without the adjectives ``affine'' or ``projective'' refers to either an affine variety or a projective variety.

\subsection{Irreducibility}
A \emph{nonempty} variety $X$ is \mydef{irreducible} if it satisfies
$$X=X_1 \cup X_2 \implies X=X_1 \text{ or } X=X_2,$$ whenever $X_1$ and $X_2$ are closed subvarieties of $X$. Otherwise, we say it is \mydef{reducible}.
A union $X=X_1\cup X_2 \cup \cdots \cup X_m$ of sets is \mydef{irredundant} if $X_i \not\subset X_j$ for any distinct $i,j \in [m]$. Note that if $X=X_1 \cup X_2$ is a witness for the reducibility of  a variety $X$, then this union is irredundant.
\begin{lemma}
\label{lem:varietiescanbedecomposed}
Every nonempty variety $X$ may be written as an irredundant union of finitely many irreducible closed subvarieties
$$X = X_1 \cup X_2 \cup \cdots \cup X_m.$$
\end{lemma}
\begin{proof}
Let $X$ be a variety. If it is irreducible, the lemma is satisfied. Otherwise, it may be written as a union $X=Y_1 \cup Y^{(1)}$ of proper closed subvarieties.
As a convention, we suppose that $Y^{(1)}$ is irreducible if $Y_1$ is, and otherwise we reorder them. Similarly, if $Y^{(1)}$ is reducible, we write $Y^{(1)}=Y_2 \cup Y^{(2)}$. Iteratively applying this process to $Y^{(j)}$ produces a proper infinite chain
$$X \supsetneq Y_1 \supsetneq Y_2 \supsetneq Y_3 \supsetneq \cdots$$
of closed subvarieties $Y_i \subset X$. 
Applying $\I$ to this chain gives an ascending chain of ideals which is proper since each $Y_i$ is closed, contradicting Hilbert's Basis Theorem. 
\end{proof}
\begin{lemma}
\label{lem:irrdecompisunique}
Let $X$ be a variety. If $X$ admits two irredundant decompositions,
$$X=X_1 \cup \cdots \cup X_m, \text{ and }
X=Y_1 \cup \cdots \cup Y_{m'},$$
into irreducible closed subvarieties,
then $m=m'$ and $\{X_1,\ldots,X_m\}=\{Y_1,\ldots,Y_{m'}\}.$
\end{lemma}
\begin{proof}
We will show that for all $i$, $Y_i=X_j$ for exactly one $j$. Consider
$$Y_i = X \cap Y_i = \bigcup_{j=1}^m X_j \cap Y_i.$$
Since $Y_i$ is irreducible, one of the sets in the union must equal $Y_i$, or equivalently, $X_j \cap Y_i=Y_i$ for some $j \in [m]$, implying that $Y_i \subseteq X_j$. Applying this argument to $X_j$ shows that $X_j \cap Y_k = X_j$ for some $k \in [m']$, implying that $X_j \subseteq Y_k$. Together, this implies $Y_i \subseteq Y_k$ and since the unions are irredundant, $Y_i=Y_k=X_j$.   Iterating this argument on $X=\bigcup_{k \neq j} X_k$ and $Y=\bigcup_{k \neq i} Y_k$ proves the result.
\end{proof}
We call the decomposition in Lemma \ref{lem:irrdecompisunique} the \mydef{irreducible decomposition} of $X$.

\begin{lemma}
Irreducible varieties are those whose ideals are prime.
\end{lemma}
\begin{proof}
Suppose $X$ is reducible, witnessed by $X_1 \cup X_2$, where $X_1,X_2$ are proper nontrivial closed subvarieties of $X$. Write $I_1 = \I(X_1)$ and $I_2 = \I(X_2)$ so that $I_1I_2 = \I(X)$. Picking $f_1 \in I_1 \smallsetminus I_2$ and $f_2 \in I_2 \smallsetminus I_1$, we see that $f_1,f_2 \not\in \I(X)$ but $f_1  f_2 \in \I(X)$ so $\I(X)$ is not prime. Conversely, suppose $\I(X)$ is not prime, witnessed by $f_1,f_2 \not\in \I(X)$ yet $f_1 f_2 \in \I(X)$. Let $I_1 = \langle f_1 \rangle+\I(X)$ and $I_2 = \langle f_2 \rangle+\I(X)$. We claim that $X=X_1\cup X_2$ where $X_1=\V(I_1)$ and $X_2=\V(I_2)$. Both $X_1,X_2 \subset X = \V(\I(X))$ since $\I(X) \subset I_1(X)$ and $\I(X) \subset I_2(X)$. Moreover, their union $X_1 \cup X_2$ is $\V(I_1I_2) = \V(\I(X))=X$.  
\end{proof}

\subsection{Dimension}

The \mydef{dimension} of an irreducible variety $X$ is the longest length $\mydefMATH{\dim(X)}$ of a  proper chain of irreducible closed subvarieties
$$\emptyset = X_{-1} \subsetneq X_0 \subsetneq X_1 \subsetneq \cdots \subsetneq X_{\dim(X)} = X.$$ If $X$ is not irreducible, then its \mydef{dimension} is the maximum dimension of its irreducible components. The \mydef{codimension} of a subvariety $X \subset Z$ is $\mydefMATH{\codim_Z(X)}=\dim(Z)-\dim(X)$. We will omit the subscript on codimension whenever $Z=\C^n$ or $Z=\P^n$ or we have specifically mentioned $Z$ and so the subscript is clear from context. If $X$ and $Y$ are both subvarieties of $Z$ and $\dim(X)=\codim(Y)$, then we say that $X$ and $Y$ have \mydef{complementary dimension}. 

A variety of codimension $1$ is the zero set of a single polynomial and  is called a \mydef{hypersurface}. Zero-dimensional varieties are finite collections of points. If $X$ is a closed subvariety of an irreducible variety $Z$ and $\dim(X)=\dim(Z)$ then $X=Z$. 

Given a variety $X$ one expects the intersection of $X$ and a hypersurface to have dimension one less than $X$. The following lemma states that the dimension is lowered by at most one in the projective setting. 
\begin{lemma}\cite[I.6.2 Corollary 5 of Theorem 4]{Shafarevich}
\label{lem:projectivelowerbounddimension}
Let $f_1,\ldots,f_k \in \C[x]$ be homogeneous polynomials and suppose $X\subset \P^{n-1}$ is a projective variety of dimension $m$. Then we have that $\dim(\V(f_1,\ldots,f_k) \cap X) \geq m-k$.
\end{lemma}
The affine analog of Lemma \ref{lem:projectivelowerbounddimension} gives a weaker conclusion.
\begin{lemma}\cite[I.6.2 Corollary 2 of Theorem 5]{Shafarevich}
\label{lem:affinelowerbounddimension}
Let $f_1,\ldots,f_k \in \C[x]$ and $X \subset \C^n$ an affine variety of dimension $m$. Every irreducible component of $\V(f_1,\ldots,f_k) \cap X \subset \C^n$ has dimension at least $m-k$. 
\end{lemma}

We distinguish the conclusions of Lemma \ref{lem:projectivelowerbounddimension} and Lemma \ref{lem:affinelowerbounddimension} in the following example.
\begin{example}
Let $X=\C^2_{x,y}$ and $f_1=xy-1,f_2=x$. Then $\V(f_1,f_2) = \emptyset$. While it is true that Lemma \ref{lem:affinelowerbounddimension} guarantees that every irreducible component of $\V(f_1,f_2)\cap X$ has dimension at least $0$, the variety $\V(f_1,f_2)\cap X$ has no irreducible components and so the lemma does not apply. 
\begin{figure}[!htpb]
\includegraphics[scale=0.4]{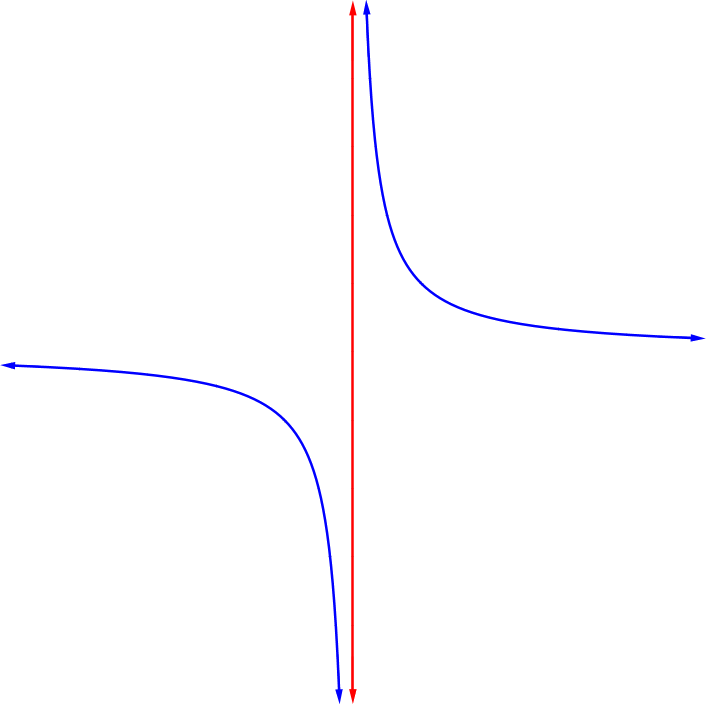}
\caption{The affine varieties $\V(xy-1)$ and $\V(x)$.}
\label{fig:hyperbolaline}
\end{figure}

Na\"ively homogenizing, take $\tilde f_1 = xy-z^2, \tilde f_2=x$ and $X=\P^2$, so that $$\V(\tilde f_1,\tilde f_2)\cap X = \{[0:1:0]\}\subset \P^2.$$ This is nonempty as guaranteed by Lemma \ref{lem:projectivelowerbounddimension}. 

Notice that with respect to this homogenization, the point $[0:1:0]$ is on the line at infinity. This aligns with our intuition as Figure \ref{fig:hyperbolaline} shows that the line $\V(x)$ and the hyperbola $\V(xy-1)$ asymptotically approach each other along the $y$-axis.\hfill $\diamond$
\end{example}
\begin{corollary}
\label{cor:nonemptyhasdimension}
If $\emptyset \neq \V(f_1,\ldots,f_k) \subset \C^n$ then $\V(f_1,\ldots,f_k)$ has dimension at least $n-k$. 
\end{corollary}

\begin{lemma}\cite[I.6.2 Theorem 6.]{Shafarevich}
Let $X$ and $Y$ be subvarieties of $\C^n$ (or $\P^n$) of dimensions $m_1$ and $m_2$ respectively. Then every irreducible component of $X \cap Y$ has dimension at least $m_1+m_2-n$. Moreover, if $X$ and $Y$ are projective and $m_1+m_2 \geq n$ then $X \cap Y \neq \emptyset$.
\end{lemma}

\section{Function fields and rational functions} 
When $X\subset \C^n$ is irreducible, $\I(X)$ is prime and so its coordinate ring $\C[X]$ is an integral domain. The field of fractions of $\C[X]$ is called the \mydef{function field} of $X$, denoted $\mydefMATH{\C(X)}$, and consists of all rational functions $g/h\colon X \dashrightarrow \C$ such that $h \not\in \I(X)$. 

If $X$ is an irreducible projective variety, the \mydef{function field} of $X$, denoted $\mydefMATH{\C(X)}$, consists of rational functions $g/h\colon X \dashrightarrow \C$ such that $g$ and $h$ have the same degree and $h \not\in \I(X)$. We use the dashed arrow notation to remind ourselves that rational functions are not defined everywhere but they are well-defined on the open subset $U=X\smallsetminus \V(h)$. Indeed if $u \in U$, then $$g(\lambda u)/h(\lambda u) = (\lambda^dg(u))/(\lambda^dh(u)) = g(u)/h(u)$$ for all $\lambda \in \C \smallsetminus \{0\}$. Unlike the affine case, the function field of $X$ is \emph{not} the field of fractions of $\C[X]$, but rather, its $0$-th graded piece. A \mydef{rational map} $\varphi\colon X \dashrightarrow \P^{m}$ from $X$ to projective space is given as
\begin{equation}
\label{eq:rationalmap}
\varphi=[\varphi_1:\cdots:\varphi_{m+1}], \quad \varphi_i \in \C(X),\quad \text{for }  i=1,\ldots,m+1,
\end{equation} where $\varphi_i = g_i/h_i$. The map $\varphi$ is defined on the open set $$V=X \smallsetminus \left(\V(g_1,\ldots,g_{m+1}) \cup \V(h_1\cdots h_{m+1})\right).$$ We may always write a rational map \eqref{eq:rationalmap} so that $\varphi_i$ are polynomials. Since each $\varphi_i$ is of the form $\varphi_i= g_i/h_i$, we simply clear denominators, 
\begin{equation}
\label{eq:rationaltopolynomial}
\varphi = [g_1/h_1:\cdots:g_{m+1}/h_{m+1}]= [f_1:\cdots:f_{m+1}]
\end{equation} 
where $f_i = (g_i/h_i)\cdot \prod_{j=1}^{m+1} h_j$. Even though the coordinates of every rational function may be written as polynomials, these are not regular functions because $\V(f_1,\ldots,f_{m+1})$ may not be empty.

Two rational functions $g/h,g'/h'\in \C(X)$ are equal whenever $gh'-g'h \in \I(X)$. Of course, they may be defined on different open subsets $U = X \smallsetminus \V(h)$ and $U' = X \smallsetminus \V(h')$, but they agree on the dense open subset $U \cap U'$. Similarly, two rational maps
$$
\varphi=[f_1:\cdots:f_{m+1}]\quad \text{ and }\quad
\varphi'=[f'_1:\cdots:f'_{m+1}],
$$written in the form \eqref{eq:rationaltopolynomial}, are the same if $f_if'_j-f_jf'_i \in \I(X)$ for all $i,j \in [m+1]$. Equivalently, $\varphi$ and $\varphi'$
agree on an dense open subset of $X$. Thus, for any dense open subset $U\subset X$, rational maps on $X$ are determined by their values on $U$. Hence, when $U$ is an affine open subvariety of $X$, $\C(U) = \C(X)$. 

\section{Products, graphs, and the degree of a variety}
Given a function $f\colon A \to B$ of sets, the graph of $f$ is simply the set $\Gamma(f)=\{(a,b) \mid a \in A, b=f(a)\}\subset A \times B$. We may similarly define the graph of a regular or rational map of algebraic varieties, however, \emph{a priori} these graphs do not come equipped with the structure of a variety. We obtain a variety structure on the graph of a map by developing a variety structure on products of varieties \emph{vis-\'a-vis} Segre maps.

\subsection{Segre maps}
Given two projective spaces $\P^{n-1}$ and $\P^{m-1}$, define the \mydef{Segre map} 
$$\mydefMATH{\sigma_{n-1,m-1}}\colon  \P^{n-1} \times \P^{m-1} \to \P^{nm-1},$$to be the function sending a pair of points $[x]\in\P^{n-1}$ and $[y] \in \P^{m-1}$ to the point whose coordinates are all possible pair-wise products of the coordinates of $[x]$ and $[y]$, namely, $$\sigma_{n-1,m-1}([x_1:\cdots:x_{n}],[y_1:\cdots:y_m]) = [x_1y_1:\cdots:x_iy_j:\cdots:x_ny_m].$$
Giving $\P^{nm-1}$ coordinates $\mydefMATH{z_{i,j}}=x_iy_j$, the image of the Segre map is 
$$\mydefMATH{\Sigma_{n-1,m-1}} = \V(z_{i,j}z_{k,l}-z_{i,l}z_{k,j}) \subset \P^{nm-1},$$
and is called the \mydef{Segre variety}.
\subsection{Products}
Defining the product of affine varieties is easy.
If $X\subset \C^n$ and $Y \subset \C^m$, the Cartesian product $\mydefMATH{X \times Y}=\{(x,y) \mid x \in X, y \in Y\}$ naturally lives in $\C^n \times \C^m \cong \C^{n+m}$ via the map $((x_1,\ldots,x_n),(y_1,\ldots,y_m)) \mapsto (x_1,\ldots,x_n,y_1,\ldots,y_m)$ and its structure as an affine variety comes from this realization of $X\times Y$ as a subvariety of $\C^{n+m}$. 

Given two projective varieties $X \subset \P^{n-1}$ and $Y \subset \P^{m-1}$, from now on, whenever we write the \mydef{product} $X \times Y$ we will mean the image of the Cartesian product $X \times Y$ under the Segre map $$\mydefMATH{X \times Y} = \{\sigma_{n-1,m-1}(x,y) \mid x \in X, y \in Y\}.$$
The Segre map is injective and so we will write elements of $X \times Y$ as $(x,y)$ where $x \in X$ and $y \in Y$. The projection maps $\mydefMATH{\pi_X}\colon X \times Y \to X$ and $\mydefMATH{\pi_Y}\colon  X \times Y \to Y$ onto the first and second coordinates are regular maps.
When $X\subset \C^{n-1}$ and $Y\subset \P^{m-1}$ we have $X \xhookrightarrow{\iota} \overline X \subset \P^{n-1}$ and so we take $X \times Y$ to be the variety $\mydefMATH{X \times Y} = \sigma_{n-1,m-1}(\iota (X), Y) \subset \P^{nm-1}$.

\subsection{Graphs}
Given a regular function $\varphi\colon X \to Y$ define the \mydef{graph of $\varphi$} to be  
$$\mydefMATH{\Gamma(\varphi)}=\{(x,y) \mid x \in X, y=\varphi(x)\}\subset X \times Y.$$
This is a closed subvariety of $X \times Y$ and the projection maps are regular. When $X$ or $Y$ are projective, we will often first take affine open subsets so that $\varphi$ is a map of affine varieties and the graph is an affine variety.
When we do this, we may assume $X=\V(f_1,\ldots,f_k)\subset \C^n$ and $Y \subset \C^m$ so the graph of $\varphi$ is the subvariety of $\C^n \times \C^m \cong \C^{n+m}$ given explicitly as 
$${\Gamma(\varphi)} = \V(f_1,\ldots,f_k,\varphi_1-x_{n+1},\ldots,\varphi_m-x_{n+m}).$$

\begin{lemma}
The closure of the image of an irreducible variety under a regular map is irreducible.
\end{lemma}
\begin{proof}
Suppose $\varphi\colon X \to Y$ is a regular map with $Y$ reducible, witnessed by $Y=Y_1 \cup Y_2$. Since $\varphi$ is continuous with respect to the Zariski topology, $X=\varphi^{-1}(Y_1) \cup \varphi^{-1}(Y_2)$ is an irredundant union of proper nonempty closed subvarieties of $X$ witnessing the reducibility of $X$.
\end{proof}

Given a rational map $\varphi\colon X \dashrightarrow \P^{m}$ of projective varieties, let $U\subset X$ be its domain of definition. We define the \mydef{graph} of $\varphi$, denoted $\mydefMATH{\Gamma(\varphi)}$, to be the closure of $\Gamma(\varphi|_U)$ in $X \times \P^{m}$ and we define the \mydef{image} of $\varphi$ to be the image of $\Gamma(\varphi)$ under $\pi_Y$.
The inverse image of a subvariety $Z \subset \P^{m}$ is $\mydefMATH{\varphi^{-1}(Z)}=\pi_X(\pi^{-1}_{\P^{m}}(Z))$. Given $Y \subset \P^{m}$, a \mydef{rational map} $\varphi\colon  X \dashrightarrow Y$ is any rational map $\varphi\colon  X \dashrightarrow \P^{m}$ whose image is contained in $Y$. 
\begin{lemma}
The image of an irreducible variety under a rational map is irreducible.
\end{lemma}

\subsection{Dominant maps}

Unfortunately, given two rational maps $\varphi\colon  X \dashrightarrow Y$, and $\psi\colon Y\dashrightarrow Z$, the composition $\psi \circ \varphi\colon  X \dashrightarrow Z$ is not always well-defined as shown in the following example.
\begin{example}
\label{ex:compnotdefined}
Let\\
\begin{minipage}{0.4\textwidth}
\begin{align*}
\varphi\colon \P^1 &\to \P^3 \\
[u:v] &\mapsto [u^3:u^2v:uv^2:v^3] \\
\end{align*}
\end{minipage}
and  \qquad
\begin{minipage}{0.4\textwidth}
\begin{align*}
\psi\colon  \P^3 &\dashrightarrow \P^2 \\
[x:y:z:w] &\mapsto [xz-y^2:yw-z^2:xw-yz]. \\
\end{align*}
\end{minipage}\\
Then $\psi \circ \varphi ([u:v]) = [0:0:0]$ is a point in $\mathbb{P}^2$.\hfill $\diamond$
\end{example}
The problem in Example \ref{ex:compnotdefined} is that the image of $\varphi$ is disjoint from the domain of definition of $\psi$. This motivates the definition of dominant maps, a subset of rational maps for which composition is always well-defined.

We say a rational map $\varphi\colon  X \dashrightarrow Y$  of varieties is \mydef{dominant} if $\varphi(X)$ is dense in $Y$. If $\varphi\colon  X \dashrightarrow Y$ is dominant with domain of definition $U$ and $\psi\colon  Y \dashrightarrow Z$ with domain of definition $V$, then the domain of definition of the composition $\psi \circ \varphi\colon  X \dashrightarrow Z$ is $U \cap \varphi^{-1}(V)$. 

In the same way that a regular map $\varphi\colon X \to Y$ of affine varieties induces a $\C$-algebra homomorphism $\varphi^*\colon  \C[Y] \to \C[X]$, a dominant map $\varphi\colon X \dashrightarrow Y$ induces an \emph{injective} $\C$-algebra homomorphism  which (when $X$ is irreducible) extends to the function field $\varphi^*\colon \C(Y) \to \C(X)$. Conversely, given an injective homomorphism $\phi\colon \C(Y) \to \C(X)$ of function fields, we obtain a dominant rational map $\phi^\#\colon X \dashrightarrow Y$.

\begin{lemma}\cite[I.6.3 Theorem 7]{Shafarevich}
\label{lem:fibredim}
Let $\varphi \colon X \to Y$ be a surjective regular map between irreducible varieties and that $\dim(X)=n$ and $\dim(Y)=m$. Then $m \leq n$ and 
\begin{enumerate}
\item $\dim(F) \geq n-m$ for any $y \in Y$ and for any component $F$ of the fiber $\varphi^{-1}(y)$.
\item there exists a nonempty open subset $U \subset Y$ such that $\dim(\varphi^{-1}(y))=n-m$ for $y \in U$.
\end{enumerate}
\end{lemma}

\begin{lemma}\cite[I.6.3 Theorem 8]{Shafarevich}
\label{lem:irreducibilityfromfibres}
Let $\varphi\colon X \to Y$ be a regular map between projective varieties with $\varphi(X)=Y$. Suppose that $Y$ is irreducible, and that all the fibers $\varphi^{-1}(y)$ for $y \in Y$ are irreducible of the same dimension. Then $X$ is irreducible.
\end{lemma}

\begin{proposition}\cite[Proposition 7.16]{HarrisBook}
\label{prop:finiteifffinitefibres}
Given a dominant map $\pi\colon X \dashrightarrow Y$, there exists an open subset $U \subset Y$ such that the fiber $\pi^{-1}(u)$ is finite if and only if $\pi^*$ expresses the field $\C(X)$ as a finite extension of the field $\C(Y)$. The number of points in a fiber over $u \in U$ is the degree of the field extension. 
\end{proposition}
\begin{proof}
We recount the proof from \cite{HarrisBook}.
Without loss of generality, replace $X$ and $Y$ with affine open subsets so that $\pi$ is a projection map $(x_1,\ldots,x_n) \mapsto (x_1,\ldots,x_{n-1})$ of affine varieties. Thus, the function field $\C(X)$ is generated over $\C(Y)$ by $x_n$. If $x_n$ is algebraic over $\C(Y)$ with minimal polynomial $$g_{(x_1,\ldots,x_{n-1})}(x_n) = a_d(x_1,\ldots,x_{n-1})x_n^d+a_{d-1}(x_1,\ldots,x_{n-1})x_n^{d-1}+\cdots,$$ we may clear denominators so that the coefficients of $g$ are regular functions. The discriminant $D$ of $g$ is a closed subset of the coefficient space since $\C$ is algebraically closed and so outside of this locus the fibers of $\pi$ consist of exactly $d$ points. 

Conversely, if $x_n$ is transcendental, then any polynomial in $\I(X)$ written in $\C(x_1,\ldots,x_{n-1})[x_n]$ must be identically zero as functions on $Y$. That is, the fiber $\pi^{-1}(y)$ for any $y \in Y$ contains infinitely many points.
\end{proof}
We remark that the locus of points $x^* \in Y$ which do not have the generic fiber size in Proposition \ref{prop:finiteifffinitefibres} come in three types:
\begin{enumerate}
\item The coefficient $x^*$ belongs to the discriminant $D$ because $g_{x^*}(x_n)$ has roots with multiplicity.
\item The coefficient $x^*$ belongs to the discriminant $D$ because $a_d(x^*)=0$. 
\item The rational coefficients $a_i(x_1,\ldots,x_{n-1})$ are not defined at $x^*$.
\end{enumerate}

A rational map $\pi\colon X \dashrightarrow Y$ satisfying Proposition \ref{prop:finiteifffinitefibres} is called a \mydef{generically finite} map. The degree of the field extension is called the \mydef{degree} of the map.

\begin{corollary}
Suppose $f\colon X \dashrightarrow Y$ is a dominant map of irreducible varieties of the same dimension. Then $f$ satisfies Proposition \ref{prop:finiteifffinitefibres}.
\end{corollary}

\subsection{Degree of a variety}
A variety cut out by linear equations is called a \mydef{linear} variety. The set of all linear subvarieties of $\P^{n}$ of dimension $k$ corresponds to the set of all $k+1$ planes in $\C^{n+1}$ through the origin. This space is called the \mydef{Grassmannian} of $(k+1)$-planes in $\C^{n+1}$ and is denoted $\mydefMATH{\text{Gr}(k+1,n+1)}$. The Grassmannian itself is a projective variety cut out by all relations amongst the minors of a $(k+1) \times (n+1)$ matrix. Similarly, a linear subvariety $L\subset \C^{n}$ of dimension $k$ corresponds to the $(k+1)$-plane $\overline L$ in $\P^{n}$.
Thus, it makes sense to talk about subvarieties and open subsets of the space of linear spaces of a particular dimension.
\begin{lemma}
\label{lem:whenhyperplanesintersect}
Let $X$ be an irreducible codimension $m$ subvariety of $\C^n$ or $\P^n$. There exists an open subset $V\subset \Gr(k+1,n+1)$ with the property  $L \in V \implies 0<|L \cap X|<\infty$ if and only if $k=m$. When $k=m$, there exists a smaller open subset $V'\subset V$ for which the number of such intersection points is constant.
\end{lemma}
\begin{proof}
The result is true for an affine variety if and only if it is true for its projective closure.
Let $X$ be projective and suppose such an open set $V\subset Y=\text{Gr}(k+1,n+1)$ exists. Consider the variety 
$$Z=\{(x,L) \mid L \in Y, x \in L \cap X\}\subset X \times Y$$
with projections $\pi_X$ and $\pi_Y$ to $X$ and $Y$ respectively. By assumption, the image of $\pi_Y$ contains $V$ and the fibers of $\pi_Y$ over a point $v \in V$ are finite. The fibers over $\pi_X$ are all irreducible of dimension $\dim(Y)-(n-k)$ and so $Z$ is irreducible of dimension $\dim(X)+\dim(Y)-(n-k)$ by Lemmas \ref{lem:fibredim} and \ref{lem:irreducibilityfromfibres}. If $\dim(X)<n-k$ then $\dim(Z)<\dim(Y)=\dim(V)$ and so $V \not\subset \pi_Y(Z)$, a contradiction. Thus, $k\geq n-\dim(X)=m$. On the other hand, if $k> m$ then by Lemma \ref{lem:projectivelowerbounddimension} the intersection $X \cap L$ is either empty or at least one-dimensional. We conclude $k=m$.

Conversely, if $k=m$, $$\dim(Z)=\dim(X)+\dim(Y)-(n-m)=\dim(Y),$$ implying that $\pi_Y$ is generically finite (such a $V$ exists). By Lemma \ref{prop:finiteifffinitefibres} there is an open subset $V' \subset V$ such that the number of points in a fiber of $\pi_Y$ over $V'$ is constant. 
\end{proof}

When $X,L \subset Z$, and $L \in V'$ as in the above lemma, then cardinality $|X \cap L|$ is some constant $d \in \mathbb{N}$. This number $d$ is called the \mydef{degree} of $X$ and is denoted $\mydefMATH{\deg(X)}$. Given an irreducible polynomial $f \in \C[x]$ The degree of a hypersurface $\V(f)$ is the degree of $f$. The degree of a collection of $d$ points is $d$.
We give the \mydef{first Bertini theorem}.
\begin{lemma}\cite[II.6.1 Theorem 1]{Shafarevich}
Let $X$ and $Y$ be irreducible varieties defined over a field of characteristic $0$ and $f\colon X \to Y$ a regular map such that $f(X)$ is dense in $Y$. Suppose that $X$ remains irreducible over the algebraic closure $\overline{\C(Y)}$ of $\C(Y)$. Then there exists an open dense set $U \subset Y$ such that all the fibers $f^{-1}(y)$ over $y \in U$ are irreducible. 
\end{lemma}
\begin{corollary}
\label{cor:hyperplaneresult1}
Let $X$ be a variety and $H$ a general hyperplane. Then
\begin{enumerate}
\item $\deg(X)=\deg(X \cap H)$.
\item $\dim(X)-1=\dim(X \cap H)$.
\item If $X$ is irreducible of dimension at least two, then $X \cap H$ is irreducible.
\end{enumerate}
\end{corollary}
\begin{proof}
Part $(1)$ follows directly from the definition of the degree of a variety. For part $(2)$, if $X$ is irreducible and $\dim(X)=\dim(X \cap H)$ then $H$ must contain $X$, but most hyperplanes do not contain a nonempty variety. For part $(3)$, suppose that $L$ is the normal line to the hyperplane $H$ and consider the linear projection $f\colon X \to L$. Then for a general point $y \in L$, the fiber $f^{-1}(y)$ is irreducible and the hyperplane slice $X \cap H$ corresponds to one such fiber.
\end{proof}

\section{Singular points}

Let $X=\V(F) \subset \C^n$ be an irreducible affine variety of dimension $m$ such that $\langle F \rangle$ is a radical ideal. We say $X$ is \mydef{smooth} at a point $p \in X$ if the rank of the \mydef{Jacobian matrix} $$DF=\left[ \frac{\partial f_i}{\partial x_j} \right]$$ evaluated at $p$ is $n-m$, otherwise we say $p$ is \mydef{singular}. We say $X$ is \mydef{smooth} if it is smooth at all of its points. The set $\mydefMATH{\text{Sing}(X)}$ of singular points of $X$ is a proper closed subvariety of $X$ \cite[Theorem 5.3]{Hartshorne} and so the set of smooth points of $X$ is open and dense. If $p$ is a point of a projective variety $X$, then $p$ is smooth on $X$ if $p$ is smooth on $U_i \cap X$ for some affine chart containing $p$. 

The following proposition is called the \mydef{second Bertini theorem}.
\begin{proposition}\cite[II.6.2 Theorem 2]{Shafarevich}
\label{prop:bertinitwo}
Let $f\colon X \to Y$ be a regular dominant map with $X$ smooth. There exists a dense open set $U \subset Y$ such that the fiber $f^{-1}(y)$ is nonsingular for every $y \in U$.
\end{proposition}
A corollary of the second Bertini theorem is fundamental to the theory of numerical algebraic geometry (Section~\ref{section:numericalalgebraicgeometry}).
\begin{corollary}
\label{lem:hyperplaneresult2}
If $X$ is a smooth variety and $H$ is a general hyperplane then  $X \cap H$ is smooth.
\end{corollary}
\begin{proof}
This follows by the same argument as in Corollary \ref{cor:hyperplaneresult1} replacing the first Bertini theorem with the second Bertini theorem.
\end{proof}

Let $\V(F) \subset \C_x^n \times \C_t$ be an irreducible affine variety of dimension one such that the projection \begin{align*}
\pi\colon  \V(F) &\to \C_t \\
(t,x_1,\ldots,x_n) &\mapsto t
\end{align*}
is dominant. The Jacobian $DF$ encodes the points $t \in \C_t$ for which $\pi^{-1}(t)$ does not have the generic cardinality as in Proposition \ref{prop:finiteifffinitefibres}. Let $D_tF = \frac{\partial F}{\partial t}$ and $D_xF=\frac{\partial F}{\partial x}$ so that $DF$ is the matrix whose first column is $D_t$ and whose last $n$ columns are $D_xF$. Given a point $p=(t^*,x^*) \in \V(F)$, $p$ is smooth on $\V(F)$ when $\rank(DF(t^*,x^*)) = n$. If $\rank(D_{x}F(p))=n-1$, then  $p$ is singular ($\rank(DF(p))=n-1$) on $\V(F)$ or the fiber $\pi^{-1}(t^*)$ has points with multiplicity. We depict this dichotomy in Figure \ref{fig:singularitiesofprojection}.

\begin{example}
\label{ex:singularitiesofprojection}
Consider the curve $\V(f)$ with 
$$f =(x-3)^2 - (t -1) (t + 1) (t +2)^2,$$
displayed in Figure \ref{fig:singularitiesofprojection}. The rank of $D_xf$ is zero at the points $(-2,3),(-1,3),$ and $(1,3)$ on $\V(f)$. The rank of $Df$ at these points is $0, 1,$ and $1$ respectively.
\begin{figure}[!htpb]
\includegraphics[scale=0.4]{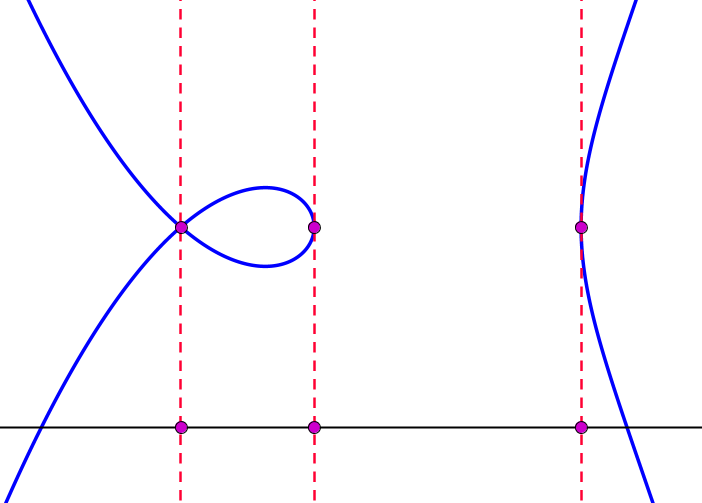}
\caption{\label{fig:singularitiesofprojection} Three points on a quartic curve in $\V(f) \subset \C_t \times \C_x$ such that the matrix $D_xf$ has rank zero when evaluated at these points.}
\end{figure}
 \hfill $\diamond$
\end{example}

\chapter{BRANCHED COVERS AND GROUPS \label{section:branchedcoversandgroups}}
Representing geometric objects as fibers of maps is a powerful method in geometry. For example, the simple problem of solving a quadratic equation $ax^2+bx+c=0$ for $a,b,c \in \C$ may be interpreted geometrically via a map 
\begin{align*}
\pi\colon \{([a:b:c],x) \in \P_{a,b,c}^2 \times \C \mid ax^2+bx+c=0\} &\to \P_{a,b,c}^2\\
([a:b:c],x) &\mapsto [a:b:c],
\end{align*}
over the parameter space $\P^2_{a,b,c}$. We identify
the solutions of a quadratic equation such as $3x^2+8x+4=0$ with the fiber $\pi^{-1}([3:8:4]) = \{([3:8:4],-2),([3:8:4],-\frac 2 3)\}$. The subset $U\subset \P^2_{a,b,c}$ of parameters whose corresponding quadratic equation has two distinct solutions is the complement of the vanishing of the discriminant $B=\V(a(b^2-4ac))$ and comprises a dense open subset of $\P_{a,b,c}^2$. Since $a \neq 0$ for $[a:b:c] \in U$, rescaling to monic quadratic equations,
\begin{align*}
\pi|_{a=1}\colon \{(b,c,x) \in \C^3 \mid x^2+bx+c=0\} &\to \C_{b,c}^2\\
(b,c,x) &\mapsto (b,c)
\end{align*}
gives a ``branched cover'' of affine varieties.
In this framework, the solutions of $3x^2+8x+4=0$ are identified with the fiber over the parameter $\left(\frac 8 3, \frac 4 3 \right) \in \C_{b,c}^2$. Figure \ref{fig:discriminant} depicts this parameter space along with the set $U|_{a=1} \subset \C^2_{b,c}$.
\begin{figure}[!htpb]
\includegraphics[scale=0.35]{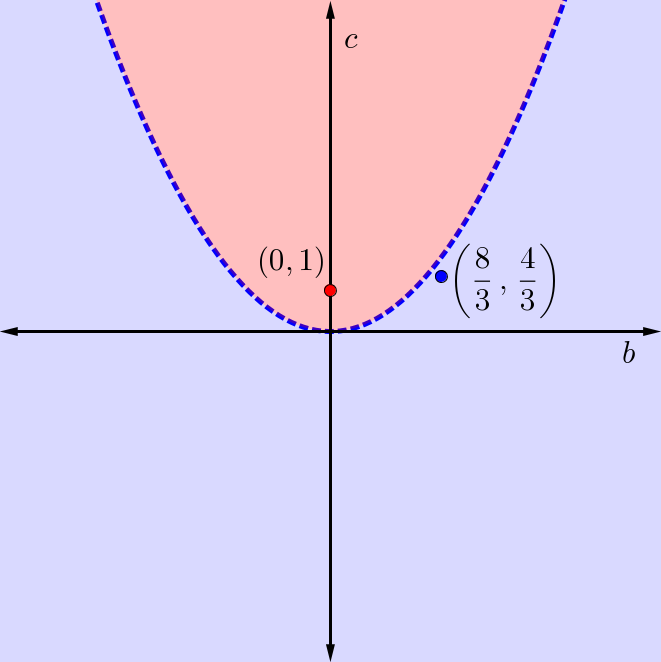}
\caption{The parameter space $\C^2_{b,c}$ along with the discriminant $\V(b^2-4c)$.}
\label{fig:discriminant}
\end{figure} 
Every point $(b,c) \in \C_{b,c}^2$ which is not on the dotted parabola in Figure \ref{fig:discriminant} is in $U|_{a=1}$. Fibers over parameters in the red region, like $\pi|_{a=1}^{-1}(0,1) = \pm \sqrt{-1}$, have two distinct (complex conjugate) nonreal points, and the fibers over points in the blue region have two distinct real points. Points on the parabola have fibers consisting of one real solution occurring with multiplicity two. 

The variety $\V(b^2-4c)$ is a hypersurface in $\C_{b,c}^2$ and thus has (complex) codimension $1$ and {real} codimension $2$. Thus, $U$ is a connected real manifold, even though it is disconnected when restricted to $\R^2_{b,c}$ as seen in Figure \ref{fig:discriminant}. 

The discussion above distills the essence of the behavior of branched covers. We give an elementary treatment of branched covers and covering spaces in Section~\ref{subsection:branchedcovers} and we provide background on permutation groups, monodromy groups, and Galois groups in Sections~\ref{subsection:permutationgroups}-\ref{subsection:galoisandmonodromy}. We conclude in Section~\ref{subsection:decomposablebranchedcovers} with a discussion of decomposable branched covers.

\section{Branched covers}
\label{subsection:branchedcovers}
An \mydef{(irreducible) branched cover} is a dominant map $\pi\colon  X \dashrightarrow Z$ of irreducible varieties of the same dimension. We may assume that we restrict to an affine open subset of $X$ so that $\pi\colon X \to Z$ is regular with $X \subset \C^n$ and $Y \subset \C^m$. Irreducible branched covers are generically finite in the sense of Proposition \ref{prop:finiteifffinitefibres} and thus there exists a number $\mydefMATH{d}$ and a dense open set $\mydefMATH{U} \subset Z$ such that for any $u \in U$, the fiber $\pi^{-1}(u)$ has cardinality $d$ and $\pi^*$ expresses the field $\C(X)$ as a degree $d$ field extension of $\C(Z)$.

More generally, a \mydef{branched cover} is a map $\pi\colon X \to Z$ such that $X$ is reducible and the restriction of $\pi$ to some top dimensional component of $X$ is an irreducible branched cover. The restriction of $\pi$ to every other top dimensional component is either dominant or the image is a proper closed subvariety of $Z$. Let $X_1,\ldots,X_k$ be those components of $X$ such that the restriction $\pi_i$ of $\pi$ to $X_i$ is dominant. Suppose $\pi_i$ has fibers of cardinality $d_i$ over any point in the dense open subset $U_i \subset Z$. Then it is immediate that for any $u \in \mydefMATH{U}=\bigcap_{i=1}^k U_i$, the cardinality of the fiber $\pi^{-1}(u)$ is $\mydefMATH{d}=\sum_{i=1}^k d_i$.
 
Given a branched cover $X \xrightarrow{\pi} Z$ as above,  $d$ is the \mydef{degree} of $\pi$,  $U$ is the set of \mydef{regular values} of $\pi$, and the complement of $U$ is the \mydef{branch locus} of $\pi$. We say $\pi$  is \mydef{trivial} if $d=1$.  With respect to the real Euclidean topologies $X$ and $Z$ inherit from their ambient spaces, there exists an open cover $\{V_{\beta}\}$ of $U$ such that for each $\beta$, the fiber $\pi^{-1}(V_\beta)$ is a disjoint union of $d$ open sets in $\pi^{-1}(U)$, each of which is mapped homeomorphically onto $V_\beta$. Such a map $\pi|_U\colon  \pi^{-1}(U) \to U$ is called a \mydef{$d$-sheeted covering space}. 

Many properties of branched covers, like the well-definedness of degree and regular values, extend immediately from  their irreducible restrictions. Therefore, in the interest of brevity, we use ``branched cover'' to refer to an irreducible branched cover, unless otherwise stated. We refrain from elaborating on branched covers which are not irreducible.

\section{Permutation groups}
\label{subsection:permutationgroups}
We recall some terminology concerning permutation groups~\cite{Wielandt}.
For $d \in \mathbb{N}$, the \mydef{symmetric group  $S_d$} on $d$ elements is the group of bijections from $[d]$ to $[d]$ under composition. Any subgroup $G \subset S_d$ of the symmetric group acts on the ordered set $\{1,2,\ldots,d\}$ by permuting its elements and is thus called a \mydef{permutation group}.  A permutation group acts \mydef{transitively} if for every $i,j \in [d]$, there exists $g \in G$ such that $g(i)=j$. For now, we will assume that $G$ acts transitively on $[d]$.

A \demph{block} of $G$ is a subset $B\subset [d]$ such that for every $g\in G$, either $gB=B$ or $gB\cap B=\emptyset$.
The subsets $\emptyset$, $[d]$, and every singleton are blocks of every permutation group.
If these trivial blocks are the only blocks, then $G$ is \demph{primitive} and otherwise it is \demph{imprimitive}. 

When $G$ is imprimitive, we have a factorization $d=ab$ with $1<a,b<d$ and there is a bijection
$[a]\times[b]\leftrightarrow[d]$ such that $G$ preserves the projection $[a]\times[b]\to [b]$.
That is, the fibers $\{[a]\times\{i\}\mid i\in [b]\}$ are blocks of $G$,  its action on this set of blocks gives a
homomorphism $G\to S_b$ with transitive image, and the kernel acts transitively on each fiber $[a]\times\{i\}$.
In particular, $G$ is a subgroup of the wreath product
$S_a\wr S_b = (S_a)^b\rtimes S_b$, where $S_b$ acts on $(S_a)^b$ by permuting factors.

We observe a second characterization of imprimitive permutation groups $G$.
Since $G$ acts transitively, if $H\subset G$ is the stabilizer of a point $c\in[d]$, then $H$ has index $d$ in $G$
and we may identify $[d]$ with the set $G/H$ of cosets.
If $B$ is a nontrivial block of $G$ containing $c$, then its stabilizer $L$ is a proper subgroup of $G$ that 
strictly contains $H$.
Furthermore, using the map $G/H \to G/L$, we see that $G$ is imprimitive if and only if the stabilizer of the point
$eH\in G/H$ is not a maximal subgroup.

\section{Monodromy groups and Galois groups}
\label{subsection:galoisandmonodromy}
Let $\pi\colon X \to Z$ be a degree $d$ branched cover so that the restriction  $\pi^{-1}(U) \xrightarrow{\pi} U$ is a $d$-sheeted covering space. A \mydef{lift} of a continuous function $\gamma\colon  Y \to U$ is a map $\widetilde{\gamma}\colon Y \to X$ such that $\pi\circ\widetilde{\gamma} = \gamma$. The \mydef{path lifting property} for a covering space says that for any path $\gamma\colon [0,1] \to U$ and any lift $\widetilde{u_0}$ of the point $u_0=\gamma(0)$, there is a unique path $\widetilde{\gamma}\colon [0,1] \to X$ which lifts $\gamma$ with the property that $\widetilde{\gamma}(0) = \widetilde{u_0}$ \cite{Hatcher}. 

Since the cardinality of the fiber $\pi^{-1}(\gamma(0))$ is $d$, there are $d$ paths $\{\tilde \gamma_i(t)\}_{i=1}^d$ lifting $\gamma$, giving a bijection $\mydefMATH{m_\gamma}$ from the fiber over $\gamma(0)$ to the fiber over $\gamma(1)$ defined by $m_\gamma(\tilde \gamma_i(0))=\tilde \gamma_i(1)$. When $\gamma(0)=\gamma(1)$, we call $\gamma$ a \mydef{(monodromy) loop based} at $\gamma(0)$. 
The set of all $m_\gamma$ such that $\gamma$ is a loop based at $u \in U$ forms a group $\mydefMATH{\mathcal M_{\pi,u}}$  called the \mydef{monodromy group} of $\pi$ based at $u$.

For any path $\gamma$ in $U$, conjugation by $m_\gamma$ gives an isomorphism $\mathcal M_{\pi,\gamma(0)}\cong \mathcal M_{\pi,\gamma(1)}$. Since $X$ is irreducible, $U$ is path-connected and so as a permutation group, the monodromy group is well-defined up to the relabelling of points in a fiber. We define the \mydef{monodromy group} of $\pi$, denoted $\mydefMATH{\mathcal M_\pi}$, to be this group. 

\begin{lemma}
\label{lem:monodromyofirreducibleistransitive}
The monodromy group of a branched cover $X \xrightarrow{\pi} Z$ is transitive.
\end{lemma}
\begin{proof} Let $p,q \in \pi^{-1}(u)$ for some $u \in U$. The set $\pi^{-1}(U)$ is path-connected and so a path  $\tau$ connecting $p$ to $q$ projects to a loop $\gamma = \pi \circ \tau$ with $\tau$ as a lift. Hence, $m_\gamma(p)=q$.  
\end{proof}

We define the \mydef{Galois group} \mydefMATH{$G_\pi$} of $\pi$ to be the Galois group of $K/\CC(Z)$ where \mydefMATH{$K$} is the Galois closure of $\CC(X)/\CC(Z)$. 
Harris~\cite{Harris} gave a modern proof of the following proposition, but this idea goes back
at least to Hermite~\cite{Hermite}.
\begin{proposition}\cite[pg. 689]{Harris}
The groups $G_\pi$ and $\mathcal M_\pi$ for a branched cover $\pi$ are equal.
\end{proposition}

\section{Decomposable branched covers}
\label{subsection:decomposablebranchedcovers}

A branched cover $\pi\colon X\to Z$ is \demph{decomposable} if there is a dense open subset $V\subset Z$ over which
$\pi$ factors
 \begin{equation}\label{Eq:Ndecomposition}
  \pi^{-1}(V)\ \xrightarrow{\; \varphi\;}\ Y\ \xrightarrow{\; \psi\;}\ V\,,
 \end{equation}
with $\varphi$ and $\psi$ both nontrivial branched covers.
The fibers of $\varphi$ over points of $\psi^{-1}(v)$ are blocks of the action of $G_\pi$ on $\pi^{-1}(v)$, which
implies that $G_\pi$ is imprimitive. 
Pirola and Schlesinger~\cite{PirolaSchlesinger} observed that decomposability of $\pi$ is equivalent to imprimitivity of
$G_\pi$.
We give a proof, as we discuss the problem of computing a decomposition.

\begin{proposition}\label{P:DecomposableIsImprimitive}
 A branched cover is decomposable if and only if its Galois group is imprimitive.
\end{proposition}  

\begin{proof}
  We need only to prove the reverse direction.
  As above, let $\CC(Z)$, $\CC(X)$, and $K$ be the function fields of $Z$, $X$, and the Galois closure of
  $\CC(X)/\CC(Z)$, respectively, and let $G_\pi$ be the Galois group of $K/\CC(Z)$.
  Let $H$ be the subgroup of $G_\pi$ such that $\CC(X)=K^H$, the fixed field of $H$.
  The set of Galois conjugates of $\CC(X)$ forms the orbit $G_\pi/H$, and the number of conjugates is the degree of
  the branched cover $X\to Z$.

  If $G_\pi$ acts imprimitively, then the stabilizer $L$ of a nontrivial block $B$ containing $\CC(X)$ is a proper subgroup 
  properly containing $H$.
  Thus its fixed field $\defcolor{M}=K^L$, which is the intersection of the conjugates of $\CC(X)$ in the block $B$, is
  an intermediate field between $\CC(Z)$ and $\CC(X)$.
  For any variety $Y'$ with function field $M$, there will be dense open subsets $Y$ of $Y'$ and $V$ of $Z$ such
  that~\eqref{Eq:Ndecomposition} holds.
\end{proof}
  
  While imprimitivity is equivalent to decomposability, the proof does not address how to compute
  the variety $Y$ of~\eqref{Eq:Ndecomposition}.
  One way is as follows.
  Replace $Z$ and $X$ by affine open subsets, if necessary, and let 
  $y_1,\dotsc,y_m\in\CC[X]$ be regular functions on $X$ that generate $M$ over $\CC(Z)$.
  Let $x_1,\dotsc,x_m$ be indeterminates and let $I\subset\CC(Z)[x_1,\dotsc,x_m]$ be the kernel of the map
  $\CC(Z)[x_1,\dotsc,x_m]\to \CC(X)$ given by $x_i\mapsto y_i$.
  This is the  zero-dimensional ideal of algebraic relations satisfied by $y_1,\dotsc,y_m$.
  Replacing $Z$ by a dense affine open subset if necessary, we may choose generators $g_1,\dotsc,g_r$ of $I$ that lie in
  $\CC[Z][x_1,\dotsc,x_m]$---their coefficients are regular functions on $Z$.
  There is an open subset $V\subset Z$ such that  the ideal $I$ defines an irreducible variety
  $\defcolor{Y}\subset V\times\CC^m$ whose projection to $V$ is a branched cover and whose function field is $M$.
  Restricting $X\to Z$ to $V$, we obtain the desired decomposition, with the map $X\to Y$ given by the functions
  $y_1,\dotsc,y_m$. 

  This does not address the practicality of computing $Y$, but it does indicate an approach.
  Given the subgroup $L$ of $G_\pi$ and a set of generators of $\CC[X]$ over $\CC[Z]$, if we apply the Reynolds averaging 
  operator~\cite{DerksenKemper} for $L$ to monomials in the generators, we obtain the desired generators $y_1,\dotsc,y_m$
  of $M$. 
  One problem is that elements of $G_\pi$ may not act on $X$, so their action on elements of $\CC[X]$ may be hard
  to describe.

  There is an exception to this.
  If $L\neq H$ normalizes $H$ in $G$ and $\pi\colon X\to Z$ is a covering space, then $\defcolor{\Gamma}=L/H$ acts freely
  on $X$, preserving the fibers---it is a group of \mydef{deck transformations} of $X\to Z$~\cite[Ch~13]{Munkres}.
  When $\Gamma$ acts on the original branched cover, $Y=X/\Gamma$ is the desired space, and both $Y$ and the map
  $X\to Y$ may be computed by applying the Reynolds operator for $\Gamma$ to generators of $\CC[X]$.
  The examples given in~\cite[Section $5$]{AmendolaRodriguez} are of this form, and the authors use this approach to compute
  decompositions.

\begin{example}\label{Ex:NoMaximal}
  Not all imprimitive groups have this property.
  Consider the wreath product $G=S_3\wr S_3$, which acts imprimitively on the nine-element set
  $[3]\times[3]$.
  The stabilizer of the point $(3,3)$ is the subgroup $H=((S_3)^2\times S_2)\rtimes S_2$, where $S_2\subset S_3$ is the
  stabilizer of $\{3\}$.
  Then $H$ is its own normalizer in $G$, as $S_2$ is its own normalizer in $S_3$.\hfill $\diamond$
\end{example}

  All imprimitive Galois groups in the Schubert calculus constructed in~\cite[Section $3$]{GIVIX} and in~\cite{SWY}
  have stabilizer $H$ equal to its normalizer.
  For these, the decomposition of the branched cover follows from a deep structural understanding of
  the corresponding Schubert problem.
  There remain many Schubert problems whose Galois group is expected to be imprimitive, yet a decomposition~\eqref{Eq:Ndecomposition} of the corresponding branched cover is unknown.

\section{Real branched covers}
\label{subsection:realbranchedcovers}
The nonreal solutions of any univariate polynomial $f \in \R[x]$ come in complex conjugate pairs. Similarly, for a multivariate polynomial system $F=(f_1,\ldots,f_k) \subset \R[x]$, any point $z \in \C^n$ satisfies $F(z)=0$ if and only if its complex conjugate $\overline z$ satisfies $F(\overline z)=0$.

When a branched cover $\pi\colon X \to Z$ with $Z \subset \C^m$ has the property that for any $z \in Z \cap \R^m$ the ideal $\I(\varphi^{-1}(z))$ can be generated by real polynomials, we say $\pi$ is a \mydef{real branched cover}. The set of real regular values of a real branched cover is possibly disconnected in $\R^m$, and we call these connected components \mydef{discriminant chambers}. 
\begin{lemma}
If $z,z' \in Z \cap \R^m$ are in the same discriminant chamber, then the number of real points in $\pi^{-1}(z)$ is equal to the number of real points in $\pi^{-1}(z')$. 
\end{lemma}
\begin{proof}
Let $z,z'$ be in the same discriminant chamber $D_z$ and let $\gamma\colon [0,1] \to Z \cap D_z$ be a path from $z$ to $z'$. For any point $\gamma(t^*)$ for  $t^* \in [0,1]$ the fiber $\pi^{-1}(t^*)$ consists of $\deg(\pi)$ distinct points. On the other hand, since nonreal points in the fiber come in complex conjugate pairs, the number of real points in a fiber over $\gamma([0,1])$ changes only if either two real points come together and become complex or two complex points come together and become real. However, this cannot happen since points in each fiber over $\gamma$ are distinct. 
\end{proof}

\begin{example}

Let 
$$f=4(\phi^2x^2-y^2)(\phi^2y^2-z^2)(\phi^2z^2-x^2)-(1+2\phi)(x^2+y^2+z^2-1^2)^2 \in \C[x,y,z],$$
where $\phi=\frac{1+\sqrt{5}}{2}$ is the golden ratio. The surface $\V(f)$ is known as the \mydef{Barth sextic}. 
The projection $$\pi\colon \C_{x,y,z}^3 \to \C^2_{x,y}$$ is a branched cover of degree $4$. 

The branch locus $B$ of $\pi$ is displayed in Figure \ref{fig:Barth} along with labels indicating the number of real points in any fiber of the corresponding discriminant chamber. Over $\mathbb{Q}$, $B$ decomposes into two lines (purple and green) and two sextics (blue and red). Over $\R$, the blue sextic curve decomposes into the union of a conic and four lines. The red curve is irreducible over $\R$.
The boxed region in Figure \ref{fig:Barth} contains a small discriminant chamber whose fibers have two real points. An enlarged depiction of this chamber is displayed in Figure \ref{fig:barthsmalldiscriminant}.
\hfill $\diamond$
\end{example}

\begin{figure}[!htpb]
\includegraphics[scale=0.35]{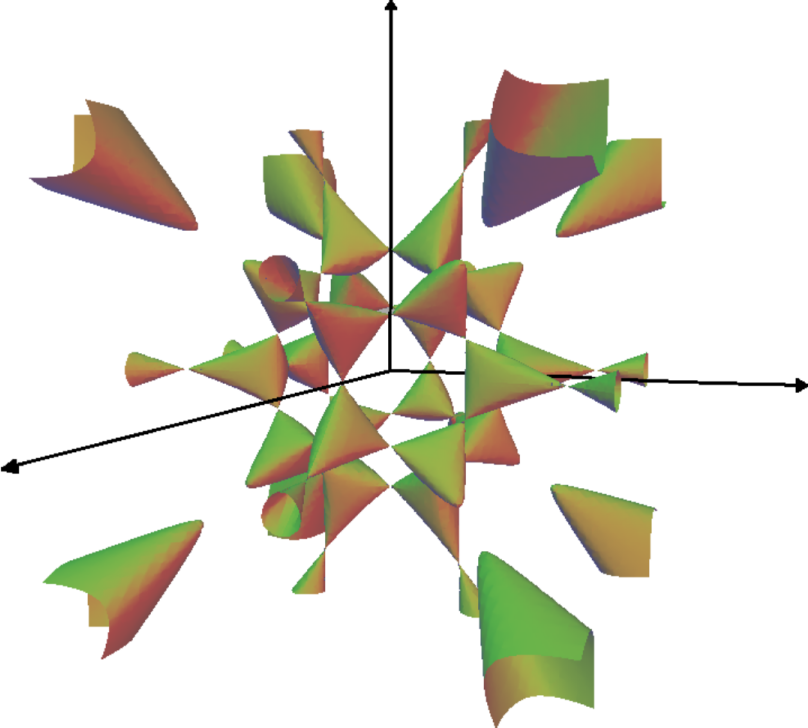}
\caption{The Barth sextic.}
\label{fig:BarthSextic}
\end{figure}

\begin{figure}[!htpb]
\includegraphics[scale=0.35]{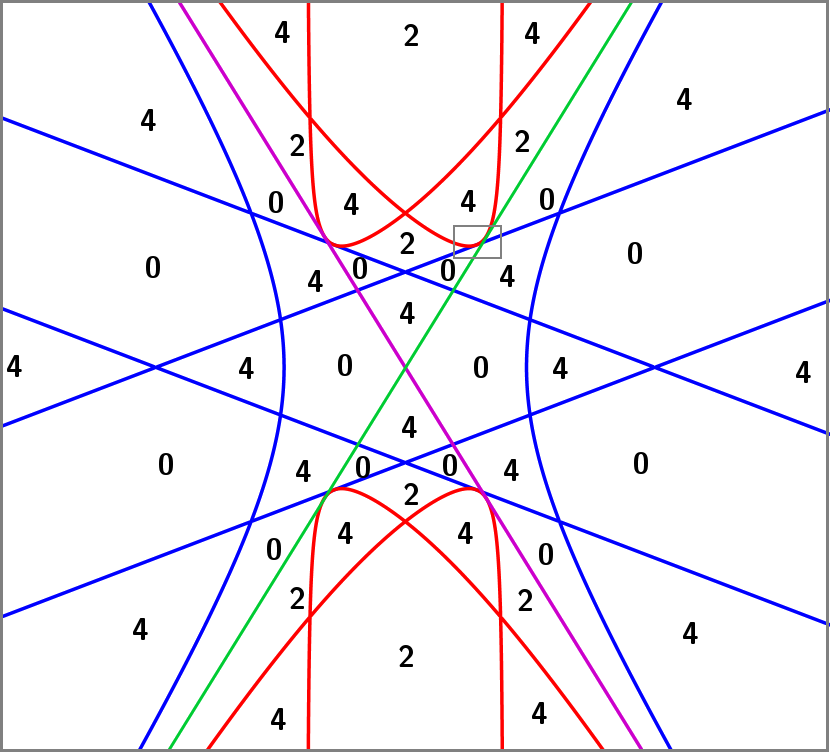}
\caption{The discriminant of the projection of the Barth sextic onto $\C^2_{x,y}$ with the number of real points in the fiber of any point in each discriminant chamber indicated.}
\label{fig:Barth}

\end{figure} 
\begin{figure}
\includegraphics[scale=0.35]{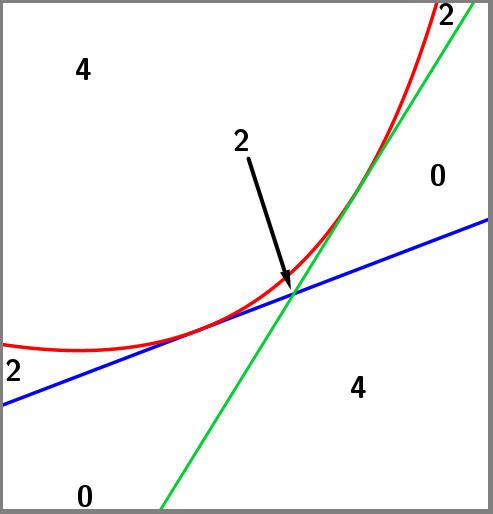}
\caption{\label{fig:barthsmalldiscriminant} One of the small discriminant chambers not easily noticeable in Figure \ref{fig:Barth}.}
\end{figure}

\chapter{NEWTON POLYTOPES, SUPPORT, TROPICAL GEOMETRY, AND SPARSE POLYNOMIAL SYSTEMS \label{section:newtonpolytopes}}
We introduce Newton polytopes, sparse polynomial systems, and tropical geometry. These connect ideas from Sections \ref{section:polytopes}, \ref{section:algebraicgeometry}, and \ref{section:branchedcoversandgroups}. 
\renewcommand*{\thefootnote}{\fnsymbol{footnote}}
Material in Section \ref{subsection:tropicalgeometry} appears in the article \cite{Bry:NPtrop} by the author\footnote{Reprinted with permission from T. Brysiewicz, ``Numerical Software to Compute Newton polytopes and Tropical Membership,''  {\it{Mathematics in Computer Science,}} 2020. Copyright 2020 by Springer Nature.}.
\renewcommand*{\thefootnote}{\arabic{footnote}}

\section{Newton polytopes}
Let $\mydefMATH{\C^\times} = \C\smallsetminus \{0\}$ be the multiplicative group of nonzero complex numbers and $\mydefMATH{(\C^\times)^n}$ be the $n$-dimensional complex torus. For each $\alpha = (\alpha_1,\ldots,\alpha_n) \in \Z^n$, the \mydef{(Laurent) monomial} with exponent $\alpha$,
$$\mydefMATH{x^\alpha}=x_1^{\alpha_1}x_2^{\alpha_2}\cdots x_n^{\alpha_n},$$
is a character (multiplicative map) $x^\alpha\colon  (\C^\times)^n \to \C^\times$.  Any finite linear combination
$$f = \sum_{\alpha \in \mathcal A} c_\alpha x^\alpha, \quad c_\alpha \in \C,$$
of monomials is a \mydef{(Laurent) polynomial} which also defines a function $f\colon  (\C^\times)^n \to \C$. When $c_\alpha \in \C^\times$ for all $\alpha \in \mathcal A$, we say that $\mathcal A$ is the \mydef{support} of $f$ and write $\mydefMATH{\supp(f)}=\mathcal A$. Otherwise, we say that $f$ is \mydef{supported on} $\mathcal A$. 
We denote the vector space of all polynomials supported on $\mathcal A$ by $\mydefMATH{\C^{\mathcal A}}$. 
Consistent with the notation for polytopes, for any $\omega \in \R^n$ we set,
$$\mydefMATH{f_\omega} = \sum_{\alpha \in \mathcal A_\omega} c_\alpha x^\alpha.$$

If $\calA$ is the support of $f$, then the support of $x^\beta f$ is $\beta+\calA$, the translation of $\calA$ by $\beta$. 
As a monomial $x^\beta$ for $\beta\in\ZZ^n$ is invertible on $(\CC^\times)^n$, the polynomials $f$ and $x^\beta f$ have the
same sets of zeros in $\C^\times$.
By translating the support of a polynomial by integer vectors we may assume that $\bzero$ is in the affine $\Z$-span of $\mathcal A$ without changing any assertions about the zeros of $f$ in $(\CC^\times)^n$, thus we define $\mydefMATH{\ZZ\calA}$ to be the lattice generated by differences $\alpha-\beta$ for $\alpha,\beta \in \calA$.
For similar reasons, the results from Section~\ref{section:algebraicgeometry} extend to this setting by shifting $\supp(f)$ to the positive orthant so that $f$ is polynomial.

The \mydef{Newton polytope} of $f$  (or of $\V(f)$) is
$$\mydefMATH{\New(f)}=\mydefMATH{\New(\V(f))}=\conv(\supp(f)).$$ We say $f$ has \mydef{dense support} in $\New(f)$ if $\supp(f) = \mathcal L(\New(f))$, the set of lattice points in $\New(f)$. The Newton polytope of a polynomial and its support both encode a considerable amount of information about the polynomial and its zero set. Moreover, these combinatorial objects behave well under certain algebro-geometric transformations on polynomials and varieties.

\subsection{Basic observations about Newton polytopes}
Let $f \in \C[x]$. Then the following observations are immediate from our definitions.
\begin{enumerate}
\item $\New(\widetilde{f}) = \widetilde{\New(f)}$ where $\widetilde{\nothing}$ denotes homogenization.
\item $f$ is homogeneous if and only if $\New(f)$ is homogeneous.
\item $\deg(f)=\deg(\New(f))$.
\item $\New(f)$ is an integral polytope.
\end{enumerate} 
For any $f,g \in \C[x]$, the Newton polytope of $f\cdot g$ is $\New(f)+\New(g)$. Indeed, Lemma \ref{lem:minkowskiproperties} implies that the vertices of $\New(f)+\New(g)$ are uniquely represented as $\alpha' + \beta'$ for some $\alpha' \in \vertices(\New(f))$ and $\beta' \in \vertices(\New(g))$. Thus, the only term of the sum $$f \cdot g = \sum_{\substack{\alpha \in \supp(f) \\\beta \in \supp(g)}} c_\alpha c_\beta x^\alpha x^\beta$$ which has exponent $\alpha'+\beta'$ is $c_{\alpha'}c_{\beta'}x^{\alpha'+\beta'}$ which is in the support of $f \cdot g$ because $c_{\alpha'}\cdot c_{\beta'} \neq 0$.   

Supports (and Newton polytopes) respect permutations of variables. For any permutation $\sigma \in S_n$, the support of $$\mydefMATH{\sigma(f)}=\sum_{\alpha \in \mathcal A} c_\alpha x_1^{\alpha_{\sigma(1)}}x_2^{\alpha_{\sigma(2)}}\cdots x_n^{\alpha_{\sigma(n)}}$$ is the set $$\mydefMATH{\sigma(\mathcal A)}=\{\mydefMATH{\sigma(\alpha)} \mid \alpha \in \mathcal A\} =\{(\alpha_{\sigma(1)},\ldots,\alpha_{\sigma(n)}) \mid \alpha \in \mathcal A\}.$$ Consequently, $\New(\sigma(f))=\mydefMATH{\sigma(\New(f))}=\conv(\sigma(\mathcal A))$. Hyperplanes containing $\New(f)$ correspond to scalings of the variables $x_1,\ldots,x_n$ which do not alter the variety $\V(f)$.
\begin{lemma}
\label{lem:torusinvarianthyperplane}
Let $f \in \C[x]$ be a polynomial with support $\calA$ and let $\omega \in \R^n$. Then $\mathcal A$ is contained in the hyperplane $\langle \alpha,\omega \rangle - h_{\mathcal A}(\omega)=0$ if and only if $\V(f)=\V(f(t^{\omega_1}x_1,\ldots,t^{\omega_n}x_n))$ for all $t \in \C^\times$.
\end{lemma}
\begin{proof}
The equality $\V(f) = \V(f(t^{\omega_1}x_1,\ldots,t^{\omega_n}x_n))$ holds for all $t \in \C^\times$ if and only if for all $a \in \V(f)$,
\begin{align}
\label{eq:torusaction}
0=f(t^{\omega_1}a_1,\ldots,t^{\omega_n}a_n) &= \sum_{\alpha \in \mathcal A} c_\alpha t^{\langle \alpha, \omega \rangle} a^\alpha \nonumber \\
&= \sum_{k=-h_\calA(\omega)}^{h_\calA(\omega)}\left( \sum_{\substack{\alpha \in \mathcal A \\ \langle \alpha, \omega \rangle =k}} t^k c_\alpha a^\omega\right) \nonumber \\
&=\sum_{k=-h_\calA(\omega)}^{h_\calA(\omega)} t^k\mydefMATH{g_k(a)}.
\end{align}
The right-most-side is a polynomial in $t$ and thus $g_k(a)=0$ for all $k = -h_\calA(\omega),\ldots,h_\calA(\omega)$ and all $a \in \V(f)$. However, this means that $\V(f) \subset \V(g_k)$ for all $k$. Since $f$ is not identically zero, at least one $g_k$ is not. Suppose $g_j \neq 0$ for some $j$. Then containment of hypersurfaces implies $\deg(g_j)\geq \deg(f)$ and since $t^jg_j$ is a summand of \eqref{eq:torusaction}, these degrees must be the same. Containment of hypersurfaces also implies that $g_j(x)=r(x)\cdot f(x)$ for some $r \in \C[x]$, but since the degrees of $g_j$ and $f$ are equal, $r$ must be a constant implying $\V(f)=\V(g_j)$. Consequently, every other summand of \eqref{eq:torusaction} must be zero, proving that $\mathcal A$ is contained in the hyperplane $\langle \alpha, \omega \rangle - h_\mathcal A(\omega)=0$. 

The converse is true since if $\langle \alpha, \omega \rangle =h_\mathcal A(\omega)$ for all $\alpha \in \mathcal A$, then $f(t^{\omega_1}x_1,\ldots,t^{\omega_n}x_n) = t^{h_\mathcal A(\omega)} f(x)$, and thus cuts out the same variety as $f$ for any $t \in \C^\times$. 
\end{proof}
\begin{remark} 
Fix $r \in \mathbb{N}$ and $k=(k_1,\ldots,k_r) \in \mathbb{N}^r$ and consider the grouping of variables $\left\{\{x_{i,j}\}_{i=1}^{k_j}\right\}_{j=1}^r$.
 By definition of projective space, the zero set of a polynomial $$f=\sum_{\alpha=(\alpha^{(1)},\ldots,\alpha^{(r)}) \in \mathcal A} c_\alpha x_{i,1}^{\alpha^{(1)}}\cdots x_{i,r}^{\alpha^{(r)}}\in \C[x_{i,j}]$$ with support $\mathcal A$ is well-defined subvariety of $\P^{k_1} \times \cdots \times \P^{k_r}$ if and only if it is invariant under scaling any of the variable groups: for each $j \in [r]$ and $t \in \C^{\times}$, the polynomial $f$ is invariant under the action which multiplies each variable in the group $\{x_{i,j}\}_{i=1}^{k_j}$ by $t$. By Lemma \ref{lem:torusinvarianthyperplane}, this is equivalent to the condition that for all $\alpha \in \mathcal A$ and $j \in [r]$, there exists $d_j$ such that $|\alpha^{(j)}|=d_j$. The vector $d=(d_1,\ldots,d_r)$ is called the \mydef{multidegree} of $\V(f)$.  
\end{remark}

Lemma \ref{lem:torusinvarianthyperplane} has strong implications when considering invariants. Fix some support $\mathcal A \subset \Z^n$ and suppose that $\mathcal F \subset \C[c_\alpha]_{\alpha \in \mathcal A}$ is a polynomial in the coefficient space $\C^{\mathcal A}$ of all polynomials 
$$f = \sum_{\alpha \in \mathcal A} c_\alpha x^\alpha \in \C[x]$$ supported on $\mathcal A$.
Observe that an action of a group $G{\curvearrowright}\C^n$  naturally induces an action $G\curvearrowright \C^{\mathcal A}$ on the coefficient space. If for all $f \in \V(\mathcal F)$ and all $\sigma \in G$, we have that $\sigma \cdot f \in \V(\mathcal F)$, then we say that $\mathcal F$ is \mydef{invariant} under the action of $G$. 

\begin{proposition}
\label{prop:invariantresult}
Suppose that $\mathcal F \in \C[c_{\alpha}]_{\alpha \in \mathcal A}$ is a homogeneous polynomial of degree $D$ with variables in the coefficient space $\C^{\mathcal A}$ of all polynomials 
$$f=\sum_{\alpha \in \mathcal A} c_\alpha x^\alpha \in \C[x],$$
supported on $\mathcal A \subset \Z^n$. Suppose further that $|\alpha| = d$ for all $\alpha \in \mathcal A$. Let $A$ be the $n\times |\mathcal A|$ matrix whose columns are points in $\mathcal A$ and whose rows are $\{\omega_{x_1},\ldots,\omega_{x_n}\}$. 
\begin{enumerate}
\item If  $\mathcal F$ is invariant under the scaling $x_i \mapsto tx_i$ for some $i \in [n]$ and all $t \in \C^\times$, then $\mathcal F_{\omega_{x_i}}=F$.
\item Suppose $\mathcal F$ is invariant under all scalings and permutations of the variables $x_i$ and that $\R\mathcal A$ is $n$ dimensional. Then $p \in \New(\mathcal F)$ solves the linear equation
$$\left({A \atop \bone}\right) p  = \left(\frac{dD}{n},\ldots,\frac{dD}{n},D\right)^T.$$
In particular, $\New(\mathcal F)\subset \R_p^{|\mathcal A|}$ is contained in an affine linear space of codimension $n$.
\end{enumerate}
\end{proposition}
\begin{proof}
Given $f = \{c_{\alpha}\}_{\alpha \in \calA}$, the action of $t \mapsto tx_1$ on $\C^n$ induces the action $$\mydefMATH{f_t}=f(tx_1,x_2,\ldots,x_n) = \left\{t^{\langle e_1,\alpha\rangle}c_{\alpha}\right\} =\left\{t^{(\omega_{x_1})_\alpha} c_\alpha\right\}$$
on the coefficients of $f$ and thus the variables of $\mathcal F$. If $\mathcal F$ is invariant under this action, then $f_t \in \V(\mathcal F)$ if and only if $f \in \V(\mathcal F)$ for all $t \in \C^\times$. Hence by Lemma \ref{lem:torusinvarianthyperplane} we have that $\mathcal F_{\omega_{x_1}}=\mathcal F$. The same argument applies for scaling any other variable.

If $\mathcal F$ is invariant under scaling any of the variables $x_1,\ldots,x_n$, then $P$ is contained in the intersection $\bigcap_{i=1}^n H_{i}$ where $$\mydefMATH{H_i} = \left\{p \in \R_p^{|\mathcal A|} \mymid \langle p, \omega_{x_i} \rangle = h_P(\omega_{x_i})\right\}$$ by part $(1)$. Since $\mathcal F$ is also invariant under the symmetric group $S_n$, the value of the support function $\mydefMATH{h}=h_P(\omega_{x_i})$ does not depend on $i$.
Since $\mathcal F$ is homogeneous, $P$ is also contained in the affine hyperplane $$\mydefMATH{H_{\deg}}=\left\{p \in \R^{|\calA|}_p \mymid \langle p, \bone \rangle = D\right\}.$$
Thus, the set $$\mydefMATH{H}=H_{\deg} \cap \left(\bigcap_{ i=1}^n H_{i} \right)$$ is the solution set of the matrix equation,
$$\left( A \atop \bone \right) p = (h,h,\ldots,h,D)^T.$$
Note that since $|\alpha|=d$ for all $\alpha \in \calA$, we have $(1,1,\ldots,1,-d)\left( A \atop \bone \right) =\bzero$. Therefore, $hn -dD=0$ and so $h=\frac{dD}{n}$.
\end{proof}

\subsection{Integer linear algebra and coordinate changes}
\label{subsubsection:integerlinear}
Supports of polynomials do not maintain their structure under generic linear changes of coordinates: for a generic linear map $\phi\colon \C^n \to \C^n$, the composition $f(\phi(z))$ has dense support $\deg(f)\Delta_{n}$. Supports do, however, respect partial evaluation in the following sense. Let $\mydefMATH{\pi_I}\colon \Z^n \to \Z^{|I|}$ be the projection onto the coordinates indexed by $I \subset [n]$. 
\begin{lemma}
Let $f \in \C[x]$ be a polynomial with support $\mathcal A$ and let $a_{k+1},\ldots,a_n \in \C^\times$ be general. Then the support of $f(x_1,\ldots,x_k,a_{k+1},\ldots,a_n)$ is the projection $\pi_{[k]}(\mathcal A)$.
\end{lemma}

Supports of polynomials transform naturally under monomial changes of coordinates. Identifying the set $\Hom((\CC^\times)^n,\CC^\times)$ of characters on $(\CC^\times)^n$ with the free abelian group
$\ZZ^n$, a homomorphism $\Phi\colon(\CC^\times)^m\to(\CC^\times)^k$ is determined by $k$ characters of $(\CC^\times)^m$,
equivalently by a homomorphism (linear map) $\varphi\colon\ZZ^k\to\ZZ^m$ of free abelian groups. Note that $\varphi$ is also the map
pulling a character of $(\CC^\times)^k$ back along $\Phi$.
In particular, an invertible map $\Phi\colon(\CC^\times)^n\to(\CC^\times)^n$ (a monomial change of coordinates) pulls back to an
invertible map $\varphi\colon\ZZ^n\to\ZZ^n$, identifying $\GL(n,\ZZ)$ with the group of possible monomial coordinate changes.
We will write 
$\Phi=\varphi^* \text{ and } \varphi=\Phi^*$ 
for these, not to be confused with the notation for the homomorphism of coordinate rings induced by a regular map of varieties. 
If $\Phi(x)=(x^{\alpha_1},\dotsc,x^{\alpha_n})$ where the integer span of $\{\alpha_1,\dotsc,\alpha_n\}$ is $\ZZ^n$, then
the map $\varphi=\Phi^*\colon\ZZ^n\xrightarrow{\sim}\ZZ^n$ sends the $i$-th standard basis vector $e_i$ to $\alpha_i$ and is
represented by the invertible matrix $A$ whose $i$-th column is $\alpha_i$.

Suppose that $f$ is a polynomial on $(\CC^\times)^n$ with support $\calA$.
Given a homomorphism $\Phi\colon(\CC^\times)^m\to(\CC^\times)^n$, the composition $f(\Phi(z))$ for $z\in(\CC^\times)^m$ is
a polynomial supported on $\varphi(\calA)$, where the coefficient of $z^\beta$ is the sum of coefficients of $x^\alpha$ for
$\alpha\in \varphi^{-1}(\beta)\cap\calA$.
For generic choices of coefficients of $x^\alpha$, this sum is nonzero and so $f(\Phi(z))$ has support $\varphi(\calA)$.

\subsection{Smith normal form}
\label{subsubsection:smithnormalform}
Let $\calA=\{0,\alpha_1,\dotsc,\alpha_m\}\subset\ZZ^n$ be a collection of integer vectors.
The sublattice $\ZZ\calA\subset\ZZ^n$ that it generates is the image of a $\ZZ$-linear map $\ZZ^m\to\ZZ^n$ and is
represented by a $n\times m$ integer matrix \defcolor{$A$} whose columns are the vectors $a_i$.
Suppose that $\ZZ\calA$ has rank $k$.
The \demph{Smith normal form} of $A$ is a factorization into integer matrices
 \begin{equation}
  A\ =\ PDQ\,,
 \end{equation}
where $P\in\GL(n,\ZZ)$ and $Q\in\GL(m,\ZZ)$ are invertible,
and $D$ is the rectangular matrix whose only nonzero entries are $d_1,\dotsc,d_k$ along the diagonal of its principal 
$k\times k$ submatrix.
These are the  \demph{invariant factors} of $A$ and they satisfy $d_1|d_2|d_3|\dotsb|d_k$.
The sublattice $\ZZ\calA\subset\ZZ^n$ has a basis given by the columns of the matrix $PD$.
If we apply the coordinate change $P^{-1}$ to $\ZZ^n$, then $\ZZ\calA$ becomes the subset of the coordinate space
$\ZZ^k\oplus\bzero^{n-k}$ given by $d_1\ZZ\oplus d_2\ZZ\oplus\dotsb\oplus d_k\ZZ\oplus\bzero^{n-k}$.

The Smith normal form is also useful in solving binomial equations over $(\C^\times)^n$. Fix a collection $F \subset \C[x]$ of binomials 
\begin{align*}
a_1x^{\alpha_1}=b_1x^{\beta_1} \quad a_2x^{\alpha_2}=b_2x^{\beta_2} \quad \ldots \quad a_nx^{\alpha_n}=b_nx^{\beta_n}
\end{align*}
 with $a_i,b_i \in \C^\times$ for $i=1,\ldots,n$. Recall that we can scale the equations and translate their support so that $a_i=1$ and  $\beta_i=0$ for all $i=1,\ldots,n$. We now assume our system $F$ is of the form
\begin{align}
\label{eq:binomial}
x^{\alpha_1}=b_1 \quad x^{\alpha_2}&=b_2  \quad \ldots \quad
x^{\alpha_n}=b_n.  
\end{align}
It is useful to use matrices as exponents. For example, we encode $x^{\alpha_1}$ as $(x_1,\ldots,x_n)^{((\alpha_1)_1,\ldots,(\alpha_1)_n)}$. Letting $A$ be the matrix whose columns are $\alpha_1,\ldots,\alpha_n$, we write $x^{ A} = x^{(\alpha_1,\ldots,\alpha_n)}=(x^{\alpha_1},\ldots,x^{\alpha_n})$ so that \eqref{eq:binomial} is written as $x^A=b$.

Assume for simplicity that $\mathcal A$ spans $\R^n$ so that $d_n$ of a  Smith normal form $ A= PDQ$ is nonzero. Then 
$$(x^A)^{Q^{-1}} = b^{Q^{-1}}$$
and setting $z^{P^{-1}} = x$ gives
\begin{equation}
\label{eq:easybinomial}
z^{P^{-1} A Q^{-1}} = z^{D} = b^{Q^{-1}}.
\end{equation}
whose solutions are clearly the set $\mathcal Z = \{z \mid z_i \text{ is a } d_i\text{-th root of } b_i\}$ of $\prod_{i=1}^n d_i$ points.
Taking $x=z^P$ expresses these solutions in terms of $x$.

\subsection{Centroids and trace curves}
Given an affine variety $X \subset \C^n$ and a generic linear space $L$ of complementary dimension to $X$, the intersection $X \cap L$ is finite and consists of $\deg(X)$ points. The \mydef{centroid} of $X \cap L$, denoted $\mu(X \cap L)$ is the coordinate-wise average of those points. A family of linear spaces $\{L_t\}_{t \in \C}$ is a \mydef{pencil} if there exists a vector $v \in \C^n$ such that $L_t=t\cdot v + L_0$ for all $t \in \C$. The following lemma is the basis for the numerical algorithm known as the \mydef{trace test} (see Section~\ref{subsubsection:solvingviamonodromy}). 
\begin{lemma}
\label{lem:traceline}
Let $X \subset \C^n$ be an irreducible affine variety and let $L_t$ be a general pencil of linear spaces of complementary dimension. The Zariski closure of the union $$\mydefMATH{\mu(X \cap L_t)}=\bigcup_{t \in \C} \mu(X \cap L_t)$$ is  an affine line. 
\end{lemma}
\begin{proof}
Let $X\subset \C^n$ be an irreducible affine variety of dimension $m$. Observe that if $L$ is a linear space of complementary dimension, then $\pi(\mu(X \cap L))=\mu(\pi(X) \cap \pi(L))$ where $\pi\colon  \C^n \to \C^{n-1}$ is any projection such that $\dim(\pi(L))=\dim(L)-1$. Projecting this way $n-m-1$ times produces $\pi'\colon \C^n \to \C^{m+1}$ so that $\dim(\pi'(L))=1$. Thus, $\pi'(X)$ is a hypersurface in $\C^{m+1}$ and $\mu(X \cap L) \in \pi'^{-1}(\mu(\pi'(X) \cap \pi'(L))$. Let $v_1,\ldots,v_{n-m}$ span $L$ and define $\pi_i\colon \C^n \to \C^{m+1}$ to be the projection such that $\dim(\pi_i(v_j))=0$ whenever $i \neq j$. Then the intersection 
$$\bigcap_{i=1}^{n-m} \pi_i^{-1}(\mu(\pi_i(X) \cap \pi_i(L)))$$ is the point $\mu(X \cap L)$. Thus, it is enough to prove the statement for when $X$ is a hypersurface.

Let $X \subset \C^n$ be a hypersurface, let  $L_t$ be a general pencil of lines, and let $P=\bigcup_{t \in \C} L_t$. Consider $X'=X \cap P$. By Lemma \ref{cor:hyperplaneresult1}, $X'$ is a curve and so it is enough to prove the statement for plane curves.

Suppose $\V(f)=X \subset \C^2$ is a plane curve of degree $d$, and $L_t$ a general pencil of lines. After an action by rotation, we may assume that $L_t$ is the family $\V(x-t)$. This rotation is a generic linear change of coordinates because the family $L_t$ is general and so the support of $f$ must be $d\Delta_{n}$. Since scaling does not change the zero set, we assume that the coefficient of $y^{d}$ is one. Then $X \cap L_t = X \cap \V(x-t)$ has points $\{(t,y_i(t))\}_{i=1}^{d}$ where $y_i(t)$ are the zeros of
$$f(t,y) = \prod_{i=1}^{d} (y-y_i(t)) = y^{d}-(y_1(t)+\cdots+y_{d}(t)) y^{d-1} + \cdots$$
for some rational functions $y_i(t)$.
On the other hand, the coefficient of $y^{d-1}$ in $f \in \C[x][y]$ is $c_{(1,d-1)}x+c_{(0,d-1)}$ and so $-(y_1(t)+\cdots+y_d(t)) = c_{(1,d-1)}x+c_{(0,d-1)}$. Since the $y$-coordinate of $\mu(X \cap L_t)$ is $\frac{1}{d} (y_1(t)+\cdots+y_d(t))$, the points satisfying $-dy=c_{(1,d-1)}x+c_{(0,d-1)}$ are the points which are centroids of this family. In other words, the centroids are on the graph of the function
\begin{equation}
\label{eq:traceline}
y=-\frac{1}{d} (c_{(1,d-1)}x+c_{(0,d-1)}).
\end{equation}
\end{proof}
The line of centroids guaranteed by Lemma \ref{lem:traceline} is called the \mydef{trace line} of $X$ with respect to $L_t$.

\begin{figure}[!htpb]
\includegraphics[scale=0.5]{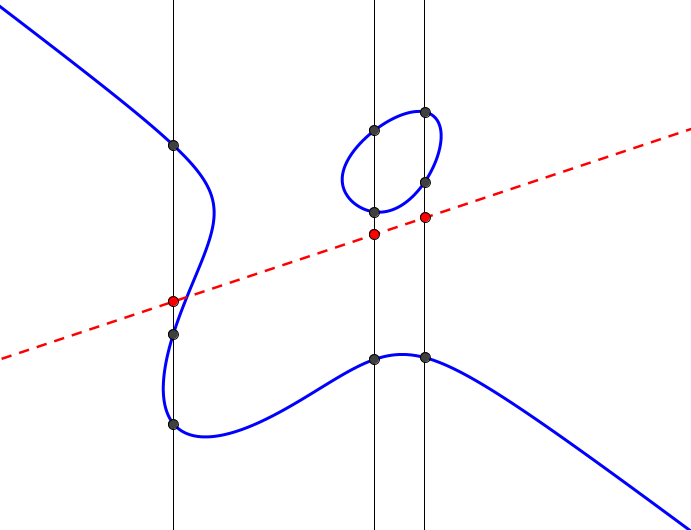}
\caption{\label{fig:tracetest}A plane cubic $\V(f)$ (blue), the trace line $\mu(\V(f,x-t))$ (red), and the specific centroids $\mu(\V(f,x+3)),\mu(\V(f,x-1)),\mu(\V(f,x-2))$.}
\end{figure}

\begin{example}
Let $$f=2-4x+x^3 +(-2-2x)y+(3-x)y^2+y^3$$
and let $L_t=\V(x-t)$
so that \eqref{eq:traceline} computes the trace line of $\V(f)$ to be $\V\left(y-\frac{1}{3}x+1\right)$. The cubic $\V(f)$ and its trace line are depicted in Figure \ref{fig:tracetest}. 
Notice that even though many lines $L_t$ do not intersect $\V(f)$ in three real points, the centroids are still real. This is because the points $\V(f) \cap L_t$ must appear in complex conjugates and so their imaginary parts will cancel in the average. \hfill $\diamond$\end{example}
When the Newton polytope of a plane curve $X$ of degree $d$ is smaller than $d\Delta_2$ the family of lines $\V(x-t)$ is not generic with respect to $X$. Therefore, Equation \ref{eq:traceline} does not compute the curve of centroids. In particular, the closure of these centroids may not be a line. The following result gives a formula for the curve of centroids when the family $\V(x-t)$ is not generic with respect to $X$.
\begin{lemma}
\label{lem:nonlineartrace}
Suppose $$f=\sum_{(i,j) \in \mathcal A} c_{i,j}x^iy^j \in \C[x,y]$$ for $\mathcal A \subset \Z_{\geq 0}^2$ and $L_t=\V(x-t)$. Then $$\overline{\bigcup_{t \in \C}\mu(\V(f) \cap L_t)}=\V\left(\sum_{i=0}^{\deg_{x}(f)} c_{i,\deg_{y}(f)-1}   x^i + \deg_{y}(f)y\left(\sum_{i=0}^{\deg_{x}(f)} c_{i,\deg_{y}(f)} x^i \right)\right),$$
where $\deg_x(f)=\max_{(i,j) \in \mathcal A} (i)$ and $\deg_y(f)=\max_{(i,j) \in \mathcal A} (j)$. 
\end{lemma}
\begin{proof}
As with the proof of Lemma \ref{lem:traceline}, we take 
\begin{align*}
f(t,y) &=  \left(\sum_{i=0}^{\deg_x(f)} c_{i,\deg_y(f)} t^i \right) y^d + \left(\sum_{i=0}^{\deg_x(f)} c_{i,\deg_y(f)-1} t^i \right)y^{\deg_y(f)-1} + \cdots
\end{align*}
and writing $f(t,y)$ as a monic polynomial tells us that 
$$-(y_1(t)+\cdots+y_{\deg_y(f)}(t)) = \frac{\left(\sum_{i=0}^{\deg_x(f)} c_{i,\deg_y(f)-1} t^i \right)}{\left(\sum_{i=0}^{\deg_x(f)} c_{i,\deg_y(f)} t^i \right)}.$$
Since $x=t$ and the $y$-coordinate of $\mu(X \cap L_t)$ is $\frac {1}{\deg_y(f)}(y_1(t)+\cdots+y_{\deg_y(f)}(t))$ we write this as $$-\deg_y(f)y=\frac{\left(\sum_{i=0}^{\deg_x(f)} c_{i,\deg_y(f)-1} x^i \right)}{\left(\sum_{i=0}^{\deg_x(f)} c_{i,\deg_y(f)} x^i \right)},$$
and clearing denominators gives the result.
\end{proof}
When the family $L_t$ is not general as in Lemma \ref{lem:nonlineartrace}, we define the \mydef{trace curve} of $X$ with respect to $L_t$ to be the closure of the set of centroids of $X \cap L_t$ for $t \in \C$.
\begin{example}
Consider the quartic curve $$f=1-x+x^2+(5+x-3x^2)y+(-3+3x-x^2)y^2$$ in $\C^2$ whose Newton polytope, support, and coefficients are depicted in Figure \ref{fig:nonlineartrace}. 
\begin{figure}[!htpb]
\includegraphics[scale=0.35]{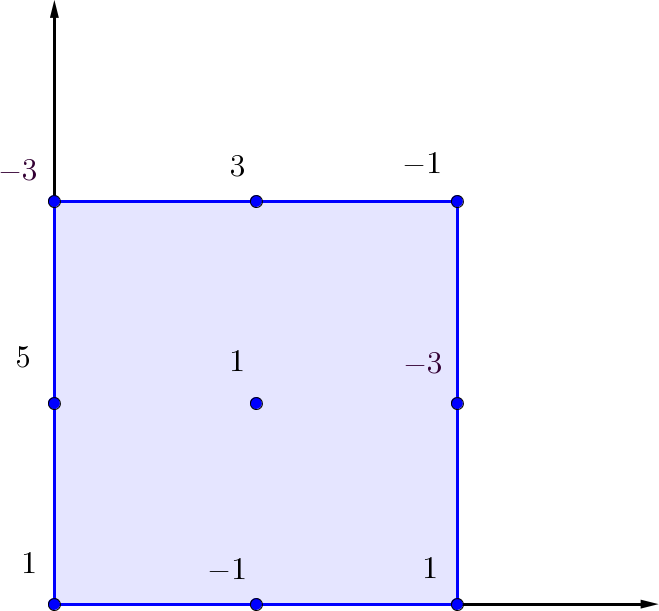}
\includegraphics[scale=0.35]{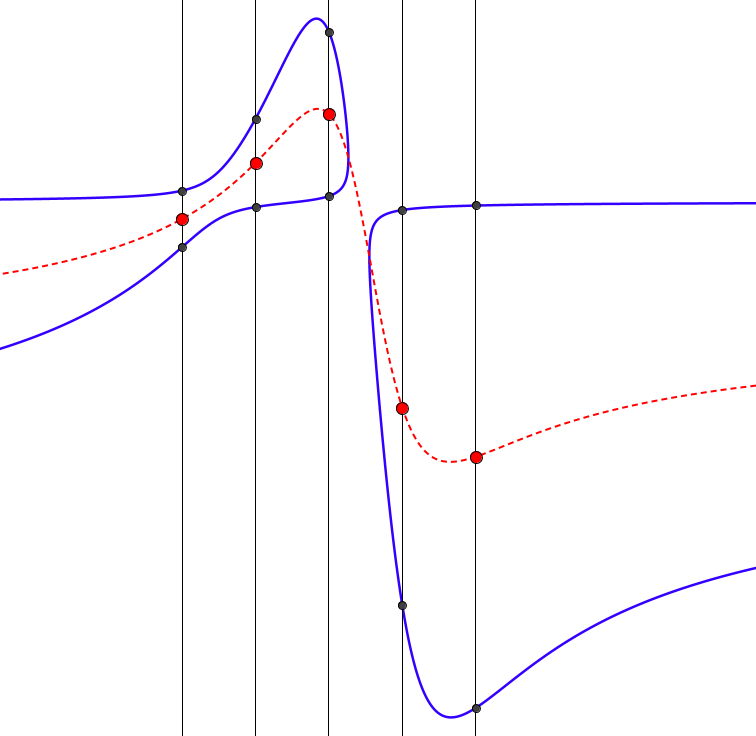}
\caption{Left: The Newton polytope, support, and coefficients of $f$. Right: The curve $\V(f)$ (blue), the trace curve of $\V(f)$ with respect to $\V(x-t)$ (red), and five lines in the family $L_t$ along with the centroids of their intersections with $\V(f)$.}
\label{fig:nonlineartrace}
\end{figure}

The equation of the trace curve of $\V(f)$ with respect to the nongeneric family of lines $\V(x-t)$ is 
$$g=(5+x-3x^2)+2y(-3+3x-x^2).$$ Lemma \ref{lem:nonlineartrace} essentially states that the equation $g$ can be read off from the coefficients of the top two rows of the polytope $\New(f)$. \hfill $\diamond$
\end{example}
If $X=X_1 \cup \cdots \cup X_r \subset \C^n$ is a reducible affine variety and $L$ is a generic linear space of complementary dimension, then $\mu(X \cap L)=\frac{1}{\deg(X)} \sum_{i=1}^r \mu(X_i \cap L)\cdot \deg(X_i)$ and so we have the following corollary.
\begin{corollary}
Let $X \subset \C^n$ be an affine variety which is possibly reducible and let $L_t$ be a general pencil of linear spaces of complementary dimension. The union of the centroids of the intersections $X \cap L_t$ is an affine line.
\end{corollary}

\section{Tropical geometry}
\label{subsection:tropicalgeometry}

Newton polytopes are intimately related to tropical geometry. We only begin to touch on the topic here and encourage the reader to reference  \cite{MacSturm} for a more extensive treatment.

The tropicalization of a variety depends on the choice of a valuation $\nu$ on the base field involved (in our case $\mathbb{C}$). Relevant to this document is the trivial valuation: $\nu(c)=0$ for all $c \in \mathbb{C}^\times$. With this valuation, the \mydef{tropicalization} of a polynomial $$f=\sum_{\alpha \in \mathcal A} c_\alpha x^\alpha,  \quad \mathcal A = \supp(f)$$ 
is the map \begin{align}\label{eqdef:tropf}\mydefMATH{\trop(f)}\colon \mathbb{R}^n &\to \mathbb{R} \\
\omega&\mapsto \max_{\alpha \in \calA}\langle \alpha, \omega \rangle \nonumber
\end{align} and the \mydef{tropicalization} of the hypersurface $\V(f)$ is 
\begin{equation}
\label{eqdef:tropX}
\mydefMATH{\trop(\V(f))}=\{\omega \in \mathbb{R}^n \mid \text{ the maximum in }\trop(f)(\omega)\text{ is attained at least twice}\}.
\end{equation}

By \eqref{eqdef:tropf}, $\trop(f)$ is the same function as $h_{\New(f)}$ and by \eqref{eqdef:tropX}, the tropicalization of $\V(f)$ is the codimension $1$ part of the normal fan of the Newton polytope of $f$, namely $\mathcal N^{(1)}(\New(f))$ (see Section~\ref{subsection:representingpolytopes}).  

Let $P=\New(f)$, and fix a monomial change of coordinates $\Phi\colon (\C^\times)^n \to (\C^\times)^n$  with $\varphi=\Phi^*$ so that we have $Q=\varphi(P)=\New(f \circ \Phi)$. The map $\varphi$ induces a map in the opposite direction on functionals $\{\alpha \mapsto \langle \alpha, \omega \rangle \mid \omega \in \R^n\}$. Consequently, $\omega$ is an element of $\trop(\V(f))$ if and only if $\varphi^{-1}(\omega) \in \trop(\V(f \circ \Phi))$ and so
$$\varphi^{-1}(\trop(\V(f))) = \trop(\V(f \circ \Phi)),$$
or equivalently,
\begin{equation}
\trop(\V(f))=\varphi(\trop(\V(f \circ \Phi))).
\label{eq:tropunderchange}
\end{equation}

The \mydef{tropicalization} of $\V(I)$ for some ideal $I \subseteq \mathbb{C}[x_1,\ldots,x_n]$ is the intersection $$\mydefMATH{\trop(\V(I))}=\bigcap_{f \in I} \trop(\V(f)).$$

Hept and Theobald in \cite{HeptTheobald}, motivated by the results of Bieri and Groves in \cite{BieriGroves},  investigated how to write $\trop(\V(I))$ as an intersection of finitely many tropical hypersurfaces coming from projections. The following is a consequence of the proof of Theorem 1.1 in \cite{HeptTheobald}.
\begin{theorem}
 If $I\subseteq \mathbb{C}[x]$ is an $m$-dimensional prime ideal, and $\{\pi_i\colon \R^n \to \R^{m+1}\}_{i=0}^{n-m}$ are generic projections, 
 $$\trop(\V(I)) =  \bigcap_{i=0}^{n-m} \pi_i^{-1}(\pi_i(\trop(\V(I)))$$ 
 where each $\pi_i^{-1}(\pi_i(\trop(\V(I))))$ is a tropical hypersurface. 
\label{thm:TropicalProjection}
\end{theorem}
Coordinate projections are not always generic and it is possible that only the proper containment $$\bigcap_{\substack{J \subseteq [n]\\ \text{codim}(\pi_J(\V(I)))=1}} \pi_J^{-1}(\pi_J(\V(I))) \subsetneq \trop(\V(I)) $$ holds where $\pi_J$ is the projection onto the coordinates indexed by $J \subset [n]$. 
\begin{remark}
The notion of genericity involved in Theorem \ref{thm:TropicalProjection} comes from that of a geometrically regular projection. Let $Y$ be a union of $m$-dimensional linear subsets of $\R^n$. A projection $\pi\colon \R^{n} \to \R^{m+1}$ is \mydef{geometrically regular} with respect to $Y\subset \R^n$ if the image of $k$-dimensional linear subspaces of $Y$ remain $k$-dimensional and $\pi$ respects containments: $\pi(Y_1) \subset \pi(Y_2) \implies Y_1 \subset Y_2$. These properties form an open dense subset within the set of projections and taking $\pi_1,\ldots,\pi_{n-m}$ distinct such projections gives 
$$Y=\bigcap_{i=1}^{n-m} \pi_i^{-1}(\pi_i(Y)). $$
A tropical variety is contained in a union of finitely many linear spaces, but requires one more projection $\pi_0$ in order to write it as the intersection of preimages; this projection determines which part of each linear space belongs to the tropical variety. \hfill $\diamond$
\end{remark}
\begin{example}
The following is 
Example 4.2.11 in \cite{Chan}. Let
\begin{align*}
I_1&=\langle xz+4yz-z^2+3x-12y+5z,xy-4y^2+yz+x+2y-z\rangle, \\
I_2&=\langle xy-3xz+3yz-1,3xz^2-12yz^2+xz+3yz+5z-1 \rangle.
\end{align*} The varieties defined by these two ideals are curves in $\C^3$ whose tropicalizations are the rays from the origin to the positive (product of coordinates is positive) and negative vertices of the cube $[-1,1]^3$ respectively. We display both curves in Figure \ref{fig:tropcube}.
\begin{figure}[htpb!]
\includegraphics[scale=0.5]{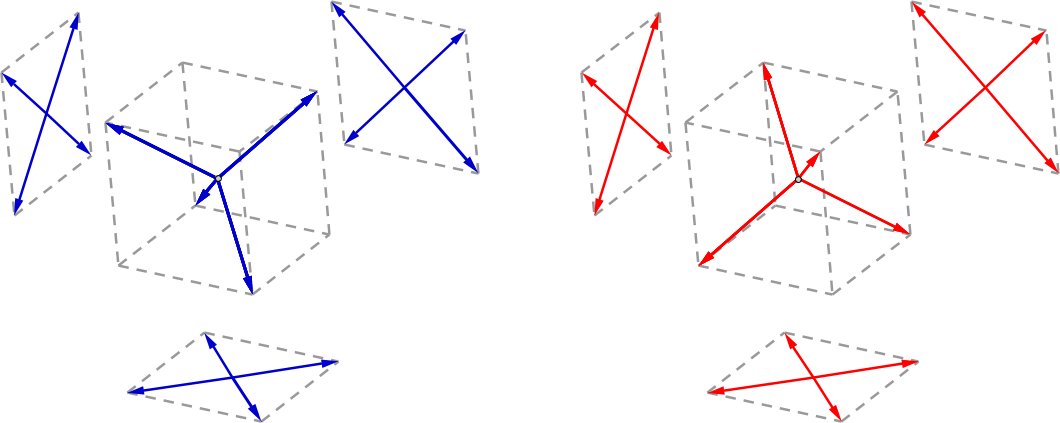}
\caption{(Reprinted from \cite{Bry:NPtrop}) An example of two tropical curves which cannot be distinguished from their coordinate projections}
\label{fig:tropcube}
\end{figure}
 Notice that for any $\{i,j\} \subset \{1,2,3\}$, we have that $\pi_{\{i,j\}}(\trop(\V(I_1)))=\pi_{\{i,j\}}(\trop(\V(I_2)))$ is the tropical plane curve whose rays are the positive span of the vertices of the square $[-1,1]^2$. Therefore, these two tropical curves cannot be distinguished from their coordinate projections. Note that these projections are not geometrically regular with respect to the union of linear spaces containing each tropical curve. \hfill $\diamond$
 \label{example:tropcube}
\end{example}

\begin{remark}
\label{remark:producingprojections}
By \eqref{eq:tropunderchange}, we have that $\trop(\V(f)) = \varphi(\trop(\V(f \circ \Phi)))$ for any monomial change of coordinates $\Phi$ and so for any $f_1,\ldots,f_m \in \C[x]$,
$$\varphi^{-1}(\trop(\V(f_1,\ldots,f_m))) = \trop(\V(f_1\circ\Phi,\ldots,f_m\circ \Phi)),$$
where $\varphi=\Phi^*$. Projecting gives
\begin{equation}
\label{eq:producingprojections}
\pi_{[k]}\varphi^{-1}(\trop(\V(f_1,\ldots,f_m))) = \pi_{[k]}\trop(\V(f_1\circ\Phi,\ldots,f_m\circ \Phi)).
\end{equation}
Thus, one way to produce a projection $A\colon \R^n \to \R^k$ on tropical varieties other than a coordinate projection is to write $A$ as $\pi_{[k]} \circ \varphi^{-1}$ such that $\varphi$ is an $n\times n$ matrix over $\ZZ$ and apply \eqref{eq:producingprojections}. \hfill $\diamond$
\end{remark}

\section{Sparse polynomial systems}

Given a collection ${\Adot}=(\calA_1,\dotsc,\calA_n)$ of nonempty finite subsets of $\ZZ^n$, write
$\defcolor{\CC^\Adot}=\CC^{\calA_1}\times\dotsb\times\CC^{\calA_n}$ for the vector space of $n$-tuples $F=(f_1,\dotsc,f_n)$ of
polynomials, where $f_i$ is supported on $\calA_i$, for each $i$.
An element $F\in\CC^\Adot$ corresponds to a system of polynomial equations
 \[
     f_1(x_1,\dotsc,x_n)\ =\ 
     f_2(x_1,\dotsc,x_n)\ =\ \dotsb\ =\ 
     f_n(x_1,\dotsc,x_n)\ =\ 0\,,
 \]
called a  \demph{sparse polynomial system} \demph{supported on $\Adot$}.
We write $F$ to refer to these equations or to their vector of coefficients, depending on context. For $\omega \in \R^n$, we let $\mydefMATH{F_\omega} = ((f_1)_\omega,\ldots,(f_n)_\omega)$. 
Letting $\Pdot=(P_1,\ldots,P_n)$ where $P_i = \conv(\mathcal A_i)$, we define the mixed volume $\mydefMATH{\MV(\Adot)}$ of $\Adot$ to be $\MV(\Pdot)$.

\subsection{Geometry of sparse polynomial systems}
\label{subsection:sparsepolynomialsystems}
Given $\Adot=(\calA_1,\dotsc,\calA_n)$, consider the incidence variety
\[
  \defcolor{X_{\Adot}}\ =\ \left\{(F,x)\in\CC^{\Adot}\times(\CC^\times)^n \mid F(x) = 0  \right\}
\]
equipped with projections $\defcolor{\pi_{\Adot}}\colon X_{\Adot}\to\CC^{\Adot}$ and
$\defcolor{p}\colon X_{\Adot}\to(\CC^\times)^n$. For $F\in\CC^{\Adot}$, the fiber $\pi_{\Adot}^{-1}(F)$ is identified with the set $\calV(F)$ of solutions in $(\CC^\times)^n$ to $F=0$.

For any $x\in(\CC^\times)^n$, the fiber $p^{-1}(x)$ is a codimension $n$ vector subspace of $\CC^{\Adot}$.
Indeed, for each $i=1,\dotsc,n$, the condition that $f_i(x)=0$ is a linear equation in the coefficients $\CC^{\calA_i}$ of 
$f_i$, and these $n$ linear equations are independent.
As a consequence $X_{\Adot}$ is irreducible of dimension 
\[
  \dim(\CC^\times)^n + \dim\CC^{\Adot} - n\ =\ \dim\CC^{\Adot}\,,
\]by Lemma \ref{lem:fibredim} and Lemma \ref{lem:irreducibilityfromfibres}.

\begin{proposition}[Bernstein-Kushnirenko]\label{prop:BKK}
  Let $F\in\CC^{\Adot}$ be a system of polynomials supported on $\Adot$.
  The number of isolated solutions in $(\CC^\times)^n$ to $F=0$ is at most $\MV(\Adot)$.
  There is a dense open subset $U\subset\CC^{\Adot}$ consisting of systems with exactly $\MV(\Adot)$ solutions.
\end{proposition}

Thus $\pi_{\Adot}\colon X_{\Adot}\to\CC^{\Adot}$ is a branched cover if and only if $\MV(\Adot)\neq 0$. When this is the case, we denote the Galois group of $\pi_{\Adot}$ by $\mydefMATH{G_{\Adot}}$. We remark that Proposition \ref{prop:BKK} gives a different way to compute the mixed volume of a collection of polytopes $P_\bullet$ than the formulas given in Section~\ref{subsection:mixedvolume}: solve a polynomial system whose Newton polytopes comprise the collection $P_\bullet$ and count the solutions in the algebraic torus.
A corollary of Proposition \ref{prop:BKK} is B\'ezout's theorem.
\begin{corollary}[B\'ezout]
Let $\Delta_\bullet=(d_1\Delta_{n},\ldots,d_n\Delta_{n})$ with $d_1,\ldots,d_n \in \mathbb{N}$. Then $\pi_{\Delta_\bullet}$ is a branched cover of degree $\prod_{i=1}^n d_i$. 
\end{corollary}

\chapter{NUMERICAL ALGEBRAIC GEOMETRY \label{section:numericalalgebraicgeometry}}
Numerical algebraic geometry refers to a collection of theoretical and computational techniques for studying algebraic varieties using numerical methods. Contrary to symbolic algorithms which use  the algebraic description of a variety as input, numerical methods represent varieties by computing approximations of points on them. This gives a computational paradigm which is almost entirely geometric, albeit, theoretically grounded in the algebra and geometry developed in Section~\ref{section:algebraicgeometry}. 

At its core, numerical algebraic geometry uses tools from numerical analysis to compute approx\-imate solutions of zero-dimensional polynomial systems. Computations on positive-dimen\-sional varieties are performed numerically via their zero-dimensional intersections with general affine linear spaces of complementary dimension. The information of such an intersection comprises the fundamental data structure in numerical algebraic geometry: a witness set. When equipped with the method of homotopy continuation, a witness set may be used to efficiently extract information from a variety. 

Understanding the basic concepts underlying numerical algebraic geometry does not require an extensive background in algebraic geometry, but the language from Section~\ref{section:algebraicgeometry} illuminates many of the ideas involved. For example, we will see that homotopy methods are conveniently chosen branched covers, a clever interpretation of the fibers, and a special (but not too special!) fiber which can be computed. 

We begin by briefly explaining the core numerical methods underlying the theory in Section~\ref{section:corenumericalmethods} and then move on to an assortment of algorithms from numerical algebraic geometry, including the polyhedral homotopy (Algorithm \ref{alg:polyhedralhomotopy}) and the monodromy solve algorithm (Algorithm \ref{alg:monodromysolve}). 
We remark that Figure \ref{fig:twistedpseudo} appears in the article \cite{Bry:NPtrop} by the author\footnote{Reprinted with permission from T. Brysiewicz, ``Numerical Software to Compute Newton polytopes and Tropical Membership,''  {\it{Mathematics in Computer Science,}} 2020. Copyright 2020 by Springer Nature.}.
\renewcommand*{\thefootnote}{\arabic{footnote}}

\section{Core numerical methods} 
\label{section:corenumericalmethods} We discuss what it means to numerically solve a polynomial system and explain two core numerical algorithms: Euler's method and Newton's method. These algorithms may be used as the predictor and corrector subroutines of a predictor-corrector method.
\subsection{Approximate solutions}
\label{subsection:approximatesolutions}
Given a polynomial map $F\colon \C^n \to \C^n$, the system $F=0$ is a collection of $n$ polynomial equations in $n$ variables and is thus called a \mydef{square system}. We suppose for now that $\V(F)$ is finite. For such a multivariate map, define  $$\mydefMATH{N_F(x)}=x-\left(DF\right)^{-1}F(x),$$
where $DF$ is the Jacobian matrix of $F$ evaluated at $x$, $\left(DF\right)^{-1}$ is its inverse, and $x$ and $F(x)$ are column vectors. We remark that $N_F(x)$ is only well-defined when $DF$ is nonsingular at $x \in \C^n$. Applying $N_F$ to a point $x_0 \in \C^n$ is called a \mydef{Newton step} on $x_0$, or a \mydef{Newton iteration}. A \mydef{Newton sequence} is a sequence of points $\{x_0,x_1,\ldots\}$ defined recursively from some initial point $x_0$ by $\mydefMATH{x_{i+1}}=N_F(x_i)$.  A sequence $\{x_0,x_1,\ldots\}$ \mydef{converges quadratically} to a point $\xi \in \C^n$ if for all $i$ 
$$||x_i-\xi|| \leq 2^{1-{2^i}}||x_0-\xi||.$$
\mydef{Newton's method} is a root-finding algorithm which iteratively applies Newton steps to some point $x_0 \in \C^n$ with the hope that  the Newton sequence $\{x_0,x_1,\ldots\}$ converges to a solution of $F=0$. 
\boxit{
\begin{algorithm}[Newton's Method]
\label{alg:newtonsmethod}
\nothing \\
\myline
{\bf Input:} \\
$\bullet$ A point $x_0 \in \C^n$\\
$\bullet$ A square polynomial system $F$\\
$\bullet$ Some number of iterations, $m \in \mathbb{N}$ \\
{\bf Output:} \\
$\bullet$ The $m$-th Newton iteration, $N^m_F(x_0)$\\
\myline
{\bf Steps:}
\begin{itemize}[nosep]
\item[1] \set $i=0$
\item[2] \while $i<m$ \mydo 
\begin{itemize}
\item[2.1] \set $x_{i+1} = x_i-(DF|_{x_i})^{-1}F(x_i)$
\item[2.2] \set $i=i+1$
\end{itemize}
\item[3] \return $x_m$
\end{itemize}
\end{algorithm}
}

\begin{lemma}\cite[Theorem 3.5]{NoceWrig06}
If $x_0$ is sufficiently near a smooth point $\xi\in \V(F)$ then a Newton sequence beginning with $x_0$ will converge quadratically to $\xi$.
\end{lemma}
A point $x_0 \in \C^n$ is an \mydef{approximate zero} of $F=0$
with \mydef{associated zero} $\xi\in \V(F)$ if the Newton sequence starting at $x_0$ converges quadratically to $\xi$. In this sense, $x_0$ is a numerical solution to $F=0$.

Certifying that a point is a numerical solution is made possible through $\alpha$-theory \cite[Ch.8]{BCSS}, developed by Smale \cite{smale} in the 1980's. We introduce the notation
\begin{align*}
\mydefMATH{\beta(F,x)} &= ||x-N_F(x)||=||DF(x)^{-1}F(x)||,\\
\mydefMATH{\gamma(F,x)} &= \sup_{k \geq 2} \bigg{|}\bigg{|} \frac{DF(x)^{-1} D^kF(x)}{k!} \bigg{|}\bigg{|}^{\frac{1}{k-1}},\\
\mydefMATH{\alpha(F,x)} &= \beta(F,x) \cdot \gamma(F,x),
\end{align*}
where $D^kF(x)$ is the symmetric tensor comprised of the $k$-th order partial derivatives of $f$. Since $D^kF$ is a linear map from the $k$-fold symmetric power of $\C^n$ to $\C^n$, so is $DF(x)^{-1}  D^kF(x)$. The norm in the definition of $\gamma(F,x)$ is the operator norm induced by the standard norms on $\C^n$ and the symmetric powers of $\C^n$. With this notation, we state a sufficient condition on quadratic convergence which forms the basis for $\alpha$-theory.
\begin{proposition}
\label{prop:alphatheory}
A point $x_0 \in \C^n$ is an approximate solution of a square system $F=0$ if $\alpha(F,x_0) <(13 -3 \sqrt{17})/4 \approx 0.15767078$.
\end{proposition}
Given a point $x \in \C^n$ and a square polynomial system $F=0$, software such as {\bf alphaCertified} \cite{HS12} and {\bf NumericalCertification} \cite{M2certification} verify the inequality in Proposition \ref{prop:alphatheory} and can thus rigorously certify that $x_0$ is an approximate solution of $F=0$.

\subsection{Euler's method} Euler's method is a standard numerical method for solving a first order ordinary linear differential equation given an initial value. 
Fix an ordinary linear differential equation encoded via a matrix equation
$$\frac{\partial x}{\partial t} = F(t;x(t)), \hspace{0.5 in} x(t_0)=x_0,$$
where $F(t;x(t))\colon  \C_t \times \C_x^n  \to \C^n$ is continuous near $(t_0;x_0)$ in $\C_t \times \C_x^n$. Fix a step size $h >0$ and define
$$\mydefMATH{E_F(t;x)}=x+hF(t;x).$$ Applying $E_F$ to a point $(t_0;x_0)$ is called an \mydef{Euler step}. An \mydef{Euler sequence} is a sequence of points $\{(t_0;x_0),(t_1;x_1),\ldots\}$ where $t_{i+1}=t_i-h$ and $x_{i+1} = E_F(t_i;x_i)$.  

Analogous to Newton's method, given a step size $h$ and a number of steps $m$, \mydef{Euler's method} attempts to compute an approximation $x_m$ of $x(t_m)$.

\boxit{
\begin{algorithm}[Euler's method]
\label{alg:eulersmethod}
\nothing \\
\myline
{\bf Input:} \\
$\bullet$ A first order linear differential equation $\frac{\partial x}{\partial t} = F(t;x)$\\
$\bullet$ An initial value $x(t_0)=x_0$\\
$\bullet$ A step size $h$ \\
$\bullet$ A number of steps $m$\\ \myline
{\bf Output:} \\
$\bullet$ An approximation $x_m$ of $x(t_m)$ \\
\myline
{\bf Steps:}
\begin{itemize}[nosep]
\item[1] \set $i=0$
\item[2] \while $i<m$ \mydo 
\begin{itemize}
\item[2.1] \set $x_{i+1}=E_F(t_i;x_i)$
\item[2.2] \set $t_{i+1}=t_i-h$
\end{itemize}
\item[3] \return $x_m$
\end{itemize}
\end{algorithm}}

\begin{example}
\label{ex:eulersmethod}
Figure \ref{fig:Euler1} displays four branches of a curve $\V(F) \subset \C^2_{t,x}$ where
$$F(t;x)=5(1-t)(x-0.1)(x-0.4)^2(x-0.6)+t(x-0.25)(x-0.5)(x-0.75)(x-0.05).$$ The branch containing the point $(1;0.75)$ is the graph of some function $x(t)\colon [0,1] \to \R^2$ satisfying $F(t,x(t))= 0$ for $t \in [0,1]$ and thus satisfying the differential equation $DF(t;x(t))=0$. After applying the chain rule, this becomes, $$\frac{\partial x}{\partial t} =- \frac{-4x^4+5.95x^3-3.1375x^2+.671875x-.0433125}{ -16tx^3+17.85tx^2+20x^3-6.275tx-22.5x^2+.671875t+7.8x-.8
}$$ We perform Algorithm \ref{alg:eulersmethod} on this differential equation using the auxiliary input $$x_0=x(1)=0.75, \quad h=0.1, \quad \text{ and } \quad m=10,$$ so that $x_m=x(0)$. The computed points $\{(t_i;x_i)\}_{i=0}^m$ are shown in Figure \ref{fig:Euler1} in green.
\begin{figure}[!htpb]
\includegraphics[scale=0.75]{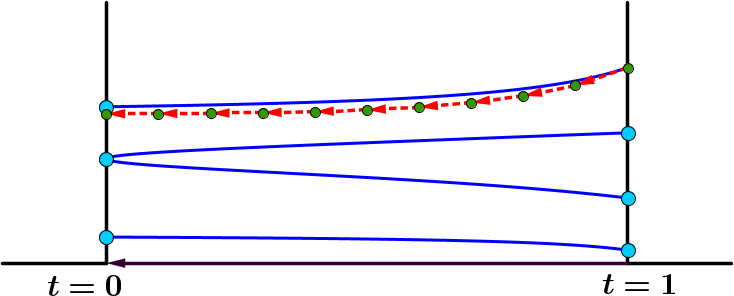}
\caption{Algorithm \ref{alg:eulersmethod} with $h=0.1$, $x(1)=0.75$, and $m=10$. }
\label{fig:Euler1}
\end{figure}
 \hfill $\diamond$
\end{example}

\subsection{Predictor-corrector methods}
Given a differential equation 
\begin{equation}
\frac{\partial x}{\partial t} = F(t;x(t))
\label{eq:diffeq}
\end{equation}
and some starting point $x(t_0)=x_0$ satisfying \eqref{eq:diffeq}, a \mydef{predictor-corrector method} attempts to analytically continue $x(t)$ as $t$ goes from $t_0$ to some $t_{m} \in \R$ (taking $h=\frac{t_m-t_0}{m}$) by interspersing applications of a predictor method (like Euler's method) and a corrector method (like Newton's method). Combining both prediction and correction increases the accuracy of  $(t_m,x_{m})$ dramatically over the use of Euler's method alone (see Example \ref{ex:betterthanjusteuler}).

Predictor-corrector methods are versatile and depend on choices of
\begin{enumerate}
\item a differential equation,
\item a predictor method,
\item a corrector method,
\item the parameters involved in both the predictor and the corrector methods.
\end{enumerate} 
We give a predictor-corrector method below when the predictor and corrector steps are Euler's method and Newton's method respectively. Thus, this algorithm requires both a differential equation and a system of equations $G$ satisfying $G(t;x(t))=0$ for $t \in [0,1]$ as input. We remark that this is an extremely simple version of such an algorithm and in practice, predictor-corrector methods are often much more nuanced, using predictor methods with higher accuracy, applying Newton's method repeatedly, and adapting the step size throughout the process as needed. 
\boxit{
\begin{algorithm}[Predictor-Corrector]
\label{alg:predictorcorrector}
\nothing \\
\myline
{\bf Input:} \\
$\bullet$ A system of equations $G(t;x)$ such that $G(t;x(t))=0$ for all $t \in [0,1]$\\
$\bullet$ A first order linear differential equation $\frac{\partial x}{\partial t} = F(t;x)$ satisfied by $x(t)$\\
$\bullet$ An initial value $x(t_0)=x_0$\\
$\bullet$ A step size $h$ \\
$\bullet$ A target $t$-value, $t'$\\  \myline
{\bf Output:}\\
$\bullet$ An approximate solution of $x(t')$\\
\myline
{\bf Steps:}
\begin{itemize}[nosep]
\item[0] \set $m=\left\lfloor \frac{(t_0-t_m)}{h}\right\rfloor$ so that $t_m-h<t'<t_m$
\item[1] \set $i=0$
\item[2] \while $i<m$ \mydo 
\begin{itemize}
\item[2.1] \set $x_{i+1}=E_F(t_i;x_i)$
\item[2.2] \set $t_{i+1}=t_i-h$
\item[2.3] \set $x_{i+1}=N_{G(t_{i+1};x)}(x_{i+1})$
\end{itemize}
\item[3] \set $x_{m+1}=E_F(t_m;x_m)$ using a stepsize of $t_m-t'$
\item[4] \set $x_{m+1}=N_{G(t';x)}(x_{m+1})$
\item[5] \return $x_{m+1}$
\end{itemize}
\end{algorithm}
}

\begin{example}
\label{ex:betterthanjusteuler}
Figure \ref{fig:predictorcorrector} illustrates the accuracy increase in Algorithm \ref{alg:predictorcorrector} compared to Euler's method alone. We list the numerical data in Table \ref{tab:numericaldata}. \hfill $\diamond$
\begin{table}[!htpb]\scriptsize{
\begin{tabular}{|l|c|c|c|c|c|c|c|c|c|c|c|}\hline
$t$-value&1&0.9&0.8&0.7&0.6&0.5 &0.4&0.3&0.2&0.1&0\\ \hline \hline
True values &.75& .69914& .668991& .649216& .6354& .625304& .617667& .611729& .607003& .603166& .6 \\ \hline \hline
Eul. Only&.75& .68175& .641803& .615861& .598597& .587539& .580981& .577285& .575157& .573848&.572984\\ \hline
Eul. Newt.&.75& .70252& .669168& .649393& .635461& .625331& .61768& .611735& .607006& .603168&.600001\\ \hline \hline
Eul. Err.&0& .01739& .027188& .033355& .036803& .037765& .036686& .034444& .031846& .029318& .027016\\ \hline
Eul. Newt. Err.&0& .00338& .000177& .000177& .000061& .000027& .000013& .000006& .000003& .000002& .000001\\ \hline
\end{tabular}}
\caption{\label{tab:numericaldata}Numerical data for Algorithm \ref{alg:predictorcorrector} on input from Example \ref{ex:eulersmethod}.}
\end{table}
\begin{figure}[!htpb]
\includegraphics[scale=0.66]{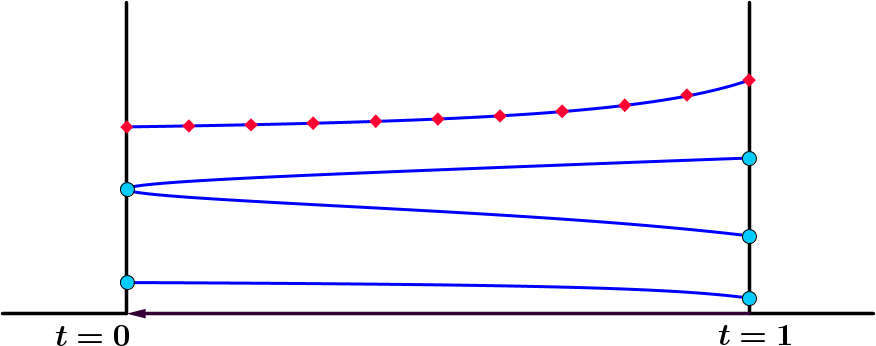}
\caption{Algorithm \ref{alg:predictorcorrector} applied to the same differential equation, initial value, and step size as in Example \ref{ex:eulersmethod}}
\label{fig:predictorcorrector}
\end{figure}
\end{example}

\subsection{Numerical errors}
\label{subsubsection:numericalerrors}
Since Algorithm \ref{alg:predictorcorrector} is numerical, it is subject to numerical errors. Due to the limits of rational computations, numerical methods require the approximation of numbers up to some precision. Applying  the linear maps relevant to Newton's method and Euler's method to these approximations could possibly increase their error. Algorithm \ref{alg:predictorcorrector} is especially prone to this whenever the matrix $DF$ evaluated at the approximation $x^*$ has a high \mydef{condition number}, $$\mydefMATH{\kappa(DF(x^*))}=||(DF(x^*))^{-1}||\cdot||DF(X^*)||.$$ When this happens, we say that the path is \mydef{ill-conditioned} at $x^*$. One way to alleviate issues coming from error accumulation due to low precision is to use \mydef{adaptive precision}. Adaptive precision involves changing the precision used during the predictor-corrector process based on indicators of the conditioning of the path being followed.

Another problem which could occur during Algorithm \ref{alg:predictorcorrector} is \mydef{path-jumping}. Path-jumping occurs when the result $(t^*,x^*)$ of an Euler step attempting to approximate $(t^*,x(t^*))$ is close enough to a solution $(t^*,\hat x(t^*)) \neq (t^*,x(t^*))$ so that a Newton sequence starting with $(t^*,x^*)$ converges to $(t^*,\hat x(t^*))$.  We display how this may occur in Figure \ref{fig:pathjumping}.
\begin{figure}[!htpb]
\includegraphics[scale=0.65]{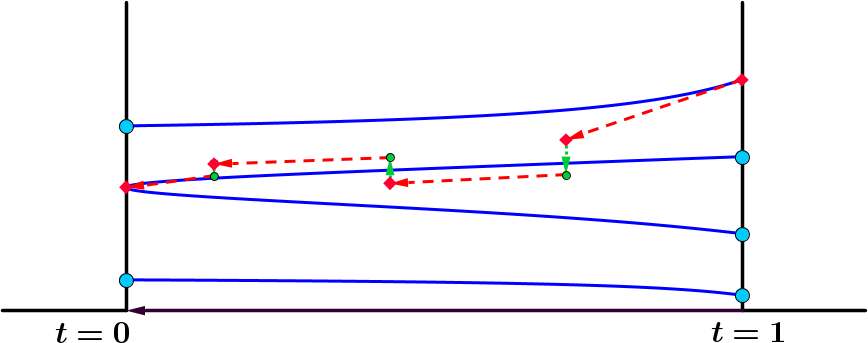}
\caption{\label{fig:pathjumping} A visual display of path-jumping.}
\end{figure}
One way to avoid path-jumping is to decrease the step size during the predictor-corrector process, particularly when $DF$ is has high condition number.

Developing robust or certifiable predictor-corrector methods is a goal of much current research \cite{BeltranLeykin, BurrYapXu, TelenBarelVerschelde}. 
For further information about these topics, we refer the reader to \cite{burgissercondition}.

\section{Homotopies} Homotopies make the idea of continuous deformations rigorous and are defined with respect to general topological spaces. For our purposes, we restrict ourselves to homotopies arising from polynomial systems.
Let $H(s;x) \in \C[s][x]$ be a system of $n$ polynomials in $m$ parameters $s^{(1)},\ldots,s^{(m)}$ and $n$ variables $x$.   Suppose the projection 
\begin{equation}
\label{eq:parametrizedcover}
\pi\colon  \V(H(s;x)) \to \C^m_s
\end{equation}
is a degree $d$ branched cover with regular values $U \subset \C_s$. The system $H(s;x)$ may also be thought of as a map
\begin{equation}
\label{eq:generalparameterhomotopy}
H(s;x)\colon  \C^m_s \times \C_x^n  \to \C^n.
\end{equation}
By the path-lifting property of covering spaces, composing $H(s;x)$ with any continuous path $\tau\colon [0,1]_t \to \C^m_s$ with $\tau(0,1] \subset U$ produces a map
\begin{equation}
\label{eq:parameterhomotopy}
H(\tau(t);x)\colon  [0,1]_t \times \C_x^n  \to \C^n
\end{equation}
along with $d$ lifts $\{\mydefMATH{x_i(t)}\}_{i=1}^d$ over $\tau(0,1]$ satisfying $H(\tau(t);x_i(t))= 0$ for all $t \in (0,1]$.
Set $\mydefMATH{H_\tau(t;x)}=H(\tau(t);x)$ so that for any $t^* \in (0,1]$, the polynomial system $H_\tau(t^*;x) \in \C[x]$ has $d$ solutions $\{x_i(t^*)\}_{i=1}^d$.
We call $H_\tau$ a \mydef{homotopy} with \mydef{start system} $H_\tau(1;x) \in \C[x]$ and \mydef{target system} $H_\tau(0;x) \in \C[x]$. We call $\{x_i(t)\}_{i=1}^d$ the \mydef{paths} of the homotopy $H_\tau$ and the set $\{x_i(1)\}_{i=1}^d \subset \C^n$ the \mydef{start solutions} of $H_\tau$. The isolated solutions of the target system $H_\tau(0;x)$ are called \mydef{target solutions}. A homotopy is called \mydef{regular} if additionally, $\tau(0) \in U$.

We omit the subscript on $H_\tau$ when convenient. 
When the limit $\lim\limits_{t \to 0} x_i(t)$ of some path exists, we extend $x_i(t)\colon (0,1] \to \C^n$ continuously by setting $x_i(0):=\lim\limits_{t \to 0}x_i(t)$.
If $H(t;x)$ is a homotopy, then for any $\lambda \in \C^\times$, we say $\lambda H(t;x)$ and $H(t;x)$ are \mydef{equivalent} and write $H(t;x)\equiv \lambda H(t;x)$ since the zeros of $\lambda H(t;x)$ are the same as those of $H(t;x)$.

\begin{lemma}
If $H(t;x)$ is a homotopy, then each target solution has the form $x_i(0)$ for some path $x_i(t)\colon [0,1] \to \C^n$ of the homotopy.
\end{lemma}
\begin{proof}
Suppose $H(t;x)=H(\tau(t);x)$ for $H(s;x) \in \C[s][x]$ and $\tau\colon[0,1] \to \C_s^m$. Let $U\subset \C_m$ be the set of regular values of $  \V(H(s;x)) \xrightarrow{\pi} \C_s$ and let  $p \in \V(H(0;x))$ be a target solution.

Since $\V(H(0;x))$ is nonempty, Corollary \ref{cor:nonemptyhasdimension} implies that $p$ belongs to to an irreducible component $C$ of $\V(H(s;x))$ of dimension at least $m$. But the dimension of $C$ is at most $m$ since $p$ is isolated in its fiber over $t=0$. Thus $\dim(C)=m$.

Since $C$ has dimension $m$ and the point $p$ in the fiber of $\pi|_C\colon C \to \C_s^m$ over $t=0$ is isolated in its fiber, $\pi(C)$ is open and dense in $\C^m_s$ and thus the intersection of $U$ and $\pi|_C(C)$ is open and dense. Considering the homotopy $H(t;x)\colon [0,1] \times \C_x^n \to \C^n$ as a map, observe that since $\pi|_C(C)$ is open and dense in $\C_s^m$, the set $H^{-1}(\bzero) \cap (0,\epsilon) \times \C_x^n$ contains points in $C$ for any $\epsilon >0$.  Such points must be of the form $x_i(\epsilon)$ for some path $x_i(t)$ of $H$ and thus $\lim\limits_{\epsilon \to 0} x_i(\epsilon)$ converges to $p \in C$.  
\end{proof}

\begin{lemma}
\label{lem:gammatrick}
Let $F(x),G(x) \in \C[x_1,\ldots,x_n]$ be square systems. Let $$H(s;x) = (1-s)F(x)+sG(x).$$
If $s=1$ is a regular value of $\pi\colon  \V(H(s;x)) \to \C_s$ then there exists a subset $S \subset \C \times \C$ of full measure such that for 
$\gamma=(\gamma_0,\gamma_1) \in S$,
\begin{equation}
\label{eq:gammatrick}
H_{\tau_\gamma}(t;x)\equiv (1-t)\gamma_0F(x)+t\gamma_1G(x)
\end{equation}
is a homotopy, where 
\begin{align*}
\tau_\gamma\colon [0,1] &\to \C_s\\
t &\mapsto \frac{t\gamma_1}{t\gamma_1+\gamma_0-t\gamma_0}.
\end{align*}
If $s=0$ is a regular value of $\pi$ as well, then $H_{\tau_\gamma}(t;x)$ is regular.
\end{lemma}

\begin{proof}
Let $\tau_\gamma(t) = \frac{t\gamma_1}{t\gamma_1+\gamma_0-t\gamma_0}$. We claim that $H(\tau_\gamma(t);x)$ has the same solutions as the right-hand-side of \eqref{eq:gammatrick} for any $t \in \C \smallsetminus \V(t\gamma_1+\gamma_0-t\gamma_0)$.
To see this, note that 
\begin{align*}
H(\tau_\gamma(t);x) &= (1-\tau_\gamma(t))F+\tau_\gamma(t)G \\
& = \left(1-\frac{t\gamma_1}{t\gamma_1+\gamma_0-t\gamma_0}\right)F+\frac{t\gamma_1}{t\gamma_1+\gamma_0-t\gamma_0}G 
\end{align*}
The denominator $t\gamma_1+\gamma_0-t\gamma_0$ is zero when $t = \frac{-\gamma_0}{\gamma_1-\gamma_0}$. 
When $t \neq \frac{\gamma_0}{\gamma_0-\gamma_1}$, the denominator is nonzero.  Thus, for $\gamma $ chosen in a subset of $\C \times \C$ of full measure, we can clear denominators without changing the solutions:
\begin{align*}
&= (t\gamma_1+\gamma_0-t\gamma_0 - t\gamma_1)F + t\gamma_1G \\
&=(1-t)\gamma_0F+t\gamma_1G.
\end{align*}

The branch locus $D$ of $\pi$ has complex codimension $1$ in $\C_t$, (i.e. $D$ is a finite set of points $d_1,\ldots,d_k$ in $\C_t$). We claim that the set of ratios $\gamma_0/\gamma_1$ with the property that $\tau_\gamma(t) = d_i$ for some $i$ and some $t \in [0,1]$ has measure zero in $\C \cong \R^2$. Because scaling does not change solutions, we may assume that $\gamma_1=1$. Note that  $\tau_\gamma(t) = \frac{t}{t+(1-t)\gamma_0}=1+\frac{1}{(1-t)}\frac{1}{\gamma_0}$ so $\tau_\gamma(t) = d_i$ if and only if $(d_i-1)(1-t) = \gamma_0^{-1}$ for some $t \in [0,1]$. Thus, the only $\gamma_0^{-1}$ for which $\tau_\gamma(t)=d_i$ for some $i=1,\ldots,k$ and  $t \in (0,1]$ are those whose inverses are contained on the finitely many half-open line segments $\{(d_i-1)(1-t)\}_{t \in (0,1]}$. This set has measure zero. Thus, the subset $S' \subset \C \times \C$ inducing such ratios has measure zero in $\C \times \C \cong \R^4$ and its complement $S=\C\times \C \smallsetminus S'$ has full measure in $\C \times \C$. Moreover, if $\tau_\gamma(0) \neq d_i$ for any $i=1,\ldots,k$, then $\tau_\gamma([0,1]) \cap \{d_1,\ldots,d_k\}=\emptyset$ for general $\gamma \in \C \times \C$ implying $H(t;x)$ is regular.
\end{proof}
A homotopy of the form 
\begin{equation}
\label{eq:straightlinehomotopy}
H(t;x)=(1-t)F(x)+tG(x)
\end{equation} is called a \mydef{straight-line homotopy}. Given two square polynomial systems $F(x)$ and $G(x)$, the construction \eqref{eq:straightlinehomotopy} may not be a homotopy, however, if $1$ is a regular value of $\pi\colon \V(H(t;x)) \to \C_t$, then \eqref{eq:gammatrick} is a homotopy with probability one: under any probability measure on the space $\C \times \C$ of choices for $\gamma$ in \eqref{eq:gammatrick}, the probability that $\gamma$ is chosen so that $H(\tau_\gamma(t);x)$ is a homotopy is one. Replacing $H(t;x)$ with \eqref{eq:gammatrick} is called the \mydef{$\gamma$-trick}. 
\begin{lemma}
\label{lem:generalgammatrick}
Suppose that $H(s;x) \in \C[s][x]$ is a square system and that $$\pi\colon \V(H(s;x)) \to \C^m_s$$ is a branched cover with regular values $U \subset \C^m_s$. If $s_1 \in U$ and $s_0 \in \C^m_s$, then there exists a path $\tau\colon [0,1] \to U$ such that $\tau(0)=s_0$ and $\tau(1)=s_1$ making $H_\tau$ a homotopy. If $s_0 \in U$ then $H_\tau$ is a regular homotopy.
\end{lemma}
\begin{proof}
Since $s_1 \in U$, the line connecting $s_0,s_1$ in $\C_s^m$  intersects the branch locus of $\pi$ in finitely many points. Parametrize this line by 
\begin{align*}
\tau'\colon \C &\to \C^m_s \\
q & \mapsto (1-q)s_0+qs_1.
\end{align*} 
Composing $\tau'$ with the map $\gamma$ from Lemma \ref{lem:gammatrick} for generic $(\gamma_0,\gamma_1) \in \C \times \C$ produces a path $\tau\colon (0,1] \to U$ so that $H_\tau$ is a homotopy, and additionally, if $s_0 \in U$ then $\tau\colon [0,1] \to U$ and $H_\tau$ is a regular homotopy.
\end{proof}
In light of this result, given a branched cover $\pi\colon X \to \C_s^m$, a value $s_0\in \C_s^m$, and a regular value $s_1\in \C_s^m$, we will henceforth use the phrase ``a homotopy from $s_1$ to $s_0$'' assuming that we take a homotopy as in Lemma \ref{lem:generalgammatrick}. 

\subsection{Homotopy continuation}
Given a homotopy $H(t;x)$ and some path $x(t)\colon (0,1] \to \C^n$ of the homotopy for which $x(1)$ is known, the method of path tracking uses the predictor-corrector algorithm  to analytically continue $x(t)$ as $t$ goes from $1$ toward $0$. Producing a differential equation satisfied by $x(t)$ is simple. By definition, $H(t;x(t))= 0$ and therefore,
\begin{equation}
\label{eq:davidenko}
DH(t;x(t))=0.
\end{equation}
Applying the chain rule to \eqref{eq:davidenko} gives
\begin{equation}
\label{eq:davidenko2}
D_tH+D_xH\cdot \frac{\partial x}{\partial t}=0.
\end{equation}
Reordering, this becomes the \mydef{Davidenko differential equation} \cite{Davidenko},
\begin{equation}
\label{eq:davidenko3}
\frac{\partial x}{\partial t}=-(D_xH)^{-1}D_tH
\end{equation}
which when used as the input to the predictor-corrector algorithm (Algorithm \ref{alg:predictorcorrector}) produces a \mydef{path tracking algorithm} for regular homotopies.

\boxit{
\begin{algorithm}[Path tracking for regular homotopies]
\label{alg:pathtrackingregular}
\nothing \\
\myline
{\bf Input:} \\
$\bullet$ A regular homotopy $H(t;x)$\\
$\bullet$ Approximate start solutions $S_1$ to $H(1;x)=0$\\
$\bullet$ A step size $h$ \\
 \myline
{\bf Output:}\\
$\bullet$ Approximate target solutions $S_0$\\
\myline
{\bf Steps:}
\begin{itemize}[nosep]
\item[1] \myfor $s \in S_1$ \mydo
\begin{itemize}[nosep]
\item[1.0] Let $x_s(t)$ be the path of $H(t;x)$ with $x_s(1)=s$
\item[1.1] \set $x_s(0)$ equal to the output of  Algorithm \ref{alg:predictorcorrector} using the input \begin{itemize}
\item[$\bullet$]Differential equation: $\frac{\partial x}{\partial t}=-(D_xH)^{-1}D_tH$
\item[$\bullet$]System of equations: $H(t;x)=0$
\item[$\bullet$] Initial value: $x_s(1)$
\item[$\bullet$] Step size: $h$
\item[$\bullet$] Target $t$-value: $0$
\end{itemize}
\end{itemize}
\item[2] \return $S_0:=\{x_s(0)\}_{s \in S}$
\end{itemize}
\end{algorithm}
}

Euler and Newton steps of the path tracking algorithm at $(t^*,x^*)$ are explicitly 
\begin{align}
\label{eq:specialeulernewton}
(x^*, t^*) &\xrightarrow{E} (t^*-h,x^*+h(D_xH(t^*;x^*))^{-1}D_t H(t^*;x^*))\\
(x^*, t^*) &\xrightarrow{N} (t^*;x^*-(D_xH(t^*;x^*))^{-1}H(t^*;x^*)). \nonumber
\end{align}

Equations \eqref{eq:specialeulernewton}  are only valid at the points $(t^*;x^*)$ where $D_xH(t^*;x^*)$ is invertible. This is the case at all points $(t;x(t))$ corresponding to a path $x(t)$ of the regular homotopy. When $H$ is not a regular homotopy, these conditions fail at $0$, but more importantly, they become computationally prohibitive near zero as described in Section~\ref{subsubsection:numericalerrors}.

\subsection{Endgames}
\label{subsection:endgames} 
We tame difficulties of homotopies at $t=0$ using endgame algorithms to produce the nonregular analog of Algorithm \ref{alg:pathtrackingregular}. Let $H(t;x)$ be a homotopy coming from lifting a path in $\C_s^m$ 
to paths $\{x_i(t)\}_{i=1}^d$ with respect to the branched cover
$$\pi\colon X \to \C_s^m.$$ If $H(t;x)$ is not regular, a path $x_i(t)$ may exhibit wild behavior near $t=0$ arising from one of two situations, each preventing the effective use of Algorithm \ref{alg:pathtrackingregular} on $H(t;x)$.
\begin{enumerate}
\item As $t \to 0$, the path $x(t)$ diverges.
\item The matrix $D_xH$ is not invertible at $p=(0;x(0))$ because
\begin{enumerate}
\item the rank of $DH|_p$ is $n-1$,
\item the rank of $DH|_p$ is $n$, but the rank of $D_xH|_p$ is $n-1$,
\item the rank of $DH_p$ is less than $n-1$.
\end{enumerate} 
\end{enumerate}
Figure \ref{fig:homotopy2} displays a homotopy where each instance occurs (in order from top to bottom).
\begin{figure}[!htpb]
\includegraphics[scale=0.5]{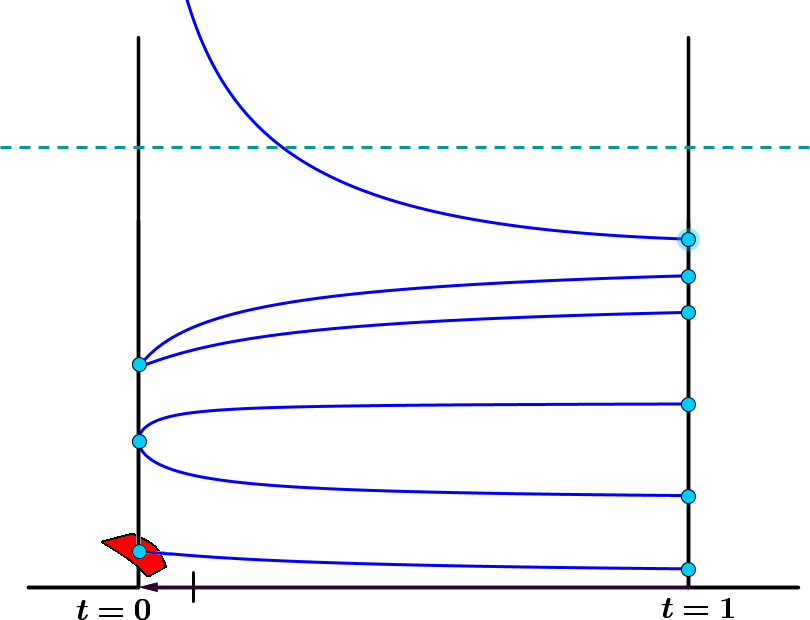}
\caption{A homotopy displaying possible behaviors at $t=0$.}
\label{fig:homotopy2}
\end{figure}
The default practical solution for handling $(1)$ is to simply truncate paths which seem to be diverging. This is assessed throughout the path tracking process by testing at each step whether $|x_i(t)|<N$ for some tolerance $N \gg 0$. If the test fails, the path is no longer tracked under the assumption that it is diverging. Another option is to homogenize the equations of $X$ and take a random dehomogenization. This involves choosing some hyperplane at infinity, and so long as this (real codimension $2$) hyperplane does not meet any of the homotopy paths (which have real dimension $1$), no path will diverge in the corresponding affine chart.

The next section deals with second case.

\subsection{Cauchy endgame}
Each instance of case $(2)$ may be handled the same way via the Cauchy endgame.

Let $x(t)$ be a path of a homotopy $H(t;x)$. We assume throughout this section that the function $H(t;x)$ extends from a function on the domain $[0,1]_t \times \C_x^n$ to a function on $\C_t \times \C_x^n$ so that $x(t)$ extends to a map $x(t)\colon U \to (H(t;x))^{-1}(\bzero)$ where $U$ are the regular values of $\C_t$.
 
There exists $\epsilon >0$ such that $0\in \C_t$ is the only branch point of the homotopy in the disc $\Delta\subset \C_t$ of radius $\epsilon$ centered at $0$ and the map $x(t)$ has a Puiseux expansion $$x(t)=\left(f_1\left(t^{\frac{1}{r}}\right),\ldots,f_n\left(t^{\frac{1}{r}}\right)\right),$$ for some $\mydefMATH{r} \in \mathbb{N}$ and complex analytic functions $f_1,\ldots,f_n$ on the disc $\mydefMATH{D}=\epsilon^{1/r}\Delta$. The number $r$ is called the \mydef{winding number} of $x(t)$.
Figure \ref{fig:cauchy} displays the graph of some $x(t)\colon \Delta\to \C_x$, with winding number $r=2$, projected onto the product of $\Delta\subset \C_t$ and the real axis of $\C_x$. 
\begin{figure}[!htpb]
\includegraphics[scale=0.35]{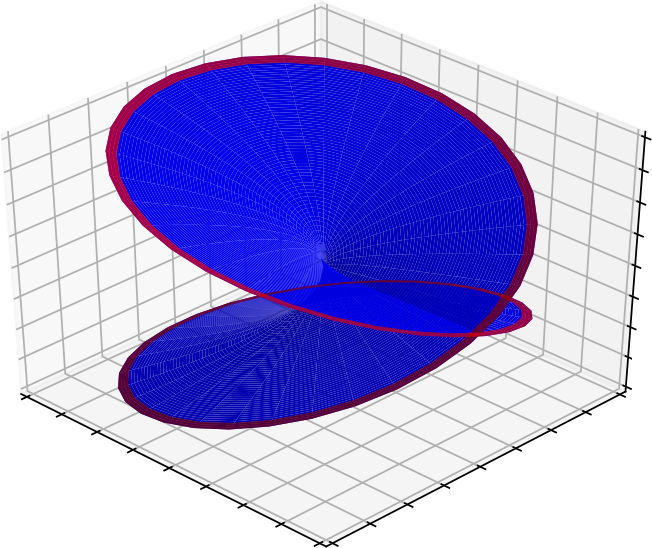}
\caption{A depiction of the local behavior of a path $x(t)$ of a homotopy near a branch point with winding number $2$.}
\label{fig:cauchy}
\end{figure}
Let  $\mydefMATH{\theta}\colon D \to \Delta$ be the map $\theta(z)=z^r$. Composing gives $\mydefMATH{f(z)}=(f_1(z),\ldots,f_n(z))=x(\theta(z))$ which is holomorphic on $D$ and has the property that $f(0)=x(0)$. 
\begin{lemma}
Suppose that $g$ is a holomorphic function on a closed disc $D \subset \C_z$ centered at the origin. Then 
 $$g(0) = \frac{1}{2\pi i} \int_{\partial D} \frac{g(z)}{z} dz.$$
\end{lemma} 

\boxit{
\begin{algorithm}[Cauchy Endgame]
\label{alg:cauchyendgame}
\nothing \\
\myline
{\bf Input:} \\
$\bullet$ A path $x(t)$ of a homotopy $H(t;x)$ \\ 
$\bullet$ An approximation of $x(\epsilon)$ such that $0 \in \Delta \subset \C_t$ is the only branch point in the disc $\Delta$ centered at $0$ with radius $\epsilon$\\ \myline
{\bf Output:}\\
$\bullet$ A numerical approximation of $x(0)$ \\
$\bullet$ The winding number of $x(t)$\\
\myline
{\bf Steps:}
\begin{itemize}[nosep]
\item[1] Use Algorithm \ref{alg:pathtrackingregular} to track the point $x(\epsilon)$ around a parametrization of the boundary of $\Delta$ by $t(s)=\epsilon e^{\sqrt{-1}s}$ to produce the points $(t(s);x(t(s)))$ (and store them) until on the $r$-th loop, $(\epsilon,x(t(0)))=(x(u(2\pi)),t_\epsilon)$
\item[2] Approximate $x' \approx x(0)$ via the path integral in the Cauchy integral formula using the stored values in step $(1)$
\item[3] \return $(x',r)$
\end{itemize}
\end{algorithm}
}

\begin{figure}[!htpb]
\includegraphics[scale=0.5]{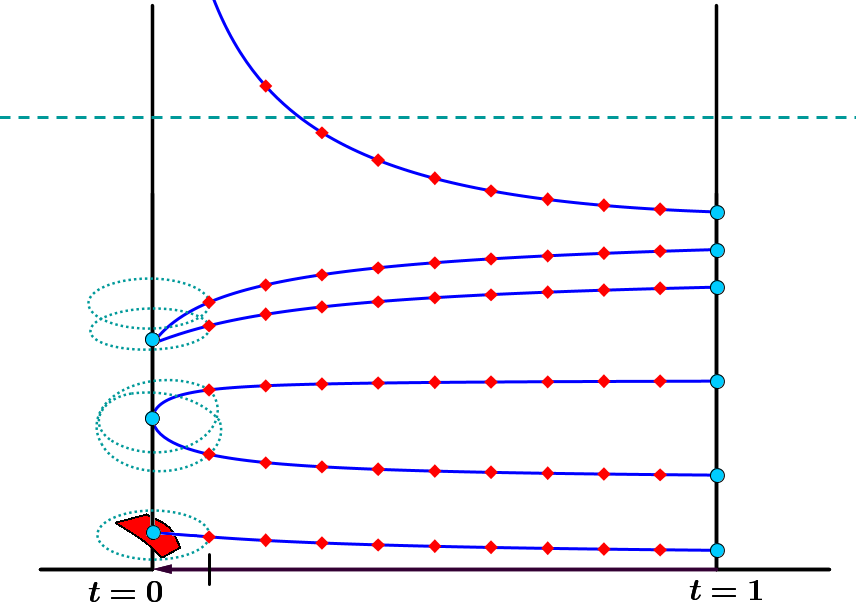}
\caption{\label{fig:endgamesfull} A visual depiction of the full path tracking algorithm. }
\end{figure}

Equipped with the Cauchy endgame, we may now state the full path tracking algorithm. 

\boxit{
\begin{algorithm}[Path tracking]
\label{alg:pathtracking}
\nothing \\
\myline
{\bf Input:} \\
$\bullet$ A homotopy $H(t;x)$\\
$\bullet$ Approximate start solutions $S_1$ to $H(1;x)$\\
$\bullet$ A tolerance $N\gg 0$ for determining divergence \\
$\bullet$ An endgame tolerance $\epsilon>0$\\
 \myline
{\bf Output:}\\
$\bullet$ Approximate target solutions $S_0$\\
\myline
{\bf Steps:}
\begin{itemize}[nosep]
\item[1] \myfor $s \in S_1$ \mydo
\begin{itemize}[nosep]
\item[1.0] Let $x_s(t)$ be the path of $H(t;x)$ with $x_s(1)=s$
\item[1.1] Compute $x_s(\epsilon)$ using Algorithm \ref{alg:pathtrackingregular}
\item[1.2] \myif $|x_s(\epsilon)|\leq N$ and there are no signs of ill-conditioning of the path $x_s(t)$, then continue tracking to $t=0$
\item[1.3] \myif $|x_s(\epsilon)|>N$ then set $x_s(0):=\infty$
\item[1.4] \myelse Use Algorithm \ref{alg:cauchyendgame} to compute $x_s(0)$
\end{itemize}
\item[2] \return $S_0:=\{x_s(0)\}_{s \in S}$
\end{itemize}
\end{algorithm}
}

We illustrate this algorithm in Figure \ref{fig:endgamesfull}.

\section{Homotopy continuation methods}
Given a zero-dimensional polynomial system $F \in \C[x]$, the process of \mydef{homotopy continuation} finds the isolated solutions of $F=0$ by the following model.
\begin{enumerate}
\item If $F$ is overdetermined (more equations than variables), construct a square polynomial system $\hat F$ so that $\V(F) \subset \V(\hat F)$.
\item Find a branched cover $\pi\colon H(t;x) \to \C_s^m$ so that the fiber $\pi^{-1}(y_0)$ is $\V(\hat F)$.
\item Compute the $\deg(\pi)$ solutions in the fiber $\pi^{-1}(y_1)$ for some $y_1 \in U$.
\item Construct a homotopy $H_\tau(t;x)$ where $\tau$ is a path connecting $\tau(1)=y_1$ to $\tau(0)=y_0$.
\item Apply a path tracking algorithm to compute the target solutions $\V(H(0;x))=\V(\hat F) = \pi^{-1}(y_0)$ from the start solutions $\V(H(1;x))=\pi^{-1}(y_1)$.
\item Determine which points of $\V(\hat F)$ are isolated points of $\V(F)$.
\end{enumerate}
Step $(1)$ is done by \mydef{squaring-up} the system $F$. If $F=(f_1,\ldots,f_k) \subset \C[x]$, then for a generic matrix $A\in \C^{k \times n}$, the system $$\hat F = \left\{\sum_{i=1}^k a_{i,j}f_i\right\}_{i=j}^n$$ is a square polynomial system such that the isolated points of $\V(F)$ are isolated points of $\V(\hat F)$. Step $(6)$ is usually performed heuristically by checking if $F(s)\approx 0$ at each isolated point $s \in \V(\hat F)$. If $F(s) \approx 0$ up to some numerical tolerance, then $s$ is deemed to be an isolated solution of $F=0$.  Recently, methods have been developed for certifying solutions of overdetermined systems \cite{AHScert, DHScert}.
In our following discussions we will assume that the polynomial systems involved are already square.

The most general homotopy method is that of a parameter homotopy \cite{LSY89,MS89}. Common special cases of parameter homotopies include the B\'ezout homotopy \cite{Garcia79}, the polyhedral homotopy \cite{HuberSturmfels,V99}, and the witness homotopy. We explain these in the following sections. For reference, we include the ingredients of steps $(1)$ and $(2)$ of each homotopy in Table \ref{table:homotopies} at the end of Section \ref{subsub:witnesscover}.

\subsection{Parameter homotopies}
\label{subsection:parameterhomotopies}
Let $H(s;x) \in \C[s][x]$ be a square parametrized polynomial system. A \mydef{parameter homotopy} is any homotopy coming from the restriction of a branched cover 
$$\pi\colon \V(H(s;x)) \to \C_s^m$$
to a path $\tau\colon [0,1] \to \C_s^m$ such that $\tau(0,1]$ is contained in the regular values $U$ of $\pi$ so that 
$$H_\tau(t;x)\colon [0,1]_t \times C_x^n \to \C^n$$ is a homotopy. In other words, every homotopy is a parameter homotopy.

The \mydef{parameter homotopy method} constructs a fiber $\pi^{-1}(s^*)$ in an \emph{ad hoc} fashion. This theoretically can always be done via the B\'ezout homotopy method, explained in the next section, but often a more immediate or efficient construction is apparent. In either case, it is standard practice to move from $\pi^{-1}(s^*)$ to a fiber $\pi^{-1}(s_1)$ over a general $s_1 \in U$ via Algorithm \ref{alg:pathtrackingregular}. Once $\pi^{-1}(s_1)$ has been computed for a general $s_1 \in \C_s^m$,  one may quickly solve for a fiber $\pi^{-1}(s_0)$ by taking $\tau$ to be a general path connecting $s_1$ to $s_0$ and applying Algorithm \ref{alg:pathtracking} to the homotopy $H_\tau(t;x)$. We note that when $s_0 \in U$, by definition, $H_\tau(t;x)$ is a regular homotopy.

\subsection{The B\'ezout homotopy}
\label{subsection:bezouthomotopy} 
The B\'ezout homotopy method solves a zero-dimensional polynomial system $\V(f_1,\ldots,f_n)$ where $f_i \in \C[x]$ has degree $d_i$. In the language of sparse polynomial systems, this method solves for any fiber of the branched cover
$$\pi_{\Delta_\bullet}\colon X_{\Delta_\bullet} \to \C^{\Delta_\bullet}$$ where $\mydefMATH{\Delta_\bullet}=(d_1\Delta_n,\ldots,d_n\Delta_n)$. By B\'ezout's theorem, this branched cover has degree $\mydefMATH{d}=\prod_{i=1}^n d_i$. The fiber over $G=\{x_i^{d_i}-1\}_{i=1}^n$ consists exactly of the points $(x_1,\ldots,x_n)$ where $x_i$ is any of the $d_i$-th roots of unity. Consequently, $|\pi_{{\Delta_\bullet}}^{-1}(G)|=d$ and so $G$ is a regular value of $\pi_{{\Delta_\bullet}}$.

Given a polynomial system $F \in \C^{\Delta_\bullet}$, if the path $\tau_\gamma\colon [0,1] \to \C^{\Delta_\bullet}$ is given by $\tau(t)=\gamma_0(1-t)F+\gamma_1tG$ for some random $\gamma_0,\gamma_1 \in \C$, then by the $\gamma$-trick, the map $H_{\tau_\gamma}(t;x)\colon [0,1]_t \times \C_x^n \to \C^n$ is a homotopy. Since the parameters of $\pi$ are linear, this homotopy is
$$H_{\tau_\gamma}(t;x) = \gamma_0(1-t)F(x) + \gamma_1tG(x).$$ If $F=0$ has $d$ solutions, then $F$ is a regular value, making $H_\tau(t;x)$ a regular homotopy.

\boxit{
\begin{algorithm}[B\'ezout homotopy]
\label{alg:bezouthomotopy}
\nothing \\
\myline
{\bf Input:} \\
$\bullet$ A square polynomial system $F=(f_1,\ldots,f_n) \subset \C[x]$ \\
 \myline
{\bf Output:}\\
$\bullet$ Approximations of the isolated solutions of $F=0$\\
\myline
{\bf Steps:}
\begin{itemize}[nosep]
\item[0] \set  $d_i=\deg(f_i)$,  $G=\{x_i^{d_i}-1\}_{i=1}^n$,
$S_1=\{a \in \C^n \mid a_i^{d_i}=1\}$, and $\gamma_0,\gamma_1 \in \C$ random complex numbers
\item[1] \return the output of Algorithm \ref{alg:pathtracking} on input homotopy $H(t;x)=\gamma_0(1-t)F+\gamma_1tG$ and start solutions $S_1$
\end{itemize}
\end{algorithm}
}

It is best to perform path tracking between two polynomial systems where at least one of them is general.
In practice, both the start system $G$ of the B\'ezout homotopy, and the target system $F$ could have special structure.  For this reason it is common to apply Algorithm \ref{alg:bezouthomotopy} to solve a random polynomial system $\widehat G \in \C^{\Delta_\bullet}$ and subsequently apply a straight-line homotopy from $\widehat G$ to $F$. This process comprises the \mydef{B\'ezout homotopy method}. 

\boxit{
\begin{algorithm}[B\'ezout homotopy method]
\label{alg:bezouthomotopymethod}
\nothing \\
\myline
{\bf Input:} \\
$\bullet$ A square polynomial system $F=(f_1,\ldots,f_n) \subset \C[x]$ \\
 \myline
{\bf Output:}\\
$\bullet$ Approximations of the isolated solutions to $F=0$\\
\myline
{\bf Steps:}
\begin{itemize}[nosep]
\item[0] \set  $d_i=\deg(f_i)$,  $\widehat{G} \in \C^{\Delta_\bullet}$ random, $\gamma_0,\gamma_1 \in \C$ random, and $H(t;x)=\gamma_0(1-t)F+\gamma_1t\widehat{G}$
\item[1] \set $\widehat S_1$ to be the output of Algorithm \ref{alg:bezouthomotopy} applied to $\widehat G$
\item[2] \return the output of Algorithm \ref{alg:pathtracking} on input homotopy $H(t;x)$ and start solutions $\widehat S_1$
\end{itemize}
\end{algorithm}
}

\subsection{The polyhedral homotopy}
\label{subsection:polyhedralhomotopy}

Generalizing the B\'ezout homotopy, the polyhedral homotopy understands a zero-dimensional polynomial system $F=\{f_1,\ldots,f_n\}$ as a member of the family $\C^\Adot$ of sparse polynomial systems supported on $\Adot=\{\mathcal A_1,\cdots, \mathcal A_n\}$ where $\supp(f_i) =\mathcal A_i$. The relevant branched cover in this scenario is $\pi_\Adot\colon X_{\Adot} \to \C^\Adot$. Unlike more basic homotopy methods, a start system is not immediately available, but must be constructed. Much of the notation in the subsequent discussion comes from Section~\ref{subsection:mixedvolume}.

Suppose $F \in \C^\Adot$ is general and let $\ell_\bullet = (\ell_1,\ldots,\ell_n)$ be a set of lifting functions $\ell_i\colon  \mathcal A_i \to \Z_{\geq 0}$ such that the induced subdivision  $S^{\ell_\bullet}$ (Definition \ref{def:inducedsubdivision}) is a fine mixed subdivision of $\Adot$. Define 
$$\mydefMATH{f_{i,\ell_i}(t;x)} = \sum_{\alpha \in \mathcal A_i} c_{i,\alpha} x^\alpha t^{\ell_i(\alpha)},$$
so that $\New(f_{i,\ell_i}) = \conv_{\ell_i}(\mathcal A_i)$ and similarly define the homotopy $$\mydefMATH{F_{\ell_\bullet}(t;x)} = \{f_{i,\ell_i}(t;x)\}_{i=1}^n,$$
coming from a path in the branched cover $\pi_{\Adot}\colon X_{\Adot} \to \C^{\Adot}$ discussed in Section~\ref{subsection:sparsepolynomialsystems}.
When $t=1$, we have $F_{\ell_\bullet} = F$ and for a general value of $t$, this is a zero-dimensional polynomial system with support $\Adot$ and so $\pi\colon \V(F_{\ell_\bullet}) \to \C_t$ is a branched cover with $\MV(\Adot)$ branches. As $t \to 0$, there are often many solutions of $F_{\ell_\bullet}(t;x)=0$ which diverge, although some may not. We understand these paths $\{x_i(t)\colon  \C_t \to \C^n\}_{i=1}^{\MV(\Adot)}$ near $t=0$ by analyzing their Puiseux expansions and changing coordinates accordingly. We explain the process below.

The branches of $F_{\ell_\bullet}(t;x)$ are functions $x=x(t)$ admitting a Puiseux expansion $$\mydefMATH{x(t)} = (z_{1}t^{\nu_1},\ldots,z_{n}t^{\nu_n}) + \text{ terms with higher powers of }t,$$ 
for some $z=(z_1,\ldots,z_n) \in \C^n$ and $\nu=(\nu_1,\ldots,\nu_n) \in \mathbb{Q}^n$. Taking the composition $F_{\ell_\bullet}(t;x(t))$ yields 
$$\{f_{i,\ell_i}(x(t))\}_{i=1}^n = \left\{\sum_{\alpha \in \mathcal A_i}c_{\alpha}z^\alpha t^{\langle \nu, \alpha \rangle + \ell_i(\alpha)}+\text{ terms with higher powers of }t\right\}_{i=1}^n.$$
The solutions of the above system approach those of 
$$\mydefMATH{F^\nu(t;z)}=\left\{\sum_{\alpha \in \mathcal A_i} c_{i,\alpha} z^\alpha t^{\langle \nu, \alpha \rangle + \ell_i(\alpha)}\right\}_{i=1}^n,$$ as $t \to 0$.  Let $\mydefMATH{\omega}=(-\nu,-1)$ and observe that the terms of 
$$\sum_{\alpha \in \mathcal A_i} c_{i,\alpha} z^\alpha t^{\langle \nu, \alpha \rangle + \ell_i(\alpha)}=\sum_{\alpha \in \mathcal A_i} c_{i,\alpha} z^\alpha t^{\langle -\omega, \Gamma_i(\alpha) \rangle}$$
with lowest power of $t$ are those $\alpha$ such that the inner product $\langle \omega, \Gamma_i(\alpha)  \rangle$ is maximized. Equivalently, these are the vectors $\alpha$ such that $\Gamma_i(\alpha)  \in (\Gamma_i(\mathcal A_i) )_\omega$. Dividing by the lowest power of $t$ occurring in each polynomial of $F^\nu(t;z)$ and evaluating at $t=0$ yields the polynomial system $\mydefMATH{G^\nu}=0$ consisting of polynomials 
$$\mydefMATH{f_i^\nu} = \sum_{\Gamma_i(\alpha)  \in (\Gamma_i(\mathcal A_i) )_\omega}  c_{i,\alpha} z^\alpha.$$
The solutions of $G^\nu$ in $(\C^\times)^n$ are the same as those of $F^\nu(0;z)$ in $(\C^\times)^n$.
Moreover, since the face of a Minkowski sum is a Minkowski sum of faces, we have that $$\conv\left(\sum_{i=1}^n (\Gamma_i(\mathcal A_i))_\omega \right)=(\conv_{\ell_\bullet}(\Adot))_\omega$$ is a face of $\conv_{\ell_\bullet}(\Adot)$. Let $\mydefMATH{\mathcal C_i(\nu)}=\supp(f_i^\nu)$ and $\mydefMATH{\mathcal C_\bullet(\nu)} = (\mathcal C_1(\nu),\ldots,\mathcal C_n(\nu))$.

\begin{lemma}
The system $G^\nu$ has a solution in the algebraic torus if and only if $\mathcal C_\bullet$ is a fine mixed cell of the fine mixed subdivision $S^{\ell_\bullet}$.
\end{lemma}
\begin{proof} 
Suppose $G^\nu$ has a solution in $(\C^\times)^n$.
Since $G^\nu$ is a general sparse polynomial system, the Bernstein-Kushnirenko Theorem asserts that the number of solutions in the algebraic torus is the mixed volume of the supports of the $f_i^\nu$. By Lemma \ref{lem:whenmixedvolumeiszero}, the polytopes $\{\conv(\Gamma_i(\mathcal A_i))_\omega\}_{i=1}^n$ form an essential set and so the dimension of $(\conv_{\ell_\bullet}(\Adot))_\omega$ is $n$. Thus, it is a facet in the lower hull of $\conv_{\ell_\bullet}(\Adot)$ and we conclude that $\mathcal C_\bullet(\nu)$ is a cell of $S^{\ell_\bullet}$.

If for any $i \in [n]$ the polynomial $f_i^\nu$ is a monomial, then $G^\nu$ has no solutions in the algebraic torus. Thus, each $\conv(\mathcal C_i(\nu))$ has dimension at least $1$ and so $\conv(\mathcal C_\bullet(\nu))$ is mixed. Since $S^{\ell_\bullet}$ is a fine mixed subdivision, $\mathcal C_\bullet(\nu)$ is fine mixed.
\end{proof}
If $\nu$ exposes a fine mixed cell, then each $f_i^\nu$ is a binomial and the binomial system $G^\nu$ may be solved using the Smith normal form (see Section~\ref{subsubsection:smithnormalform}) to produce all $\mydefMATH{d_\nu}=\vol(\mathcal C_\bullet(\nu))$ solutions $\mydefMATH{X^{\nu}}$ to ${G}^\nu$. These $d_\nu$ solutions correspond to limits of paths of $F_{\ell_\bullet}(t;x)$ and may be tracked from $t=0$ to $t=1$ by first predicting their values at some $\epsilon >0$ in the $z$-coordinates, applying a corrector method, and changing coordinates back to $F_{\ell_\bullet}(t;x)$ to complete the path tracking from $t=\epsilon$ to $t=1$. This gives $d_\nu$ points $\mydefMATH{X_\nu}$ of $\V(F) \cap (\C^\times)^n$. By Lemma \ref{lem:mixedvolume5}, this comprises all of the $\MV(\Adot)$ branches of this homotopy.

We illustrate the polyhedral homotopy with an example.

\begin{example}
Let $\Adot$ be as in Example \ref{ex:subdivision} and take 
\begin{align*}
f_1 &= 3+4x-2y+xy \\
f_2 &= 6-2xy^2+x^2y
\end{align*}
and let 
\begin{align*}
\ell_1(0,0)&=2,\quad \ell_1(0,1)=\ell_1(1,0)=\ell_1(1,1) = 3,\\
\ell_2(0,0)&=\ell_2(1,2)=\ell_2(2,1) = 1.
\end{align*}
so that $$F(t;x)=\{3t^2+4xt^3-2yt^3+xyt^6,6t-2xy^2t+x^2yt\}.$$
 The three directions $\omega_i=(-\nu_i,-1)$ exposing facets in the lower hull of $\conv_{\ell_\bullet}(\Adot)$ which correspond to fine mixed cells are
$$\omega_1=(2,2,-1),\quad \omega_2=(-2,1,-1),\quad \omega_3=(1,-2,-1),$$
and so $\nu_1=(-2,-2),\nu_2=(2,-1),\nu_3=(-1,2)$. 

\begin{figure}[!htpb]
\includegraphics[scale=0.4]{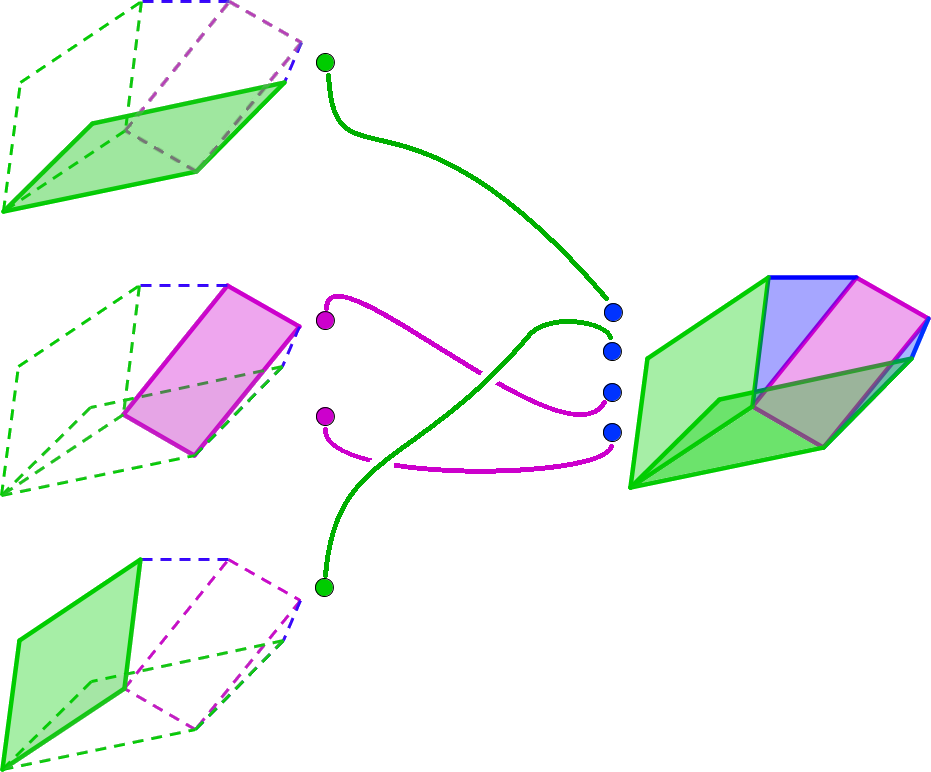}
\caption{A cartoon describing the Polyhedral homotopy.}
\label{fig:polyhedralcartoon}
\end{figure}

Construct
$F^{\nu_1}(t;x(t))=\{3t^2+4z_{1}t-2z_{2}t+z_{1}z_{2}t^2,6t-2z_{1}z_{2}^2t^{-5}+z_{1}^2z_{2}t^{-5}\}$ and divide out by the lowest powers of $t$ to produce
$G^{\nu_1}(z)=\{3+z_{1}z_{2},-2z_{1}z_{2}^2+z_{1}^2z_{2}\}$
which has $2=\vol(\mathcal C_\bullet(\nu_1))$ solutions: $X^{\nu_1}=\left\{(\sqrt{-3/2},\sqrt{-6}),(-\sqrt{-3/2},-\sqrt{-6})\right\}.$

Similarly, for $\nu_2$ construct
$F(t;x(t))^{\nu_2}=\{3t^2+4z_{1}t^5-2z_{2}t^2+z_{1}z_{2}t^7,6t-2z_{1}z_{2}^2t+z_{1}^2z_{2}t^4\}$
and
$G^{\nu_2}(z)=\{3-2z_{2},6-2z_{1}z_{2}^2\}.$
The system $G^{\nu_2}(z)$ has $1=\vol(\mathcal C_\bullet(\nu_2))$ solution, namely $X^{\nu_2}=\left\{\left(4/3,3/2\right)\right\}$.

Finally, for $\nu_3$ we have
$F^{\nu_3}(t;x(t))=\{3t^2+4z_{1}t^2-2z_{2}t^5+z_{1}z_{2}t^7,6t-2z_{1}z_{2}^2t^4+z_{1}^2z_{2}t\}$
and
$G^{\nu_3}(z)=\{3+4z_{1},6+z_{1}^2z_{2}\}$
which has $1=\vol(\mathcal C_\bullet(\nu_3))$ solution: $X^{\nu_3} = \left\{(-3/4,-32/3)\right\}$.

Each solution set $X^{\nu_i}$ may be used to approximate $d_{\nu_i}$ solutions of $F_{\ell_\bullet}(t;x)$ at $t=\epsilon >0$ via a predictor and corrector step followed by a coordinate change. Subsequently, we may track these solutions of $X_{\nu_i} \subset \V(F) \cap (\C^\times)^n$ via the homotopy $F_{\ell_\bullet}(t;x)$ as $t$ goes from $\epsilon$ to $1$.
 \hfill $\diamond$
\end{example}

\boxit{
\begin{algorithm}[Polyhedral homotopy]
\label{alg:polyhedralhomotopy}
\nothing \\
\myline
{\bf Input:} A general sparse polynomial system $F \in \C^{\Adot}$\\ \myline
{\bf Output:} All solutions of $F=0$ in the algebraic torus $(\C^\times)^n$\\
\myline
{\bf Steps:}
\begin{itemize}[nosep]
\item[0] \set $\texttt{solutions}=\emptyset$
\item[1] Choose lifting functions $\ell_\bullet$ such that $S^{\ell_\bullet}$ is a fine mixed subdivision of $\Adot$
\item[2] Compute the mixed cells $\mathcal C_\bullet^{(1)},\ldots,\mathcal C_\bullet^{(m)}$
\item[3] \myfor each mixed cell $\mathcal C_\bullet$ \mydo 
\begin{itemize}
\item [3.1] Compute the vector $(-\nu,-1)$ exposing $\conv_{\ell_\bullet}(\mathcal C_\bullet)$
\item [3.2] Compute $X^\nu =\V(G^\nu)\cap (\C^\times)^n$ using Smith normal form
\item [3.3] Move the solutions $X^\nu$ from $t=0$ to $t=\epsilon >0$ via a prediction and correction step and change coordinates to $x$
\item[3.4] Track the solutions of $F_{\ell_\bullet}(\epsilon;x)$ in step $(3.3)$ from $t=\epsilon$ to $t=1$ (backwards) under the homotopy $F_{\ell_\bullet}(t;x)$ and append the resulting solutions $X_\nu$ to the list $\texttt{solutions}$
\end{itemize}
\item[4] \return $\texttt{solutions}$
\end{itemize}
\end{algorithm}}

Since we usually do not \emph{a priori} know whether or not a sparse polynomial is general in the sense of Proposition \ref{prop:BKK}, given $F \in \C^\Adot$, we solve for the isolated solutions of $F$ in the algebraic torus via the following method.

\boxit{
\begin{algorithm}[Polyhedral homotopy method]
\label{alg:polyhedralhomotopymethod}
\nothing \\
\myline
{\bf Input:} A sparse polynomial system $F \in \C^{\Adot}$\\ \myline
{\bf Output:} All isolated solutions of $F=0$ in the algebraic torus $(\C^\times)^n$\\
\myline
{\bf Steps:}
\begin{itemize}[nosep]
\item[0] Pick a general sparse polynomial system $G \in \C^\Adot$
\item[1] Apply Algorithm \ref{alg:polyhedralhomotopy} to $G$ to produce all isolated points $\V(G) \cap (\C^\times)^n$
\item[2] Track the points $\V(G) \cap (\C^\times)^n$  to the points $\V(F) \cap (\C^\times)^n$ via the straight-line homotopy $H(t;x) = \gamma_0(1-t)F+\gamma_1tG$
\item[3] \return $\V(F) \cap (\C^\times)^n$
\end{itemize}
\end{algorithm}}

\section{Witness sets}
\label{subsection:witnesssets}

Positive-dimensional varieties are represented in numerical algebraic geometry by slicing them with sufficiently many general hyperplanes which cut out degree-many points. Numerical approximations of these points are computed  using homotopy methods and stored in the fundamental data structure of numerical algebraic geometry, a witness set.

\begin{definition}
Let $X$ be an irreducible variety. A \mydef{witness set} for $X$ is a triple $(F,L,S)$ where
\begin{itemize}
\item $F$: a finite set of polynomials such that $X$ is an irreducible component of $\V(F)$.
\item $L$: a general affine linear space of complementary dimension to $X$.
\item $S$: a set containing approximations of each of the points in $X \cap L$. 
\end{itemize}
If $X$ is reducible with top-dimensional components $X_1,\ldots,X_k$ then a witness set for $X$ is $(F,L,S)$ where $S=S_1 \cup \cdots \cup S_k$ such that $(F,L,S_i)$ is a witness set for $X_i$.
\end{definition}
We refer to $L$ as a \mydef{witness slice} and $S$ as \mydef{witness points}. One immediate way to compute a witness set is by using Algorithm \ref{alg:bezouthomotopy}.  

\boxit{
\begin{algorithm}[Constructing a witness set]
\label{alg:witness}
\nothing \\
\myline
{\bf Input:} \\ $\bullet$  A polynomial system $F=(f_1,\ldots,f_{n-m}) \subset \C[x]$ such that $X$ is the union of irreducible components of $\V(F)$ of dimension $m=\dim(\V(F))$\\ \myline
{\bf Output:} \\ $\bullet$ A witness set $(F,L,S)$ for $X$\\
\myline
{\bf Steps:}
\begin{itemize}[nosep]
\item[1] Choose $m$ random linear polynomials $L = \{\ell_1,\ldots,\ell_m\} \subset \C[x]$
\item[2] Apply Algorithm \ref{alg:bezouthomotopy} to $F \cup \ell$ to produce $S$
\item[3] \return $(F,L,S)$
\end{itemize}
\end{algorithm}
}

If $X=X_1 \cup \cdots \cup X_k \subset \C^n$ is the irreducible decomposition of a variety whose components are not all of the same dimension, let $\mydefMATH{\text{Dim}(X)}$ denote the set $\{\dim(X_i)\}_{i=1}^k$. A \mydef{witness superset} of $X$ is a collection $\{(F,L_i,S_i)\}_{i \in \text{Dim}(X)}$ where $(F,L_i,S_i)$ is a witness set for the union of all irreducible components of $X$ of dimension $i$. Methods for computing witness supersets include ``working dimension by dimension'' and the ``cascade algorithm''. These are discussed in \cite[Ch. 9.3-9.4]{BertiniBook}.

\subsection{The witness cover}
\label{subsub:witnesscover}
Given an irreducible variety $X \subset \C^n$ of dimension $m$, let 
$$\mydefMATH{W(X)}=\left\{(x,L) \mymid x \in X \cap \V(L),\text{ and } L \in (\C^{\Delta_n})^m\}\subset X \times (\C^{\Delta_n})^m\right\}$$ be the incidence variety of points on $X$ with linear varieties cut out by $m$ linear polynomials. Then $W(X)$ is irreducible of dimension $m(n+1)$ and the map,
$$\mydefMATH{\pi_{W(X)}}\colon W(X) \to (\C^{\Delta_n})^m,$$
is a degree $\deg(X)$ branched cover (by Lemma \ref{lem:whenhyperplanesintersect}) called the \mydef{witness cover} of $X$. 
With this language, it is straightforward to describe how to ``move'' witness sets.

\boxit{
\begin{algorithm}[Regular Witness homotopy]
\label{alg:regularwitnesshomotopy}
\nothing \\
\myline
{\bf Input:} \\ $\bullet$ A witness set $(F,L,S)$ for $X$\\$\bullet$ A general linear space $L'$ of dimension $n-m$\\
\myline
{\bf Output:} \\ $\bullet$ A witness set $(F,L',S')$ for $X$\\
\myline
{\bf Steps:}
\begin{itemize}[nosep]
\item[1] \set $H(t;x) = \gamma_0(1-t)[F|L]+\gamma_1t[F|L']$
\item[2] Track the witness points $S$ via $H(t;x)$ to the solutions $S'$ of $F=L'=0$
\item[3] \return $(F,L',S')$
\end{itemize}
\end{algorithm}
}

We may deform witness sets of a variety to special linear intersections via a similar algorithm.

\boxit{
\begin{algorithm}[Witness homotopy]
\label{alg:witnesshomotopy}
\nothing \\
\myline
{\bf Input:} \\ $\bullet$  A polynomial system $F=(f_1,\ldots,f_{n-m}) \subset \C[x]$ such that $X$ is an irreducible component of $\V(F)$ of dimension $m=\dim(\V(F))$\\ $\bullet$ A witness set $(F,L,S)$ for $X$\\$\bullet$ A linear space $L'$ of dimension $n-m$\\
\myline
{\bf Output:} \\ $\bullet$ The points $S'=X \cap L'$\\
\myline
{\bf Steps:}
\begin{itemize}[nosep]
\item[1] \set $H(t;x) = \gamma_0(1-t)[F|L]+\gamma_1t[F|L']$
\item[2] Track the witness points $S$ via $H(t;x)$ to the solutions $S'$ of $F=L'=0$
\item[3] \return $S'$
\end{itemize}
\end{algorithm}
}

Since $L$ and $L'$ are regular values of the branched cover $\pi_{W(X)}$, Lemma \ref{lem:gammatrick} guarantees that Algorithm \ref{alg:witnesshomotopy} computes a witness set $(F,L',S')$. We remark that Algorithm \ref{alg:witnesshomotopy} functions just as well for reducible varieties $X=\bigcup_{i=1}^k X_i$ where the map $\pi_{W(X)}$ becomes a branched cover which is not irreducible.

Now that we have explained each of the four homotopy methods mentioned in the introduction, we provide a reference table (Table \ref{table:homotopies}) for their ingredients.
\begin{table}[!htpb]
\begin{tabular}{|l|l|l|l|}
\hline
\textbf{Homotopy} & \textbf{Relevant systems} & \textbf{Branched Cover} & \textbf{Start system} \\
\hline

\begin{minipage}{0.15\textwidth} \vspace{0.1 in} Parameter \\ Section~\ref{subsection:parameterhomotopies}\vspace{0.1 in} \end{minipage} & $F_s \in \C[s][x]$ & $\V(F_s) \xrightarrow{\pi} \C_s^m$ & $\V(F_{s_1})$ \\  \hline 
\begin{minipage}{0.15\textwidth} \vspace{0.1 in} B\'ezout \\ Section~\ref{subsection:bezouthomotopy} \vspace{0.1 in}\end{minipage} & $F \in \C^{\Delta_\bullet}=\C^{d_1\Delta_n,\ldots,d_n\Delta_n}$ & $X_{\Delta_\bullet} \xrightarrow{\pi_{\Delta_\bullet}} \C^{\Delta_\bullet}$ & $\{x_i^{d_i}-1\}_{i=1}^n$ \\
\hline 

\begin{minipage}{0.15\textwidth} \vspace{0.1 in} Polyhedral \\ Section~\ref{subsection:polyhedralhomotopy}\vspace{0.1 in} \end{minipage} & $F \in \C^{\Adot}$ & $X_{\calA_\bullet} \xrightarrow{\pi_{\calA_\bullet}} \C^{\calA_\bullet}$ & \begin{minipage}{0.23\textwidth} Constructed from a \\fine mixed subdivision \end{minipage}\\  \hline

\begin{minipage}{0.15\textwidth} \vspace{0.1 in} Witness \\ Section~\ref{subsection:witnesssets}\vspace{0.1 in} \end{minipage} & \begin{minipage}{0.28\textwidth}  $X \cap L$ where \\ $\dim(X)=m=\codim(L)$ \end{minipage} & \small{$W(X) \xrightarrow{\pi_{W(X)}} (\C^{\Delta_n})^{m}$} & Any witness set for $X$ \\   \hline 
\end{tabular}
\caption{Ingredients for homotopy methods.}
\label{table:homotopies}
\end{table}
Parameter homotopies and the B\'ezout homotopy are implemented in most numerical algebraic geometry software including \textbf{Bertini}, \textbf{PHCPack}, \textbf{HOM4PS}, \textbf{homotopycontinuation.jl} and \textbf{NAG4M2} \cite{V99,Bertini,HCjl, HOM4PS, NAG4M2}.
The polyhedral homotopy is implemented in \textbf{PHCPack} and \textbf{HOM4PS} \cite{V99,HOM4PS}.

\subsection{Witness sets for images of maps}

Much of the strength of numerical algebraic geometry stems from the fact that witness sets (the fundamental data structure in numerical algebraic geometry) can often be computed even when their symbolic analogs, Gr\"obner bases, cannot.  In some sense, this is because Gr\"obner bases transparently express so much information about a variety  while witness sets do not. Rather, witness sets offer users the option to discover information as-needed, similar to the oracles from Section~\ref{section:polytopes}. 

One particular instance where witness sets can be easily computed is when the variety of interest is a projection. Because witness sets are geometric in nature, they behave well with respect to projections. 
\begin{definition}
A \mydef{pseudo-witness set} for an affine variety $Z$ is a quadruple $(F,\varphi,\varphi^{-1}(L_1),S)$ where 
\begin{itemize}
\item $F$: a finite set of polynomials such that $X$ is the union of top dimensional components of $\V(F)\subset \C^N$.
\item $\varphi$: a coordinate projection $\varphi\colon \C^N \to \C^n$ such that $Z=\overline{\varphi(X)}$ and $\dim(Z)=\dim(X)$. 
\item $L_1$: a general affine linear space in $\C^n$ of complementary dimension to $Z$.
\item $S$: a set containing approximations of each of the points in $X \cap \varphi^{-1}(L_1)$. 
\end{itemize}
\end{definition}
Often, one desires a pseudo-witness set for the image $Z=\overline{\varphi(X)}$ of a map $X \xrightarrow{\varphi} \C^n$ where the dimension of $X$ is larger than its image. When this is the case, one may take $\dim(X)-\dim(Z)$ generic linear equations $L_2 \subset \C[x_1,\ldots,x_N]$ so that $X \cap L_2$ so that that the image of $X\cap L_2 \xrightarrow{\varphi} \C^n$ is $Z$ and $\dim(X \cap L_2)=\dim(Z)$. By factoring $\varphi$ through its graph, it is enough to be able to compute witness sets for projections. We do this in the following way.

\boxit{
\begin{algorithm}[Constructing a pseudo-witness set]
\label{alg:pseudowitness}
\nothing \\
\myline
{\bf Input:}  \\ $\bullet$ A witness set $(F,L,S)$ for a variety $X$ of dimension $m$\\
$\bullet$ A coordinate projection $\varphi\colon \C^N \to \C^n$ such that $Z=\overline{\varphi(X)}$ and $\dim(Z)=m$\\ \myline
{\bf Output:}\\ $\bullet$ A pseudo-witness set $(F,\varphi,\varphi^{-1}(L^*),S)$ for $Z$\\
\myline
{\bf Steps:}
\begin{itemize}[nosep]
\item[0] Assume that $\varphi(x_1,\ldots,x_N) = (x_1,\ldots,x_n)$
\item[1] Fix $\dim(Z)$ random linear forms $ L^* \subset \C[x_1,\ldots,x_n]$
\item[2] Use Algorithm \ref{alg:witnesshomotopy} to compute a witness set $(F,L^*,S^*)$ for $X$  by moving $(F,L,S) \to (F,L^*, S^*)$
\item[3] \return $(F,\varphi,\varphi^{-1}(L^*),S)$
\end{itemize}
\end{algorithm}}

\begin{example}
\label{ex:pseudowitness}
We give two examples which exhibit subtleties in pseudo-witness sets. The first is the twisted cubic $C \subset \C^3$ with the projection $\varphi\colon \C^3 \to \C^2$ such that the image is a parabola. During the homotopy which constructs a pseudo-witness, one of the three points of intersection with the twisted cubic diverges towards infinity. 
\begin{figure}
\includegraphics[scale=0.4]{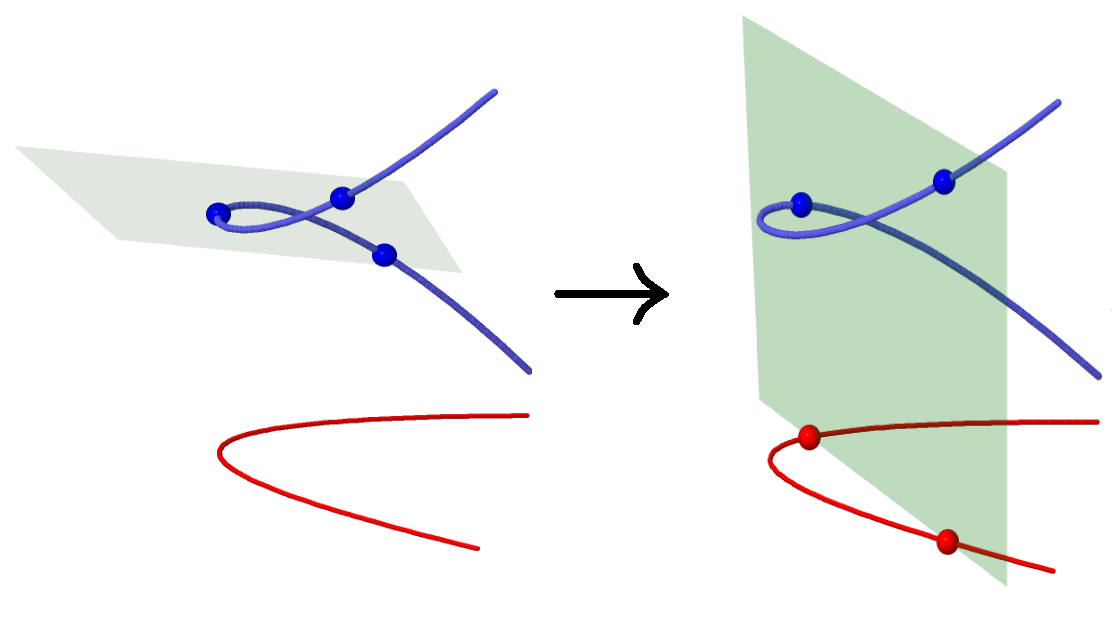}
\caption{
\label{fig:twistedpseudo}(Reprinted from \cite{Bry:NPtrop}) Constructing a pseudo-witness set for a projection of a twisted cubic}
\end{figure}

The second example involves the necessity of a dimension reduction. The variety $X$ in this case is the cylinder $\V(x^2+y^2-1) \subset \C_{x,y,z}^3$ along with the projection $\pi\colon \C_{x,y,z}^3 \to \C_{x,y}$ whose image is the circle defined by the same equation in $\C[x,y]$. In this case, to construct a pseudo-witness set for the circle, we must first slice the cylinder by a hyperplane to produce the red curve $C$ in Figure \ref{fig:cylinderpseudo}. The dimension of $C$ is the same as the dimension of its image and so one may simply deform a witness set for $C$ to be vertical with respect to $\pi$ to produce a pseudo-witness set for the circle.
\begin{figure}
\includegraphics[scale=0.55]{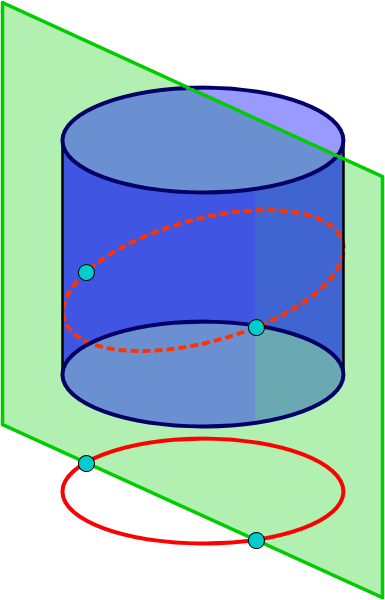}
\caption{
\label{fig:cylinderpseudo}Constructing a pseudo-witness set for a projection of a cylinder via slicing.}
\end{figure}
 \hfill $\diamond$
\end{example}

Any algorithm that may be performed using solely the witness cover of a variety may also be performed using pseudo-witness sets since these may be moved just as easily with respect to the a pseudo-witness cover. Suppose $\varphi\colon X \to Z$ is a projection. To construct a pseudo-witness cover for $Z$, replace $X$ and $Z$ with affine open sets, intersect $X$ with a linear space $L_2$ of codimension $\dim(X)-\dim(Z)$, and relabel variables so that $\varphi(x_1,\ldots,x_N)=(x_1,\ldots,x_n)$ is a degree $d$ branched cover of affine varieties $X \subset \C^N, Z\subset \C^n$ of dimension $m$. Take
$$\mydefMATH{PW(Z,\varphi)} = \left\{(x,L_1) \mymid x \in X \cap L_1, \text{ and } L_1 \in (\C^{\Delta_n})^m\right\}$$
to be the incidence variety of points on $X$ with linear varieties cut out by $m$ polynomials in the first $n$ coordinates. Define the \mydef{pseudo-witness cover} of $Z$ with respect to $\varphi$ (and $L_2$) to be the map
$$\mydefMATH{\pi_{PW(Z,\varphi)}}\colon PW(X) \to (\C^{\Delta_n})^m.$$  
Any pseudo-witness set of the form $(F,\varphi,\varphi^{-1}(L_1),S)$ is a fiber $S=\pi_{PW(X)}^{-1}(L_1)$ of $\pi_{PW(X)}$ by construction. In particular, we see that $Z$ has degree $|S|/d$ witnessed by the $|S|/d$ points $\varphi(S)\subset Z \cap L_1$. Note, that $|S|$ is not necessarily equal to the degree of $X$, as shown in the first part of Example \ref{ex:pseudowitness}. For more information about pseudo-witness sets, see \cite{WSP}.

\section{Monodromy}
Recall the background on monodromy groups of branched covers in Section~\ref{section:branchedcoversandgroups}. 

Let $\mydefMATH{F(s;x)} \subset \C[s][x]$ be a parametrized polynomial system in $m$ parameters $s$ and $n$ variables $x$ so that 
$$\mydefMATH{\pi}\colon \V(F(s;x)) \to \C_s^m$$
is a degree $\mydefMATH{d}$ branched cover with regular values $U \subset \C_s^m$. We do not assume $\pi$ is an irreducible branched cover. For $s_1,s_2 \in U$ and $c_1\in \C$ let  
\begin{align*}
\mydefMATH{\tau_{s_1,s_2,c_1}}\colon [0,1]_t &\to \C_s^m \\
t &\mapsto  (1-t)s_1+c_1 t s_2
\end{align*}
be a path in $\C_s^m$. For this section, we will assume that $c_1$ is in the Euclidean dense subset of $\C$ which satisfies $\tau_{s_1,s_2,c_1}([0,1])\subset U$. Applying the path tracking algorithm for regular homotopies (Algorithm \ref{alg:pathtrackingregular}) to the homotopy $\mydefMATH{H_{s_1,s_2,c_1}(t;x)}$ produced by $\pi$ and $\tau_{s_1,s_2,c_1}$ gives the bijection $m_{\tau_{s_1,s_2,c_1}}$ discussed in Section~\ref{subsection:galoisandmonodromy}. Picking another generic complex number $c_2$, the composition
$$m_{\tau_{s_2,s_1,c_2}} \circ m_{\tau_{s_1,s_2,c_1}}\colon  \pi^{-1}(s_1) \to \pi^{-1}(s_1)$$
 is a monodromy permutation of the $d$ points in the fiber over $s_1$. This permutation is the monodromy element $m_\gamma$ where $\gamma$ is the loop in $\C_s^m$ formed by following the concatenation of the paths $\tau_{s_1,s_2,c_1}$ and $\tau_{s_2,s_1,c_2}$.

This leads immediately to a heuristic algorithm for computing elements of the monodromy group of a branched cover.

\boxit{
\begin{algorithm}[Extract monodromy group element]
\label{alg:extractmonodromyelements}
\nothing\\
\myline
{\bf Input:} \\
$\bullet$ A parametrized polynomial system $F(s;x) \subset \C[s][x]$ such that $\pi\colon \V(F(s;x)) \to \C_s^m$ is an irreducible branched cover of degree $d$ \\
$\bullet$ A fiber $S_1=\pi^{-1}(s_1)$ \\ \myline
{\bf Output:}\\
$\bullet$ An element $g$ of the monodromy group $\mathcal M_\pi$ \\
\myline
{\bf Steps:}
\begin{itemize}[nosep]
\item[1] Label $S_1=\{p_1,\ldots,p_d\}$ so that $p_i$ is identified with $i \in [d]$
\item[2] Pick $s_2, \in U$ and generic $c_1,c_2\in \C$
\item[3] Track all points in $S_1$ along $H_{s_1,s_2,c_1}$ to produce $S_2$
\item[4] Track all points in $S_2$ along $H_{s_2,s_1,c_2}$ to produce $(m_{\tau_{s_2,s_1,c_2}} \circ m_{\tau_{s_1,s_2,c_1}})(S)$
\item[5] Determine $g=m_{\tau_{s_2,s_1,c_2}} \circ m_{\tau_{s_1,s_2,c_1}}(p_i)$ for each $i \in [n]$ to determine $g$
\item[6] \return $g$
\end{itemize}
\end{algorithm}}

\begin{example}
\label{ex:monodromyloopexample}
Rather than producing loops with only two parameter values in our examples, we use three parameters, $s_1, s_2,$ and $s_3$, to clarify the ideas and images. 
Figure \ref{fig:singlemonodromyloop} shows a schematic of the computation of a monodromy group element using numerical algebraic geometry. 
\begin{figure}[!htpb]
 \includegraphics[scale=0.5]{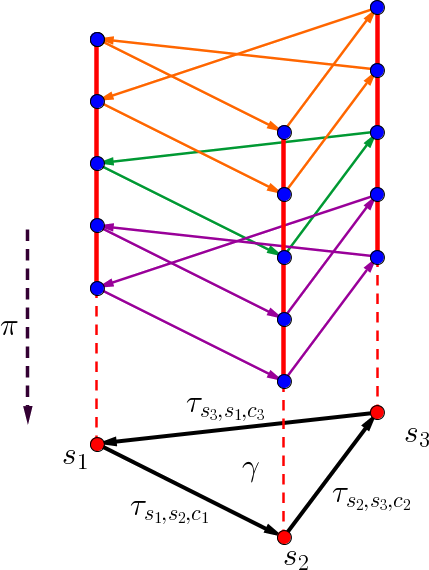}
\caption{\label{fig:singlemonodromyloop} A schematic of a single monodromy loop tracked numerically.}
\end{figure}
Labeling the points of $\pi^{-1}(s_1)$ from bottom to top as $1,2,\ldots,5$, the element $m_\gamma \in M_\pi$ is $m_\gamma=(1,2)(4,5)$, written in cycle notation. Its cycles are depicted in distinct colors in Figure \ref{fig:singlemonodromyloop}.  \hfill $\diamond$
\end{example}

To determine the monodromy group $\mathcal M_\pi$ of a branched cover, one may repeatedly extract group elements  using Algorithm \ref{alg:extractmonodromyelements} until the group generated by these elements fails to grow after many runs of the algorithm. This is, of course, heuristic. A more rigorous way to compute the monodromy group is to restrict the parameter space $\C_s$ to a generic line $\C_t \subset \C_s$. The branch locus of $\pi$ restricted to $\C_t$ consists of finitely many points $b_1,\ldots,b_k$. A theorem of Zariski \cite{Zariski} implies that the monodromy group $\mathcal M_\pi$ is generated by loops around each $b_i$. For more information about computing monodromy groups of branched covers using numerical algebraic geometry, we refer the reader to \cite{NumGal}.

\subsection{Solving via monodromy}
\label{subsubsection:solvingviamonodromy} 
Recall that the monodromy group of an irreducible branched cover is transitive (Lemma \ref{lem:monodromyofirreducibleistransitive}). Thus, the observation that monodromy permutations can be explicitly computed using numerical algebraic geometry suggests one way to compute $\pi^{-1}(s_1)$ given some point $q \in \pi^{-1}(s_1)$: pick a random monodromy loop $\gamma$, use a homotopy to track $q$ via a lift of $\gamma$ thus computing $m_\gamma(q)$, and repeat. 
This is the na\"ive version of the \mydef{monodromy solve} algorithm.

\boxit{
\begin{algorithm}[Na\"ive monodromy solver]
\label{alg:monodromysolve}
\nothing\\
\myline
{\bf Input:} \\
$\bullet$ A parametrized polynomial system $F(s;x) \subset \C[s][x]$ such that $\pi\colon \V(F(s;x)) \to \C_s^m$ is an irreducible branched cover \\
$\bullet$ A single point $q$ in some fiber $\pi^{-1}(s_1)$ \\ \myline
{\bf Output:}\\
$\bullet$ All points in $\pi^{-1}(s_1)$ \\
\myline
{\bf Steps:}
\begin{itemize}[nosep]
\item[1] \set $S_1=\{q\}$
\item[2] \while $\pi^{-1}(s_1)$ has not been fully computed \mydo
\begin{itemize}[nosep]
\item[2.1] Pick $s_2, \in U$ and generic $c_1,c_2\in \C$
\item[2.2] Track all points in $S_1$ along $H_{s_1,s_2,c_1}$ to produce $S_2$
\item[2.3] Track all points in $S_2$ along $H_{s_2,s_1,c_2}$ to produce  $$S_1'=(m_{\tau_{s_2,s_1,c_2}} \circ m_{\tau_{s_1,s_2,c_1}})(S) \subset \pi^{-1}(s_1)$$
\item[2.4] \set $S_1 = S_1 \cup S_1'$
\end{itemize}
\item[3] \return $S_1$
\end{itemize}
\end{algorithm}}

\begin{remark}
Conditions for determining when the fiber has been ``fully computed'' in step $(2)$ are called \mydef{stopping criteria} and are not obvious. When the degree of the cover is known, then a stopping criterion for Algorithm \ref{alg:monodromysolve} is ``stop when $\deg(\pi)$ points of $\pi^{-1}(s_1)$ have been computed''. We give an alternative stopping criterion when $\pi$ is the witness cover of a variety.  \hfill $\diamond$
\end{remark}
By Lemma \ref{lem:traceline}, the centroids of witness points on a pencil of witness slices of an irreducible affine variety lie on an affine line. A stronger result is true.
\begin{lemma}
\label{lem:thetracetest}
Let $X$ be an irreducible affine variety and $L_t$ a general pencil of linear spaces of complementary dimension.
Given a subset of witness points $S_0\subset L_0 \cap X$, the centroids of the paths starting at $S_0$ along the homotopy over $L_t$ in the witness cover moves affine linearly if and only if $S_0 = L_0 \cap X$.  
\end{lemma}

Let $X$ be an irreducible variety of dimension $m$.
When performing Algorithm \ref{alg:monodromysolve} on the witness cover 
$$\pi_{W(X)}\colon W(X) \to (\C^{\Delta_n})^m$$ a stopping criterion is that the condition in Lemma \ref{lem:thetracetest} holds. This may be checked during the monodromy algorithms by moving a witness set to three slices in a pencil and numerically taking the centroids of the witness points.
As witness points are only numerical approximations, this test relies on determining whether the midpoint $m_{pq}$ of two centroids $p,q \in \C^n$ satisfies $\frac{p+q}{2}-m_{pq}=0$: it requires assessing whether or not a numerical value is zero. Although extremely reliable in practice, this means the trace test does not certify the computation of all witness points. Developing an algorithm to certify the trace task is an important open task in numerical algebraic geometry.

\subsection{Monodromy solving for real branched covers}

Algorithm \ref{alg:monodromysolve} is not optimal. For example, suppose the first two loops of  Algorithm \ref{alg:monodromysolve}  are $\gamma$ and $\gamma'$ depicted in Figure \ref{fig:monodromytwoloops}. These loops induce a transitive subgroup of $S_5$, but Algorithm \ref{alg:monodromysolve} would only use each once, discovering a total of three points after the second loop.

\begin{figure}[!htpb]
\includegraphics[scale=0.5]{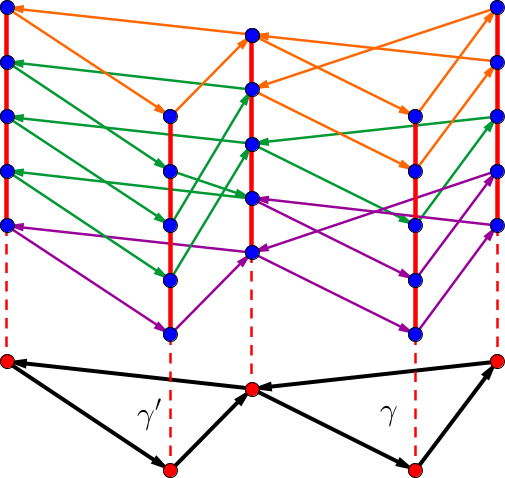}
\caption{\label{fig:monodromytwoloops}An example of two monodromy loops generating a transitive subgroup of a monodromy group.}
\end{figure}

A model for monodromy algorithms as well as a strategy-analysis for choosing monodromy loops are given in \cite{Duff}. We propose improving Algorithm \ref{alg:monodromysolve} by taking advantage of automorphisms of fibers guaranteed by the structure of $\pi$. For example, many polynomial systems of interest are defined over the real numbers, whose solutions come in complex conjugate pairs (see Section~\ref{subsection:realbranchedcovers}) and so computing a nonreal solution $x \in \pi^{-1}(s_1)$ immediately computes its conjugate $\overline x \in \pi^{-1}(s_1)$. This occurs, in particular, whenever $\pi$ is a real branched cover and $s_1$ is real. Thus, we propose an additional step to Algorithm \ref{alg:monodromysolve}.

\boxit{
\begin{algorithm}[Monodromy solver for real branched covers]
\label{alg:monodromysolveoverR}
\nothing \\
\myline
{\bf Input:} \\
$\bullet$ A parametrized polynomial system $F(s;x) \subset \C[s][x]$ such that $\pi\colon \V(F(s;x)) \to \C_s^m$ is an irreducible real branched cover \\
$\bullet$ A single point $q$ in some fiber $\pi^{-1}(s_1)$ with $s_1$ real \\ \myline
{\bf Output:}\\
$\bullet$ All points in $\pi^{-1}(s_1)$ \\
\myline
{\bf Steps:}
\begin{itemize}[nosep]
\item[1] \set $S_1=\{q\}$
\item[2] \while $\pi^{-1}(s_1)$ has not been fully computed \mydo
\begin{itemize}[nosep]
\item[2.1] Pick $s_2, \in U$ and generic $c_1,c_2\in \C$
\item[2.2] Track all points in $S_1$ along $H_{s_1,s_2,c_1}$ to produce $S_2$
\item[2.3] Track all points in $S_2$ along $H_{s_2,s_1,c_2}$ to produce  $$S_1'=(m_{\tau_{s_2,s_1,c_2}} \circ m_{\tau_{s_1,s_2,c_1}})(S) \subset \pi^{-1}(s_1)$$
\item[2.4] \set $S_1 = S_1 \cup S_1'\cup \overline{S_1'}$
\end{itemize}
\item[3] \return $S_1$
\end{itemize}
\end{algorithm}
}

\begin{example}
\begin{figure}
\includegraphics[scale=0.5]{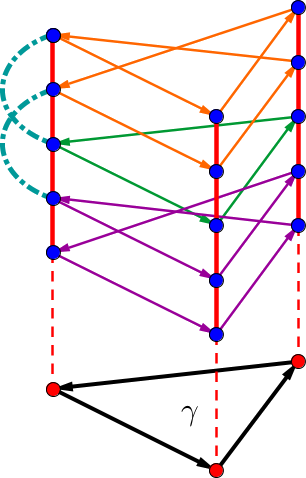}
\caption{\label{fig:monodromyreal}A schematic describing step $(2)$ of Algorithm \ref{alg:monodromysolveoverR}.}
\end{figure}
Figure \ref{fig:monodromyreal} depicts a schematic for Algorithm \ref{alg:monodromysolveoverR} showing a monodromy loop along with complex conjugation generating a transitive subgroup of $S_5$. Algorithm \ref{alg:monodromysolveoverR} computes all solutions in three steps using only one monodromy loop along with complex conjugation.
 \hfill $\diamond$
\end{example}
We remark that step $(2.4)$ in the above algorithm may be replaced by any operation $g$ preserving the fiber $\pi^{-1}(s_1)$. In particular, if $g$ is a deck transformation of the cover $\pi$, then one may append the orbit $gS_1$ to $S_1$ in step $(2.4)$ at a nominal cost.

\subsection{Expected success of monodromy solving}
The authors of \cite{Duff} address the question of how many monodromy loops are necessary to induce a transitive subgroup of $\mathcal M_\pi$ under the assumptions 
\begin{enumerate}
\item $\mathcal M_\pi$ is the full symmetric group.
\item Random choices of $s_2,c_1,$ and $c_2$  samples elements of $\mathcal M_\pi$ uniformly at random. 
\end{enumerate}
For branched covers $\pi$ of degree $d$, they prove a generalization of Dixon's theorem \cite{Dixon}, a result which implies that the probability of two random elements of $S_d$ generating a transitive subgroup of $S_d$ approaches $1$ as $d \to \infty$.

One may hope for an analogous result with respect to Algorithm \ref{alg:monodromysolveoverR}. That is, given a branched cover $\pi$ of degree $d$ and a regular value $s_1$ whose fiber is known to be fixed under the action $\iota$ of complex conjugation, what is the probability that a random monodromy element along with complex conjugation generate a transitive subgroup of $\pi^{-1}(s_1)$?

As in \cite{Duff}, we must make decide how to model the action of complex conjugation on a fiber. A fiber of $\pi$ whose points are fixed under complex conjugation may consist entirely of real points (in which case Algorithm \ref{alg:monodromysolveoverR} is no different than Algorithm \ref{alg:monodromysolve}) or entirely of nonreal points, or some number in between. Thus, we analyze the case when conjugation on a fiber is modeled by random involutions in $S_d$ and the case that it is modeled by a fixed-point free involution. The latter case is relevant in applications as there are many instances  where we can guarantee that a fiber contains no real points (such as computing witness sets for varieties which are compact over $\R$).

We fix some notation. Let $R_i\subset S_i$ be a subset of the symmetric group with the property that whenever the subgroup generated by $\sigma \in S_i$ and $\tau \in R_i$ has $k_i$ orbits of size $i$, then $\tau \in (R_1)^{k_1} \times (R_2)^{k_2} \times \cdots \times (R_d)^{k_d}$ where each factor of $R_i$ acts on an orbit of size $i$. Note that not every sequence of subsets of $S_i$ has this property. For example, any sequence of subsets starting as
$$R_1=\{(1)\},\quad R_2=\{(1)\}, \quad  R_3=\{(1,2)\},\ldots$$ does not have this property since the permutation $(1,2)$ has orbits of sizes $(k_1,k_2)=(1,1)$, but $(1,2) \not\in R_1 \times R_2$.

\begin{proposition}[]
\label{prop:probabilities}
Let $R_i$ be a sequence of subsets of $S_i$ such that whenever $(\sigma,\tau) \in S_d \times R_d$ has $k_i$ orbits of size $i$, then $(\sigma,\tau) \in (S_1 \times R_1)^{k_1}\times \cdots \times (S_d \times R_d)^{k_d}$. Let $t_d$ be the probability that $(\sigma,\tau) \in S_d \times R_d$ generates a transitive subgroup of $S_d$. Then the $t_i$ satisfy the recursion
$$d|R_d| = \sum_{i=1}^d it_i|R_i|\cdot |R_{d-i}|.$$
\end{proposition}
\begin{proof}
We use the same strategy as \cite{Duff, Dixon} to determine a recursion for the probabilities $t_d$. 

Let $K_d = \{ \overline{k} \in \mathbb{N}^d \mid \sum i k_i = d\}$ be the set of number partitions of $d$. The number of set partitions which have parts corresponding to some $\overline k$ is $\left(\frac{d!}{\prod_{i=1}^d (i!)^{k_i} k_i!}\right)$. This is since there are $d!$ ways to place the numbers $\{1,\ldots,n\}$ into a sequence of cycles of sizes $k_1,\ldots,k_d$, but we have over counted since each cycle can be permuted $i!$ ways, and cycles of the same size may be permuted as well. Let $(\sigma, \tau) \in S_d \times R_d$, then if  $\langle \sigma, \tau \rangle$ is a subgroup whose orbits have sizes $\overline k$, then $\sigma, \tau $ must respect these partitions so $(\sigma, \tau) \in (S_1 \times R_1)^{k_1}\times \cdots \times (S_d \times R_d)^{k_d}$. So we may assume $\sigma$ and $\tau$ have been uniformly chosen from $S_1^{k_1}\times \cdots \times S_d^{k_d}$ and $R_1^{k_1} \times \cdots \times R_d^{k_d}$ respectively. Therefore, using the probabilities $t_i$ we may count the elements in $S_d \times R_d$ via
\begin{align*}
|S_d \times R_d| &= \sum_{\overline k \in K_d } \left(\frac{d!}{\prod_{i=1}^d (i!)^{k_i} k_i!}\right) \prod_{i=1}^d \left[t_i(i! \cdot |R_i|) \right]^{k_i} \\
&=d! \sum_{\overline k \in K_d} \prod_{i=1}^d \left(\frac{t_i i! |R_i|}{i!} \right)^{k_i} \frac{1}{k_i!}\\
&=d! \sum_{\overline k \in K_d} \prod_{i=1}^d \frac{(t_i\cdot |R_i|)^{k_i}}{k_i!}
\end{align*}
and since $|S_d \times R_d| = d!\cdot |R_d|$ we have
\begin{equation}
\label{Eq:Rec1}
|R_d|=\sum_{\overline{k} \in K_d} \prod_{i=1}^d\frac{(t_i\cdot |R_i|)^{k_i}}{k_i!}.
\end{equation}
Using the theory of generating functions, we extract a recursion on the numbers $t_i$. Let $\hat F(x)$ be the generating function of $F(d)=|R_d|$ and recall the formal identity 
$$\text{exp}\left(\sum_{i=1}^\infty y_ix^i \right) = \sum_{d=0}^\infty x^d \left(\sum_{\overline k \in K_d} \prod_{i=1}^d \frac{y_i^{k_i}}{k_i!} \right).$$ Applying this formula to $\hat F(x)$ using Equation \eqref{Eq:Rec1} gives us  
$$ \text{exp}\left(\sum_{i=1}^\infty  t_i|R_i| x^i \right)=\hat F(x).$$
Now consider $\hat F'(x)$:
\begin{align*}
\sum_{d=1}^\infty d |R_d| x^{d-1} &= \frac{\partial}{\partial x} \hat F(x) \\
&=  \frac{\partial}{\partial x} \text{exp}\left(\sum_{i=1}^\infty t_i |R_i| x^i \right).
\end{align*}
When we apply the chain rule to differentiate this expression we get
\begin{align*}
&= \hat F(x) \cdot  \left(\sum_{i=1}^\infty it_i|R_i|x^{i-1} \right)\\
&=\left [ \sum_{d=1}^\infty |R_d| x^d \right] \cdot  \left[ \sum_{i=1}^\infty i t_i |R_i| x^{i-1} \right]\\
&=\sum_{d=1}^\infty \left(\sum_{i=1}^\infty i t_i |R_i| \cdot |R_{d}| x^{d+i-1} \right)
\end{align*}
making the substitution $d'=d+i$ yields the equality
$$\sum_{d=1}^\infty d |R_d| x^{d-1} = \sum_{d'=1}^\infty x^{d'-1} \left(\sum_{i=1}^{d'} i t_i |R_i|\cdot |R_{d'-i}|\right).$$
Equating the coefficients of $x^d$ gives a recursion
\begin{equation}
\label{Eq:Rec2}
d|R_d| = \sum_{i=1}^d i t_i |R_i|\cdot |R_{d-i}|
\end{equation}
\begin{equation}
\label{Eq:Rec3}
t_d=1-\sum_{i=1}^{d-1} \frac{i}{d} t_i \frac{|R_i|\cdot |R_{d-i}|}{|R_d|},
\end{equation}
completing the proof.
\end{proof}
Even though we followed exactly the same argument used in the results of \cite{Duff, Dixon} in our proof of Proposition \ref{prop:probabilities}, we are not aware of this elementary result in the literature. 
We remark that when $R_i$ is a subgroup of $S_i$, the analysis of the probabilities that random elements of $R_i \times S_i$ generating either $A_i$ or $S_i$ has been done \cite{Babai}.

As mentioned, there are three natural choices of $R_i$ to analyze: 
\begin{enumerate}
\item$R_i=S_i$
\item $R_i=T_i$
\item $R_{2i}=\{$all fixed point free involutions$\}=:\mathbb{T}_{2i}$.
\end{enumerate} The first case is the subject of Dixon's theorem. The second case corresponds to choosing a random involution to model complex conjugation on a fiber of the monodromy algorithm. The third case corresponds to modeling complex conjugation on a fiber where every solution is nonreal (in particular $d$ is even). 
We list the first few terms of the probabilities $t_i$ in each case in Table \ref{Table:ThreeCases}.

\begin{table}[htpb]
\begin{center}
\begin{tabular}{c|c|c|c|c|c|c|c|c}
$d$ & $1$ & $2$ & $3$ & $4$ & $5$ & $10$ & $20$ & $30$ \\
\hline
$S_d$ & $1$ & $0.75$ & $0.722$ & $0.739$ & $0.768$ & $0.881$ & $0.946$ & $ 0.965$\\ 
\hline
$T_d$ & $1$ & $0.75$ & $0.583$ & $0.575$ & $0.546$ & $0.607$ & $0.731$  & $0.792$\\
\hline
$\mathbb{T}_d$ & -&$1$ & - & $0.833$ & -  & $0.863$ & $0.937$ & $0.962$
\end{tabular}
\caption{\label{Table:ThreeCases} Some probabilities of generating a transitive action by uniformly choosing from $S_d \times R_d$}
\end{center}
\end{table}
We remark that the only property of a sequence $\{R_i\}_{i \in \mathbb{N}}$ determining the probabilities $t_i$ are the cardinalities $|R_i|$.

\begin{corollary}
For $d=2n$, the probability that a fixed point free involution and a random element of $S_d$ generate a transitive subgroup of $S_d$ approaches $1$ as $d \to \infty$.
\end{corollary}
\begin{proof}
Let $p_d=1-t_d$ be the probability that a random element of $S_d$ and a fixed point free involution do \emph{not} generate a transitive subgroup of $S_d$. Note that if $d$ is odd, then there are no fixed point free involutions. We let $a_j$ be the number of fixed point free involutions on a set of cardinality $2j$ so that 
$$a_j = (2j-1)!! = \frac{(2j)!}{j!2^j}.$$
These may be recursively defined by $a_j=(2j-1)a_{j-1}$ and $a_1=1$. 

Set $d=2n$ so that \eqref{Eq:Rec2} becomes
$$2n\cdot a_n = \sum_{j=1}^{n}(2j)\cdot t_{2j}\cdot a_j\cdot a_{n-j}$$
which we may rearrange so that 
$$t_{2n} = 1-\sum_{j=1}^{n-1} \frac{2j}{2n}t_{2j}\frac{a_j\cdot a_{n-j}}{a_n}=1-\sum_{j=1}^{n-1} \frac{j}{n}t_{2j}\frac{a_j\cdot a_{n-j}}{a_n}.$$
Consequently, to show that $p_{d} \to 0$ as $d \to \infty$ we show that 
$$\lim_{n \to \infty} p_{2n} =\lim_{n \to \infty}\sum_{j=1}^{n-1} \frac{j}{n}t_{2j}\frac{a_j\cdot a_{n-j}}{a_n}=0.$$
Let $m=\lfloor \frac{n-1}{2} \rfloor$ and observe that since the $t_{2j}$ are probabilities, they are bounded by $1$ so
$$
p_d \leq \sum_{j=1}^{n-1} \frac j n \frac{a_j \cdot a_{n-j}}{a_n}.$$
By symmetry, if $n$ is even, we have
$$p_{2n} \leq \frac 1 2 \frac{a_{\frac{n}{2}}^2}{a_n}+\sum_{j=1}^{m} \frac{a_j \cdot a_{n-j}}{a_n}\leq \left(\frac{1}{2}\right)^m + \sum_{j=1}^m \frac{a_j\cdot a_{n-j}}{a_n}$$
and so $p_{2n}$ will approach $0$ as $n \to \infty$ if and only if 
$$\sum_{j=1}^{m} \frac{a_j\cdot a_{n-j}}{a_n}$$ does. If $n$ is odd, this bound also holds.
Since the $a_j$ satisfy the recursion $a_j=(2j-1)a_{j-1}$, we know that $$a_ja_{n-j} = a_{j+1}a_{n-j-1}\frac{2n-2j}{2j+1} > a_{j+1}a_{n-j-1}$$ for all $1 \leq j \leq m$. Thus,

\begin{align*}
\lim_{n \to \infty} \sum_{j=1}^m \frac{a_j\cdot a_{n-j}}{a_n} &\leq \left(\lim_{n \to \infty} \frac{a_1\cdot a_{n-1}}{a_n} + \sum_{j=2}^m \frac{a_2\cdot a_{n-2}}{a_n} \right)\\&=\lim_{n \to \infty} \left( \frac{1}{2n-1} + \sum_{j=2}^m \frac{3}{(2n-1)(2n-3)}\right) \\ &=\lim_{n \to \infty} \frac{1}{2n-1} + \frac{3(m-1)}{(2n-1)(2n-3)} =0
\end{align*}
showing that $\lim\limits_{n \to \infty} p_{2n} \to 0$ so $\lim\limits_{n \to \infty} t_{2n} = 1$. 
\end{proof}

\chapter{NEWTON POLYTOPES AND TROPICAL MEMBERSHIP VIA NUMERICAL ALGEBRAIC GEOMETRY \label{section:numericalNP}}
A major theme of numerical algebraic geometry is the extraction of information about a variety $X$ using witness sets. For varieties arising as the image of a map, the algebraic information of generators of the ideal $\I(X)$ may not be readily available. Finding these generators is the problem of {implicitization}. While this may be done using symbolic methods involving Gr\"obner bases, this technique is often computationally prohibitive for moderate to large problems. Even when $X$ is a hypersurface, its defining polynomial may be so large that it is not human-readable. Thus, one naturally desires a coarser description of the polynomial, such as its Newton polytope.

In 2012, Hauenstein and Sottile \cite{NP} sketched a numerical algorithm which functions as a vertex oracle for the Newton polytope of a hypersurface, relying only on the computation of witness sets. In Section~\ref{sec:hsalgorithm}, we explain how this algorithm, which we call the HS-algorithm, actually functions as a numerical oracle and is therefore stronger than originally anticipated. 

Following ideas from Hept and Theobald,  we extend the HS-algorithm to an algorithm for computing the tropicalization of an ideal in Section~\ref{sec:tropicalalgorithms}. In Section~\ref{sec:convergencerates} we analyze the convergence of the HS-algorithm. We discuss our implementation in the {\bf Macaulay2} \cite{M2} package {\bf NumericalNP.m2} \cite{taylor} in Section~\ref{sec:numericalnp} and give large examples showcasing our software in Sections~\ref{sec:AlgebraicVision} and \ref{sec:Luroth}. 
\renewcommand*{\thefootnote}{\fnsymbol{footnote}}
Much of this material is contained in the article \cite{Bry:NPtrop} by the author\footnote{Reprinted with permission from T. Brysiewicz, ``Numerical Software to Compute Newton polytopes and Tropical Membership,''  {\it{Mathematics in Computer Science,}} 2020. Copyright 2020 by Springer Nature.}.
\renewcommand*{\thefootnote}{\arabic{footnote}}

\section{The HS-Algorithm}
\label{sec:hsalgorithm}

Let $\mydefMATH{\mathcal H} \subseteq \mathbb{C}^n$ be a degree $\mydefMATH{d}$ hypersurface defined by $$\mydefMATH{f}=\sum_{\alpha \in \mathcal A} c_\alpha x^\alpha  \in \mathbb{C}[x]\hspace{0.5 in} c_\alpha \neq 0, \mathcal A \subseteq \mathbb{N}^n, |\mathcal A|<\infty$$ so that $\supp(f)=\mydefMATH{\mathcal A}$.
Suppose that $\mydefMATH{a},\mydefMATH{b} \in (\C^{\times})^n$ are general so that the line parametrized by $s \mapsto (a_1s-b_1,\ldots,a_ns-b_n)$ intersects $\mathcal H$ at $d$ points, making the map
\begin{align*}
\mydefMATH{\mathcal W}&\mydefMATH{(\mathcal H)} = \{(p_1,\ldots,p_n,s) \mid f(p_1(a_1s-b_n),\ldots,p_n(a_ns-b_n))=0\} \\
\mydefMATH{\pi}\downarrow & \\
\C&^n_{p}
\end{align*}
a degree $d$ branched cover. For any direction $\mydefMATH{\omega} \in \mathbb{R}^n$, the path $t \mapsto (t^{\omega_1},\ldots,t^{\omega_n})$ in $\C^n_p$ corresponds to a family of lines $\mydefMATH{\mathcal L_t}$ parametrized by
\begin{align*}
\mydefMATH{{\bf L}_t}\colon  \C_s &\to \C^n \\
s &\mapsto (t^{\omega_1}(a_1s-b_1),\ldots,t^{\omega_n}(a_ns-b_n)).
\end{align*} This family of lines lifts to $d$ paths $\{\mydefMATH{s_i(t)}\}_{i=1}^d$ in $\mathcal W(\mathcal H)$ via  $\pi$, each corresponding to an intersection point of $\mathcal L_t$ and $\mathcal H$. In other words, these paths comprise the data of witness points and so $\pi$ essentially functions as a witness cover.

Each path $s_i(t)$ may be tracked numerically with respect to the homotopy
\begin{align*}
f(\textbf{L}_t(s))\colon (0,1] \times \C_s &\to \C \\
(t^{-1},s) &\mapsto f(\textbf{L}_t(s))
\end{align*}
in the variables $t^{-1}$ and $s$. We remark that the use of  $t^{-1}$ is a definitional technicality and that we will mostly work in $t$ using the map 
\begin{align}
\label{eq:hshomotopy}
\mydefMATH{\pi_\omega}\colon f(\textbf{L}_t(s))^{-1}(0) &\to (1,\infty)\\
(t^{-1},s) & \mapsto t \nonumber
\end{align} so that the fiber of $\pi_\omega$ over $t$ is identified with $\{s_i(t)\}_{i=1}^d$. The following lemma is the basis of the HS-algorithm.
\begin{lemma}
As $t\to \infty$ the $s$-coordinates of the fibers $\pi_\omega^{-1}(t)$ converge to the solutions of $f_\omega({\bf L}_{1}(s))$.
\end{lemma}
\begin{proof}
Using the notation $\mydefMATH{(as-b)^\alpha}=(a_1s-b_1)^{\alpha_1}\cdots(a_ns-b_n)^{\alpha_n}$, observe that 
\begin{align}
f({\bf L}_t)&=\sum_{\alpha \in \mathcal A} c_\alpha [t^{\omega_1}(a_1s-b_1)]^{\alpha_1} \cdots [t^{\omega_n}(a_ns-b_n)]^{\alpha_n}\nonumber \\ 
&= \sum_{\alpha \in \mathcal A} t^{\langle \omega,\alpha \rangle}c_\alpha(as-b)^{\alpha}\nonumber \\
 \label{eq:composition0}
&= \sum_{\alpha \in \mathcal A_\omega}t^{h_{\mathcal A}(\omega)} c_\alpha(as-b)^\alpha + \sum_{\alpha \in \mathcal A_\omega^c}t^{\langle \omega,\alpha \rangle}c_\alpha(as-b)^{\alpha} 
\end{align}
Since $t$ is not zero, we may scale \eqref{eq:composition0} by $t^{-h_\mathcal A(\omega)}$ without changing its zeros. Thus the solutions of \eqref{eq:composition0} are the same as those of 
\begin{equation}
\label{eq:composition}
\sum_{\alpha \in \mathcal A_\omega} c_\alpha(as-b)^\alpha + \sum_{\alpha \in \mathcal A_\omega^c}t^{\langle \omega,\alpha \rangle-h_{\mathcal A}(\omega)}c_\alpha(as-b)^{\alpha} 
\end{equation}
where $\mydefMATH{\mathcal A_\omega^c}$ is the complement of $\mathcal A_\omega$ in $\mathcal A$. Note that $\pi_\omega^{-1}(t) = \{(t,s_i(t))\}_{i=1}^d$ where $\{s_i(t)\}_{i=1}^d$ are the solutions of \eqref{eq:composition}. Finding the values of each $s_i(t)$ as $t \to \infty$ is the same as finding the values of $s_i(t^{-1})$ as $t \to 0$ and since these paths are continuous, we substitute $t^{-1}$ for $t$ in \eqref{eq:composition} and take the limit as $t \to 0$:
\begin{align}
\sum_{\alpha \in \mathcal A_\omega}& c_\alpha(as-b)^\alpha + \sum_{\alpha \in \mathcal A_\omega^c}(t^{-1})^{\langle \omega,\alpha \rangle-h_{\mathcal A}(\omega)}c_\alpha(as-b)^{\alpha} \nonumber \\
\label{eq:composition2}
\sum_{\alpha \in \mathcal A_\omega}& c_\alpha(as-b)^\alpha + \sum_{\alpha \in \mathcal A_\omega^c}(t)^{-\langle \omega,\alpha \rangle+h_{\mathcal A}(\omega)}c_\alpha(as-b)^{\alpha}.
\end{align}
Note that $\langle -\omega,\alpha \rangle+h_{\mathcal A}(\omega)>0$ for all $\alpha \in \mathcal A_\omega^c$ by definition, and so evaluating \eqref{eq:composition2} at $t=0$ gives $\V(\sum_{\alpha \in \mathcal A_\omega}c_\alpha(as-b)^\alpha)=\V(f_\omega({\bf L}_{1}(s))).$
\end{proof}
Suppose a hypersurface $\mathcal H \subset \C^m$ is the image of a map $\varphi\colon X \to \mathcal H$. Recall that a (pseudo)-witness set for $\mathcal H$ can be computed by computing a witness set for the graph of $\varphi$ and applying Algorithm \ref{alg:pseudowitness}.
Consequently, since fibers of $\pi_\omega$ are essentially witness points, we may compute them without access to the defining equation for $\mathcal H$. 
The HS-algorithm follows from the above observations.

We remind the reader that $\widetilde{f}$ denotes the homogenization of a polynomial $f$ and $\mathcal O_P$ denotes a numerical oracle for a polytope $P$. 

\boxit{
\begin{algorithm}[HS-Algorithm]
\nothing \\
\myline
{\bf Input:} \\
$\bullet$ A witness set, or pseudo-witness set, $W$ for a hypersurface $\mathcal H \subseteq \C^n$\\
$\bullet$ A direction $\omega \in \R^n$ \\
 \myline
{\bf Output:}\\
$\bullet$ $\mathcal O_{\New(\widetilde{f})}(\omega)$ where $\mathcal H = \V(f)$\\
\myline
{\bf Steps:}
\begin{enumerate}[nosep]
\item[1] Pick random $a,b \in (\C^{\times})^n$ and construct $\textbf{L}_t$ 
\item[2] Track the witness points in $W$ to the intersection $\mathcal H \cap \mathcal L_1$
\item[3] Track all points $\{s_i(1)\}_{i=1}^d$ along \eqref{eq:hshomotopy} from $t=1 \to 2$. 
\item[4] If none of the points move, \return \texttt{EEP}
\item[5] Initialize $\beta = (0_1,0_2,\ldots,0_n,0_\infty) \in \mathbb{N}^{n+1}$
\item[6] \myfor $i$ from $1$ to $d$ \mydo
\begin{enumerate}[nosep]
\item[6.1] Track the point $s_i(1)$ along \eqref{eq:hshomotopy} as $t \to \infty$
\item[6.2] If $s_i(t)$ converges or diverges, stop tracking it
\begin{enumerate}
\item[6.2.1] If $s_i(t)$ converged to $\rho_i$, increment $\beta_i$ by one
\item[6.2.2] If $s_i(t)$ diverged, increment $\beta_\infty$ by one
\end{enumerate}
\end{enumerate}
\item[7] \return $\beta$
\end{enumerate}
\label{alg:hsalgorithm}
\end{algorithm}}

\textit{Proof of correctness:}
We claim that Algorithm \ref{alg:hsalgorithm} is a numerical oracle (see Definition~\ref{def:numericaloracle}) for $\New(\widetilde{f})$.
We consider three situations, dependent on $\omega \in \R^n$, which result in different behaviors of the set $\{s_i(t)\}_{i=1}^d$ as $t \to \infty$.
\begin{enumerate}
\item {\bf ($\omega$ exposes a single point):} This means that $f_\omega({\bf L}_1(s))= c_\beta(as-b)^\beta$ is a monomial with exponent $\mydefMATH{\beta}=(\beta_1,\ldots,\beta_n) \in \mathcal A$. This clearly has roots of $\mydefMATH{\rho_i}=b_i/a_i$ appearing with multiplicity $\beta_i$. Note that if $|\beta|<d$ then there are $\mydefMATH{\beta_\infty}=d-|\beta|$ paths which diverge as $t \to \infty$. One way to see this is to observe that $\beta_\infty$ is the exponent of the homogenizing variable in the term $\widetilde{f}_\omega$. 

\item {\bf ($\omega$ exposes $\mathcal A$):}
If $\omega$ exposes the entire polytope defined by $\mathcal A$, then the roots $\{s_i(t)\}_{i=1}^d$ remain constant as $t$ varies since $f({\bf L}_t)=t^{h_{\mathcal A}(\omega)}f({\bf L}_1)$.

\item {\bf ($\omega$ exposes a proper subset of $\mathcal A$ consisting of more than one point):}
If $\omega$ exposes a proper non-singleton subset of $\mathcal A$, then there is more than one term in $f_\omega$. We remark this happens exactly when $\omega \in \trop(\V(f))$. The terms of $f_\omega$ will have a common factor of $\prod_{i=1}^n (a_is-b_i)^{m_i}$ where the vector $\mydefMATH{m}$ is the coordinate-wise minimum of the points in $\mathcal A_\omega$. Therefore,  $m_i$ roots will converge to $\rho_i$ and $\mydefMATH{m_\infty}=\min_{\beta \in \mathcal A_\omega} \left(d-|\beta|\right)$ points will diverge. All other roots will converge somewhere else in $\mathbb{C}$. 
\end{enumerate}
In each case, the output is that of a numerical oracle. \hfill $\square$

\section{Tropical membership}
\label{sec:tropicalalgorithms}
A direct consequence of the HS-algorithm is a tropical membership algorithm for hypersurfaces. Recall that for a polynomial $f \in \C[x]$ of degree $d$, a direction $\omega \in \R^n$ is an element of $\trop(f)$ if and only if  $\omega$ exposes a positive-dimensional face of $\New(\widetilde{f})$. Equivalently, recall that $\omega \in \trop(f)$ if and only if a numerical oracle outputs a vector $v=\mathcal O_{\New(\widetilde{f})}(\omega) \in \Z^{n+1}$ satisfying $|v|<d$. Thus, since the HS-algorithm functions as a numerical oracle, it may be used  as a tropical membership algorithm for hypersurfaces. 

\boxit{\begin{algorithm}[Tropical Membership for Hypersurfaces]
\nothing \\
\myline
{\bf Input:} \\
$\bullet$ A witness set, or pseudo-witness set, $W$ for a hypersurface $\mathcal H \subseteq \C^n$\\
$\bullet$ A direction $\omega \in \R^n$ \\
 \myline
{\bf Output:}\\
$\bullet$ \texttt{true} if $\omega \in \trop(\mathcal H)$ and \texttt{false} otherwise.\\
\myline
{\bf Steps:}
\begin{enumerate}[nosep]
\item[0] \set $d=\deg(\mathcal H)$ and \set $\beta$ to be the output of the HS-algorithm on input $W$ and $\omega$
\item[1] \myif $|\beta|<d$ \then \return \texttt{true},
 \myelse 
 \return \texttt{false}

\end{enumerate}
\label{alg:tropmembershiphypersurfaces}
\end{algorithm}
}

Given an arbitrary variety $\V(I) \subset \C^n$, recall that the tropicalization of $\V(I)$ may be realized as the intersection of preimages of projections of $\trop(\V(I))$ (Lemma \ref{thm:TropicalProjection}). When the coordinate projections are sufficiently generic, Algorithm \ref{alg:tropmembershiphypersurfaces} extends immediately to a tropical membership algorithm for the tropicalization of $\V(I)$. When they are not, this algorithm can only yield false positives. To handle this, we may obtain new projections of $\trop(\V(I))$ by taking the coordinate projections of $\trop(\V(I))$ after a linear change of coordinates on $\trop(\V(I))$. We recall that a linear change of coordinates on $\trop(\V(I))$ amounts to a monomial change of coordinates on $I$ (see Remark \ref{remark:producingprojections}) which will likely increase the degree of $\V(I)$. While this makes the computation of a witness set more difficult, it is often still manageable.

\boxit{
\begin{algorithm}[Tropical Membership]
\nothing \\
\myline
{\bf Input:} \\
$\bullet$ An $m$-dimensional variety $X=\V(I)\subseteq \mathbb{C}^n$\\
$\bullet$ A direction $\omega \in \R^n$ \\
 \myline
{\bf Output:}\\
$\bullet$ \texttt{true} if $\omega \in \trop(\mathcal H)$ and \texttt{false} otherwise\\
\myline
{\bf Steps:}
\begin{enumerate}[nosep]
\item[1] Replace $I$ with its image under a generic monomial map $\Phi$ so that the coordinate projections of $\V(I)$ are generic 
\item[2] Replace $\omega$ with $\varphi^{-1}\omega$ where $\varphi = \Phi^*$
\item[3] Compute a witness set $W$ for $X$.
\item[4] \myfor each coordinate projection $\{\pi_J\}_{J \subseteq [n]}$ with $|J|=n-m-1$ \mydo 
\begin{enumerate}[nosep]
\item[4.1] Compute a pseudo-witness set $W_J$ for $\pi_J(X)$
\item[4.2] \myif Algorithm \ref{alg:tropmembershiphypersurfaces} returns $\texttt{false}$ on input $(W_J,\pi_J(\omega))$, \then \texttt{STOP} and \return \texttt{false}
\end{enumerate}
\item[5] \return \texttt{true}
\end{enumerate}
\label{alg:TropicalMembership}
\end{algorithm}
}

\section{Convergence rates of the HS-Algorithm}
\label{sec:convergencerates}

Theorem 8 of \cite{NP} gives an analysis of the convergence of the HS-algorithm whenever $\omega$ exposes a vertex. We generalize this result to include the case where $\omega \in \trop(\V(f))$. First, we introduce some notation. As before, let 
$$f= \sum_{\alpha \in \mathcal A} c_\alpha x^\alpha$$
be a polynomial with support $\mathcal A$ and let $\omega \in \R^n$.  The polynomial $f_\omega$ may be written as $f_\omega=x^\mydefMATH{m}\cdot \mydefMATH{g(x)}$ for some polynomial $g(x) \in \C[x]$ whose terms have no common monomial factor. After choosing generic points $a,b \in \C^n$, we write $f_\omega({\bf L}_1)$ in its factored form as 
$$f_\omega({\bf L}_1)= (as-b)^m g({\bf L}_1)=(as-b)^m\cdot \mydefMATH{K} \cdot (s-\mydefMATH{\tau})^\mydefMATH{k}.$$
Note that the $\mydefMATH{\tau}=(\tau_1,\ldots,\tau_{n'})$ are the $\mydefMATH{n'}$ complex roots of $g({\bf L}_t)$ and so $k$ is some $n'$-tuple satisfying $|k|<d-|m|$.
In the case that $\New(f)_\omega=\beta$ is a vertex, we have $m=\beta$, $g(x)=1$, and $K=c_\beta$. We define the following constants based on the coefficients $c_\alpha$, the support $\mathcal A$, and the constant $K$.
\begin{center}
$$\mydefMATH{C}={\max}\left\{\frac{|c_\alpha|}{|K|} \mymid \alpha \in \mathcal A\right\}, \qquad \mydefMATH{d_\omega}=h_{\mathcal A}(\omega)-h_{ \mathcal A_\omega^c}(\omega)$$
$$\mydefMATH{a_{\min}}={\min}\{1,|a_i| \mid i=1,\ldots,n\},
 \quad \mydefMATH{a_{\max}}={\max}\{1,|a_i| \mid i=1,\ldots,n\}$$
\end{center}
Finally, based on the positions of the $\rho_i=\frac{b_i}{a_i}$ and the $\tau_j$ appearing as roots of $g({\bf L}_1)$ we define the following constants for any $z \in \{\rho_i\}_{i=1}^n \cup \{\tau_j\}_{j=1}^{n'}$.
$$
\mydefMATH{\gamma_{z}}={\min}\left\{a_{\min},\frac{|z-\hat z|}{2} \mymid \hat z \in \{\rho_i\}_{i=1}^n \cup \{\tau_j\}_{j=1}^{n'} \smallsetminus z \right\}$$
$$
\mydefMATH{\Gamma_{z}}={\max}\left\{\frac{2}{a_{\max}},|z-\rho_i|\mymid i=1,\ldots,n \right\}
$$ 
The constant $d_\omega$ describes how close $\omega$ is to exposing a positive-dimensional face of $\New(f)$. The constant $\gamma_z$ is defined so that any point inside the circle of radius $\gamma_z$ centered at $z$ is closer to $z$ than any other point in $\{\rho_i\}_{i=1}^{n'} \cup \{\tau_j\}_{j=1}^{n'}$. 
We include helpful graphics describing this notation in Figure \ref{fig:numericaldefs} and Figure \ref{fig:dw}.
\begin{figure}[!htpb]
\includegraphics[scale=0.55]{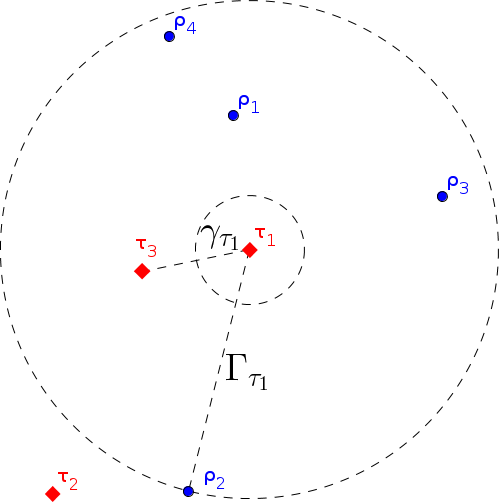}
\caption{\label{fig:numericaldefs}An example of locations of $\{\rho_i\}_{i=1}^4$ and $\{\tau_j\}_{j=1}^3$ in $\C_s$. The smaller circle has radius $\gamma_{\tau_1}$ and the larger circle has radius $\Gamma_{\tau_1}$.}
\end{figure}
\begin{figure}[!htpb]
\includegraphics[scale=0.55]{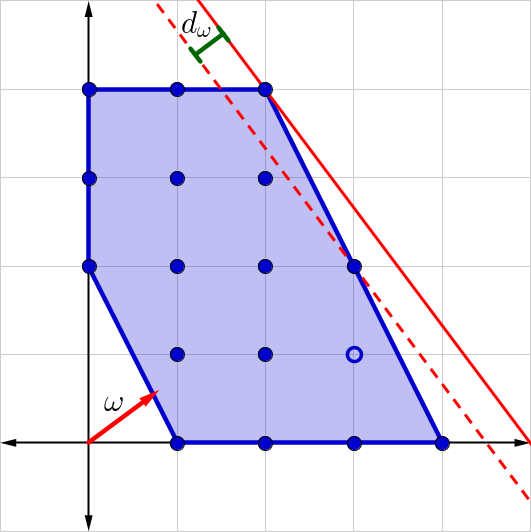}
\caption{\label{fig:dw}A unit vector $\omega$ and a geometric description of $d_\omega$.}
\end{figure}

\begin{theorem}
\label{thm:convergencerates}
Suppose $\omega \in \R^n$. Let $s(t)$ be a path of the HS-algorithm converging to $z$ as $t \to \infty$ and let $\beta$ be the number of such paths converging to $z$. Let $t_1 \geq 0$ be a number such that if $t >t_1$ then 
$|s(t) - z| \leq \gamma_z$.
Then for all $t > t_1$ 
$$|s(t)-z|^\beta \leq t^{-d_\omega} \cdot C \cdot |\mathcal A_\omega^c| \cdot \left(\frac{a_{\max}}{a_{\min}} \left(1+\frac{\Gamma_z}{\gamma_z}\right)\right)^d.$$
\end{theorem}
\begin{proof}
Recall from \eqref{eq:composition} we have
\begin{equation}
\label{eq:numproof1}
t^{-h_\mathcal A(\omega)}\cdot f({\bf L}_t)=\sum_{\alpha \in \mathcal A_\omega} c_\alpha(as-b)^\alpha + \sum_{\alpha \in \mathcal A_\omega^c}t^{\langle \omega,\alpha \rangle-h_{\mathcal A}(\omega)}c_\alpha(as-b)^{\alpha}.
\end{equation}
Suppose $s(t)\colon (1,\infty) \to \C_s$ is a continuous path in \eqref{eq:hshomotopy} so that $f({\bf L}_t(s(t))) = 0$ for all $t > 1$. Then \eqref{eq:numproof1} gives 
\begin{align}
\label{eq:numproof2}
|f_\omega(as(t)-b)| &= \left|\sum_{\alpha \in \mathcal A_\omega^c} t^{\langle \omega, \alpha \rangle-h_\mathcal{A}(\omega)}c_\alpha(as(t)-b)^\alpha \right| 
\end{align}
and so after dividing through by $K$ and extracting the largest power of $t$ from the sum,
\begin{align}
 \label{eq:numproof3}|(as-b)^m \cdot (s-\tau)^k|&\leq t^{-d_\omega} \cdot \sum_{\alpha \in \mathcal A_\omega^c} \left| \frac{c_\alpha}{K} \right| \cdot |(as(t)-b)^\alpha|  \\
 &\leq t^{-d_\omega} \cdot C \cdot \sum_{\alpha \in \mathcal A_\omega^c}   |(as(t)-b)^\alpha|.
\end{align}
Recalling that $|s(t)-z|\leq \gamma_z$ by hypothesis, we bound the right-hand summands,
\begin{align*}
|a_js(t)-b_j| = |a_j| \cdot |s(t)-\rho_j| &\leq a_{\max}\cdot |s(t)-z+z-\rho_j| \\
&\leq a_{\max}\cdot (|s(t)-z|+|z-\rho_j|) \\
&\leq a_{\max} \cdot (\gamma_z + \Gamma_{z})
\end{align*}
and so since $2 \leq a_{\max}\Gamma_z$,
\begin{align}
\label{eq:numproof4}
|(as(t)-b)^\alpha| \leq (a_{\max}(\gamma_z+\Gamma_z))^{|\alpha|} \leq (a_{\max}(\gamma_z+\Gamma_z))^{d}.
\end{align}
Substituting \eqref{eq:numproof4} into \eqref{eq:numproof3} gives
\begin{align*}
|(as-b)^m \cdot (s-\tau)^k| &\leq t^{-d_\omega} \cdot C \cdot |\mathcal A_\omega^c| \cdot (a_{\max} \cdot (\gamma_z+\Gamma_z))^d.
\end{align*}
We now bound the factors on the left-hand-side of \eqref{eq:numproof3},
\begin{align*}
|s(t)a_j-b_j| = |a_j| \cdot |s(t)-\rho_j| = |a_j| \cdot |s(t)-z+z-\rho_j| &\geq a_{\min} \Bigl| |z-\rho_j|-|s(t)-z| \Bigr| \\
& \geq a_{\min} \cdot (2\gamma_z-\gamma_z)=a_{\min}\gamma_z.
\end{align*}
Similarly, 
\begin{align*}
|s(t)-\tau_j| = |s(t)-z+z-\tau_j| &\geq \Bigl| |z-\tau_j|-|s(t)-z| \Bigr| \\
&\geq 2\gamma_z-\gamma_z=\gamma_z \geq a_{\min}\gamma_z
\end{align*}
and since $a_{\min}\gamma_z \leq 1$ and $|m|+|k| \leq d$ we have that
$$|(as(t)-b)^m(s(t)-\tau)^k| \geq (a_{\min}\gamma_z)^d.$$
So for either $z=\tau_j$ or $z=\rho_i$ we have 
$$\left| \frac{(as(t)-b)^m(s(t)-\tau)^k}{(s(t)-\tau_j)^{k_j}} \right| \geq (a_{\min} \gamma_{\tau_j})^{d-k_j},$$
$$\left| \frac{(as(t)-b)^m(s(t)-\tau)^k}{(a_is(t)-b_i)^{m_i}} \right| \geq (a_{\min} \gamma_{\rho_i})^{d-m_i},$$
respectively.

We now suppose that $z=\rho_i$ and essentially recover Theorem 8 of \cite{NP}. Note that,
$$|s(t)-\rho_i|^{m_i} = \frac{1}{|a_i|^{m_i}} \cdot\left| \frac{(as(t)-b)^m(s(t)-\tau)^k}{\left( \prod_{j \neq i}(a_js(t)-b_j)^{m_j}  \right)\cdot (s(t)-\tau)^k}\right| $$
and so putting our bounds together gives 
$$|s(t)-\rho_i|^{m_i} \leq  t^{-d_\omega} \cdot C \cdot |\mathcal A_\omega^c| \cdot (a_{\max} \cdot (\gamma_{\rho_i}+\Gamma_{\rho_i}))^d \cdot \frac{1}{a_{\min}^{m_i}} \cdot \frac{1}{(a_{\min}\gamma_{\rho_i})^{d-m_i}}.$$
Recall that $1 \geq a_{\min} \geq \gamma_{\rho_i}$ so 
$$|s(t)-\rho_i|^{m_i} \leq  t^{-d_\omega} \cdot C \cdot |\mathcal A_\omega^c| \cdot  \left(\frac{a_{\max}}{a_{\min}} \left(1 + \frac{\Gamma_{\rho_i}}{\gamma_{\rho_i}} \right)\right)^d.$$
On the other hand, if $z=\tau_j$, we have 

$$|s(t)-\tau_j|^{k_j} \leq  \frac{(as(t)-b)^m(s(t)-\tau)^k}{\left( \prod_{j \neq i}(s(t)-\tau_j)^{k_j}  \right)\cdot (as(t)-b)^m} $$
and so putting our bounds together gives 
$$|s(t)-\tau_j|^{k_j} \leq t^{-d_\omega} \cdot C \cdot |\mathcal A_\omega^c| \cdot (a_{\max} \cdot (\gamma_{\tau_j}+\Gamma_{\tau_j}))^d \cdot \frac{1}{(a_{\min}\gamma_{\tau_j})^{d-k_j}}.$$
Since $1 \geq a_{\min} \geq \gamma_{\tau_j}$, we obtain
\begin{align*}
|s(t)-\tau_j|^{k_j}&\leq  t^{-d_\omega} \cdot C \cdot |\mathcal A_\omega^c| \cdot (a_{\max} \cdot (\gamma_{\tau_j}+\Gamma_{\tau_j}))^d \cdot \frac{1}{(a_{\min}\gamma_{\tau_j})^{d}} \\
& \leq t^{-d_\omega} \cdot C \cdot |\mathcal A_\omega^c| \cdot  \left(\frac{a_{\max}}{a_{\min}} \left(1 + \frac{\Gamma_{\tau_j}}{\gamma_{\tau_j}} \right)\right)^d,
\end{align*}
completing the proof.
\end{proof}

We give an example displaying the convergence rates in Theorem \ref{thm:convergencerates}.
\begin{example}
\label{ex:simulatedconvergence}
Consider the plane curve (hypersurface) given by $\V(f)\subset \C^2$ where
\begin{align*}
f=&x+20x^2-4x^3+x^4-4xy+10x^2y+y^2+8xy^2+\\ +&4x^2y^2+x^3y^2-4y^3-6xy^3+4x^2y^3+4y^4-4xy^4+x^2y^4.
\end{align*}
Its Newton polytope and tropicalization are displayed in Figure \ref{fig:newtandtrop}. 
\begin{figure}[!htpb]
\includegraphics[scale=0.55]{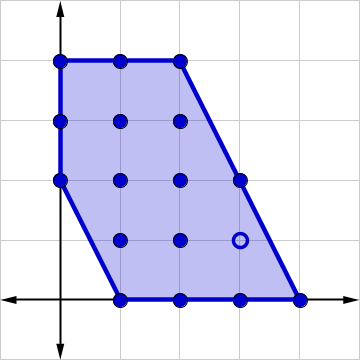}
\hspace{0.5 in}
\includegraphics[scale=0.33]{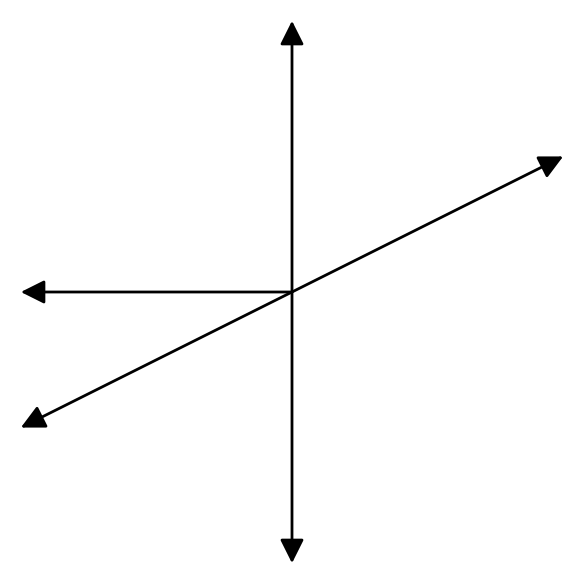}
\caption{The Newton polytope and tropicalization of a hypersurface.}
\label{fig:newtandtrop}
\end{figure}
Note that the only lattice point of the Newton polytope not appearing in the support of $f$ is the point $(3,1)$.
Figure \ref{fig:simulatedconvergence} displays the convergence rates in Theorem \ref{thm:convergencerates} as follows. For a uniform sample of unit vectors $\omega \in S^2 \subset \R^2$, we draw a ray in the direction of $\omega$ with length equal to the minimum of $1$ and $\frac{1}{d^{\omega}}$, the exponent appearing in Theorem \ref{thm:convergencerates}. 
\begin{figure}[!htpb]
\includegraphics[scale=0.5]{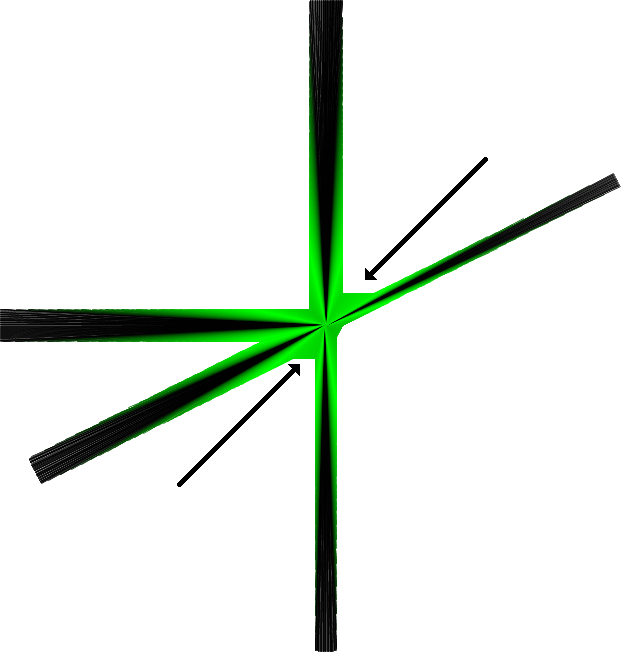}
\caption{For directions $\omega \in S^1$ we draw the ray in direction $\omega$ with length $\min(1,d^{-\omega})$ describing the convergence rate proven in Theorem \ref{thm:convergencerates}.}
\label{fig:simulatedconvergence}
\end{figure}
We remark that setting the length of the rays to be the minimum of $d^{-\omega}$ and $1$ models the feature that when this algorithm is used in practice, the user must specify a tolerance describing how far to track $t$ to see convergence.  We also point out that the ridges indicated in Figure \ref{fig:simulatedconvergence} occur because $d_\omega$ depends not only on the vertices of $\New(f)$ but also on the monomials in $\supp(f)$. \hfill $\diamond$ 
\end{example}
Example \ref{ex:simulatedconvergence} shows that in practice, the numerical oracle for the Newton polytope of a hypersurface coming from the HS-algorithm comes with a cost associated to inputs near the tropicalization of the hypersurface: as the input of the HS-algorithm approaches the tropical hypersurface, the convergence rate becomes arbitrarily slow. Due to this feature, pairing Algorithm \ref{alg:hsalgorithm} with Algorithm \ref{alg:notbeneathbeyond} may be too computationally expensive for computing large Newton polytopes.
\begin{remark}
\label{rem:positivemeasure}
Figure \ref{fig:simulatedconvergence} exposes an important drawback of the algorithms involved in this dissertation: many of our algorithms require a blind random choice of parameters avoiding some forbidden set of measure zero (in this case, the tropical variety). 

The first issue with this is that choosing parameters \emph{near} the forbidden set can cause computations to take arbitrarily long. Consequently, the true space of parameters which we want to avoid in our computations has positive measure. In the case of Figure \ref{fig:simulatedconvergence}, this is the set of directions which correspond to a black ray. 
Another example is the choice of $\gamma$ in Lemma \ref{lem:gammatrick} and Lemma \ref{lem:generalgammatrick}. If $\gamma$ is chosen near the set of measure zero, then the condition number involved in path tracking can become large. This will either cause an instance of path-jumping, or if one is using adaptive precision, can cause the computation to take arbitrarily long.

The second issue is that our computations inherently work over a subset of rational numbers with bounded height. This technically causes problems with observations such as Remark \ref{rem:oracletovertex} where the set of directions which do not expose a vertex form a finite subset of the finite set of rational numbers with bounded height; that is, a set of positive measure. 

Nonetheless, even with positive measure, the forbidden sets involved in our computations remain heuristically small and in practice the algorithms remain effective. \hfill $\diamond$
\end{remark}

\section{Implementation of Algorithms \ref{alg:hsalgorithm} and \ref{alg:TropicalMembership}}
\label{sec:numericalnp}
We describe our implementation of Algorithm \ref{alg:hsalgorithm} and Algorithm \ref{alg:TropicalMembership} along with the relevant supporting functions in our {\bf Macaulay2} package {\bf NumericalNP.m2} \cite{taylor}. This package contains four main user functions, the first three of which implement the HS-algorithm and the last implements the tropical membership algorithm. All numerical computations are piped to {\bf Bertini} \cite{Bertini} through the package {\bf Bertini.m2} \cite{B4M2}.

Function \ref{witnessForProjection} computes a pseudo-witness set for the image of a variety $X \subseteq \mathbb{C}^N$  under a projection $\pi\colon  \mathbb{C}^N \to \mathbb{C}^n$. 

\boxit{
\begin{funct}{\texttt{witnessForProjection}}
\label{witnessForProjection}\\
{\bf Input:}\\
$\bullet$ I: Ideal defining $X \subseteq \mathbb{C}^N$ \\
$\bullet$ ProjCoord: List of coordinates which are forgotten by $\pi$\\
$\bullet$ OracleLocation (option): Path in which to create witness files\\
{\bf Output:} \\ $\bullet$ A subdirectory \texttt{/OracleLocation/WitnessSet} containing\\
- witnessPointsForProj: Preimages of witness points of $\overline{\pi(X)}$\\
- projectionFile: List of coordinates in ProjCoord\\
- equations: List of equations defining $X' \subseteq X$ such that $\pi|_{X'}$ is generically finite and that $\overline{\pi(X')}=\overline{\pi(X)}$
\end{funct}}

Given a hypersurface $\mathcal H$, Function \ref{witnessToOracle}, \texttt{witnessToOracle},  creates all necessary {\bf Bertini} files to track the witness set $\mathcal H \cap \mathcal L_t$ as $t \to \infty$ for any $\omega \in \mathbb{R}^n$. These files treat $\omega$ as a parameter so that the user who wants to query many directions needs only to produce these files once. 

\boxit{\begin{funct}{\texttt{witnessToOracle}}
\label{witnessToOracle}\\
{\bf Input:} \\
$\bullet$ OracleLocation: Path containing the directory \texttt{/WitnessSet} \\
{\bf Optional Input:}\\
$\bullet$ \texttt{PointChoice}: Prescribes $a$ and $b$ explicitly (see Algorithm \ref{alg:hsalgorithm})\\
$\bullet$ \texttt{TargetChoice}: Prescribes targets $b_i/a_i$ \\
$\bullet$ \texttt{NPConfigs}: List of {\bf Bertini} path tracking configurations \\
{\bf Output:} \\ 
$\bullet$ A subdirectory \texttt{/OracleLocation/Oracle} containing all necessary files to run the homotopy described in Algorithm \ref{alg:hsalgorithm}.

\end{funct}
}

Function \ref{witnessToOracle}, by default, chooses  $a,b \in \mathbb{C}^n$ such that $\rho_i =a_i/b_i$ are the $n$-th roots of unity. One may choose to either specify $a$ and $b$ (\texttt{PointChoice}), or $\rho_i=a_i/b_i$ (\texttt{TargetChoice}) or request that these choices are random. When random, the function ensures that the points $\rho_i$ are far from each other so that convergence to $\rho_i$ is easily distinguished from convergence to $\rho_j$.  
{\bf Bertini} is called to track the points in \texttt{/OracleLocation/WitnessSet} to the points $\overline{\pi(X)} \cap \mathcal L_1$. These become start solutions of the homotopy described in Algorithm \ref{alg:hsalgorithm} with parameters $\omega$ and $t$. There are many numerical choices for {\bf Bertini}'s native path-tracking algorithms which can be specified via \texttt{NPConfigs}.

The fundamental function of \textbf{NumericalNP.m2} is \texttt{oracleQuery}. It runs the homotopy described in the HS-algorithm on a hypersurface $\mathcal H=\V(f)$, monitors convergence, and outputs the result of the numerical oracle.

\boxit{
\begin{funct}{\texttt{oracleQuery}}
\label{queryOracle}

{\bf Input:}\\
$\bullet$  OracleLocation (Option): Location containing the directory \texttt{/Oracle}\\
$\bullet$  $\omega$: A vector in $\mathbb{R}^n$\\
{\bf Optional Input:}\\ 
$\bullet$ \texttt{Certainty} \hspace{0.05 in}  $\bullet$ \texttt{Epsilon} \hspace{0.05 in}  $\bullet$ 
  \texttt{MinTracks} \hspace{0.05 in}  $\bullet$ 
 \texttt{MaxTracks} \hspace{0.05 in}
   $\bullet$
 \texttt{StepResolution} \hspace{0.05 in} \\ {$\bullet$ 
 \texttt{MakeSageFile}}\\  
{\bf Output:}\\
$\bullet$ $\mathcal O_{\New(\widetilde{f})}(\omega)$ or \texttt{Reached MaxTracks}\\
$\bullet$  A subdirectory \texttt{/OracleLocation/OracleCalls/Call\#} containing\\
-  \texttt{SageFile}: Sage code animating the paths $s(t)$\\
-  \texttt{OracleCallSummary}: a human-readable file summarizing the results
\end{funct}}

To monitor convergence of points $s(t)$ the software tracks $t \to \infty$ in discrete steps. The option \texttt{StepResolution} specifies these step sizes. In each step and for each path $s_{i}(t)$, a numerical derivative is computed to heuristically determine convergence or divergence of the solution. If the solution is large and the numerical derivative exceeds $10^{\texttt{Certainty}}$ in two consecutive steps, then the path is declared to diverge, and if the numerical derivative is below $10^{-\texttt{Certainty}}$ in two consecutive steps, then the point is declared to converge. If a converged point is at most \texttt{Epsilon} from some $\rho_i$, then the software deems that it has converged to $\rho_i$. When a point is declared to converge or diverge, it is not tracked further.
\begin{figure}[!htpb]
\includegraphics[scale=0.32]{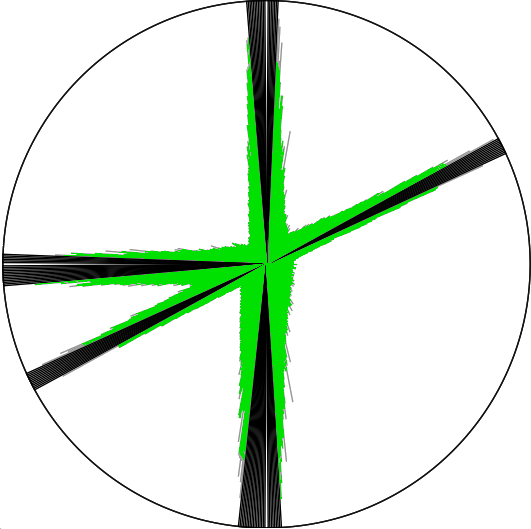} \hspace{0.15 in} 
\includegraphics[scale=0.45]{numericalnppolytope.png} \hspace{0.1 in}
\includegraphics[scale=0.32]{fig3b.png}
\caption{(Reprinted from \cite{Bry:NPtrop}) Left: Values of $t$ (magnitude of rays) such that queryOracle finishes for different $\omega$ (direction of rays) on a hypersurface with Newton polytope (center) and normal fan (right).}
\label{koosh}
\end{figure}
The option \texttt{MaxTracks} allows the user to specify how long to wait for convergence of the paths $s(t)$.

\begin{example}
Figure \ref{koosh} shows the Newton polytope of the same plane sextic as in Example \ref{ex:simulatedconvergence}. It also shows the convergence rate of the algorithm on different directions $\omega \in S^1$. The length of each green ray is proportional to the number of steps required for \texttt{oracleQuery} to finish and the black rays indicate that this convergence was not observed within the limit specified by \texttt{MaxTracks}.  We note that the striking resemblance of Figure \ref{fig:simulatedconvergence} and \ref{koosh} indicates that the value of $t$ at which our implementation recognizes convergence is approximately proportional to the convergence rate we prove in Theorem \ref{thm:convergencerates}. We include the image of the tropicalization of this curve to illustrate how the convergence rate involved in the HS-algorithm slows as $\omega$ approaches directions in the tropical variety. Nonetheless, we remind the reader that this slow convergence rate does not occur when $\omega$ is in the tropical variety. \hfill $\diamond$
\end{example}

One may also specify \texttt{MinTracks} which indicates the step at which convergence begins to be monitored. The option to create a Sage \cite{sage} animation (see Figure \ref{SAGE}) of the solution paths helps the user recognize pathological behavior in the numerical computations and fine-tune parameters such as \texttt{Certainty}, \texttt{StepResolution}, or \texttt{Epsilon} accordingly.
\begin{example}
\label{ex:example}
Consider the curve in $X \subseteq \mathbb{C}^3$ defined by $$I=\langle xyt-(x-y-t)^2+3x+t,x+y^2+t^2 \rangle\subseteq \mathbb{C}[x,y,t]$$ and let $\pi$ be the projection forgetting the $t$ coordinate. The following code written in {\bf Macaulay2} computes a witness set for $\mathcal C=\overline{\pi(X)}$, prepares oracle files for the HS-algorithm and then runs the HS-algorithm in the direction $(3,2)$. The software returns the list $\{2,4,0\}$ indicating that $\New(\overline{\pi(X)})_{(3,2)}=(2,4)$.
{\small
\begin{verbatim}
i1: loadPackage("NumericalNP");
i2: R=CC[x,y,t];
i3: I=ideal(x*y*t-(x-y-t)^2+3*x+t,x+y^2+t^2);
i4: witnessForProjection(I,{2},OracleLocation=>"Example");
i5: witnessToOracle("Example") ;
i6: time oracleQuery({3,2},OracleLocation=>"Example",MakeSageFile=>true)
     -- used 0.178448 seconds
o6: {2,4,0} 
\end{verbatim}}
The equation of $\pi(X)$ is the polynomial in Example \ref{ex:simulatedconvergence} and so its Newton polytope is displayed in Figure \ref{koosh}. Snapshots of the Sage animation created by \texttt{queryOracle} are shown in Figure \ref{SAGE}. There, the circles are centered at $\rho_1=1$ and $\rho_2=-1$ and have radius \texttt{epsilon}.
\begin{figure}[!htpb]
\centering
\includegraphics[scale=0.45]{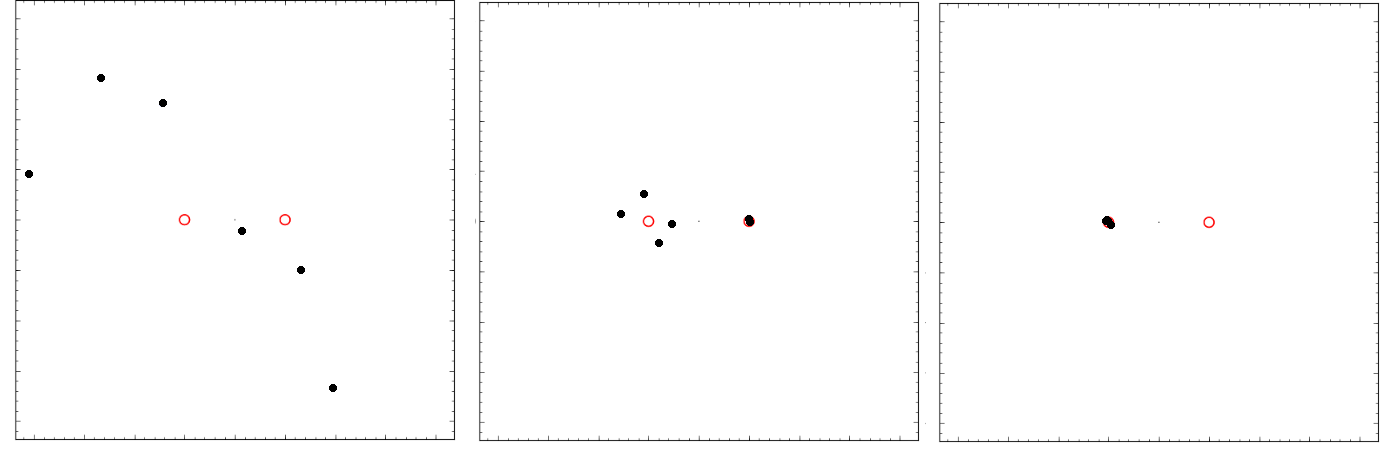}
\caption{Three snapshots of Sage animation from example with viewing window $[-4,4]^2$}
\label{SAGE}
\end{figure} The first image shows the intersections (in the $s$-coordinates) of the sextic $\overline{\pi(X)}$ with $\mathcal L_1$ in the complex plane $\C_s$. The second image is a snapshot showing two points converging to $s=1$ and the third image shows four other points converging to $s=-1$. \hfill $\diamond$
\end{example}

Given an ideal $I$, the fourth function \texttt{tropicalMembership} computes a pseudo-witness set for each coordinate projection $\pi(\V(I))$ whose image is a hypersurface. The algorithm subsequently checks that \texttt{oracleQuery} indicates that $\pi(\omega) \in \trop(\pi(\V(I))$. If this is true for each coordinate projection, the algorithm returns \texttt{true} and otherwise returns \texttt{false}. The numerical options fed to \texttt{tropicalMembership} are passed along to \texttt{oracleQuery}.

\boxit{
\begin{funct}{\texttt{tropicalMembership}}
\label{tropicalMembership}

{\bf Input:}\\
$\bullet$   $I:$ Ideal defining $X \subseteq \mathbb{C}^n$\\
$\bullet$  $\omega$: A vector in $\mathbb{R}^n$\\
{\bf Optional Input:}\\
$\bullet$ \texttt{Certainty} \hspace{0.03 in}  $\bullet$ \texttt{Epsilon}\hspace{0.03 in}  \hspace{0.03 in}  $\bullet$ 
  \texttt{MinTracks} \hspace{0.03 in}  $\bullet$
 \texttt{MaxTracks} \hspace{0.03 in}$\bullet$ \texttt{StepResolution} \hspace{0.03 in} \\ {$\bullet$ \texttt{MakeSageFile}}\\  
{\bf Output:}\\
$\bullet$ A list of oracle queries of $\pi(X)$ in directions $\pi(\omega)$ where $\pi$ runs through all coordinate projections such that $\pi(X)$ is a hypersurface.\\
$\bullet$ \texttt{true} if all oracle queries exposed positive-dimensional faces and \texttt{false} otherwise
\end{funct}}

\begin{example}
\label{ex:trop} We return to Example \ref{example:tropcube} of two tropical space curves which are different, yet have the same tropicalized coordinate projections. 
We depict these tropical curves again in Figure \ref{fig:badtrop} and illustrate their behavior with our software.
\begin{figure}[!htpb]
\includegraphics[scale=0.5]{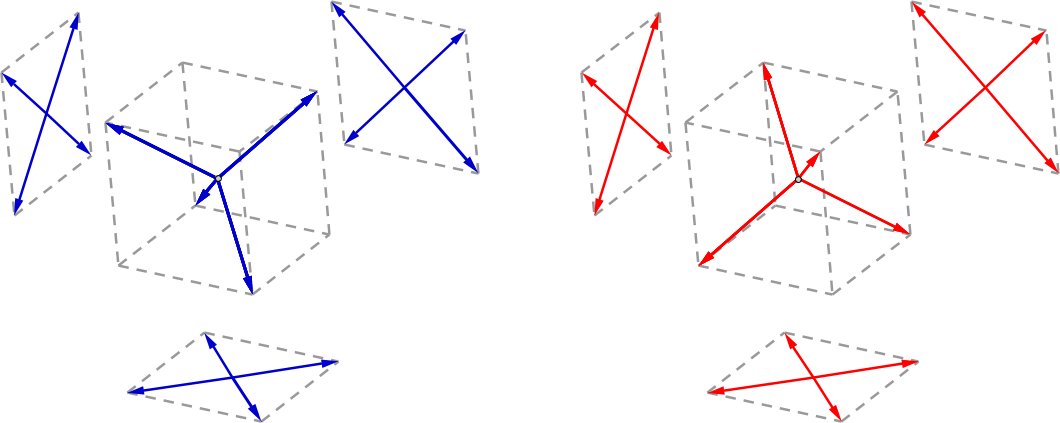} 
\caption{(Reprinted from \cite{Bry:NPtrop}) Two tropical space curves with the same tropical coordinate projections}
\label{fig:badtrop}
\end{figure}
{\small
\begin{verbatim}
i1 : loadPackage("NumericalNP");
i2 : R=QQ[x,y,z];
i3 : I_1=ideal {x*z+4*y*z-z^2+3*x-12*y+5*z,x*y-4*y^2+y*z+x+2*y-z};
i4 : I_2=ideal{x*y-3*x*z+3*y*z-1,3*x*z^2-12*y*z^2+x*z+4*y*z+5*z-1};
i5 : I_1==I_2
o5 = false
i6 : directions:={{1,1,1},{1,1,-1},{1,-1,1},
{1,-1,-1},{-1,1,1},{-1,1,-1},{-1,-1,1},{-1,-1,-1}};
i7 : apply(directions,d->tropicalMembership(I_2,d))
o7 = {true, true, true, true, true, true, true, true}
i8 : apply(directions,d->tropicalMembership(I_1,d))
o8 = {true, true, true, true, true, true, true, true}
\end{verbatim}}\hfill $\diamond$
\end{example}
Every projection of every vertex of $\text{cube}(3)$ is in the tropicalization of the corresponding projection of $\V(I_1)$ and $\V(I_2)$. Nonetheless, the tropicalizations of $\V(I_1)$ and $\V(I_2)$ are disjoint subsets of the vertices of the cube, exemplifying that an output of \texttt{true} from \texttt{tropicalMembership} is not a certification of membership in the tropical variety. Unfortunately, we cannot {\it a priori} decide whether or not our coordinate projections are generic. \\
\\
{\textbf{ Example \ref{ex:trop} (continued).}}
Consider the monomial change of coordinates $\Phi$ given by $$\Phi(x)=xyz,\quad \Phi(y)=y,\quad  \text{and} \quad \Phi(z)=z,$$ and let $\Phi^*=\varphi:\Z^3 \to \Z^3$ be the linear map corresponding to $\Phi$.  Let $F$ and $G$ be the generators used in the above code of $I_1$ and $I_2$ respectively. By Equation \eqref{eq:tropunderchange} of Section~\ref{subsection:tropicalgeometry} we have that $$\trop(\V(F)) = \varphi(\trop(\V(F \circ \Phi))), \hspace{0.2 in} \varphi={
\begin{bmatrix}
1 & 1 & 1 \\
0 & 1 & 0 \\
0 & 0 & 1
\end{bmatrix}}={ \begin{bmatrix}
1 & -1 & -1 \\
0 & 1 & 0 \\
0 & 0 & 1
\end{bmatrix}^{-1}.}
$$
The linear transformation $\varphi$ produces generic coordinate projections in the sense of Theorem \ref{thm:TropicalProjection} and the function \texttt{tropicalMembership} is able to distinguish $\trop(\V(I_1))$ from $\trop(\V(I_2))$.
{\small
\begin{verbatim}
i9 : I'_1=ideal apply((I_1)_*,f->sub(f,{x=>x*y*z,y=>y,z=>z}));
i10 : I'_2=ideal apply((I_2)_*,f->sub(f,{x=>x*y*z,y=>y,z=>z}));
i11 : directions'=apply(directions,d->{d#0-d#1-d#2,d#1,d#2})
o11 = {{-1, 1, 1}, {1, 1, -1}, {1, -1, 1}, {3, -1, -1},
      {-3, 1, 1}, {-1, 1, -1}, {-1, -1, 1}, {1, -1, -1}}
i12 : apply(directions',d->tropicalMembership(I'_1,d))
o12 = {false, true, true, false, true, false, false, true}
i13 : apply(directions',d->tropicalMembership(I'_2,d))
o13 : {true, false, false, true, false, true, true, false}
\end{verbatim}
}

\section{A hypersurface from algebraic vision}
\label{sec:AlgebraicVision}
The following example is a hypersurface in the space of $3 \times 2 \times 2$ tensors coming from a multiview variety of a pinhole camera and a two slit camera. This example can be found in Proposition $7.5$ of \cite{Ponce2017}, where the authors computed the polynomial symbolically via elimination with respect to another variety in the space of $3 \times 3 \times 3$ tensors. Although this computation is not new, it serves to demonstrate the strength of our implementation.

Consider the matrix 
$$\begin{bmatrix}
A&B&C
\end{bmatrix}
=
\begin{bmatrix}
a_{1,1} & a_{1,2} & a_{1,3} & b_{1,1} & b_{1,2} & c_{1,1} & c_{1,2} \\
a_{2,1} & a_{2,2} & a_{2,3} & b_{2,1} & b_{2,2} & c_{2,1} & c_{2,2} \\
a_{3,1} & a_{3,2} & a_{3,3} & b_{3,1} & b_{3,2} & c_{3,1} & c_{3,2} \\
a_{4,1} & a_{4,2} & a_{4,3} & b_{4,1} & b_{4,2} & c_{4,1} & c_{4,2} 
\end{bmatrix}.
$$
The matrix $A$ represents a pinhole camera and $(B,C)$, a two slit camera. The corresponding multi-view variety $X$ is a hypersurface in $\mathbb{P}^{11}$.
Let $f_{i,j,k}$ be the minor corresponding to the submatrix which ignores columns $a_i,b_j,$ and $c_k$. Then $X$ is parametrized by these twelve minors 
\begin{align*}
F\colon \mathbb{C}^{28} &\to \mathbb{C}^{12}\\
\begin{bmatrix}
a_{1,1} & a_{1,2} & a_{1,3} & b_{1,1} & b_{1,2} & c_{1,1} & c_{1,2} \\
a_{2,1} & a_{2,2} & a_{2,3} & b_{2,1} & b_{2,2} & c_{2,1} & c_{2,2} \\
a_{3,1} & a_{3,2} & a_{3,3} & b_{3,1} & b_{3,2} & c_{3,1} & c_{3,2} \\
a_{4,1} & a_{4,2} & a_{4,3} & b_{4,1} & b_{4,2} & c_{4,1} & c_{4,2} 
\end{bmatrix}
&\overset{F}{\longmapsto}
[f_{i,j,k}]_{i \in \{1,2,3\},j,k \in \{1,2\}}.
\end{align*}

 This map has $17$-dimensional fibers. Rather than taking generic linear slices in $\C^{28}$, we find constant replacements for $17$ of the variables under the condition that the Jacobian of $F$ does not drop rank. This substitution gives a new map $\mathcal F\colon  \mathbb{C}^{11} \to \mathbb{C}^{12}$ whose image is $X$.
 
We order the $f_{i,j,k}$ variables lexicographically,
$$(f_{111},f_{112},f_{121},f_{122},
f_{211},f_{212},f_{221},f_{222},
f_{311},f_{312},f_{321},f_{322}).$$
The polynomial $f$ which cuts out $X$ is homogeneous of degree $6$ in $12$ variables, giving an {\it a priori} upper bound of $12,376$ possible monomials appearing in $\mathcal A= \supp(f)$. There is a group action of $G \cong S_3 \times S_2 \times S_2$ on the coordinates in $\mathbb{C}^{12}$ taking $f_{i,j,k} \to f_{\sigma(i),\tau(j),\nu(k)}$ which extends to an action on the vertices of the polytope. This action is transitive on $\{f_{i,j,k}\}_{i \in \{1,2,3\}, j,k \in \{1,2\}}$ and so to get a bound on the size of any coordinate $\alpha \in \mathcal A$, it is enough to bound one. An oracle query in the $(1,0,\ldots,0)$ direction returns the vector $(2,0,\ldots,0)$ along with four points which converge somewhere other than a target. As such,  $\New(f) \subset \bigcap_{i=1}^{12} \R^{12}_{e_i,2} = 2 \cdot [0,1]^{12}$. This reduces the possible number of lattice points in $\New(f)$ to $8,074$. Querying the oracle in the independent directions
\begin{align*}
(1,1,1,1,0,0,0,0,0,0,0,0)&\qquad
(0,0,0,0,1,1,1,1,0,0,0,0)\\
(1,0,1,0,1,0,1,0,1,0,1,0) \quad
(1,1,0,0,1,1&,0,0,1,1,0,0) \quad
(1,1,1,1,1,1,1,1,1,1,1,1)
\end{align*} returns \texttt{Exposes entire polytope}.
Thus, $\New(f)$ is a subset of a $7$-dimensional subspace of $\mathbb{R}^{12}$. The following four directions expose four vertices of $\New(f)$ which, after applying symmetries of $G$, become $60$ vertices $V$ of a $7$-dimensional polytope $P_*\subset \New(f)$ containing $60+6$ lattice points.
\begin{align*}
\mathcal{O}_{\New(f)}&(6,-3.5,-1,0.4,0.16,.6,0.2,1.33,.66,.9,4,-4)\\
&=(2,0,0,0,0,0,0,2,0,1,1,0)\\
\mathcal{O}_{\New(f)}&(.31, -.31, -.31, .31, -.31, .09, -.31, .31, .31, -.31, .09, -.31)\\
&=(1,0,0,1,0,1,0,1,1,0,1,0)\\
\mathcal{O}_{\New(f)}&(-.31, -.31, .31, .09, -.31, .31, .31, -.31, .09, .31, -.31, -.31)\\
&=(0,0,1,1,0,1,1,0,1,1,0,0)\\
\mathcal{O}_{\New(f)}&(.19, -.39, .13, .19, .04, .08, -.33, .04, .25, -.20, -.13, .71)\\
&=(1,0,1,0,1,1,0,0,0,0,0,2).
\end{align*}
We conclude that $\New(f)$ is $7$-dimensional. Two more oracle queries,
\begin{align*}
\mathcal{O}_{\New(f)}(-11,-3,-3,5,-11,-3,-3,5,1,9,9,-31)&=(0,0,0,0,0,0,0,0,0,0,0,0)\\
\mathcal{O}_{\New(f)}(-5,3,3,-5,-5,3,3,-5,-5,3,3,-5)&=(0,0,0,0,0,0,0,0,0,0,0,0)
\end{align*}
imply that a positive-dimensional face of $\New(f)$ is exposed in each of these directions. The facets of $P_*$ are also exposed by these directions but no other pair of points within $P^*=\bigcap_{i=1}^{12} \R^{12}_{e_i,2}$ are exposed. Thus, $\New(f)=P_*$.

Knowing the support of $f$, interpolation successfully recovers the polynomial computed in \cite{Ponce2017}:
\singlespacing
{{
$$
f=f_{111}^2f_{212}f_{221}f_{322}^2-f_{111}^2f_{212}f_{222}f_{321}f_{322}-f_{111}^2f_{221}f_{222}f_{312}f_{322}+\\
$$
$$
f_{111}^2f_{222}^2f_{312}f_{321}-f_{111}f_{112}f_{211}f_{221}f_{322}^2+f_{111}f_{112}f_{211}f_{222}f_{321}f_{322}-\\
$$
$$
f_{111}f_{112}f_{212}f_{221}f_{321}f_{322}+f_{111}f_{112}f_{212}f_{222}f_{321}^2+
f_{111}f_{112}f_{221}^2f_{312}f_{322}+\\
$$
$$
f_{111}f_{112}f_{221}f_{222}f_{311}f_{322}-f_{111}f_{112}f_{221}f_{222}f_{312}f_{321}-f_{111}f_{112}f_{222}^2f_{311}f_{321}-\\
$$
$$
f_{111}f_{121}f_{211}f_{212}f_{322}^2+f_{111}f_{121}f_{211}f_{222}f_{312}f_{322}+
f_{111}f_{121}f_{212}^2f_{321}f_{322}-\\
$$
$$
f_{111}f_{121}f_{212}f_{221}f_{312}f_{322}+f_{111}f_{121}f_{212}f_{222}f_{311}f_{322}-f_{111}f_{121}f_{212}f_{222}f_{312}f_{321}+\\
$$
$$
f_{111}f_{121}f_{221}f_{222}f_{312}^2-f_{111}f_{121}f_{222}^2f_{311}f_{312}+f_{111}f_{122}f_{211}f_{212}f_{321}f_{322}+\\
$$
$$
f_{111}f_{122}f_{211}f_{221}f_{312}f_{322}-2f_{111}f_{122}f_{211}f_{222}f_{312}f_{321}-f_{111}f_{122}f_{212}^2f_{321}^2-\\
$$
$$
2f_{111}f_{122}f_{212}f_{221}f_{311}f_{322}+2f_{111}f_{122}f_{212}f_{221}f_{312}f_{321}+
f_{111}f_{122}f_{212}f_{222}f_{311}f_{321}-\\
$$
$$
f_{111}f_{122}f_{221}^2f_{312}^2+f_{111}f_{122}f_{221}f_{222}f_{311}f_{312}+
f_{112}^2f_{211}f_{221}f_{321}f_{322}-\\
$$
$$
f_{112}^2f_{211}f_{222}f_{321}^2-f_{112}^2f_{221}^2f_{311}f_{322}+f_{112}^2f_{221}f_{222}f_{311}f_{321}\\
$$
$$
+f_{112}f_{121}f_{211}^2f_{322}^2-f_{112}f_{121}f_{211}f_{212}f_{321}f_{322}-f_{112}f_{121}f_{211}f_{221}f_{312}f_{322}-\\
$$
$$
2f_{112}f_{121}f_{211}f_{222}f_{311}f_{322}+2f_{112}f_{121}f_{211}f_{222}f_{312}f_{321}+
2f_{112}f_{121}f_{212}f_{221}f_{311}f_{322}-\\
$$
$$
f_{112}f_{121}f_{212}f_{222}f_{311}f_{321}-f_{112}f_{121}f_{221}f_{222}f_{311}f_{312}+
f_{112}f_{121}f_{222}^2f_{311}^2\\
$$
$$
-f_{112}f_{122}f_{211}^2f_{321}f_{322}+f_{112}f_{122}f_{211}f_{212}f_{321}^2+
f_{112}f_{122}f_{211}f_{221}f_{311}f_{322}-\\
$$
$$
f_{112}f_{122}f_{211}f_{221}f_{312}f_{321}+f_{112}f_{122}f_{211}f_{222}f_{311}f_{321}-f_{112}f_{122}f_{212}f_{221}f_{311}f_{321}+\\
$$
$$
f_{112}f_{122}f_{221}^2f_{311}f_{312}-f_{112}f_{122}f_{221}f_{222}f_{311}^2+f_{121}^2f_{211}f_{212}f_{312}f_{322}-\\
$$
$$
f_{121}^2f_{211}f_{222}f_{312}^2-f_{121}^2f_{212}^2f_{311}f_{322}+
f_{121}^2f_{212}f_{222}f_{311}f_{312}-\\
$$
$$
f_{121}f_{122}f_{211}^2f_{312}f_{322}+f_{121}f_{122}f_{211}f_{212}f_{311}f_{322}-f_{121}f_{122}f_{211}f_{212}f_{312}f_{321}+\\
$$
$$
f_{121}f_{122}f_{211}f_{221}f_{312}^2+f_{121}f_{122}f_{211}f_{222}f_{311}f_{312}
+f_{121}f_{122}f_{212}^2f_{311}f_{321}-\\
$$
$$
f_{121}f_{122}f_{212}f_{221}f_{311}f_{312}-f_{121}f_{122}f_{212}f_{222}f_{311}^2
+f_{122}^2f_{211}^2f_{312}f_{321}-\\
$$
$$
f_{122}^2f_{211}f_{212}f_{311}f_{321}-f_{122}^2f_{211}f_{221}f_{311}f_{312}+f_{122}^2f_{212}f_{221}f_{311}^2
$$}}
\doublespacing

\section{The L\"uroth invariant}
\label{sec:Luroth}
\subsection{The L\"uroth invariant, hypersurface, and polytope.}
Let $\C_q^{15}$ be the vector space spanned by all homogeneous quartic plane curves with coefficients $\{q_{ijk}\}_{i+j+k=4}$ so that a quartic $Q\in\C_q^{15}$ is written as
$$Q=\sum_{i+j+k=4} q_{ijk}x^iy^jz^k.$$
Such a quartic $\V(Q) \subset \P^2$ is called \mydef{L\"uroth} if it passes through the ten intersection points of five lines in $\P^2$. We display one such quartic in Figure \ref{fig:lurothquartic}.

\begin{figure}[!htpb]
\includegraphics[scale=0.4]{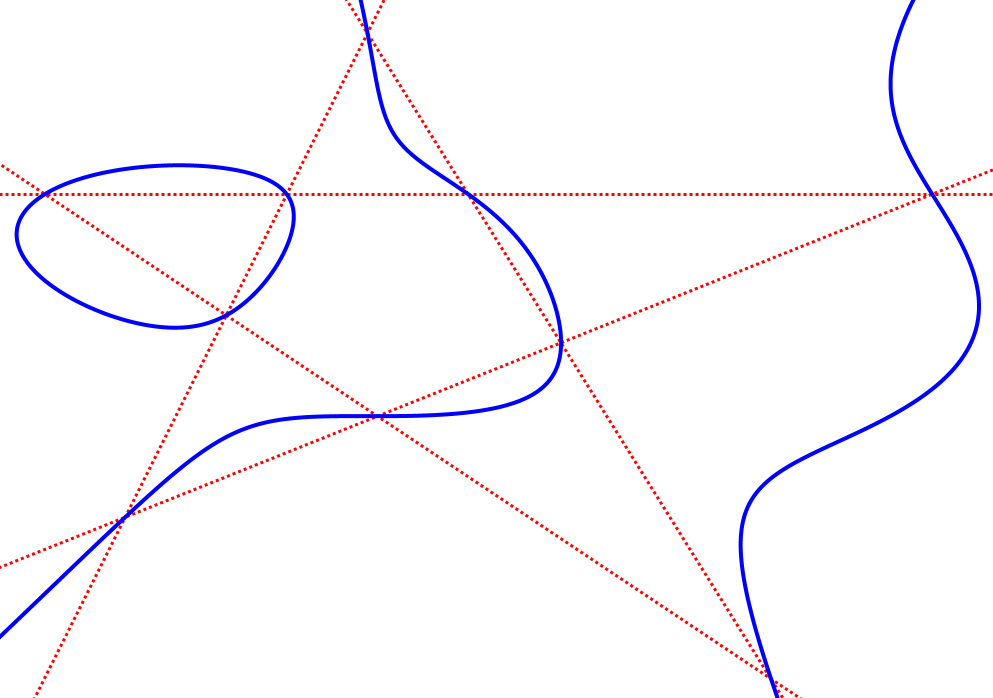}
\caption{A L\"uroth quartic.}
\label{fig:lurothquartic}
\end{figure}

The set of all L\"uroth quartics $\mydefMATH{\mathbb{L}}$ is a hypersurface of degree $54$ in $\P_q^{14}$ called the \mydef{L\"uroth hypersurface}. The group $\text{PGL}(3,\C)$ of all projective linear transformations of $\P^2$ acts on a plane quartic $\V(Q)$ by some element $A \in \text{PGL}(3,\C)$ in the natural way: $\V(Q) \mapsto A\cdot \V(Q)=\V(Q(A^*(x,y,z)))$. This action preserves intersection points and so if $\V(Q)$ is a L\"uroth quartic, so is $A \cdot \V(Q)$.  The defining equation $\mydefMATH{\Lambda}$ of the L\"uroth hypersurface is called the \mydef{L\"uroth invariant}.  

The L\"uroth hypersurface is parametrized by the coefficients of five homogeneous linear polynomials $\mydefMATH{\ell_i} = a_ix+b_iy+c_iz \in \C[x,y,z]$. This parametrization is 
\begin{align}
\label{eq:lurothparametrization}
\mydefMATH{\varphi}\colon  \P((\mathbb{C}^3)^5) &\dashrightarrow \mathbb{P}^{14}\\(\ell_1,\ldots,\ell_5)&\mapsto \sum_{j=1}^5 \prod_{i \neq j} \ell_i=\sum_{i+j+k=4} q_{ijk}x^iy^jz^k.
\end{align}
Finding $\Lambda$ using symbolic elimination algorithms is computationally infeasible. Moreover, it is expected that $\Lambda$ in its expanded form is not human-readable. Thus, we attempt to determine the \mydef{L\"uroth polytope}, $\mydefMATH{\mathfrak{P}}=\New(\Lambda)\subset \R^{15}$ using Algorithm \ref{alg:hsalgorithm}. Before discussing computations, we explain some reductions to the problem.

\begin{corollary}
\label{cor:lurothresult}
Every point $p$ in the L\"uroth polytope $\mathfrak{P}$ solves the linear equation
$$\begin{pmatrix}
4&3&3&2&2&2&1&1&1&1&0&0&0&0&0 \\
0&1&0&2&1&0&3&2&1&0&4&3&2&1&0 \\
0&0&1&0&1&2&0&1&2&3&0&1&2&3&4 \\
1&1&1&1&1&1&1&1&1&1&1&1&1&1&1
\end{pmatrix} p
= \begin{pmatrix} 72 \\ 72  \\ 72 \\54
\end{pmatrix}.$$
\end{corollary}
\begin{proof}
The L\"uroth invariant is a homogeneous polynomial of degree $54$ in the coefficient space $\C_q^{4\Delta_3} \cong \P^{14}$ which is invariant under permutations and scalings of variables $x,y,$ and $z$. Observing that $\frac{54 \cdot 4}{3}=72$ and applying Lemma \ref{prop:invariantresult} gives the result.
\end{proof}

We order the coordinates of the space $\R^{15}_p$ containing $\mathfrak{P}$ as follows 
\begin{equation}
\label{eq:variableorder}\{\underset{0}{p_{400}},\underset{1}{p_{310}},\underset{2}{p_{301}},\underset{3}{p_{220}},\underset{4}{p_{211}},\underset{5}{p_{202}},\underset{6}{p_{130}},\underset{7}{p_{121}},\underset{8}{p_{112}},\underset{9}{p_{103}},
\underset{10}{p_{040}},\underset{11}{p_{031}},\underset{12}{p_{022}},\underset{13}{p_{013}},\underset{14}{p_{004}}\}.
\end{equation}
We may identify coordinates by their numerical bijection with $0,1,\ldots,14$ as listed above. For example, $p_{301}$ may be written as $p_2$. 
Under this bijection, the permutation group $S_3$ acting on coordinates of the subscripts of $p_{ijk}$ induces the following involutions,
\begin{align*}
\mydefMATH{\sigma_{xy}}&=(0,10)(1,6)(2,11)(4,7)(5,12)(9,13)\\
\mydefMATH{\sigma_{yz}}&=(1,2)(3,5)(6,9)(7,8)(10,14)(11,13)\\
\mydefMATH{\sigma_{xz}}&=\sigma_{xy}\circ \sigma_{yz} \circ \sigma_{xy}= (0,14)(1,13)(2,9)(3,12)(4,8)(6,11),\\
\end{align*}
written in cycle notation. We write $\mydefMATH{G}$ for this subgroup $S_3 \hookrightarrow S_{15}$. Corollary \ref{cor:lurothresult} gives the \emph{a priori} bounds of
\begin{equation}
\label{eq:aprioribounds}
p_{400} \in [0,18], \quad
p_{310} \in[0,24], \quad
p_{220} \in [0,36],\quad
p_{211} \in [0,36]
 \end{equation}on the sizes of each coordinate $p_{ijk}$ of a point in $\mathfrak P$.
In Section~\ref{subsubsection:verticesofluroth}, we use the HS-algorithm to show that these bounds are not sharp.

\subsection{Computational setup}
To perform computations, we dehomogenize the domain $\P((\C^{3})^5)$ of the parametrization \eqref{eq:lurothparametrization} with respect to a random linear polynomial and work with the restricted map ${\varphi}\colon  \C^{14} \to \C_q^{15}$.

We parametrize the lines $\mathcal L_t$ in the HS-algorithm by $t \overset{\textbf{L}_t} \mapsto \{t^{\omega_i}(a_is-b_i)\}_{i=0}^{14}$ and chose $a,b \in (\C^\times)^{15}$ so that the target points $\rho_i = b_i/a_i$ are the $15$-th roots of unity $\{\zeta_{15}^i\}_{i=0}^{14}$ where $\mydefMATH{\zeta_{15}}=e^{2\pi \sqrt{-1}/15}$. We do this under the heuristic assumption that choosing targets far from one another decreases the chances of the implementation \textbf{NumericalNP.m2} misattributing which target is the true limit of a path. 

To compute the intersection points of $\mathcal L_t \cap \Lambda$ in the parameters $s$, we employ the fiber product construction in Figure \ref{eq:fibreproductluroth} 
\begin{figure}[t]
\begin{tikzpicture}
  \matrix (m) [matrix of math nodes,row sep=3em,column sep=4em,minimum width=2em]
  {
     \C^{14}_a \times \C_s& \C_s \\
     \C^{14}_a  & \C_q^{15} \\};
  \path[-stealth]
    (m-1-1) edge node [left] {$\pi_a$} (m-2-1)
            edge node [below] {$\pi_s$} (m-1-2)
    (m-2-1.east|-m-2-2) edge node [below] {$\varphi$}
            node [above] {} (m-2-2)
    (m-1-2) edge node [right] {$\textbf{L}_t$} (m-2-2)
            ;
\end{tikzpicture}
\caption{
\label{eq:fibreproductluroth}A fiber product construction to compute witness points.}
\end{figure}
 and solve the equations 
 \begin{equation}
  \mydefMATH{F_t}=\{q_{ijk}(a)-t^{\omega_\iota}(a_\iota s - b_\iota) \}_{\iota=0}^{14}
 \end{equation}
in $\C_a^{14} \times \C_s$,  where $ijk \leftrightarrow \iota$ is the identification of $\{(i,j,k) \mid i,j,k \geq 0, \quad i+j+k = 4\}$ with $\{0,\ldots,14\}$ given in \eqref{eq:variableorder}. During the homotopy process, we project a solution $(a,s) \in \C_a^{14} \times \C_s$ to $s\in \C_s$ to monitor convergence.

Whenever querying the numerical oracle in a direction $\omega \in S^{14} \subset \R^{15}$, by Theorem \ref{thm:convergencerates} it is best to attempt to maximize $d_\omega$. Given that we do not \emph{a priori} know the polytope $\mathfrak{P}$, this is generally difficult. Nonetheless, we always project $\omega$ onto the kernel of the matrix in Corollary \ref{cor:lurothresult} and rescale this projection to be a unit vector. Given that $|\omega|=1$, this process increases $d_\omega$ thus increasing the convergence rate of the HS-algorithm. 

Given the size of these computations, we expect numerical errors to occur. However, assessing whether something went wrong during the path tracking process can be done in several ways.

\begin{remark} 
Suppose that a numerical implementation of the HS-algorithm returns a vertex $\mathcal O_{\mathfrak P}(\omega)=v$ on the direction $\omega \in \R^{15}$. A numerical error has occurred if any of the following are true.
\begin{enumerate}
\item $v_\infty \neq 0$, (Since $\Lambda$ is homogeneous, no points in the HS-algorithm will diverge)
\item $\mathcal O_{\mathfrak{P}}(\sigma(\omega))\neq \sigma(v)$ for any $\sigma \in G$, ($\Lambda$ is invariant under $G$)
\item $v$ does not solve the matrix equation in Corollary \ref{cor:lurothresult}.
\item $v \not\in H_{\mathfrak P}(\nu)$ where $H_{\mathfrak P}(\nu)$ is the halfspace containing $\mathfrak P$ implied by an oracle call $\mathcal O_{\mathfrak P}(\nu)$ we have already performed on $\nu \in \R^{15}$, (Since $\mathfrak P \subset H_{\mathfrak P}(\nu)$ for any $\nu \in \R^{15}$).
\end{enumerate}
In the last case, an error has occurred on either $\mathcal O_{\mathfrak P}(\omega)$ or $\mathcal O_{\mathfrak P}(\nu)$. \hfill $\diamond$
\end{remark}

\subsection{Vertices of the L\"uroth polytope}
\label{subsubsection:verticesofluroth}
 Using \textbf{NumericalNP.m2}, we reproduce the result of \cite{NP} that  $$\mathfrak{P}_{(3,-5,3,2,3,-2,-1,4,-3,-2,3,1,-5,3,-5)=}(6,0,6,0,0,0,0,30,0,0,0,0,0,12,0),$$ indicating that $q_{400}^6q_{301}^6q_{121}^{30}q_{013}^{12}$ is a monomial in the support of $\Lambda$.
Acting on the exponents by $G$ reveals that  
\begin{align*}
q_{040}^6q_{301}^6q_{211}^{30}q_{103}^{12}, \quad
q_{400}^6q_{310}^6&q_{112}^{30}q_{031}^{12}, \quad
q_{040}^6q_{130}^6q_{112}^{30}q_{301}^{12},\\
q_{004}^6q_{013}^6q_{211}^{30}q_{130}^{12},&\quad
q_{004}^6q_{130}^6q_{121}^{30}q_{310}^{12},
\end{align*}
are also monomials of $\Lambda$. 

Figures \ref{fig:sage1}-\ref{fig:sage3} display snapshots from the Sage \cite{sage} animation (produced by Function \ref{queryOracle} for $t=1,4,8,20,30,$ and $75$ respectively) of the paths $\{s_i(t)\}_{i=1}^{54}$ in the HS-algorithm. The second image shows a clustering of $12$ points toward $\zeta_{15}^{13}$. The third image shows a clustering of six points toward $\zeta_{15}^{0}$ and $1$ point converging to $\zeta_{15}^{2}$. Next, a large cluster of $30$ points move toward $\zeta_{15}^{7}$ and five points move toward $\zeta_{15}^{2}$. They converge in the last image.
 \begin{figure}[!htpb]
 \includegraphics[scale=0.42]{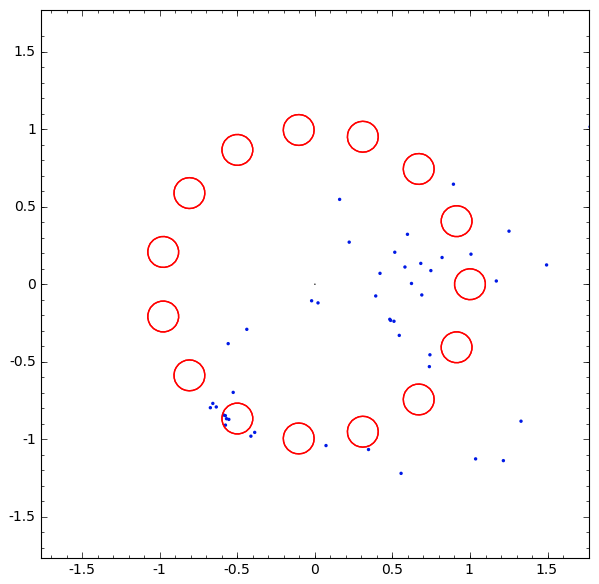} \qquad \qquad
  \includegraphics[scale=0.42]{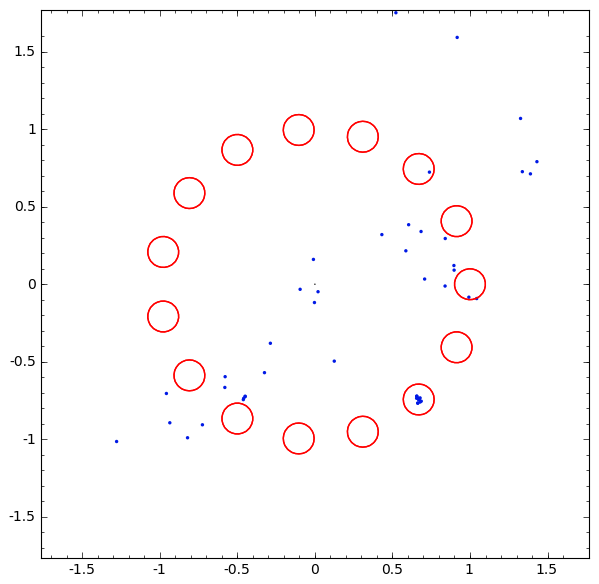}
\caption{(Reprinted from \cite{Bry:NPtrop}) Two snapshots of a Sage animation of the paths $\{s_i(t)\}_{i=1}^{54}$ (for $t=1$ and $t=4$ resp.) of Algorithm \ref{alg:hsalgorithm}.}
\label{fig:sage1}
\end{figure}
\begin{figure}[!htpb]
 \includegraphics[scale=0.42]{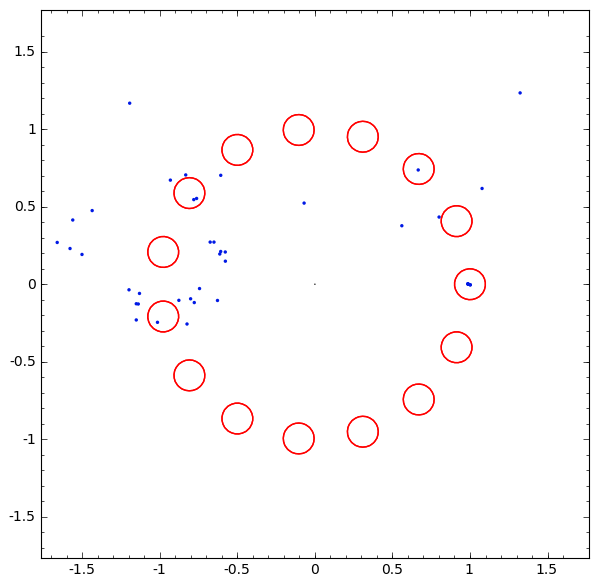}\qquad \qquad
 \includegraphics[scale=0.42]{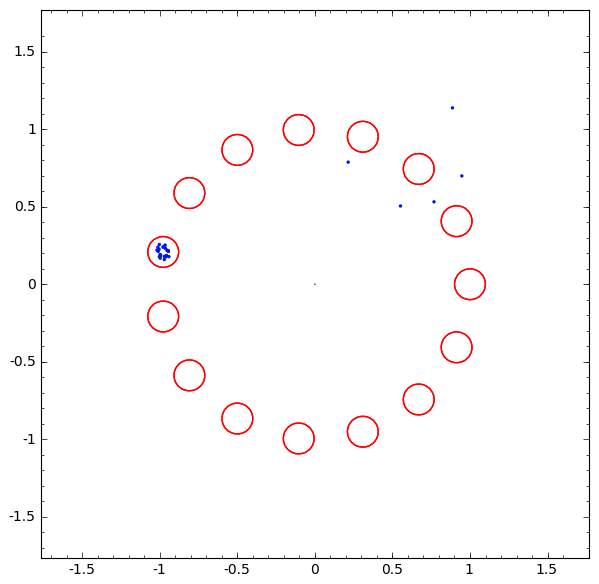}\caption{(Reprinted from \cite{Bry:NPtrop}) Two snapshots of a Sage animation of the paths $\{s_i(t)\}_{i=1}^{54}$ (for $t=8$ and $t=20$ resp.) of Algorithm \ref{alg:hsalgorithm}.}
\label{fig:sage2}
\end{figure}
\begin{figure}[!htpb]
 \includegraphics[scale=0.42]{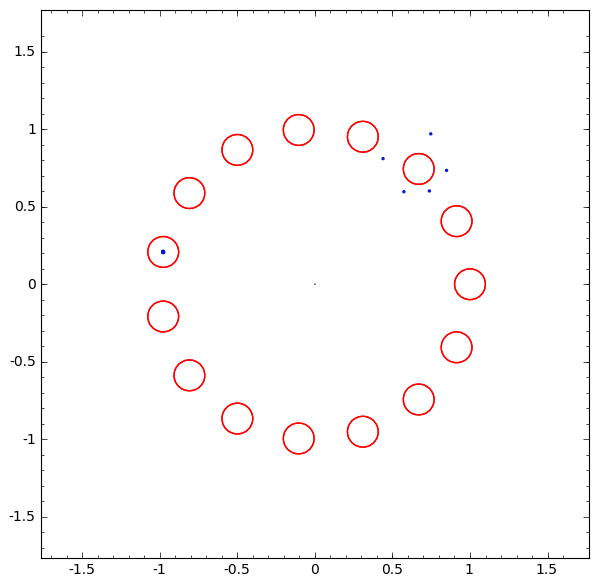}\qquad \qquad
 \includegraphics[scale=0.42]{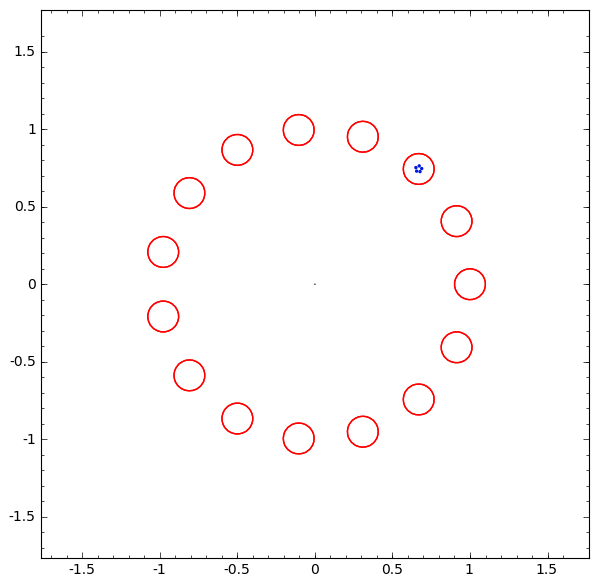}
\caption{(Reprinted from \cite{Bry:NPtrop}) Two snapshots of a Sage animation of the paths $\{s_i(t)\}_{i=1}^{54}$ (for $t=30$ and $t=75$ resp.) of Algorithm \ref{alg:hsalgorithm}.}
\label{fig:sage3}
\end{figure}
As $t \to \infty$, there are two clusters of points which converge to $\zeta_{15}^2$: one of size one and another of size five. This suggests that these paths have winding numbers one and five respectively (see the Cauchy endgame in Section~\ref{subsection:endgames}). We do not have any conjectures about what this means for the polytope $\mathfrak{P}$.
 
Querying the oracle in the coordinate directions with respect to $p_{400}, p_{310}, p_{220}$, and $p_{211}$ returns $18e_{400}, 24e_{310}, 28e_{220}$, and $32e_{211}$ respectively where $e_{ijk}$ is the standard basis vector in the coordinate $p_{ijk}$. This gives new bounds of 
\begin{equation}
\label{eq:newbounds}
p_{400} \in [0,18], \quad 
p_{310} \in [0,24], \quad
p_{220} \in [0,28], \quad
p_{211} \in [0,32]
\end{equation}  improving the bounds \eqref{eq:aprioribounds} for $p_{220}$ and $p_{211}$. 

The initial upper bound on the number of terms in $f$ based on homogeneity and degree is ${{54+15-1}\choose{54}}=123,234,279,768,160$. Taking into account the linear space containing $\mathfrak{P}$ we see that the $p_{400},p_{040}$, and $p_{004}$ coordinates of a point in $\mathfrak{P}$ are determined by the rest. Thus, the number of points in $\Z^{15}$ subject to the bounds of \eqref{eq:newbounds} and the matrix equation of Corollary \ref{cor:lurothresult} is $$[x^{54}]\left( \frac{1-x^{24}}{1-x} \right)^6\left( \frac{1-x^{28}}{1-x} \right)^3\left( \frac{1-x^{32}}{1-x} \right)^3=879,008,719,165$$ which improves the initial bound by a factor of $\approx 140$. Nonetheless, this number remains so large that interpolation is infeasible.

In total, we have found $1713$ vertices, belonging to $1,1,28,$ and $271$ orbits of sizes $1,2,3,$ and $6$ using our implementation \textbf{NumericalNP.m2}. 
The orbits of size one and two which we found are $$\{q_{400}^{18}q_{040}^{18}q_{004}^{18}\}\quad \text{ and } \quad \{q_{301}^{18}q_{130}^{18}q_{013}^{18},q_{103}^{18}q_{031}^{18}q_{310}^{18}\},$$ respectively.  We list the other orbits in Table \ref{table:luroth1} along with an orbit representative, the number of times a representative of each orbit was found in our search, and the number of elements of each orbit. The two orbits colored in blue were found by Hauenstein and Sottile \cite{NP} in their computation of the Newton polytope of the hypersurface  of even L\"uroth quartics which has five vertices, two having a $G$-orbit of size one and three belonging to a $G$-orbit of size three. We did not find the vertex $(4,0,0,14,0,14,0,0,0,0,4,0,14,0,4)$.
Up-to-date computations regarding the L\"uroth polytope can be found at the author's webpage \cite{taylor}. 
 \begin{table}[!htpb]
 \begin{center}
{
 \begin{tabular}{|c|c|c|}
\hline
Vertex $v$ & $\#$  & $|G\cdot v|$ \\ \hline
$(0,0,24,0,0,0,0,0,0,0,18,0,0,0,12)$&$360$&$ 6$\\ \hline
${\color{blue}{(4,0,0,0,0,28,0,0,0,0,18,0,0,0,4)}}$&$126$&$ 3$\\ \hline
$(0,0,0,0,0,25,22,0,0,0,0,0,3,0,4)$&$117$&$ 6$\\ \hline
$(0,0,18,0,0,0,0,0,0,18,18,0,0,0,0)$&$113$&$ 3$\\ \hline
$(0,0,0,0,32,0,0,0,0,8,10,0,0,0,4)$&$106$&$ 6$\\ \hline
$(0,0,24,0,0,0,0,0,0,0,7,0,22,0,1)$&$88$&$ 6$\\ \hline
$(0,0,0,0,0,25,22,0,0,0,0,0,0,6,1)$&$88$&$ 6$\\ \hline
$(0,0,18,0,0,0,18,0,0,0,0,0,0,18,0)$&$78$&$ 2$\\ \hline
$(0,0,0,0,32,0,0,0,0,8,6,0,8,0,0)$&$72$&$ 6$\\ \hline
$(0,0,8,0,0,24,0,0,0,0,18,0,0,0,4)$&$64$&$ 6$\\ \hline
$(0,0,0,0,30,0,0,0,0,12,6,6,0,0,0)$&$64$&$ 6$\\ \hline
$(0,0,0,0,0,24,22,0,0,2,0,0,0,6,0)$&$64$&$ 6$\\ \hline
${\color{blue}{(18,0,0,0,0,0,0,0,0,0,18,0,0,0,18)}}$&$63$&$ 1$\\ \hline
$(4,0,0,0,28,0,0,0,0,0,11,0,0,0,11)$&$60$&$ 3$\\ \hline
$(0,0,0,0,30,0,0,0,0,12,10,0,0,2,0)$&$54$&$ 6$\\ \hline
$(0,0,0,0,0,25,21,1,0,0,0,0,0,7,0)$&$50$&$ 6$\\ \hline
$(0,0,12,0,0,0,22,0,0,14,0,0,0,6,0)$&$49$&$ 6$\\ \hline
$(0,0,12,0,0,0,16,0,0,20,6,0,0,0,0)$&$48$&$ 6$\\ \hline
$(0,0,0,0,30,0,6,0,0,6,0,0,12,0,0)$&$46$&$ 3$\\ \hline
$(0,0,12,0,0,0,24,0,0,12,0,0,0,0,6)$&$45$&$ 6$\\ \hline
$(0,0,0,0,0,24,16,0,0,8,6,0,0,0,0)$&$45$&$ 6$\\ \hline
$(0,0,0,0,0,28,0,16,0,0,10,0,0,0,0)$&$44$&$ 3$\\ \hline
$(0,0,0,0,0,27,0,18,0,0,9,0,0,0,0)$&$43$&$ 3$\\ \hline
$(0,0,8,24,0,0,0,0,0,0,6,0,0,0,16)$&$41$&$ 6$\\ \hline
$(0,0,0,0,0,28,16,0,0,0,6,0,0,0,4)$&$41$&$ 6$\\ \hline
$(0,0,8,0,24,0,0,0,0,0,12,0,0,0,10)$&$40$&$ 6$\\ \hline
$(0,0,0,0,24,6,12,0,0,0,0,0,0,12,0)$&$40$&$ 6$\\ \hline
$(0,0,0,0,29,0,0,0,0,14,10,1,0,0,0)$&$39$&$ 6$\\ \hline
$(0,0,12,0,18,0,0,0,0,0,10,0,0,14,0)$&$37$&$ 6$\\ \hline
$(0,0,12,0,0,0,20,0,0,16,0,0,6,0,0)$&$36$&$ 6$\\ \hline
$(0,0,1,24,0,0,0,0,0,21,0,8,0,0,0)$&$31$&$ 6$\\ \hline
$(0,0,13,0,0,0,24,0,0,9,0,0,0,0,8)$&$29$&$ 6$\\ \hline
$(0,0,0,0,28,1,0,0,0,14,11,0,0,0,0)$&$29$&$ 6$\\ \hline
$(0,0,12,0,0,0,18,0,0,18,0,6,0,0,0)$&$28$&$ 6$\\ \hline
$(0,0,0,0,32,0,0,0,0,8,6,4,0,4,0)$&$27$&$ 6$\\ \hline
$(0,0,24,0,0,0,0,0,0,0,17,0,0,4,9)$&$26$&$ 6$\\ \hline
$(0,0,2,0,0,21,24,0,0,0,0,0,0,0,7)$&$25$&$ 6$\\ \hline
$(0,0,1,0,27,0,15,0,0,0,0,0,0,0,11)$&$25$&$ 6$\\ \hline
$(0,0,14,0,0,0,21,0,9,0,0,0,0,0,10)$&$24$&$ 6$\\ \hline
$(0,0,11,0,9,0,21,0,0,0,0,0,0,0,13)$&$21$&$ 6$\\ \hline
\end{tabular}
}
\end{center}
\caption{Vertices of the L\"uroth polytope found.}
\label{table:luroth1}
\end{table}

\newpage
 \begin{table}[!htpb]
 \begin{center}
{
 \begin{tabular}{|c|c|c|}
\hline
Vertex $v$ & $\#$  & $|G\cdot v|$ \\ \hline
$(0,0,21,0,0,0,0,0,0,9,16,0,0,8,0)$&$18$&$ 6$\\ \hline
$(0,0,18,6,0,0,0,0,0,6,0,20,0,0,4)$&$18$&$ 6$\\ \hline
$(0,0,0,0,30,0,6,0,0,6,0,8,0,0,4)$&$18$&$ 6$\\ \hline
$(0,0,20,0,0,0,12,0,0,0,9,0,0,0,13)$&$17$&$ 6$\\ \hline
$(0,0,12,0,0,0,18,0,18,0,0,0,0,0,6)$&$17$&$ 6$\\ \hline
$(0,0,4,28,0,0,0,0,0,4,4,0,0,0,14)$&$17$&$ 6$\\ \hline
$(3,0,0,0,0,16,0,28,0,0,4,0,0,0,3)$&$16$&$ 3$\\ \hline
$(0,0,17,0,0,0,21,0,0,0,0,0,0,9,7)$&$16$&$ 6$\\ \hline
$(0,0,8,0,24,0,0,0,0,0,6,4,0,12,0)$&$16$&$ 6$\\ \hline
$(0,0,24,0,0,0,0,0,0,0,15,4,0,0,11)$&$15$&$ 6$\\ \hline
$(0,0,8,18,0,0,12,0,0,0,0,0,0,0,16)$&$14$&$ 6$\\ \hline
$(0,0,4,16,0,0,0,0,28,0,3,0,0,0,3)$&$14$&$ 6$\\ \hline
$(0,0,0,4,0,28,0,8,0,0,12,0,0,0,2)$&$14$&$ 6$\\ \hline
$(0,0,4,28,0,0,0,0,4,0,3,0,0,0,15)$&$13$&$ 6$\\ \hline
$(0,0,0,0,20,12,8,0,0,0,0,0,14,0,0)$&$12$&$ 6$\\ \hline
$(0,0,19,0,0,0,6,0,0,9,0,18,0,0,2)$&$11$&$ 6$\\ \hline
$(0,0,0,0,0,24,19,0,0,5,0,3,3,0,0)$&$11$&$ 6$\\ \hline
$(0,0,24,0,0,0,0,0,0,0,6,2,21,0,1)$&$10$&$ 6$\\ \hline
$(0,0,0,3,0,22,22,0,0,0,0,0,0,0,7)$&$10$&$ 6$\\ \hline
$(0,0,0,0,32,0,2,0,0,6,4,4,0,6,0)$&$10$&$ 6$\\ \hline
$(0,12,12,0,0,0,0,0,0,0,1,0,28,0,1)$&$9$&$ 3$\\ \hline
$(0,0,8,0,0,20,0,0,0,8,18,0,0,0,0)$&$9$&$ 3$\\ \hline
$(0,0,0,0,28,0,4,0,0,12,2,8,0,0,0)$&$9$&$ 6$\\ \hline
$(0,0,0,0,0,28,16,0,0,0,0,6,3,0,1)$&$9$&$ 6$\\ \hline
$(0,0,0,0,0,24,22,0,0,2,0,0,3,0,3)$&$9$&$ 6$\\ \hline
$(0,0,23,0,0,0,3,0,0,0,9,9,0,0,10)$&$8$&$ 6$\\ \hline
$(0,0,8,12,12,0,0,0,0,0,0,12,0,0,10)$&$8$&$ 6$\\ \hline
$(0,0,0,0,28,0,0,0,8,8,6,4,0,0,0)$&$8$&$ 6$\\ \hline
$(0,0,0,0,27,0,9,0,0,9,0,0,9,0,0)$&$8$&$ 3$\\ \hline
$(0,0,0,0,16,8,4,0,20,0,6,0,0,0,0)$&$8$&$ 6$\\ \hline
$(0,0,24,0,0,0,0,0,0,0,6,13,0,9,2)$&$7$&$ 6$\\ \hline
$(0,0,17,0,0,0,0,21,0,0,6,0,0,6,4)$&$7$&$ 6$\\ \hline
$(0,0,6,0,24,0,6,0,0,0,4,0,0,14,0)$&$7$&$ 6$\\ \hline
$(0,0,2,20,0,0,3,0,23,0,0,0,0,0,6)$&$7$&$ 6$\\ \hline
$(0,0,1,23,0,0,0,0,0,23,6,0,1,0,0)$&$7$&$ 6$\\ \hline
$(0,0,12,0,0,0,2,24,0,10,4,0,0,2,0)$&$6$&$ 6$\\ \hline
$(0,0,8,0,0,8,8,24,0,0,0,0,0,0,6)$&$6$&$ 6$\\ \hline
$(0,0,0,12,0,20,0,0,8,0,10,0,0,0,4)$&$6$&$ 6$\\ \hline
\end{tabular}
}\\ \nothing \\  \nothing \\ 
\textbf{Table \ref{table:luroth1}} Continued.
\end{center}
\end{table}

\newpage

 \begin{table}[!htpb]
 \begin{center}
{
 \begin{tabular}{|c|c|c|}
\hline
Vertex $v$ & $\#$  & $|G\cdot v|$ \\ \hline
$(0,0,0,2,0,24,20,0,0,0,0,0,0,8,0)$&$6$&$ 6$\\ \hline
$(0,0,0,0,26,2,12,0,0,4,0,0,0,10,0)$&$6$&$ 6$\\ \hline
$(0,0,12,0,0,0,6,18,0,12,0,6,0,0,0)$&$5$&$ 3$\\ \hline
$(0,0,8,0,0,21,6,0,0,0,0,18,0,0,1)$&$5$&$ 6$\\ \hline
$(0,0,6,0,26,0,2,0,0,0,4,4,0,12,0)$&$5$&$ 6$\\ \hline
$(0,0,2,0,24,0,6,0,0,12,0,10,0,0,0)$&$5$&$ 6$\\ \hline
$(0,0,1,0,0,24,21,0,0,0,0,0,1,7,0)$&$5$&$ 6$\\ \hline
$(0,0,0,0,26,6,8,0,0,0,0,4,0,10,0)$&$5$&$ 6$\\ \hline
$(0,0,0,0,26,6,8,0,0,0,0,0,8,6,0)$&$5$&$ 6$\\ \hline
$(0,0,0,0,24,8,0,0,8,0,6,0,8,0,0)$&$5$&$ 6$\\ \hline
$(2,0,0,0,0,19,0,26,0,0,5,0,0,0,2)$&$4$&$ 3$\\ \hline
$(0,0,20,6,0,0,0,0,0,0,6,12,0,0,10)$&$4$&$ 6$\\ \hline
$(0,0,8,6,18,0,0,0,0,0,0,10,0,12,0)$&$4$&$ 6$\\ \hline
$(0,0,6,0,24,0,6,0,0,0,0,6,0,12,0)$&$4$&$ 6$\\ \hline
$(0,0,4,0,0,18,24,0,0,0,0,0,0,0,8)$&$4$&$ 6$\\ \hline
$(0,0,3,28,0,0,0,0,7,0,0,3,0,0,13)$&$4$&$ 3$\\ \hline
$(0,0,0,12,0,20,0,0,0,8,12,0,0,0,2)$&$4$&$ 6$\\ \hline
$(0,0,0,10,0,22,0,8,0,0,0,12,0,0,2)$&$4$&$ 6$\\ \hline
$(0,0,0,4,0,28,8,0,0,0,10,0,0,0,4)$&$4$&$ 6$\\ \hline
$(0,0,0,0,30,0,6,0,0,6,0,6,0,6,0)$&$4$&$ 3$\\ \hline
$(0,0,0,0,28,0,6,0,0,10,0,8,0,2,0)$&$4$&$ 6$\\ \hline
$(0,0,0,0,24,0,6,0,12,6,4,0,0,2,0)$&$4$&$ 6$\\ \hline
$(0,0,0,0,20,12,0,0,8,0,10,0,0,4,0)$&$4$&$ 6$\\ \hline
$(0,0,0,0,0,28,14,2,0,0,0,6,4,0,0)$&$4$&$ 6$\\ \hline
$(0,0,0,0,0,24,16,0,6,2,0,6,0,0,0)$&$4$&$ 6$\\ \hline
$(0,0,0,0,0,24,14,2,8,0,0,6,0,0,0)$&$4$&$ 6$\\ \hline
$(3,0,18,0,0,0,0,0,0,6,3,20,0,0,4)$&$3$&$ 6$\\ \hline
$(3,0,0,0,29,0,0,0,0,2,7,5,0,0,8)$&$3$&$ 6$\\ \hline
$(0,2,10,0,0,0,0,26,0,10,4,0,0,2,0)$&$3$&$ 3$\\ \hline
$(0,0,12,0,0,0,12,0,24,0,0,4,0,0,2)$&$3$&$ 6$\\ \hline
$(0,0,8,6,0,18,0,0,0,0,0,20,0,0,2)$&$3$&$ 6$\\ \hline
$(0,0,6,0,20,0,12,0,2,0,0,0,0,14,0)$&$3$&$ 6$\\ \hline
$(0,0,6,0,0,12,6,24,0,0,0,0,0,6,0)$&$3$&$ 6$\\ \hline
$(0,0,2,0,30,0,6,0,0,0,0,6,0,6,4)$&$3$&$ 6$\\ \hline
$(0,0,2,0,0,27,12,0,0,0,0,12,0,0,1)$&$3$&$ 6$\\ \hline
$(0,0,2,0,0,22,0,22,0,0,7,0,0,0,1)$&$3$&$ 6$\\ \hline
$(0,0,2,0,0,20,18,0,0,8,0,6,0,0,0)$&$3$&$ 6$\\ \hline
$(0,0,1,22,0,1,0,0,0,23,7,0,0,0,0)$&$3$&$ 6$\\ \hline
\end{tabular}
}\\ \nothing \\ \nothing \\ 
\textbf{Table \ref{table:luroth1}} Continued.
\end{center}
\end{table}

\newpage

 \begin{table}[!htpb]
 \begin{center}
{
 \begin{tabular}{|c|c|c|}
\hline
Vertex $v$ & $\#$  & $|G\cdot v|$ \\ \hline
$(0,0,1,21,2,0,0,0,0,23,7,0,0,0,0)$&$3$&$ 6$\\ \hline
$(0,0,0,4,0,21,0,22,0,0,5,0,0,0,2)$&$3$&$ 6$\\ \hline
$(0,0,0,3,0,28,0,10,0,0,11,0,0,2,0)$&$3$&$ 6$\\ \hline
$(0,0,0,0,32,0,2,0,0,6,4,0,8,2,0)$&$3$&$ 6$\\ \hline
$(0,0,0,0,29,0,0,1,0,13,10,0,0,1,0)$&$3$&$ 6$\\ \hline
$(0,0,0,0,28,0,6,0,0,10,0,6,4,0,0)$&$3$&$ 6$\\ \hline
$(0,0,0,0,24,0,2,10,0,12,4,2,0,0,0)$&$3$&$ 6$\\ \hline
$(0,0,0,0,18,6,4,0,20,0,4,2,0,0,0)$&$3$&$ 6$\\ \hline
$(0,0,0,0,8,16,8,0,16,0,6,0,0,0,0)$&$3$&$ 6$\\ \hline
$(0,0,0,0,6,22,14,2,0,0,0,0,10,0,0)$&$3$&$ 6$\\ \hline
$(0,0,0,0,4,22,0,20,0,0,7,0,0,0,1)$&$3$&$ 6$\\ \hline
$(0,0,0,0,3,21,19,0,0,5,0,0,6,0,0)$&$3$&$ 6$\\ \hline
$(0,0,0,0,0,28,10,6,0,0,0,10,0,0,0)$&$3$&$ 3$\\ \hline
$(0,0,0,0,0,28,8,8,0,0,6,0,4,0,0)$&$3$&$ 6$\\ \hline
$(0,0,0,0,0,25,19,0,3,0,3,0,0,0,4)$&$3$&$ 6$\\ \hline
$(0,0,0,0,0,24,18,6,0,0,0,0,0,6,0)$&$3$&$ 6$\\ \hline
$(0,3,18,0,0,0,3,0,0,6,0,20,0,0,4)$&$2$&$ 6$\\ \hline
$(0,0,23,0,0,0,3,0,0,0,3,9,12,0,4)$&$2$&$ 6$\\ \hline
$(0,0,14,0,0,0,0,21,0,9,6,2,0,0,2)$&$2$&$ 6$\\ \hline
$(0,0,12,7,0,0,2,20,0,0,3,0,0,0,10)$&$2$&$ 6$\\ \hline
$(0,0,12,7,0,0,0,22,0,0,3,0,0,2,8)$&$2$&$ 6$\\ \hline
$(0,0,12,6,0,12,0,0,0,0,0,20,0,0,4)$&$2$&$ 6$\\ \hline
$(0,0,12,3,0,0,0,26,0,4,3,0,0,2,4)$&$2$&$ 6$\\ \hline
$(0,0,12,0,0,0,16,0,18,2,0,0,0,6,0)$&$2$&$ 6$\\ \hline
$(0,0,12,0,0,0,6,24,0,6,0,2,0,0,4)$&$2$&$ 6$\\ \hline
$(0,0,9,4,2,0,0,22,0,11,2,2,2,0,0)$&$2$&$ 6$\\ \hline
$(0,0,8,0,20,4,0,0,0,0,10,0,0,12,0)$&$2$&$ 6$\\ \hline
$(0,0,8,0,20,0,0,8,0,0,6,0,0,12,0)$&$2$&$ 6$\\ \hline
$(0,0,6,24,0,0,0,0,0,6,0,8,0,0,10)$&$2$&$ 6$\\ \hline
$(0,0,5,27,0,0,0,0,0,3,3,2,0,0,14)$&$2$&$ 6$\\ \hline
$(0,0,4,27,0,0,0,0,6,0,0,4,0,0,13)$&$2$&$ 3$\\ \hline
$(0,0,4,24,4,0,0,4,0,0,3,0,0,0,15)$&$2$&$ 6$\\ \hline
$(0,0,4,24,0,4,4,0,0,0,3,0,0,0,15)$&$2$&$ 6$\\ \hline
$(0,0,4,24,0,0,8,0,0,4,0,0,0,0,14)$&$2$&$ 6$\\ \hline
$(0,0,1,0,27,0,0,0,1,14,11,0,0,0,0)$&$2$&$ 6$\\ \hline
$(0,0,0,12,0,18,0,0,0,12,12,0,0,0,0)$&$2$&$ 6$\\ \hline
$(0,0,0,10,0,22,4,4,0,0,0,4,10,0,0)$&$2$&$ 3$\\ \hline
\end{tabular}
}\\ \nothing \\  \nothing \\ 
\textbf{Table \ref{table:luroth1}} Continued.
\end{center}
\end{table}

\newpage

 \begin{table}[!htpb]
 \begin{center}
{
 \begin{tabular}{|c|c|c|}
\hline
Vertex $v$ & $\#$  & $|G\cdot v|$ \\ \hline
$(0,0,0,10,0,20,12,0,0,0,0,0,8,0,4)$&$2$&$ 6$\\ \hline
$(0,0,0,6,18,8,0,0,8,0,0,6,8,0,0)$&$2$&$ 6$\\ \hline
$(0,0,0,6,0,26,8,0,0,0,0,8,6,0,0)$&$2$&$ 3$\\ \hline
$(0,0,0,4,0,23,0,18,0,0,5,0,4,0,0)$&$2$&$ 3$\\ \hline
$(0,0,0,2,8,20,12,0,0,0,0,0,12,0,0)$&$2$&$ 6$\\ \hline
$(0,0,0,0,32,0,4,0,0,4,2,0,10,0,2)$&$2$&$ 3$\\ \hline
$(0,0,0,0,28,4,0,0,0,8,10,0,0,4,0)$&$2$&$ 6$\\ \hline
$(0,0,0,0,28,1,14,0,0,0,0,0,1,0,10)$&$2$&$ 6$\\ \hline
$(0,0,0,0,26,6,6,0,2,0,0,6,0,8,0)$&$2$&$ 6$\\ \hline
$(0,0,0,0,26,2,6,0,10,0,0,6,0,0,4)$&$2$&$ 6$\\ \hline
$(0,0,0,0,26,0,2,0,12,6,4,4,0,0,0)$&$2$&$ 6$\\ \hline
$(0,0,0,0,26,0,0,10,0,10,6,0,0,2,0)$&$2$&$ 6$\\ \hline
$(0,0,0,0,24,6,0,0,0,12,12,0,0,0,0)$&$2$&$ 6$\\ \hline
$(0,0,0,0,24,0,4,0,14,6,4,2,0,0,0)$&$2$&$ 6$\\ \hline
$(0,0,0,0,24,0,2,12,0,10,4,0,0,2,0)$&$2$&$ 6$\\ \hline
$(0,0,0,0,22,10,8,0,0,0,4,0,0,10,0)$&$2$&$ 6$\\ \hline
$(0,0,0,0,20,4,8,0,16,0,0,4,0,0,2)$&$2$&$ 6$\\ \hline
$(0,0,0,0,16,16,0,8,0,0,6,0,8,0,0)$&$2$&$ 6$\\ \hline
$(0,0,0,0,16,11,0,2,16,0,9,0,0,0,0)$&$2$&$ 3$\\ \hline
$(0,0,0,0,12,16,0,16,0,0,6,0,0,4,0)$&$2$&$ 6$\\ \hline
$(0,0,0,0,2,22,20,0,4,0,0,0,0,6,0)$&$2$&$ 6$\\ \hline
$(0,0,0,0,0,28,15,1,0,0,0,6,3,1,0)$&$2$&$ 6$\\ \hline
$(0,0,0,0,0,24,8,8,8,0,6,0,0,0,0)$&$2$&$ 6$\\ \hline
$(3,0,4,0,0,10,0,28,0,0,4,0,0,0,5)$&$1$&$ 6$\\ \hline
$(3,0,2,0,27,0,0,0,0,0,10,0,0,5,7)$&$1$&$ 6$\\ \hline
$(3,0,2,0,27,0,0,0,0,0,3,9,0,6,4)$&$1$&$ 6$\\ \hline
$(3,0,0,0,25,0,0,0,8,2,3,9,0,0,4)$&$1$&$ 6$\\ \hline
$(2,0,4,26,0,0,0,0,0,0,4,0,2,0,16)$&$1$&$ 6$\\ \hline
$(2,0,0,0,30,0,0,0,0,4,4,7,0,5,2)$&$1$&$ 6$\\ \hline
$(1,0,0,4,4,13,0,26,0,0,2,0,0,0,4)$&$1$&$ 6$\\ \hline
$(0,9,11,0,0,0,12,0,0,0,0,9,0,0,13)$&$1$&$ 6$\\ \hline
$(0,4,4,0,24,0,0,0,0,0,3,0,16,0,3)$&$1$&$ 3$\\ \hline
$(0,3,20,0,0,0,3,0,0,0,6,12,0,0,10)$&$1$&$ 6$\\ \hline
$(0,2,9,2,0,0,0,28,0,7,2,0,0,2,2)$&$1$&$ 6$\\ \hline
$(0,0,20,0,6,0,0,0,0,0,12,6,0,0,10)$&$1$&$ 6$\\ \hline
$(0,0,20,0,6,0,0,0,0,0,6,0,21,0,1)$&$1$&$ 6$\\ \hline
$(0,0,20,0,3,0,6,0,0,0,0,15,0,6,4)$&$1$&$ 6$\\ \hline
$(0,0,20,0,0,0,0,12,0,0,11,0,2,0,9)$&$1$&$ 6$\\ \hline
$(0,0,19,0,2,0,0,11,0,0,12,0,0,0,10)$&$1$&$ 6$\\ \hline
\end{tabular}
}\\ \nothing \\  \nothing \\ 
\textbf{Table \ref{table:luroth1}} Continued.
\end{center}
\end{table}

\newpage

 \begin{table}[!htpb]
 \begin{center}
{
 \begin{tabular}{|c|c|c|}
\hline
Vertex $v$ & $\#$  & $|G\cdot v|$ \\ \hline
$(0,0,18,0,0,2,0,14,0,0,11,0,0,0,9)$&$1$&$ 6$\\ \hline
$(0,0,17,0,0,0,12,0,9,0,0,9,0,0,7)$&$1$&$ 6$\\ \hline
$(0,0,16,9,0,0,0,0,6,0,0,16,0,0,7)$&$1$&$ 3$\\ \hline
$(0,0,15,0,0,0,0,23,0,4,6,0,0,2,4)$&$1$&$ 6$\\ \hline
$(0,0,15,0,0,0,0,21,0,6,6,2,0,0,4)$&$1$&$ 6$\\ \hline
$(0,0,14,0,9,0,12,0,0,0,0,9,0,0,10)$&$1$&$ 6$\\ \hline
$(0,0,14,0,2,0,0,16,10,0,7,0,0,0,5)$&$1$&$ 6$\\ \hline
$(0,0,12,18,0,0,0,0,0,0,4,0,0,20,0)$&$1$&$ 6$\\ \hline
$(0,0,12,9,0,0,0,18,0,0,3,2,0,0,10)$&$1$&$ 6$\\ \hline
$(0,0,12,3,0,0,0,24,2,4,3,0,2,0,4)$&$1$&$ 6$\\ \hline
$(0,0,12,0,9,0,6,0,12,0,0,11,0,0,4)$&$1$&$ 6$\\ \hline
$(0,0,12,0,2,0,0,22,0,10,6,0,0,2,0)$&$1$&$ 6$\\ \hline
$(0,0,12,0,0,2,0,23,0,9,6,0,0,2,0)$&$1$&$ 6$\\ \hline
$(0,0,12,0,0,0,16,0,20,0,0,0,0,4,2)$&$1$&$ 6$\\ \hline
$(0,0,12,0,0,0,4,18,0,14,6,0,0,0,0)$&$1$&$ 6$\\ \hline
$(0,0,12,0,0,0,2,20,2,12,6,0,0,0,0)$&$1$&$ 6$\\ \hline
$(0,0,11,3,0,2,0,25,0,4,4,0,0,0,5)$&$1$&$ 6$\\ \hline
$(0,0,11,2,0,0,2,26,0,7,2,0,0,2,2)$&$1$&$ 6$\\ \hline
$(0,0,10,9,2,0,0,20,0,0,3,0,0,0,10)$&$1$&$ 6$\\ \hline
$(0,0,10,5,0,0,0,17,0,15,7,0,0,0,0)$&$1$&$ 6$\\ \hline
$(0,0,10,0,14,0,14,0,0,0,0,0,0,16,0)$&$1$&$ 6$\\ \hline
$(0,0,9,7,0,0,0,21,0,10,4,0,0,0,3)$&$1$&$ 6$\\ \hline
$(0,0,9,7,0,0,0,18,0,13,4,0,3,0,0)$&$1$&$ 6$\\ \hline
$(0,0,9,4,2,0,0,26,0,7,2,0,0,2,2)$&$1$&$ 6$\\ \hline
$(0,0,8,24,0,0,0,0,0,0,1,0,0,20,1)$&$1$&$ 6$\\ \hline
$(0,0,8,3,0,8,0,26,0,0,3,0,0,2,4)$&$1$&$ 6$\\ \hline
$(0,0,8,0,24,0,0,0,0,0,10,0,0,8,4)$&$1$&$ 6$\\ \hline
$(0,0,8,0,24,0,0,0,0,0,8,0,8,0,6)$&$1$&$ 6$\\ \hline
$(0,0,6,6,20,0,0,2,0,0,0,8,0,12,0)$&$1$&$ 6$\\ \hline
$(0,0,6,4,22,0,2,0,0,0,0,8,0,12,0)$&$1$&$ 6$\\ \hline
$(0,0,6,0,22,0,2,8,0,0,4,0,0,12,0)$&$1$&$ 6$\\ \hline
$(0,0,4,25,0,0,6,0,4,0,0,0,0,0,15)$&$1$&$ 6$\\ \hline
$(0,0,4,24,0,4,0,4,0,0,4,0,0,0,14)$&$1$&$ 6$\\ \hline
$(0,0,4,21,4,0,6,4,0,0,0,0,0,0,15)$&$1$&$ 6$\\ \hline
$(0,0,4,21,0,4,10,0,0,0,0,0,0,0,15)$&$1$&$ 6$\\ \hline
$(0,0,4,16,0,0,0,0,25,3,3,0,0,3,0)$&$1$&$ 6$\\ \hline
$(0,0,3,20,0,0,0,0,0,23,8,0,0,0,0)$&$1$&$ 6$\\ \hline
$(0,0,3,18,0,0,8,0,0,19,0,0,6,0,0)$&$1$&$ 6$\\ \hline
\end{tabular}
}\\ \nothing \\  \nothing \\ 
\textbf{Table \ref{table:luroth1}} Continued.
\end{center}
\end{table}

\newpage

 \begin{table}[!htpb]
 \begin{center}
{
 \begin{tabular}{|c|c|c|}
\hline
Vertex $v$ & $\#$  & $|G\cdot v|$ \\ \hline
$(0,0,3,7,0,22,0,5,0,0,0,16,0,0,1)$&$1$&$ 6$\\ \hline
$(0,0,2,0,26,4,6,0,0,0,0,6,0,10,0)$&$1$&$ 6$\\ \hline
$(0,0,2,0,0,20,24,0,0,2,0,0,0,0,6)$&$1$&$ 6$\\ \hline
$(0,0,1,28,0,0,0,0,1,12,1,0,0,11,0)$&$1$&$ 6$\\ \hline
$(0,0,1,27,0,0,2,0,1,12,0,0,0,11,0)$&$1$&$ 6$\\ \hline
$(0,0,1,23,0,0,0,0,0,23,5,2,0,0,0)$&$1$&$ 6$\\ \hline
$(0,0,1,22,0,0,2,0,0,23,5,0,1,0,0)$&$1$&$ 6$\\ \hline
$(0,0,1,0,0,27,15,0,0,0,0,6,4,1,0)$&$1$&$ 6$\\ \hline
$(0,0,0,14,0,18,8,0,0,0,0,0,8,4,2)$&$1$&$ 6$\\ \hline
$(0,0,0,14,0,18,4,4,0,0,0,0,10,4,0)$&$1$&$ 6$\\ \hline
$(0,0,0,14,0,18,4,0,0,4,0,4,10,0,0)$&$1$&$ 6$\\ \hline
$(0,0,0,14,0,17,0,10,0,0,0,8,0,0,5)$&$1$&$ 6$\\ \hline
$(0,0,0,12,0,18,12,0,0,0,0,0,0,12,0)$&$1$&$ 6$\\ \hline
$(0,0,0,10,4,18,4,4,0,0,0,0,14,0,0)$&$1$&$ 6$\\ \hline
$(0,0,0,10,0,20,8,4,0,0,0,0,10,0,2)$&$1$&$ 6$\\ \hline
$(0,0,0,10,0,19,0,14,0,0,0,8,0,0,3)$&$1$&$ 6$\\ \hline
$(0,0,0,9,0,21,0,12,0,0,0,10,0,0,2)$&$1$&$ 6$\\ \hline
$(0,0,0,8,0,22,4,8,0,0,2,0,10,0,0)$&$1$&$ 6$\\ \hline
$(0,0,0,8,0,22,0,12,0,0,4,0,8,0,0)$&$1$&$ 3$\\ \hline
$(0,0,0,6,20,6,0,2,6,0,0,8,0,6,0)$&$1$&$ 6$\\ \hline
$(0,0,0,6,8,18,8,0,0,0,0,0,14,0,0)$&$1$&$ 6$\\ \hline
$(0,0,0,6,0,26,8,0,0,0,0,12,0,0,2)$&$1$&$ 6$\\ \hline
$(0,0,0,6,0,19,0,22,0,0,4,0,0,0,3)$&$1$&$ 6$\\ \hline
$(0,0,0,6,0,19,0,22,0,0,3,0,0,4,0)$&$1$&$ 6$\\ \hline
$(0,0,0,4,0,28,8,0,0,0,2,8,4,0,0)$&$1$&$ 3$\\ \hline
$(0,0,0,1,0,25,20,0,0,0,0,1,0,7,0)$&$1$&$ 6$\\ \hline
$(0,0,0,0,29,0,1,0,0,13,10,0,0,0,1)$&$1$&$ 6$\\ \hline
$(0,0,0,0,28,4,6,0,0,2,0,6,0,8,0)$&$1$&$ 6$\\ \hline
$(0,0,0,0,28,0,2,8,0,6,4,0,0,6,0)$&$1$&$ 6$\\ \hline
$(0,0,0,0,28,0,0,8,0,8,6,0,0,4,0)$&$1$&$ 6$\\ \hline
$(0,0,0,0,28,0,0,6,0,10,6,0,4,0,0)$&$1$&$ 6$\\ \hline
$(0,0,0,0,27,0,7,0,0,11,6,0,0,0,3)$&$1$&$ 6$\\ \hline
$(0,0,0,0,27,0,5,0,0,13,3,6,0,0,0)$&$1$&$ 6$\\ \hline
$(0,0,0,0,26,6,2,0,6,0,4,4,0,6,0)$&$1$&$ 6$\\ \hline
$(0,0,0,0,26,2,0,0,6,10,10,0,0,0,0)$&$1$&$ 6$\\ \hline
$(0,0,0,0,26,0,6,4,4,6,0,0,8,0,0)$&$1$&$ 3$\\ \hline
$(0,0,0,0,26,0,0,8,2,10,6,0,2,0,0)$&$1$&$ 6$\\ \hline
\end{tabular}
}\\ \nothing \\  \nothing \\ 
\textbf{Table \ref{table:luroth1}} Continued.
\end{center}
\end{table}

\newpage

 \begin{table}[!htpb]
 \begin{center}
{
 \begin{tabular}{|c|c|c|}
\hline
Vertex $v$ & $\#$  & $|G\cdot v|$ \\ \hline
$(0,0,0,0,26,0,0,8,0,12,6,2,0,0,0)$&$1$&$ 6$\\ \hline
$(0,0,0,0,24,8,0,0,0,8,6,8,0,0,0)$&$1$&$ 6$\\ \hline
$(0,0,0,0,24,6,6,0,6,0,0,0,12,0,0)$&$1$&$ 6$\\ \hline
$(0,0,0,0,24,2,0,8,0,12,8,0,0,0,0)$&$1$&$ 6$\\ \hline
$(0,0,0,0,24,0,7,0,6,11,3,3,0,0,0)$&$1$&$ 6$\\ \hline
$(0,0,0,0,24,0,4,8,0,12,5,0,0,0,1)$&$1$&$ 6$\\ \hline
$(0,0,0,0,24,0,4,8,0,12,2,4,0,0,0)$&$1$&$ 6$\\ \hline
$(0,0,0,0,22,4,6,12,0,2,0,0,0,8,0)$&$1$&$ 6$\\ \hline
$(0,0,0,0,22,2,10,0,14,0,0,2,0,0,4)$&$1$&$ 6$\\ \hline
$(0,0,0,0,20,4,12,0,12,0,0,0,0,4,2)$&$1$&$ 6$\\ \hline
$(0,0,0,0,20,4,6,14,4,0,0,0,0,2,4)$&$1$&$ 6$\\ \hline
$(0,0,0,0,18,8,6,14,0,0,0,0,0,8,0)$&$1$&$ 6$\\ \hline
$(0,0,0,0,18,6,14,0,0,10,0,0,6,0,0)$&$1$&$ 6$\\ \hline
$(0,0,0,0,16,16,8,0,0,0,2,0,12,0,0)$&$1$&$ 6$\\ \hline
$(0,0,0,0,16,16,0,8,0,0,10,0,0,0,4)$&$1$&$ 6$\\ \hline
$(0,0,0,0,16,8,10,0,14,0,0,0,6,0,0)$&$1$&$ 6$\\ \hline
$(0,0,0,0,14,10,7,17,0,0,0,1,0,0,5)$&$1$&$ 6$\\ \hline
$(0,0,0,0,12,20,6,0,0,2,0,14,0,0,0)$&$1$&$ 6$\\ \hline
$(0,0,0,0,12,20,0,8,0,0,10,0,0,4,0)$&$1$&$ 6$\\ \hline
$(0,0,0,0,12,12,6,18,0,0,0,0,0,6,0)$&$1$&$ 6$\\ \hline
$(0,0,0,0,8,20,14,0,2,0,0,0,10,0,0)$&$1$&$ 6$\\ \hline
$(0,0,0,0,6,21,12,6,0,0,0,0,9,0,0)$&$1$&$ 6$\\ \hline
$(0,0,0,0,2,22,14,0,10,0,0,6,0,0,0)$&$1$&$ 6$\\ \hline
$(0,0,0,0,1,24,21,0,1,0,0,0,0,7,0)$&$1$&$ 6$\\ \hline
$(0,0,0,0,0,28,15,1,0,0,3,3,0,4,0)$&$1$&$ 6$\\ \hline
$(0,0,0,0,0,27,9,9,0,0,0,9,0,0,0)$&$1$&$ 3$\\ \hline
$(0,0,0,0,0,25,20,2,0,0,0,0,4,0,3)$&$1$&$ 6$\\ \hline
$(0,0,0,0,0,25,19,3,0,0,0,3,0,0,4)$&$1$&$ 6$\\ \hline
$(0,0,0,0,0,25,16,0,6,0,0,6,0,0,1)$&$1$&$ 6$\\ \hline
$(0,0,0,0,0,24,20,2,2,0,0,0,0,6,0)$&$1$&$ 6$\\ \hline
$(0,0,0,0,0,24,18,6,0,0,0,0,3,0,3)$&$1$&$ 6$\\ \hline
$(0,0,0,0,0,24,16,3,0,5,0,6,0,0,0)$&$1$&$ 6$\\ \hline
$(0,0,0,0,0,24,15,9,0,0,0,3,0,0,3)$&$1$&$ 6$\\ \hline
$(0,0,0,0,0,24,12,9,0,3,0,6,0,0,0)$&$1$&$ 6$ \\ \hline
\end{tabular}
}\\ \nothing \\  \nothing \\ 
\textbf{Table \ref{table:luroth1}} Continued.
\end{center}
\end{table}

\chapter{SOLVING SPARSE DECOMPOSABLE SYSTEMS \label{section:solvingsparsedecomposablesystems}}
We describe how to use a numerical homotopy to solve polynomial systems corresponding to fibers of decomposable branched covers. Recall that a branched cover $\pi\colon X \to Z$ is decomposable if there is a dense open subset $V \subset Z$ over which $\pi$ factors as
 \begin{equation}\label{Eq:decomposition}
  \pi^{-1}(V)\ \longrightarrow\  Y\ \longrightarrow\  V
 \end{equation}
 with $\varphi$ and $\psi$ both nontrivial branched covers. As discussed in Section~\ref{subsection:decomposablebranchedcovers}, a result of Pirola and Schlesinger~\cite{PirolaSchlesinger} states that the Galois group $G_\pi$ acts imprimitively if and only if $\pi$ is decomposable.

Am\'endola and Rodriguez ~\cite{AmendolaRodriguez} explained how to use an explicit decomposition to compute fibers
$\pi^{-1}(z)$ using monodromy.
They also showed how several examples from the literature involve a decomposable branched cover; for these, the variety $Y$
and intermediate maps were determined using invariant theory as there was a finite group acting as automorphisms of
$\pi\colon X\to Z$.
In general, it is  nontrivial to determine a decomposition~\eqref{Eq:decomposition} of a branched cover $\pi\colon X\to Z$
with imprimitive Galois group, especially when the cover has trivial automorphism group. 

Esterov~\cite{Esterov} determined which systems of sparse polynomials have an imprimitive Galois group.
One goal was to classify those which are solvable by radicals.
He identified two simple structures which imply that the system is decomposable. In these cases, the decomposition is transparent.
He also showed that the Galois group is full symmetric when neither structure occurs.
We use Esterov's classification to give a recursive numerical homotopy continuation algorithm for solving decomposable
sparse systems.

The first such structure is when a polynomial system is composed with a monomial map. For example, if $f(x)=g(x^3)$ then to solve $f(x)=0$, first solve $g(y)=0$ and then for each solution $y$, extract its third roots.
The second structure is when the system is triangular, such as
 \[
     f(x,y)\ =\ g(y)\ =\ 0\,.
 \]
To solve this, first  solve $g(y)=0$ and then for each solution $y$, solve $f(x,y)=0$.

In general, Esterov's classification leads to a sequence of branched covers, each corresponding to 
a sparse system with symmetric monodromy or to a monomial map.
Our algorithm identifies this structure and uses it to recursively solve a decomposable system.
We give some examples which demonstrate that, despite its overhead, this algorithm is a significant improvement over a
direct use of the polyhedral homotopy (Algorithm \ref{alg:polyhedralhomotopymethod}).

By the Bernstein-Kushnirenko Theorem (Proposition \ref{prop:BKK}), a general system of sparse polynomials has the same
number of solutions as a system whose supports have the same convex hull.
When the system supported on the vertices is decomposable, we propose using it as a start system in a parameter
homotopy to solve the original system.
This is similar in spirit to the B\'ezout homotopy (Algorithm \ref{alg:bezouthomotopymethod}).

We remind the reader of the general background we developed in Section~\ref{section:branchedcoversandgroups} on Galois groups of branched covers as well as our explanation of the  relation
between decompositions of branched covers and imprimitivity of the corresponding Galois groups.

In Section~\ref{SS:Decompositions}, we explain Esterov's classification and describe
how to compute the corresponding decompositions in Section~\ref{S:SNF}.
We present our algorithms for solving sparse decomposable systems in Section~\ref{S:Algorithm}, and give an application to
furnish start systems for parameter homotopies.
Section~\ref{S:computations} gives timings and information on the performance of our algorithm.
Much of the material in this section appears in the paper of the same name \cite{BRSY:DecSpaSys} with Rodriguez, Sottile, and Yahl.
\section{Decompositions of sparse polynomial systems}\label{SS:Decompositions}
Let $\Adot=(\mathcal A_1, \mathcal A_2,\ldots, \mathcal A_n)$ be a collection of supports $\mathcal A_i \subset \Z^n$. 
We describe two properties that a collection $\Adot$ may have, lacunary and (strictly) triangular, and then
state Esterov's theorem about the Galois group $G_\Adot$.
We then present explicit decompositions of the projection $\pi\colon X_{\Adot}\to \CC^{\Adot}$ when $\Adot$ is lacunary
and when $\Adot$ is triangular.
These form the basis for our algorithms.

Assume that $\MV(\Adot)>1$.
We say that $\Adot$ is \mydef{lacunary} if $\ZZ\Adot\neq\ZZ^n$ (it has rank $n$ as $\MV(\Adot)\neq 0$).
We say that $\Adot$ is \mydef{triangular} if there is a nonempty proper subset $\emptyset\neq I\subsetneq[n]$ such that
$\rank(\ZZ\calA_I)=|I|$, or equivalently, the defect of the collection of polytopes $\{\conv(\mathcal A_i)\}_{i \in I}$ is zero.
As we explain in Section~\ref{S:SNF}, we may change coordinates and assume that $\ZZ\calA_I\subset\ZZ^{|I|}$ so that
$\MV(\calA_I)$ is defined using $\conv(\calA_i)\subset\RR^{|I|}$ for $i\in I$.
A system $\Adot$ of triangular supports is \mydef{strictly triangular} if for some $\emptyset\neq I\subsetneq[n]$ with 
$\rank(\ZZ\calA_I)=|I|$, we have $1<\MV(\calA_I)<\MV(\Adot)$.
It is elementary that if  $\Adot$ is either lacunary or strictly triangular, then the
branched cover $X_\Adot\to\CC^\Adot$ is decomposable and therefore $G_\Adot$ is an imprimitive permutation group. We do this explicitly in Sections \ref{SSS:Lacunary} and \ref{SSS:triangular}. 
  
\begin{proposition}[Esterov~\cite{Esterov}]
  Let $\Adot$ be a collection of supports with $\MV(\Adot)\neq 0$.
  The Galois group $G_\Adot$ is equal to the symmetric group $S_{\MV(\Adot)}$ if and only if $\Adot$ is neither lacunary
  nor strictly triangular.
\end{proposition}

\subsection{Lacunary support}\label{SSS:Lacunary}

Let us begin with an example when $n=2$.
Let
\[
  \calA_1\ =\ \left(\begin{matrix} 0&0&3&6&12\\0&4&3&6&0 \end{matrix}\right)
   \qquad\mbox{and}\qquad
  \calA_2\ =\ \left(\begin{matrix} 0&3&6&9&9\\0&7&2&1&5 \end{matrix}\right)
\]
be supports in $\ZZ^2$.
Then $\ZZ\Adot$ has index 12 

in $\ZZ^2$ as the map $\varphi(a,b)^T=(3a,4b-a)^T$ is an isomorphism
$\varphi\colon\ZZ^2\xrightarrow{\sim}\ZZ\Adot$, and $\det(\begin{smallmatrix}3&0\\-1&4\end{smallmatrix})=12$.
If we set $\mydefMATH{\calB_i}=\varphi^{-1}(\calA_i)$, then 
\[
  \calB_1\ =\ \left(\begin{matrix} 0&0&1&2&4\\0&1&1&2&1 \end{matrix}\right)
   \qquad\mbox{and}\qquad
  \calB_2\ =\ \left(\begin{matrix} 0&1&2&3&3\\0&2&1&1&2 \end{matrix}\right)\ .
\]
We display $\calA_1$, $\calA_2$, $\calB_1$, and $\calB_2$ in Figure \ref{fig:sdslacunary}.
\begin{figure}[!htpb]
\includegraphics[scale=0.5]{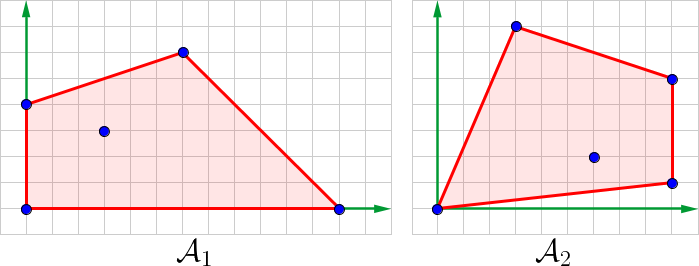} \hspace{0.2 in}
\includegraphics[scale=0.5]{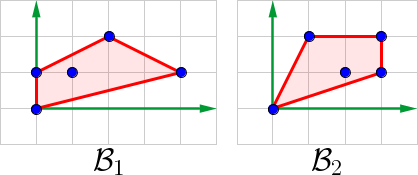}
\caption{\label{fig:sdslacunary}The lacunary image $(\mathcal A_1,\mathcal A_2)$ of the support $(\mathcal B_1,\mathcal B_2)$ under the map $\varphi$.}
\end{figure}
Then the map $\Phi=\varphi^*\colon(\CC^\times)^2\twoheadrightarrow(\CC^\times)^2$ is given by
$\Phi(x,y)=(x^3y^{-1},y^4)=(z,w)$.
If 
\begin{align*}
  f_1\ &=\ 1+2y^4+4x^3y^3+8x^6y^6+16x^{12}\\
  f_2\ &=\ 3+5x^3y^7+7x^6y^2+11x^9y+13x^9y^5\ ,
\end{align*}
which is a polynomial system with support $\Adot$, then $f_i=g_i\circ\Phi$, where  
 \begin{align*}
  g_1\ &=\ 1+2w+4zw+8z^2w^2+16z^4w\\
  g_2\ &=\ 3+5zw^2+7z^2w+11z^3w+13z^3w^2\,,
\end{align*}
is a polynomial system with support $\Bdot$. 
Therefore, the branched cover $X_{\Adot}\to\CC^{\Adot}$ factors as $X_{\Adot}\to X_{\Bdot}\to\CC^{\Bdot}=\CC^{\Adot}$
with the map $X_{\Adot}\to X_{\Bdot}$ induced by $\Phi$.
Consequently, this implies that $G_{\Adot}\subset (\ZZ/12\ZZ)^{10}\rtimes S_{10}$, as $\ZZ^2/\ZZ\Adot\simeq\ZZ/12\ZZ$, $\Bdot$ is neither
lacunary nor triangular, and $\MV(\Bdot)=10$.

We generalize this example.
Suppose  that $\Adot=(\calA_1,\dotsc,\calA_n)$ is lacunary.
Then $\ZZ\Adot$ has rank $n$ but $\ZZ\Adot\neq\ZZ^n$.
Let $\varphi\colon\ZZ^n\xrightarrow{\,\sim\,}\ZZ\Adot$ be an isomorphism.
Then the corresponding map $\Phi=\varphi^*\colon(\CC^\times)^n\to(\CC^\times)^n$ 
is a surjection with kernel $\Hom(\ZZ^n/\ZZ\Adot,\CC^\times)$.
For each $i=1,\dotsc,n$, set $\mydefMATH{\calB_i}=\varphi^{-1}(\calA_i)$.
Then $\Bdot=(\calB_1,\dotsc,\calB_n)$ is a collection of supports with $\ZZ\Bdot=\ZZ^n$.
Since $\varphi$ is a bijection, we identify $\CC^{\calB_i}$ with $\CC^{\calA_i}$ and $\CC^{\Bdot}$ with
$\CC^{\Adot}$.
Given a system $F\in\CC^{\Adot}$, let $\mydefMATH{\iota(F)}\in\CC^{\Bdot}$  be the corresponding system with support
$\Bdot$. 

\begin{lemma}\label{L:Lacunary}
 Suppose that $\Adot$ is lacunary, $\varphi\colon\ZZ^n\xrightarrow{\sim}\ZZ\Adot$ is an
 isomorphism with corresponding surjection  $\Phi\colon(\CC^\times)^n\to(\CC^\times)^n$.
 Let $\Bdot=\varphi^{-1}(\Adot)$ and suppose that $\MV(\Bdot)>1$.
 Then the branched cover $X_\Adot\to\CC^\Adot$ is decomposable and
 $X_\Adot\to X_\Bdot\to\CC^\Adot=\CC^\Bdot$ is a nontrivial decomposition of branched covers induced by the map $\Phi$.
\end{lemma}

\begin{proof}
If $g$ is a polynomial with support $\calB\subset\ZZ^n$, then the composition $g\circ\Phi$ is a polynomial with support
$\varphi(\calB)$, with the coefficient of $x^\beta$ in $g$ equal to the coefficient of $x^{\varphi(\beta)}$ in
$g\circ\Phi$.
Since $\varphi(\calB_i)=\calA_i$, this gives the natural identifications 
$\iota\colon\CC^{\calA_i}\xrightarrow{\,\sim\,}\CC^{\calB_i}$ and $\iota\colon\CC^\Adot\xrightarrow{\,\sim\,}\CC^\Bdot$
mentioned before the lemma.
Under this identification, we have $\iota(f)(\Phi(x))=f(x)$.

Since  $\MV(\Bdot)>1$, the branched cover $X_\Bdot\to\CC^\Bdot$ is nontrivial by definition.
The identification $\iota\colon\CC^\Adot\to\CC^\Bdot$ extends to a commutative diagram
 \begin{equation}\label{Eq:Lacunary}
  \raisebox{-25pt}{\begin{picture}(97,58)(-4,0)
                          \put(35,52){\small$\iota\times\Phi$}  
    \put(0,45){$X_\Adot$} \put(24,49){\vector(1,0){47}} \put(73,45){$X_\Bdot$}
      \put(5,40){\vector(0,-1){27}}                    \put(78,40){\vector(0,-1){27}}
      \put(-4,25){\small$\pi$}                          \put(80,25){\small$\pi$}
                          \put(45,7){\small$\iota$}  
    \put(0,0){$\CC^\Adot$} \put(25,4){\vector(1,0){45}} \put(73, 0){$\CC^\Bdot$}
    \end{picture}}\ 
 \end{equation}
where $\iota\times\Phi$ is the restriction of the map
$\iota\times\Phi\colon \CC^\Adot\times(\CC^\times)^n \to \CC^\Bdot\times(\CC^\times)^n$ to $X_\Adot$.
The map $\iota\times\Phi\colon X_\Adot\to X_\Bdot$ is a map of branched covers with $\ker\Phi$ acting freely on the fibers.
If we restrict the diagram~\eqref{Eq:Lacunary} to the open subset $V$ of $\CC^\Bdot$ over which
$X_\Bdot\to\CC^\Bdot$ is a covering space, we obtain a composition of covering spaces with $\ker\Phi$ acting as deck
transformations on ${\pi^{-1}(V)}\subset X_\Adot$.
Thus $X_\Adot\to\CC^\Adot$ is decomposable.
\end{proof}

\subsection{Triangular support}\label{SSS:triangular}
This requires more discussion before we can state the analog of Lemma~\ref{L:Lacunary}.
Let us begin with an example when $n=3$.
Suppose that
\[
  \calA_1\ =\ \calA_2\ =\ \calA\ =\ 
  \begin{pmatrix} 0&1&1&1&2&2&2&3\\
                  0&0&1&2&0&1&2&1\\
                  0&1&2&3&2&3&4&4\end{pmatrix}
  \quad\mbox{and}\quad
 \calA_3\ =\
  \begin{pmatrix} 0&0&0&0&1&1\\
                  0&0&0&1&0&1\\
                  0&2&4&5&3&4\end{pmatrix}.
\]
The span $\ZZ\calA$ of the first two supports is isomorphic to $\ZZ^2$, with 
$\varphi(a,b)^T\mapsto(a,b,a+b)^T$ an isomorphism $\varphi\colon\ZZ^2\xrightarrow{\sim}\ZZ\Adot$.
Set $\mydefMATH{\calB}=\varphi^{-1}(\calA)$.
We display $\calA$, $\calA_3$, and $\calB$ in the horizontal plane together on the left in Figure \ref{fig:sdstriangular}, and $\calB$ on the right.
\begin{figure} [!htpb]
\includegraphics[scale=0.4]{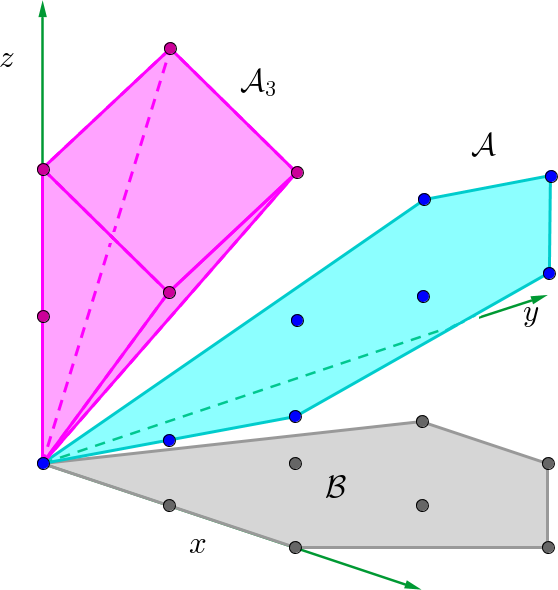} \hspace{0.8 in}
\includegraphics[scale=0.6]{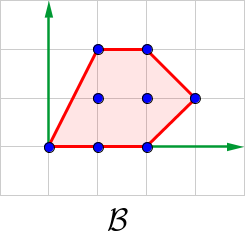}
\caption{\label{fig:sdstriangular}An example of triangular support.}
\end{figure}
Consider the polynomial system $F=(f_1,f_2,f_3)\in\CC[x,y,z]$ with support $\Adot$,
\begin{align*}
  f_1 &=\ 1 + 2xz+3xyz^2+4xy^2z^3+5x^2z^2+6x^2yz^3+7x^2y^2z^4+8x^3yz^4\\
  f_2 &=\ 2 + 3xz+5xyz^2+7xy^2z^3+11x^2z^2+13x^2yz^3+17x^2y^2z^4+19x^3yz^4\\
  f_3 &=\ 1 + 3z^2 + 9 z^4 + 27 yz^5 + 81 xz^3 + 243 xyz^4\,.
\end{align*}
Let $\Phi\colon(\CC^\times)^3\to(\CC^\times)^2$ be given by $\Phi(x,y,z)=(xz,yz)=(u,v)$.
If 
\begin{align*}
  g_1 &=\ 1 + 2u+3uv+4uv^2+ 5u^2+ 6u^2v+ 7u^2v^2+ 8u^3v\\
  g_2 &=\ 2 + 3u+5uv+7uv^2+11u^2+13u^2v+17u^2v^2+19u^3v\,,
\end{align*}
then $f_i=g_i\circ\Phi$ for $i=1,2$.
To compute $\calV(F)$, we first may compute $\calV(g_1,g_2)$ which consists of eight points.
For each solution $(u_0,v_0)\in \calV(g_1,g_2)$, we may identify the fiber $\Phi^{-1}(u_0,v_0)$ with $\CC^\times$
by $z\mapsto (u_0z^{-1},v_0z^{-1},z)$.
Then the restriction of $f_3$ to this fiber is
\[
  1 + (3 + 81 u_0 + 243 u_0v_0)z^2 + (9 + 27 v_0 )z^4\,,
\]
which is a lacunary univariate polynomial with support $\{0,2,4\}$, and has four solutions (counted with multiplicity) when  
$v_0\neq -1/3$.

This example generalizes to all triangular systems.
Suppose  that $\Adot=(\calA_1,\dotsc,\calA_n)$ is triangular.
Let $\emptyset\neq I\subsetneq[n]$ be a proper subset witnessing the triangularity, so that 
$\rank(\ZZ\calA_I)=|I|$.
Set $\mydefMATH{J}=[n]\smallsetminus I$.
Let
\[
  \mydefMATH{\ZZ^I}\ =\ \QQ\calA_I\cap\ZZ^n\ =\
  \{ v\in\ZZ^n\mid \exists m\in\NN\mbox{ with }mv\in\ZZ\calA_I\}\,,
\]
be the saturation of $\ZZ\calA_I$, which is a free abelian group of rank $|I|$.
As it is saturated, $\mydefMATH{\ZZ_J}=\ZZ^n/\ZZ^I$ is free abelian of rank $n-|I|=|J|$.

Applying $\Hom(\bullet,\CC^\times)$ to the short exact sequence
$\ZZ^I\hookrightarrow \ZZ^n\twoheadrightarrow\ZZ_J$ gives the short exact sequence of tori
(whose characters are $\ZZ_J$, $\ZZ^n$, and $\ZZ^I$) with indicated maps,
\begin{equation}
\label{Eq:Triangular}
  (\CC^\times)^{|J|}\simeq\mydefMATH{\TT_J}:=\Hom(\ZZ_J,\CC^\times)\
  \lhra\ (\CC^\times)^n\ \stackrel{\Phi}{\lthra}\
  \mydefMATH{\TT^I}:=\Hom(\ZZ^I,\CC^\times)\simeq(\CC^\times)^{|I|}\,.
\end{equation}
A  polynomial $f$ with support in $\ZZ^I$ determines polynomial functions on $(\CC^\times)^n$ and on $\TT^I$ with the first
the pullback of the second. 
Let $f$ be a polynomial on $(\CC^\times)^n$ with support $\calA\subset\ZZ^n$.
Then its restriction to a fiber $\Phi^{-1}(y_0)$ of $\Phi$ is a regular function $\overline{f}$ on the fiber, which
is a coset of $\TT_J$. 
Choosing an identification of $\TT_J\simeq\Phi^{-1}(y_0)$, we obtain a polynomial
\mydefMATH{$\overline{f}$} on  $\TT_J$ whose
support is the image \mydefMATH{$\overline{\calA}$} of $\calA$ in $\ZZ_J=\ZZ^n/\ZZ^I$.
This polynomial $\overline{f}$ depends upon the identification of the fiber with $\TT_J$.
Let $\mydefMATH{\overline{\calA_J}}$ be the image in $\ZZ_J$ of the collection $\calA_J$ of supports.
Then we have the product formula (see~\cite[Lem.~6]{ThSt} or~\cite[Thm.~1.10]{Esterov})
\begin{equation}\label{Eq:MVproduct}
  \MV(\Adot)\ =\ \MV(\calA_I)\cdot\MV(\overline{\calA_J})\,.
\end{equation}

Since $\Adot = \calA_I\sqcup\calA_J$, we have the identification $\CC^\Adot=\CC^{\calA_I}\oplus\CC^{\calA_J}$.
Suppose that $F\in\CC^\Adot$ is a polynomial system with support $\Adot$.
Write $\mydefMATH{F_I}\in\CC^{\calA_I}$ for its
restriction to the indices in $I$, and the same for $F_J$.
We have the diagram
 \begin{equation}\label{Eq:Triangular_Decomposition}
  \raisebox{-25pt}{\begin{picture}(97,60)(-4,0)
                          \put(32,52){\small$p_I\times\Phi$}  
    \put(0,45){$X_\Adot$} \put(24,49){\vector(1,0){47}} \put(73,45){$X_{\calA_I}$}
      \put(5,40){\vector(0,-1){27}}                    \put(78,40){\vector(0,-1){27}}
      \put(-4,25){\small$\pi$}                          \put(80,25){\small$\pi$}
                          \put(44,8){\small$p_I$}  
    \put(0,0){$\CC^\Adot$} \put(25,4){\vector(1,0){45}} \put(73, 0){$\CC^{\calA_I}$}
    \end{picture}}\ 
 \end{equation}
where $p_I\times\Phi$ is the restriction of the map
$p_I\times\Phi\colon\CC^\Adot\times(\CC^\times)^n\to \CC^{\calA_I}\times\TT^I$ to $X_\Adot$.
 
Let $\mydefMATH{V_\Adot}\subset\CC^\Adot$ be the dense open subset over which $X_\Adot$ is a covering space.
This is the set of polynomial systems $F$ with support $\Adot$ which have exactly
$\MV(\Adot)$ solutions in $(\CC^\times)^n$.
Similarly, let $\mydefMATH{V_{\calA_I}}\subset\CC^{\calA_I}$ be the subset where $X_{\calA_I}\to\CC^{\calA_I}$ is a covering space.
We will show that under the projection $\CC^{\Adot}\to\CC^{\calA_I}$, the image of $V_\Adot$ is a subset of $V_{\calA_I}$.
Define $\mydefMATH{Y_{\Adot}}\to V_\Adot$ to be the restriction of $X_\Adot\to\CC^{\Adot}$ to the dense open set $V_\Adot$.
Also define $\mydefMATH{Y_{\calA_I}}\to V_\Adot$ to be the pullback of $X_{\calA_I}\to\CC^{\calA_I}$ along the map
$V_\Adot\to V_{\calA_I}$.
Write $\Phi\colon Y_\Adot\to Y_{\calA_I}$ for the map induced by $\Phi$.

\begin{lemma}\label{L:Triangular}
 Suppose that $\Adot$ is a triangular set of supports in $\ZZ^n$ witnessed by $I\subsetneq[n]$.
 Then $Y_\Adot\to Y_{\calA_I}\to V_{\Adot}$ a composition of covering spaces.
 If $1<\MV(\calA_I)<\MV(\Adot)$, then this decomposition is nontrivial, so that $X_\Adot\to\CC^\Adot$ is decomposable.

 Furthermore, each fiber of the map $Y_\Adot\to Y_{\calA_I}$ may be identified with the set of solutions of a polynomial
 system with support $\overline{\calA_J}$. 
\end{lemma}

\begin{proof}
  Let $F\in V_{\Adot}$.
  Then its number of solutions is $\#\calV(F)=\MV(\Adot)$.
  If $x\in\calV(F)$, then $\Phi(x)\in\TT^I$ is a solution of $f_i=0$ for $i\in I$.
  Thus $\Phi(\calV(F))\subset\calV(F_I)$, the latter being the solutions of $F_I$ on $\TT^I$.
  For any $y\in\calV(F_I)$, if we choose an identification $\TT_J \simeq\Phi^{-1}(y)$ of the fiber, 
  then the restriction of $F$ to $\Phi^{-1}(y)$ is the system $\mydefMATH{\overline{F_J}}=\{\overline{f_j}\mid j\in J\}$.
  By the Bernstein-Kushnirenko Theorem, this has at most $\MV(\overline{\calA_J})$ solutions.
  By the product formula~\eqref{Eq:MVproduct} and our assumption on $\#\calV(F)$, we conclude that
  the system $F_I$ has $\MV(\calA_I)$ solutions, and for each $y\in\calV(F_I)$, the system $\overline{F_J}$ has 
  $\MV(\overline{\calA_J})$ solutions.

  In particular, this implies that the image of $V_{\Adot}$ in $\CC^{\calA_I}$ is a subset of $V_{\calA_I}$.
  As  $V_{\Adot}$ is open and dense in $\CC^\Adot$, its image contains an open dense subset.
  This proves the assertion that  $Y_\Adot\to Y_{\calA_I}\to V_{\Adot}$ is a decomposition of covering spaces.
  We have already shown that each fiber  of the map $Y_\Adot\to Y_{\calA_I}$ is a polynomial system with support
  $\overline{\calA_J}$ with exactly $\MV(\overline{\calA_J})$ solutions.
  Thus when $1<\MV(\calA_I)<\MV(\Adot)$, we have $\MV(\overline{\calA_J})>1$, which shows that this
  decomposition is nontrivial.
\end{proof}

\section{Computing the decompositions}\label{S:SNF}
We show how to compute the decompositions of $X_\Adot\to\CC^{\Adot}$ from 
Section~\ref{SS:Decompositions} when $\Adot$ is either lacunary or strictly triangular.

Let us consider the Smith normal form (see Section~\ref{subsubsection:smithnormalform}),
\begin{equation}
\label{Eq:SNF}
\calA = PDQ,
\end{equation} when $\calA$ is the matrix whose columns are the vectors in $\Adot$ and
$\MV(\Adot)>0$. 
Then $d_n> 0$ as $\ZZ\Adot$ has rank $n$, and $\Adot$ is lacunary when $d_n> 1$.
In this case, an identification $\varphi\colon\ZZ^n\xrightarrow{\sim}\ZZ\calA$ is given by $PD_n$, where
\mydefMATH{$D_n$} is the principal $n\times n$ submatrix of $D$.
Recall that the 
corresponding surjection $\varphi^*=\Phi\colon(\CC^\times)^n\to(\CC^\times)^n$ has kernel
$\Hom(\ZZ^n/\ZZ\Adot,\CC^\times)$.
Let $\mydefMATH{\psi}=P^{-1}$.
Then $\psi\circ\varphi=D_n$, so that if we set $\mydefMATH{\Psi}=\psi^*$, then
$\Phi\circ\Psi\colon(\CC^\times)^n\to(\CC^\times)^n$  is diagonal, 
 \begin{equation}\label{Eq:diagonal}
  \Phi\circ\Psi(x_1,\dotsc,x_n)\ =\ (x_1^{d_1},\dotsc,x_n^{d_n})\,.
 \end{equation}
Let $y=(y_1,\dotsc,y_n)\in(\CC^\times)^n$.
If we set $\mydefMATH{\rho_i}=|y_i|$ and $\mydefMATH{\zeta_i}=\arg(z_i)$ so that $y_i=\rho_i e^{\sqrt{-1}\zeta_i}$, then 
$(\Phi\circ\Psi)^{-1}(y)$ is the set
 \begin{equation}\label{Eq:PhiInv}
    \left\{ \left(\rho_1^{1/d_1}e^{\sqrt{-1}\theta_1}\,,\,\dotsc\,,\,\rho_n^{1/d_n}e^{\sqrt{-1}\theta_n}\right)
   \,\middle\vert\,  \theta_i = \tfrac{\zeta_i{+}2\pi j}{d_i}\mbox{ for } j=0,\dotsc,d_i{-}1\right\}\vspace{3pt}
 \end{equation}
 as explained in Section~\ref{subsubsection:smithnormalform}.
 %

Suppose that $\Adot$ is triangular, and let us use the notation of Section~\ref{SSS:triangular}.
We suppose that $I=[k]=\{1,\dotsc,k\}$ and $J=\{k{+}1,\dotsc,n\}$.
Given a polynomial $f$ on $(\CC^\times)^n$, its restriction $\overline{f}$ to a fiber of $\Phi\colon(\CC^\times)^n\to\TT^I$
is a regular function on the fiber, which is isomorphic to $\TT_J$.
To represent $\overline{f}$ as a polynomial on $\TT_J$ depends on the choice of a point in that fiber. 
Indeed, suppose that $f=\sum_{\alpha\in\calA}c_\alpha x^\alpha$.
Let $y\in\TT^I$ and $y_0\in\Phi^{-1}(y)$ be a point in
the fiber above $y$, so that $\TT_J\ni z\mapsto y_0z\in\Phi^{-1}(y)$ parameterizes $\Phi^{-1}(y)$.
If we write \mydefMATH{$\overline{\alpha}$} for the image of $\alpha\in\ZZ^n$ in $\ZZ_J=\ZZ^n/\ZZ^I$, then 
 \begin{equation}\label{Eq:substituteFibre}
   \overline{f}(z)\ =\ \sum_{\alpha\in\calA} c_\alpha (y_0z)^\alpha\
    =\ \sum_{\beta\in\overline{\calA}} z^\beta \ 
   \biggl(\,\sum_{\alpha\in\calA\ \mbox{\scriptsize with}\ \overline{\alpha}=\beta} c_\alpha y_0^\alpha\biggr)\,.
 \end{equation}
A uniform choice of a point in each fiber is given by fixing a splitting
$\TT^I\hookrightarrow(\CC^\times)^n$ of the map $\Phi\colon(\CC^\times)^n\twoheadrightarrow\TT^I$.
This gives an identification $(\CC^\times)^n=\TT^I\times\TT_J$.
Then points $y\in\TT^I$ are canonical representatives of cosets of $\TT_J$.
As $k=|I|$, we may further fix isomorphisms $\TT^I\simeq(\CC^\times)^k$ giving $\ZZ^I\simeq\ZZ^k$ and
$\TT_J\simeq(\CC^\times)^{n-k}$ giving $\ZZ_J\simeq\ZZ^{n-k}$.


Suppose now that $\calA=\calA_I$, and we compute a decomposition~\eqref{Eq:SNF}.
Since $\ZZ\calA_I$ has rank $k$, the diagonal matrix $D$ has $k$ nonzero invariant factors.
The saturation $L$ of  $\ZZ\calA_I$ is the image of $PI_k$, where $I_k$ is the $n\times n$
matrix whose only nonzero entries are in its principal $k\times k$ submatrix, which forms an identity matrix.
Then $\varphi=PI_k$ and $\Phi=\varphi^*$.
Applying the coordinate change $\mydefMATH{\psi}=P^{-1}$ to $\ZZ^n$ identifies this saturation as the coordinate plane 
$\ZZ^k\oplus\bzero^{n-k}$ and the lattice $\ZZ\calA_I$ as $d_1\ZZ\oplus d_2\ZZ\oplus\dotsb\oplus d_k\ZZ\oplus\bzero^{n-k}$.
As in Section~\ref{SSS:triangular}, this identifies $\ZZ/L$ with the complementary coordinate plane,
$\bzero^k\oplus\ZZ^{n-k}$.
Setting $\mydefMATH{\Psi}=\psi^*$ , the composition $\Phi\circ\Psi$ is the projection to the first $k$ coordinates,
 \begin{equation}\label{Eq:projection}
  \Phi\circ\Psi\;\colon\; (\CC^\times)^n\ \lthra\ (\CC^\times)^k
 \end{equation}
and we identify $\TT_J=1^k\times (\CC^\times)^{n-k}$ and 
$\TT^I=(\CC^\times)^k\times 1^{n-k}$.

 
\section{Solving decomposable sparse systems}
\label{S:Algorithm}
We describe algorithms that use Esterov's condi\-tions to solve sparse decomposable systems and suggest an application for computing
a start system for solving a general (not necessarily decomposable) sparse polynomial system.
In each, we let \mydefMATH{$\texttt{SOLVE}$} be an arbitrary algorithm for solving a polynomial system.
We assume that the system $F$ to be solved is general given its support $\Adot$ in that it has $\MV(\Adot)$ 
solutions in $(\CC^\times)^n$.
If not, then one may instead solve a general polynomial system with support $\Adot$ and
then use a parameter homotopy together with endgames to compute $\calV(F)$. 
Recall the identification in \eqref{Eq:Lacunary} for Algorithm~\ref{alg:Lacunary}
and the notation ${F_I}\in\CC^{\calA_I}$ used in \eqref{Eq:Triangular_Decomposition} for Algorithm~\ref{alg:Triangular}.

\boxit{
\begin{algorithm}[SolveLacunary]\ 
\label{alg:Lacunary}\\ \myline 
{\bf Input:} \\
$\bullet$  A general polynomial system $F$ whose support $\Adot$ is  lacunary.
\\ \myline 
{\bf Output:} \\
$\bullet$ All solutions $\calV(F)\subset (\CC^\times)^n$
\\ \myline 
{\bf Steps:}
\begin{enumerate}[nosep]
\item[1]  Compute the Smith normal form \eqref{Eq:SNF} of $\Adot$, giving  $\varphi=PD_n$, $\Phi=\varphi^*$, $\psi = P^{-1}$, and $\Psi=\psi^*$,
  so that $\Phi\circ\Psi$ is diagonal~\eqref{Eq:diagonal}
\item[2] Use $\texttt{SOLVE}$ to compute $\calV(\iota(F))\subset (\CC^\times)^n$

\item[3] Using the formula~\eqref{Eq:PhiInv} to compute $(\Phi\circ\Psi)^{-1}(y)$ for $y\in\calV(\iota(F))$, \return
  \[
    \left\{ \Psi(w) \mymid w\in \bigcup_{z \in\calV(\iota(F))} (\Phi\circ\Psi)^{-1}(z)\right\}
  \]
\end{enumerate}
\end{algorithm}}

\begin{proof}[Proof of Correctness.]
  By Lemma~\ref{L:Lacunary}, $\calV(F)=\Phi^{-1}(\calV(\iota(F)))$.
  We apply $\Psi$ to the points of $(\Phi\circ\Psi)^{-1}(z)$ for $z\in\calV(\iota(F))$ to obtain
  points of $\calV(F)$ in their original coordinates.
\end{proof}

\boxit{
\begin{algorithm}[SolveTriangular]\ 
\label{alg:Triangular}\\ \myline
{\bf Input:}\\
$\bullet$  A general polynomial system $F$ whose support $\Adot$ is triangular, witnessed by $0<k<n$ such that 
 $\rank(\ZZ\calA_{[k]})=k$ \\ \myline
{\bf Output:}\\
$\bullet$ All solutions of $\calV(F)\subset(\CC^\times)^n$ \\ \myline
{\bf Steps:}
\begin{enumerate}[nosep]
\item[1]  Compute the Smith normal form \eqref{Eq:SNF} of $\calA_{[k]}$, giving $\varphi=PI_k$, $\Phi=\varphi^*$, $\psi = P^{-1}$, and
  $\Psi=\psi^*$, so that  $\Phi\circ\Psi$ is the projection~\eqref{Eq:projection}
  
\item[2] Use $\texttt{SOLVE}$ to compute $\calV(F_{[k]})\subset (\CC^\times)^k$

\item[3] Choose $y_0\in\calV(F_{[k]})$
  Use $\texttt{SOLVE}$ to compute the points of the fiber $(\Phi\circ\Psi)^{-1}(y_0)$ in $Y_{\Adot}$,
  which are $\calV(\overline{F_J})\subset\{y_0\}\times(\CC^\times)^{n-k}$, where $\overline{F_J}$ has support
  $\overline{\calA_J}$ and $J=[n]\smallsetminus[k]$
 
\item[4] \myfor each $y\in\calV(F_{[k]})$ use a parameter homotopy with start system
  $\calV(\overline{F_J})$ to compute $(\Phi\circ\Psi)^{-1}(y)$ and \return
  \[
    \left\{ \Psi(w) \mymid w\in \bigcup_{y \in\calV(F_{[k]})} (\Phi\circ\Psi)^{-1}(y)\right\}
  \]
\end{enumerate}
\end{algorithm}}

\begin{proof}[Proof of Correctness.]
  By Lemma~\ref{L:Triangular}, every solution $x\in\calV(F)$ lies over a solution $y=\Phi(x)$ to $F_{[k]}$ in
  $(\CC^\times)^k$.
  As explained in Section~\ref{S:SNF}, the map $\Phi\circ\Psi$ is a coordinate projection and
  $(\Phi\circ\Psi)^{-1}(y)=\calV(\overline{F_J})$.
  Here, $\overline{F_J}=(\overline{f_{k+1}},\dotsc,\overline{f_n})$ where $\overline{f_j}$ has support $\overline{\calA_j}$
  and is computed using~\eqref{Eq:substituteFibre}.
  We apply $\Psi$ to convert these points to the original coordinates.  
\end{proof}

Our main algorithm  takes a sparse system and checks Esterov's criteria for decomposability.  
If the system is decomposable, the algorithm calls Algorithm~\ref{alg:Lacunary} (if lacunary) or
Algorithm~\ref{alg:Triangular} (if triangular), and in each of these algorithms calls to the solver $\texttt{SOLVE}$ are assumed to
be recursive calls back to Algorithm~\ref{alg:SDS}.
If the polynomial system is indecomposable, then Algorithm~\ref{alg:SDS}
calls a black box solver $\mydefMATH{\texttt{BLACKBOX}}$.

\boxit{
\begin{algorithm}[SolveDecomposable]\ 
\label{alg:SDS} \\ \myline
{\bf Input:}\\ $\bullet$ A generic polynomial system $F$ with support $\Adot$ \\ \myline
{\bf Output:} \\ $\bullet$ All solutions of $\calV(F)\subset (\CC^\times)^n$ \\ \myline
{\bf Steps:}
\begin{enumerate}[nosep]
\item[1] Compute the Smith normal form $PDQ$~\eqref{Eq:SNF} of $\Adot$
\item[2]  \myif $d_n>1$, then {\return} SolveLacunary$(F)$
\item[3] \myif $d_n=1$, then
\begin{enumerate}[nosep]
  \item[3.1]\myfor all $\emptyset\neq I\subsetneq [n]$
    compute the Smith normal form  $PD_IQ$~\eqref{Eq:SNF} of $\calA_I$

 \item[3.2]\myif $\rank(D_I)=|I|$ for some $I$, reorder so $I=[k]$ and {\return} SolveTriangular$(F,k)$
\item[3.3]  \myelse neither of Esterov's conditions hold and \return $\texttt{BLACKBOX}(F)$
    \end{enumerate}
\end{enumerate}
\end{algorithm}}

\begin{proof}[Proof of Correctness.]
  First note that if the algorithm halts, then it returns the solutions $\calV(F)$.
  Halting is clear in Case (3), but the other cases involve recursive calls back to  Algorithm~\ref{alg:SDS}.
  In Case (1), SolveLacunary will call Algorithm~\ref{alg:SDS} on a system $\iota(F)$ whose mixed volume is
  less than $\MV(\Adot)$.
  In Case (2), SolveTriangular  will call Algorithm~\ref{alg:SDS} on systems $F_{[k]}$ and $\overline{F_J}$, each involving
  fewer variables than $F$.
  Thus, in each recursive call back to  Algorithm~\ref{alg:SDS}, either the mixed volume or the number of variables
  decreases, which proves that the algorithm halts.
\end{proof}

%
\section{Start systems}
The start system in the B\'ezout homotopy (Algorithm \ref{alg:bezouthomotopy}) is a highly decomposable sparse polynomial system consisting of supports which are subsets of the original support of $F$, but have the same mixed volume.
We propose a generalization, in which Algorithm~\ref{alg:SDS} is used to compute a start system.

\begin{example}\label{Ex:start}
  Suppose that we have supports $\calA_1=\calA_2=\calA$, shown in Figure \ref{fig:sdsvert1} which are given by the columns of the matrix
  $(\begin{smallmatrix}0&0&1&1&2&3&3&3&4&5&5&6\\0&2&0&1&3&0&1&4&2&3&4&4\end{smallmatrix})$.
  \begin{figure}[!htpb]
\includegraphics[scale=0.5]{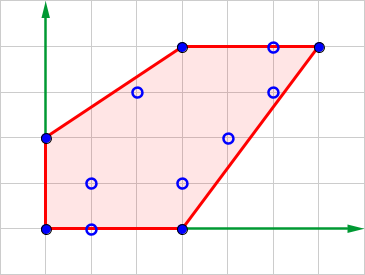}
  \caption{\label{fig:sdsvert1}A support $\mathcal A$ such that $(\mathcal A, \mathcal A)$ is neither triangular nor lacunary.}
  \end{figure}
  Then $\MV(\calA_1,\calA_2)=2!\vol(\conv(\calA))=30$.
  Let $\calB_1=\calB_2=(\begin{smallmatrix}0&0&3&3&6\\0&2&0&4&4\end{smallmatrix})$\vspace{1pt} be the set of vertices
  of $\conv(\calA)$.
  Given a general system $F\in\CC^\Adot$, let $G\in\CC^\Bdot\subset\CC^\Adot$ be obtained from $F$ by restriction to the
  monomials in $\calB$.
  (That is, we set coefficients of monomials $x^\alpha$ in $F$ to zero if $\alpha\not\in\calB$.)
  Then $\Bdot$ is lacunary with the map $\Phi(x_1,x_2)=(x_1^3,x_2^2)$, and $\iota(G)$ has five solutions.
  We may use Algorithm~\ref{alg:SDS} (more specifically, Algorithm~\ref{alg:Lacunary}) to compute $\calV(G)$, and then
  compute $\calV(F)$ using the straight-line homotopy $H(t;x)$ with start system $G=H(1;x)$ and tracking from the solutions $\calV(G)$ at $t=1$.\hfill$\diamond$ 
\end{example}

Example~\ref{Ex:start} motivates our final algorithm.
For a collection $\Adot=(\calA_1,\dotsc,\calA_n)$ of supports, let $\mydefMATH{\vertices(\Adot)}=(\vertices(\calA_1),\dotsc,\vertices(\calA_n))$ where $\mydefMATH{\vertices(\calA_i)}=\vertices(\conv(\calA_i))$.
Note that if $G \in \CC^{\vertices(\Adot)}$ is a regular value of the branched cover
$\pi|_{X_{\vertices(\Adot)}}\colon X_{\vertices(\Adot)} \to \CC^{\vertices(\Adot)}$ then $G$ is also a regular value of
$\pi\colon X_\Adot \to \CC^\Adot$.
As such, $G$ may be taken as a start system for a straight-line homotopy and used to compute $\calV(F)$
for any $F\in\CC^{\Adot}$ with $\calV(F)$ finite.
 The benefit of this approach is that $\pi|_{X_{\vertices(\Adot)}}$ decomposes if $\pi$ does. Therefore, as seen in Example~\ref{Ex:start}, $\pi|_{X_{\vertices(\Adot)}}$ is more likely (and no less likely) than $\pi$ to be decomposable.

\boxit{
\begin{algorithm}[Decomposable Start System]\ 
\label{alg:SDSS}\\ 
{\bf Input:}\\ 
$\bullet$ A set $\Adot$ of supports \\ 
{\bf Output:} \\ $\bullet$ A start system $G$ for a homotopy coming from $\pi_{\Adot}$ and start solutions $\V(G)$\\ 
{\bf Steps:}
\begin{enumerate}[nosep]
  \item[1] Choose a general system $G\in\CC^{\vertices(\Adot)}$
  \item[2] Compute $\calV(G)$ using Algorithm~\ref{alg:SDS}
  \item[3] \return the pair $(G,\calV(G))$
\end{enumerate}

\end{algorithm}
}

\begin{proof}[Proof of Correctness.]
  As $G\in\CC^{\vertices(\Adot)}$ is  general, it has $\MV(\vertices(\Adot))$ solutions.
  Since for each $i$, $\conv(\calA_i)=\conv(\vertices(\calA_i))$, we have $\MV(\vertices(\Adot))=\MV(\Adot)$.
  Finally, $\CC^{\vertices(\Adot)}$ is the subspace of $\CC^{\Adot}$ where the coefficients of non-extreme monomials in each
  polynomial are zero.
  Thus $G\in\CC^{\Adot}$, which shows that $(G,\calV(G))$ is a start system for $\Adot$.
\end{proof}

\begin{remark}\label{R:more}
  The B\'ezout homotopy motivated Algorithm~\ref{alg:SDSS}.
  However, if we apply Algorithm~\ref{alg:SDSS} to the system of supports $\Adot$, where $\calA_i$ consists of all monomials of
  degree at most $d_i$, then we will not get the start system for the B\'ezout homotopy.
  For example, when $n=2$, $d_1=2$, and $d_2=3$, the supports are  as shown in Figure \ref{fig:sdsbezout}.
 Here, $\calB_1$ and $\calB_2$ are the supports of the start system for the B\'ezout homotopy.
 
   \begin{figure}[!htpb]
\includegraphics[scale=0.45]{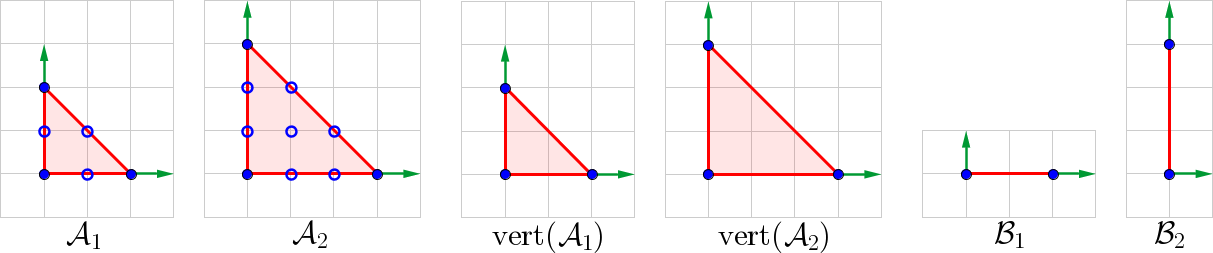}
 \caption{\label{fig:sdsbezout} Dense support $(\mathcal A_1, \mathcal A_2)$, the support $(\vertices(\calA_1),\vertices(\calA_2))$, and the support of the B\'ezout start system.} 
 \end{figure}

 We leave open the challenge of finding a simple, general method to replace each set $\calA_i$ by a subset (or superset)
 $\calB_i$ of $\calA_i$, so that $\MV(\Adot)=\MV(\Bdot)$ and $\pi\colon X_{\Bdot} \to \CC^\Bdot$ is decomposable.

 A possible first step would be to take advantage of the results on monotonicity developed in Section \ref{subsection:monotonicity}. For example, if $\mathcal A_1 = (\begin{smallmatrix}1&3&1&3&2\\1&1&3&3&4\end{smallmatrix})$ and $\mathcal A_2=\mathcal B_1=\mathcal B_2= (\begin{smallmatrix}2&0&2&4\\0&2&4&2\end{smallmatrix})$ then $\Adot=(\mathcal A_1,\mathcal A_2)$ is neither lacunary nor triangular, but $\Bdot=(\mathcal B_1,\mathcal B_2)$ is lacunary. Moreover, $\MV(\Adot)=\MV(\Bdot) = 8$ and so a general sparse polynomial system supported on $\Adot$ corresponds to a regular value of $\pi_\Bdot$. Thus, one may solve a general sparse decomposable system on $\Bdot$, and subsequently solve a system supported on $\Adot$ via a parameter homotopy. 
\hfill$\diamond$
\end{remark}


\section{A computational experiment}\label{S:computations}
We explored the computational cost of using Algorithm~\ref{alg:SDS} to solve sparse decomposable systems, comparing
timings to \textbf{PHCPack} \cite{V99,PHC_M2} on a family of related systems.

Let $\calA_1=(\begin{smallmatrix}0&1&2&0&1\\0&0&0&1&1\end{smallmatrix})$, 
$\calA_2=(\begin{smallmatrix}1&0&1&2&1\\0&1&1&1&2\end{smallmatrix})$, 
$\calB_1=(\begin{smallmatrix}0&2&0&2\\0&0&1&3\end{smallmatrix})$, and 
$\calB_2=(\begin{smallmatrix}0&1&2&0&2&0\\0&0&0&1&1&2\end{smallmatrix})$.
We display these supports and their convex hulls in Figure \ref{fig:sdsexperiment}.
\begin{figure}[!htpb]
\includegraphics[scale=0.6]{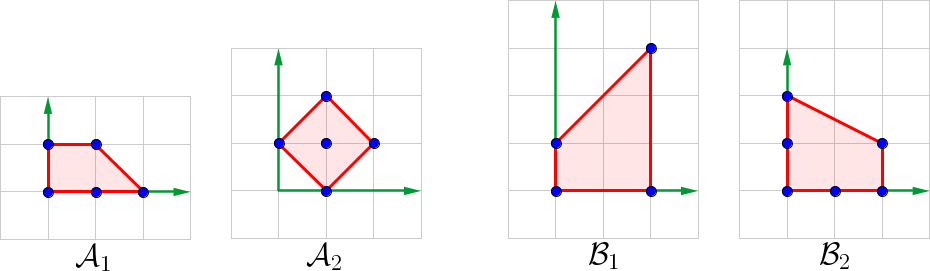}
\caption{\label{fig:sdsexperiment}The four supports involved in a computational experiment.}
\end{figure}
Let $\mydefMATH{\calC}=\{0,1\}^5$ be the vertices of the five-dimensional cube.
We construct sparse decomposable systems from $\mydefMATH{\Adot}=(\calA_1,\calA_2)$,
$\mydefMATH{\Bdot}=(\calB_1,\calB_2)$, and $\calC$ as follows.

Choose two injections $\mydefMATH{\imath},\mydefMATH{\jmath}\colon\ZZ^2\to\ZZ^5$ such that
$\imath(\ZZ^2)\cap\jmath(\ZZ^2)=\{0\}$.
For example, choose four linearly independent vectors $\imath_1,\imath_2,\jmath_1,\jmath_2\in\ZZ^5$, and define
$\imath(a,b)=a\imath_1+b\imath_2$, and the same for $\jmath$.
Let us set
\[
  \mydefMATH{\calA(\imath,\jmath)}\ =\
  \bigl(\imath(\calA_1)\,,\,\imath(\calA_2)\,,\ \jmath(\calB_1)\,,\,\jmath(\calB_2)\,,\ \calC\bigr)\,.
\]

%
%
%
\begin{example}\label{Ex:computation1}
We now illustrate Algorithm~\ref{alg:SDS} in detail on $\calA(\imath,\jmath)$ by considering the case when
$\imath_1,\imath_2,\jmath_1,\jmath_2$ are the first four standard unit vectors $e_1,\ldots,e_4$.  
  Suppose $F = (f_1,f_2,g_1,g_2,h)$ is a system of polynomials $\mathbb{C}[x_1,x_2,y_1,y_2,z]$ with support
  $\mathcal A(e_1,e_2,e_3,e_4)$. 
  We use superscripts to distinguish different calls of the same algorithm.
  When \texttt{SolveDecomposable}$^{(1)}(F)$ is called, it first checks if $F$ is lacunary (it is not as $\ZZ\calC=\ZZ^5$),
  and then recognizes that $F$ is triangular witnessed by $(f_1,f_2)$.
  As such, it calls \texttt{SolveTriangular}$^{(1)}(F,2)$ which computes the $\MV(\Adot)=5$ solutions $p_1,\ldots,p_5$ to
  $\calV(f_1,f_2)$ with \textbf{PHCPack}, our choice of \texttt{BLACKBOX}.

  As its penultimate task, \texttt{SolveTriangular}$^{(1)}$ computes a fiber of the first solution $p_1$ by
  performing the substitution $(x_1,x_2)=p_1$ in $g_1,g_2$ and $h$, and recursively calls
  \texttt{Solve\-Decomposable}$^{(2)}$ on the system $(g_1(p_1,y,z),g_2(p_1,y,z),h(p_1,y,z))\in\CC[y_1,y_2,z]$.
  This system is recognized to be triangular witnessed by $(g_1,g_2)$ and \texttt{SolveTriangular}$^{(2)}(g_1,g_2)$
  computes the $\MV(\Bdot)=10$ solutions $q_1,\ldots,q_{10}$ using \textbf{PHCPack}.
  Next, \texttt{SolveTriangular}$^{(2)}$ computes a fiber above $q_1$ by performing  the substitution
  $y=(y_1,y_2)=q_1$ in $h(p_1,y,z)$ producing the univariate polynomial $h(p_1,q_1,z)$ of degree $1$ which
  has solution $(p_1,q_1,z_1)$.
  Finally, \texttt{SolveTriangular}$^{(2)}$ performs a parameter homotopy from $q_1$ to $q_i$ to populate the fibers above
  each $q_i$.
  Thus \texttt{SolveTriangular}$^{(1)}$ populates the fiber above $p_1$ consisting of $10 \cdot 1=10$ solutions.
  As its final step, \texttt{SolveTriangular}$^{(1)}$  uses parameter homotopies from $p_1$ to $p_i$ to populate all
  fibers producing all $5 \cdot 10 = 50$ solutions of $\calV(F)$.

\begin{figure}[htpb!]
\includegraphics[scale=0.6]{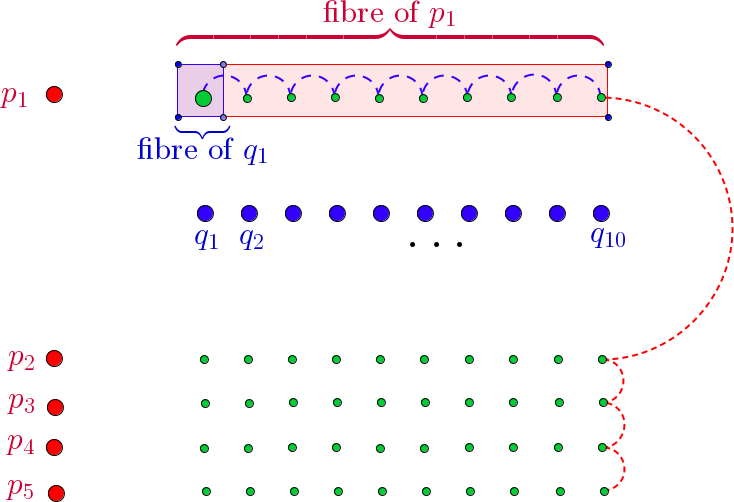}
\caption{\label{fig:DSSexperiment}A schematic of the process in Example \ref{Ex:computation1}.}
\end{figure}
  
   Figure \ref{fig:DSSexperiment} depicts a schematic of this process. Red objects correspond to the function \texttt{solve}\-$\texttt{Decomposable}^{(1)}$ and blue objects correspond to $\texttt{solveDecomposable}^{(2)}$. The largest points represent solutions which were computed directly. The dotted lines represent the use of monodromy to move fibers. \hfill$\diamond$
\end{example}

The overhead of this algorithm includes the computation of Smith normal forms and the search for subsets witnessing
triangularity.
Additionally, it often requires more path-tracking than a direct use of \textbf{PHCPack}.
Nonetheless, the overhead seems nominal, and compared to the paths tracked in \textbf{PHCPack}, the paths tracked in our
algorithm either involve fewer variables or  polynomials of smaller degree. 

For example, in Example \ref{Ex:computation1}, our algorithm called \textbf{PHCPack} to solve two sparse polynomial systems with $5$
and $10$ solutions respectively.
A parameter homotopy was called $10-1=9$ times on a system with $1$ solution, then a different parameter homotopy was
called $5-1=4$ times on a system with $10$ solutions.
In total, $5+10+9+40=64$ individual paths were tracked.
In contrast, a direct use of \textbf{PHCPack} involves tracking exactly $\MV(\mathcal A(e_1,e_2,e_3,e_4)) = 50$ paths, albeit
in a higher dimensional space.  

For more general 
$\imath$ and $\jmath$, the recursive structure of our computation is similar to
Example~\ref{Ex:computation1}. 
Some notable differences include 
\begin{enumerate}
\item $\imath(\Adot)$ or $\jmath(\Bdot)$ may be lacunary which induces further decompositions. 
\item Monomial changes must be computed as $\imath(\Adot)$ or $\jmath(\Bdot)$ could involve all variables.
\item For most $\imath,\jmath$ the univariate polynomial obtained from $h$ has degree $5$ and is solved by computing
  eigenvalues of its companion matrix.
\end{enumerate}
For example,
 if we choose $e_1-e_2, e_2-e_3,e_3-e_4, e_4-e_5$ for $\imath_1,\imath_2,\jmath_1,\jmath_2$, then again, no
system in the algorithm is lacunary, but the univariate polynomial obtained from $h$ has support
$\{0,1,2,3,4,5\}$, so that $\MV(\calA(\imath,\jmath))=250$.

\begin{figure}[htpb]
	{\includegraphics[scale=0.36]{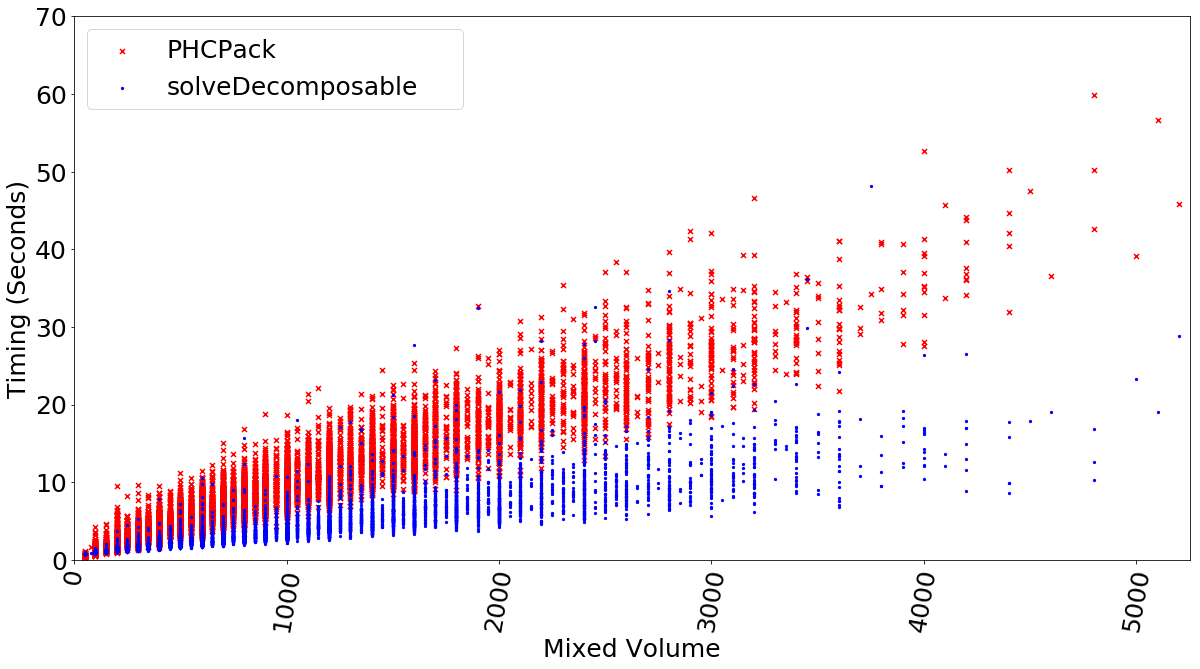}}
\caption{Scatter plot of timings}
\label{fig:scatter}
\end{figure}
\begin{figure}[htpb]
	{\includegraphics[scale=0.36]{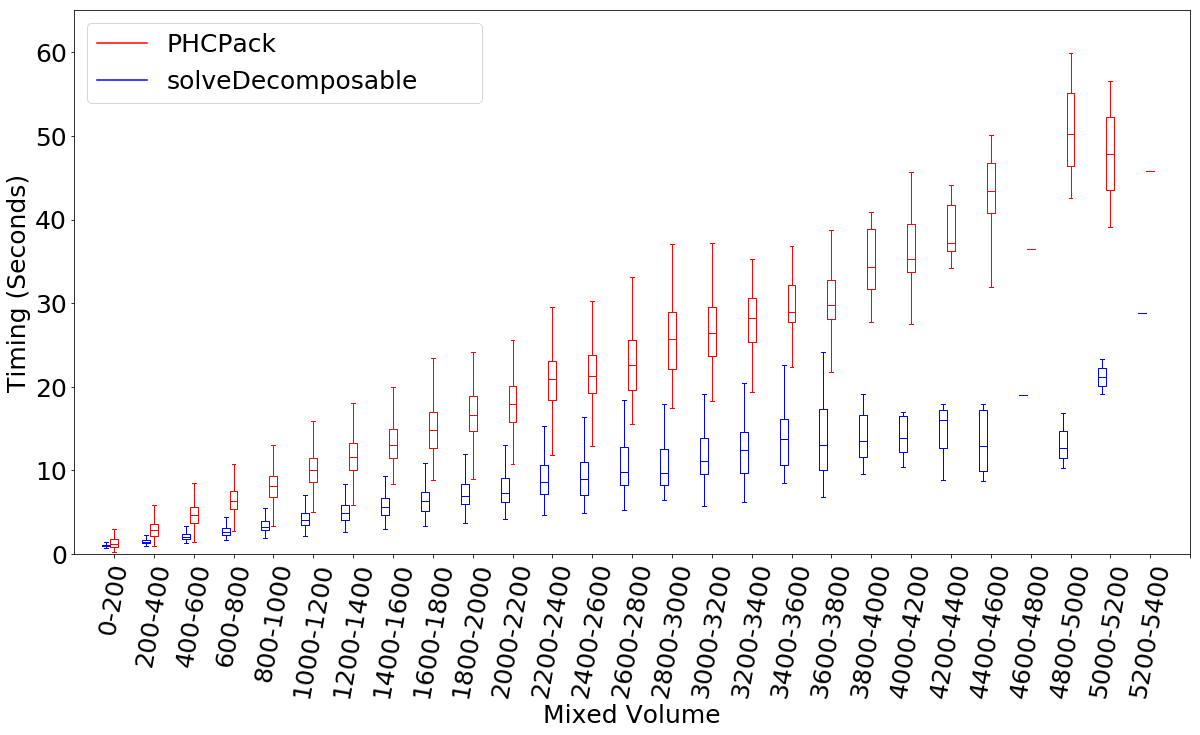}}
\caption{Box plot of timings}
\label{fig:boxplots}
\end{figure}

In our computational experiment, we produced $13563$  instances of $\mathcal A(\imath,\jmath)$ and solved each instance using our implementation of Algorithm
\ref{alg:SDS} as well as with \textbf{PHCPack}.
Due to ill-conditioning 
and heuristic choices of tolerances, some computations failed to produce all solutions.
We only include the $10962$ instances such that both \textbf{PHCPack} and Algorithm \ref{alg:SDS} computed all solutions.

We give a scatter plot of the elapsed timings in Figure \ref{fig:scatter} with respect to the mixed volume of the
system. Figure \ref{fig:boxplots} displays box plots of the timings of each algorithm grouped by sizes of mixed
volumes.
The boxes range from the first quartile $q_1$ to the third quartile $q_3$ of the group data with whiskers
extending to the smallest and largest data points which are not outliers.
Outliers are the data points which are smaller than $q_1-1.5I$ or larger than $q_3+1.5I$ where $I$ is the length of the
interquartile range $(q_1,q_3)$.

A more detailed account of these computations, along with our implementation in \textbf{Macaulay2}, may be found
at the website~\cite{ThomasWebsite}. 
%
\chapter{SUMMARY}

Newton polytopes provide a rich combinatorial structure by which we may delineate polynomials and thus polynomial systems. The geometric nature of numerical algebraic geometry lends itself to algorithms which can extract and use this combinatorial data via the HS-algorithm and polyhedral homotopy algorithm, respectively. We augment both of these algorithms. 

Using the HS-algorithm as a subroutine, we develop a tropical membership algorithm. We implement both the HS-algorithm and the tropical membership algorithm in a computer algebra system and analyze their convergence rates. We use the HS-algorithm to completely identify a large polynomial defining a hypersurface from algebraic vision. With the same software, we also determine many vertices of the L\"uroth polytope.

We augment the polyhedral homotopy by developing and implementing an algorithm which recognizes when a sparse polynomial system is decomposable. It then uses this decomposition to numerically and recursively solve the sparse system. We compare timings of our software against the use of a polyhedral homotopy.


\let\oldbibitem\bibitem
\renewcommand{\bibitem}{\setlength{\itemsep}{0pt}\oldbibitem}
\bibliographystyle{ieeetr}

\phantomsection
\addcontentsline{toc}{chapter}{REFERENCES}

\renewcommand{\bibname}{{\normalsize\rm REFERENCES}}

\bibliography{myReference}

\def\cprime{$'$}
\begin{thebibliography}{10}

\bibitem{Bry:NPtrop}
T.~Brysiewicz, ``{Numerical Software to Compute Newton Polytopes and Tropical
  Membership},'' {\em Mathematics in Computer Science}, 2020.

\bibitem{HuberSturmfels}
B.~Huber and B.~Sturmfels, ``A polyhedral method for solving sparse polynomial
  systems,'' {\em Mathematics of Computation}, vol.~64, no.~212,
  pp.~1541--1555, 1995.

\bibitem{Esterov}
A.~Esterov, ``Galois theory for general systems of polynomial equations,'' {\em
  Compositio Mathematica}, vol.~155, no.~2, pp.~229--245, 2019.

\bibitem{HeptTheobald}
K.~Hept and T.~Theobald, ``Tropical bases by regular projections,'' {\em
  Proceedings of the American Mathematical Society}, vol.~137, no.~7,
  pp.~2233--2241, 2009.

\bibitem{Ziegler}
G.~M. Ziegler, {\em Lectures on polytopes}.
\newblock New York: Springer, 1995.

\bibitem{OToV}
I.~A. Emiris, V.~Fisikopoulos, C.~Konaxis, and L.~Penaranda, ``An oracle-based,
  output-sensitive algorithm for projections of resultant polytopes,'' {\em
  International Journal of Computational Geometry and Applications}, vol.~23,
  no.~04n05, 2013.

\bibitem{Minkowski}
H.~Minkowski, ``Theorie der konvexen k\"orper, insbesondere begr\"undung ihres
  oberfl\"achenbegriffs,'' {\em Gesammelte Abhandlungen}, vol.~II,
  pp.~131--229, 1911.

\bibitem{Ewald}
G.~Ewald, {\em Combinatorial convexity and algebraic geometry}, vol.~168 of
  {\em Graduate Texts in Mathematics}.
\newblock Springer, New York, 1996.

\bibitem{Steffens}
R.~J. Steffens, {\em Mixed volumes, mixed {E}hrhart theory and applications to
  tropical geometry and linkage configurations}.
\newblock PhD thesis, Goethe Universit\"at, Goethe Universitat Frankfurt, 2009.

\bibitem{triangulations}
J.~De~Loera, J.~Rambau, and F.~Santos, {\em Triangulations: Structures for
  Algorithms and Applications}.
\newblock Algorithms and Computation in Mathematics, Springer Berlin
  Heidelberg, 2010.

\bibitem{Rojas}
J.~Rojas, ``A convex geometric approach to counting the roots of a polynomial
  system,'' {\em Theoretical Computer Science}, vol.~133, pp.~105--140, 10
  1994.

\bibitem{EsterovMonotonicity}
A.~Esterov, ``Indices of $1$-forms and newton polyhedra,'' {\em
  Mathematicheskii Sbornik}, vol.~197, no.~7, pp.~1085--1108, 2006.

\bibitem{Chen}
T.~Chen, ``Unmixing the mixed volume computation,'' {\em Discrete {\&}
  Computational Geometry}, vol.~62, pp.~55--86, 2019.

\bibitem{Bihan}
B.~Fr\'ed\'eric and I.~Soprunov, ``Criteria for strict monotonicity of the
  mixed volume of convex polytopes,'' {\em Advances in Geometry}, 02 2017.

\bibitem{CLO}
D.~Cox, J.~Little, and D.~O'Shea, {\em Ideals, Varieties and Algorithms: An
  Introduction to Computational Algebraic Geometry and Commutative Algebra}.
\newblock Springer, 1991.

\bibitem{HarrisBook}
J.~Harris, {\em Algebraic Geometry: A First Course}.
\newblock Graduate Texts in Mathematics, Springer, 1992.

\bibitem{Hartshorne}
R.~Hartshorne, {\em Algebraic Geometry}.
\newblock Graduate Texts in Mathematics, Springer New York, 2013.

\bibitem{Shafarevich}
M.~Reid and I.~Shafarevich, {\em Basic Algebraic Geometry 1}.
\newblock Springer Berlin Heidelberg, 2013.

\bibitem{Hilbert}
D.~Hilbert, ``\"uber die theorie der algebraischen formen,'' {\em Mathematische
  Annalen}, vol.~36, pp.~473--530, 1890.

\bibitem{HilbertNull}
D.~Hilbert, ``\"uber die vollen invariantensysteme,'' {\em Mathematische
  Annalen}, vol.~42, pp.~313--373, 1893.

\bibitem{Wielandt}
H.~Wielandt, {\em Finite permutation groups}.
\newblock Translated from the German by R. Bercov, Academic Press, New
  York-London, 1964.

\bibitem{Hatcher}
A.~Hatcher, {\em Algebraic topology}.
\newblock Cambridge: Cambridge University Press, 2002.

\bibitem{Harris}
J.~Harris, ``Galois groups of enumerative problems,'' {\em Duke Mathematical
  Journal}, vol.~46, no.~4, pp.~685--724, 1979.

\bibitem{Hermite}
C.~Hermite, ``Sur les fonctions alg{\'e}briques,'' {\em Comptes rendus de
  l'Acad\'emie des Sciences (Paris)}, vol.~32, pp.~458--461, 1851.

\bibitem{PirolaSchlesinger}
G.~P. Pirola and E.~Schlesinger, ``Monodromy of projective curves,'' {\em
  Journal of Algebraic Geometry}, vol.~14, no.~4, pp.~623--642, 2005.

\bibitem{DerksenKemper}
H.~Derksen and G.~Kemper, {\em Computational invariant theory}.
\newblock Invariant Theory and Algebraic Transformation Groups, I, Springer,
  Berlin, 2002.
\newblock Encyclopaedia of Mathematical Sciences, 130.

\bibitem{Munkres}
J.~R. Munkres, {\em Topology: a first course}.
\newblock Prentice-Hall, Inc., Englewood Cliffs, N.J., 1975.

\bibitem{AmendolaRodriguez}
C.~Am\'endola and J.~Rodriguez, ``Solving parameterized polynomial systems with
  decomposable projections,'' 2016.
\newblock {arXiv:1612.08807}.

\bibitem{GIVIX}
A.~Mart\'in~del Campo-Sanchez, F.~Sottile, and R.~Williams, ``Classification of
  {S}chubert {G}alois groups in ${G}r(4,9)$.'' {arXiv.org/1902.06809}, 2019.

\bibitem{SWY}
F.~Sottile, R.~Williams, and L.~Ying, ``Galois groups of compositions of
  {S}chubert problems.'' {arXiv.org/1910.06843}, 2019.

\bibitem{MacSturm}
D.~Maclagan and B.~Sturmfels, {\em Introduction to Tropical Geometry},
  vol.~161.
\newblock Providence, RI: American Mathematical Society, 2015.

\bibitem{BieriGroves}
R.~Bieri and J.~Groves, ``The geometry of the set of characters induced by
  valuations.,'' {\em Journal f\"ur die reine und angewandte Mathematik},
  vol.~347, pp.~168--195, 1984.

\bibitem{Chan}
A.~Chan, ``Gr\"obner bases over fields with valuation and tropical curves by
  coordinate projections,'' {\em Ph.D. thesis, University of Warwick}, 2013.

\bibitem{NoceWrig06}
J.~Nocedal and S.~J. Wright, {\em Numerical Optimization}.
\newblock New York, NY, USA: Springer, second~ed., 2006.

\bibitem{BCSS}
L.~Blum, F.~Cucker, M.~Shub, and S.~Smale, {\em Complexity and real
  computations}.
\newblock Springer, New York, 1998.

\bibitem{smale}
S.~Smale, ``Newton's method estimates from data at one point,'' in {\em The
  merging of disciplines: new directions in pure, applied, and computational
  mathematics}, pp.~185--196, Springer, New York, 1986.

\bibitem{HS12}
J.~D. Hauenstein and F.~Sottile, ``Algorithm 921: alphacertified: certifying
  solutions to polynomial systems,'' {\em ACM Transactions on Mathematical
  Software (TOMS)}, vol.~38, no.~4, p.~28, 2012.

\bibitem{M2certification}
K.~Lee, ``Numericalcertification.'' Distributed with Macaulay 2.

\bibitem{BeltranLeykin}
C.~Beltr\'an and A.~Leykin, ``Certified numerical homotopy tracking,'' {\em
  Experimental Mathematics}, vol.~21, no.~1, pp.~69--83, 2012.

\bibitem{BurrYapXu}
M.~Burr, C.~Yap, and J.~Xu, ``An approach for certifying homotopy continuation
  paths: Univariate case,'' {\em In Proceedings of the 43rd International
  Symposium on Symbolic and Algebraic Computation}, pp.~399--406, 2018.

\bibitem{TelenBarelVerschelde}
S.~Telen, M.~Van~Barel, and J.~Verschelde, ``A robust numerical path tracking
  algorithm for polynomial homotopy continuation,'' 2019.
\newblock {arXiv:1909.04984 }.

\bibitem{burgissercondition}
P.~B{\"u}rgisser and F.~Cucker, {\em Condition: The geometry of numerical
  algorithms}, vol.~349.
\newblock Springer Science \& Business Media, 2013.

\bibitem{Davidenko}
D.~Davidenko, ``Ob odnom novom methode chislennovo resheniya sistem nelineinykh
  uravenii,'' {\em Doklady Akademii Nauk SSR}, vol.~87, no.~4, pp.~601--602,
  1953.

\bibitem{AHScert}
T.~Akoglu, J.~D. Hauenstein, and A.~Szanto, ``Certifying solutions to
  overdetermined and singular polynomial systems over $\mathbb{Q}$,'' {\em
  Journal of Symbolic Computation}, vol.~84, pp.~147--171, 2018.

\bibitem{DHScert}
T.~Duff, N.~Hein, and F.~Sottile, ``Certification for polynomial systems via
  square subsystems,'' 2019.
\newblock {arXiv:1812.02851}.

\bibitem{LSY89}
T.~Y. Li, T.~Sauer, and J.~A. Yorke, ``The cheater's homotopy: an efficient
  procedure for solving systems of polynomial equations,'' {\em SIAM Journal on
  Numerical Analysis}, vol.~26, no.~5, pp.~1241--1251, 1989.

\bibitem{MS89}
A.~P. Morgan and A.~J. Sommese, ``Coefficient-parameter polynomial
  continuation,'' {\em Applied Mathematics and Computation}, vol.~29, no.~2,
  part II, pp.~123--160, 1989.

\bibitem{Garcia79}
C.~B. Garcia and W.~I. Zangwill, ``Finding all solutions to polynomial systems
  and other systems of equations,'' {\em Mathematical Programming}, vol.~16,
  no.~1, pp.~159--176, 1979.

\bibitem{V99}
J.~Verschelde, ``Algorithm 795: {PHC}pack: A general-purpose solver for
  polynomial systems by homotopy continuation,'' {\em ACM Transactions on
  Mathematical Software}, vol.~25, no.~2, pp.~251--276, 1999.
\newblock Available at {http://www.math.uic.edu/{\~{}}jan}.

\bibitem{BertiniBook}
D.~J. Bates, J.~D. Hauenstein, A.~J. Sommese, and C.~W. Wampler, {\em
  Numerically solving polynomial systems with Bertini}.
\newblock SIAM, 2013.

\bibitem{Bertini}
D.~J. Bates, J.~D. Hauenstein, A.~J. Sommese, and C.~W. Wampler, ``Bertini:
  Software for numerical algebraic geometry.'' Available at bertini.nd.edu with
  permanent doi: dx.doi.org/10.7274/R0H41PB5.

\bibitem{HCjl}
P.~Breiding and S.~Timme, ``Homotopycontinuation.jl - a package for solving
  systems of polynomial equations in julia,'' {\em Mathematical Software ICMS
  2018, Lecture Notes in Computer Science}, 2018.
\newblock Available at juliahomotopycontinuation.org.

\bibitem{HOM4PS}
T.-L. Lee, T.~Li, and C.~Tsai, ``Hom4ps-2.0: A software package for solving
  polynomial systems by the polyhedral homotopy continuation method,'' {\em
  Computing}, vol.~83, pp.~109--133, 2008.

\bibitem{NAG4M2}
A.~Leykin, ``Numerical algebraic geometry for macaulay2.''
  http://people.math.gatech.edu/ ~aleykin3/NAG4M2.

\bibitem{WSP}
J.~D. Hauenstein and A.~J. Sommese, ``Witness sets of projections,'' {\em
  Applied Mathematics and Computation}, vol.~217, no.~7, pp.~3349--3354, 2010.

\bibitem{Zariski}
O.~Zariski, ``A theorem on the poincar\'e group of an algebraic hypersurface,''
  {\em Annals of Mathematics}, vol.~38, no.~1, pp.~131--141, 1937.

\bibitem{NumGal}
J.~D. Hauenstein, J.~Rodriguez, and S.~F., ``Numerical computation of galois
  groups,'' {\em Foundations of Computational Mathematics}, vol.~18,
  pp.~867--890, 2018.

\bibitem{Duff}
T.~Duff, C.~Hill, A.~Jensen, K.~Lee, A.~Leykin, and J.~Sommars, ``Solving
  polynomial systems via homotopy continuation and monodromy,'' {\em IMA
  Journal of Numerical Analysis}, 2018.

\bibitem{Dixon}
J.~Dixon, ``The probability of generating the symmetric group,'' {\em Math Z},
  vol.~110, pp.~199--205, 1969.

\bibitem{Babai}
T.~Hayes and L.~Babai, ``The probability of generating the symmetric group when
  one of the generators is random,'' {\em Publicationes Mathematicae Debrecen},
  vol.~69, pp.~271--280, 10 2006.

\bibitem{NP}
J.~D. Hauenstein and F.~Sottile, ``Newton polytopes and witness sets,'' {\em
  Mathematics in Computer Science}, vol.~8, no.~2, pp.~235--251, 2012.

\bibitem{M2}
D.~R. Grayson and M.~E. Stillman, ``Macaulay2, a software system for research
  in algebraic geometry.'' Available at {http://www.math.uiuc.edu/Macaulay2/}.

\bibitem{taylor}
T.~Brysiewicz, ``Numerical computations of {N}ewton polytopes.'' Available at
  {http://www.math.tamu.edu/\~tbrysiewicz/NumericalNP}, 2018.

\bibitem{B4M2}
D.~J. {Bates}, E.~{Gross}, A.~{Leykin}, and J.~{Rodriguez}, ``{Bertini for
  Macaulay2},'' Oct. 2013.

\bibitem{sage}
W.~Stein {\em et~al.}, {\em {S}age {M}athematics {S}oftware ({V}ersion x.y.z)}.
\newblock The Sage Development Team, 2017.
\newblock {http://www.sagemath.org}.

\bibitem{Ponce2017}
J.~Ponce, B.~Sturmfels, and M.~Trager, ``Congruences and concurrent lines in
  multi-view geometry,'' {\em Advances in Applied Mathematics}, vol.~88,
  pp.~62--91, 2017.

\bibitem{BRSY:DecSpaSys}
T.~Brysiewicz, J.~Rodriguez, F.~Sottile, and T.~Yahl, ``{Solving Decomposable
  Sparse Systems},'' {\em arXiv:2001.04228}, 2019.

\bibitem{ThSt}
R.~Steffens and T.~Theobald, ``Mixed volume techniques for embeddings of
  {L}aman graphs,'' {\em Computational Geometry. Theory and Applications},
  vol.~43, no.~2, pp.~84--93, 2010.

\bibitem{PHC_M2}
E.~Gross, S.~Petrovi\'{c}, and J.~Verschelde, ``Interfacing with {PHC}pack,''
  {\em The Journal of Software for Algebra and Geometry}, vol.~5, pp.~20--25,
  2013.

\bibitem{ThomasWebsite}
T.~Brysiewicz, J.~Rodriguez, F.~Sottile, and T.~Yahl, ``Software for
  decomposable sparse polynomial systems,'' 2020.
\newblock https://www.math.tamu.edu/\~{}thomasjyahl/ research/DSS/DSSsite.html.

\end{thebibliography}


\end{document}